%% file: main.tex
\documentclass[
a4paper,
11pt, 
DIV=14,
toc=bibliography, 
headsepline, 
parskip=half-,
abstract=true,
]{scrartcl}

\usepackage[english]{babel}
\usepackage[utf8]{inputenc}
\usepackage[T1]{fontenc}
\usepackage[english,cleanlook]{isodate}
\usepackage{lmodern}
\usepackage[babel=true]{microtype}
\usepackage{amsmath}
\usepackage{mathtools}
\usepackage{amssymb}
\usepackage{amsthm}
\usepackage[markcase=noupper]{scrlayer-scrpage} 
\usepackage{enumitem}
\usepackage{tabularx}
\usepackage{tikz} 
\usepackage{tikz-3dplot} 
\usepackage{mathrsfs}
\setkomafont{pageheadfoot}{\footnotesize}
\automark[section]{section}
\ohead{A.\,Carlotto, M.\,B.\,Schulz, D.\,Wiygul}
\ihead{\headmark}

\providecommand{\R}{\mathbb{R}}
\providecommand{\C}{\mathbb{C}}
\providecommand{\N}{\mathbb{N}}
\providecommand{\Z}{\mathbb{Z}}
\providecommand{\Sp}{\mathbb{S}}
\providecommand{\B}{\mathbb{B}}
\providecommand{\K}{\mathbb{K}}
\providecommand{\tow}{\mathbb{M}}
\providecommand{\towquot}{\widetilde{\tow}}
\providecommand{\towbent}{\widehat{\tow}}
\providecommand{\dih}{\mathbb{D}}
\providecommand{\apr}{\mathbb{A}}
\providecommand{\pri}{\mathbb{P}}
\providecommand{\pyr}{\mathbb{Y}}
\providecommand{\cyc}{\mathbb{Z}}
\providecommand{\st}{\, :\ }
\providecommand{\axis}[1]{\widehat{#1}}
\providecommand{\dist}{d}
\newcommand{\cutoff}[2]{\Psi_{#1,\,#2}}
\providecommand{\cyl}{\Lambda}
\providecommand{\catr}{K}
\providecommand{\towr}{M}
\providecommand{\discr}{B}
\providecommand{\towcokergen}{v_{\tow^+}^{\mathrm{dislocate}}}
\providecommand{\towcoker}{H_{\tow^+}^{\mathrm{dislocate}}}
\providecommand{\cokerval}[1]{H_{\Sigma_{m,#1}}^{\mathrm{dislocate}}}
\providecommand{\coker}{\cokerval{\xi}}

\providecommand{\dvol}[1]{dV_{#1}}
\providecommand{\Riem}{\mathrm{Riem}}
  
\DeclarePairedDelimiter\abs{\lvert}{\rvert}
\DeclarePairedDelimiter\nm{\lVert}{\rVert}
\DeclarePairedDelimiter\sk{\langle}{\rangle}
\DeclarePairedDelimiter\interval{]}{[}
\DeclarePairedDelimiter\Interval{[}{[}
\DeclarePairedDelimiter\intervaL{]}{]}
\DeclarePairedDelimiter\IntervaL{[}{]}
\DeclarePairedDelimiter\set{\{}{\}}

\renewcommand{\Re}{\operatorname{Re}}
\renewcommand{\Im}{\operatorname{Im}}
\DeclareMathOperator{\Arg}{Arg}
\DeclareMathOperator{\Aut}{Aut}
\DeclareMathOperator{\dom}{dom}

\DeclareMathOperator{\tubularexp}{Exp}
\DeclareMathOperator{\Ogroup}{O}
\DeclareMathOperator{\SOgroup}{SO}
\DeclareMathOperator{\genus}{genus}
\DeclareMathOperator{\refl}{\underline{\mathsf{R}}}
\DeclareMathOperator{\rot}{\mathsf{R}}

\DeclareMathOperator{\sgn}{sgn}
\DeclareMathOperator{\spt}{spt}
\DeclareMathOperator{\trans}{\mathsf{T}}
\DeclareMathOperator{\graph}{graph}
\DeclareMathOperator{\area}{area}
\DeclareMathOperator{\End}{End}
\DeclareMathOperator{\morseindex}{index}

\usepackage[initials]{amsrefs}

\usetikzlibrary{calc,intersections,plotmarks,matrix}
\makeatletter
\tikzset{
    scale plot marks/.is choice,
    scale plot marks/true/.style={},	
    scale plot marks/false/.code={
        \def\pgfuseplotmark##1{\pgftransformresetnontranslations\csname pgf@plot@mark@##1\endcsname}
    },
every mark/.append style={scale plot marks=false},
plus/.style={mark=+,mark size=2.25pt},
vdash/.style={mark=|,mark size=2.25pt},
hdash/.style={mark=-,mark size=2.25pt},
bullet/.style={mark=*,mark size=1.125pt},
declare function={acosh(\x)=ln(\x+sqrt(\x*\x-1));},
}
\makeatother
\pgfmathsetmacro{\unitscale}{3.33} 
\pgfmathsetmacro{\fbmsscale}{0.46}

\usepackage[
pdftitle={Infinitely many pairs of free boundary minimal surfaces with the same topology and symmetry group},
pdfauthor={Alessandro Carlotto, Mario B. Schulz and David Wiygul}, 
pdfborder={0 0 0},
]{hyperref}

\usepackage[automake,toc]{glossaries}
\newglossarystyle{lwlong3colheader}{%
\setglossarystyle{long3colheader}%
}
\setglossarystyle{lwlong3colheader}%
\input{glossaryentries}

\makeglossaries
\setglossarypreamble[main]{In the following glossary we have collected a selection of main notational items; for each item we indicate a brief verbal description and the page(s) of principal reference (i.\,e. most often the page where the symbol in question is introduced within the paper).}

\theoremstyle{plain}
\newtheorem{theorem}{Theorem}[section]
\newtheorem{lemma}[theorem]{Lemma}
\newtheorem{corollary}[theorem]{Corollary} 
\newtheorem{proposition}[theorem]{Proposition}
\newtheorem{conjecture}[theorem]{Conjecture}

\theoremstyle{definition}

\theoremstyle{remark}
\newtheorem{remark}[theorem]{Remark}

\numberwithin{equation}{section}

\addtokomafont{title}{}
\title{\vspace*{-3ex}Infinitely many pairs of free boundary minimal surfaces  \\ 
with the same topology and symmetry group}
\author{Alessandro Carlotto, Mario B. Schulz and David Wiygul}

\date{
\vspace*{-3ex}
} 

\newcommand\printaddress{{
\setlength{\parindent}{17pt}
\bigskip
\par
\vbox{
{\scshape Alessandro Carlotto}
\newline ETH D-Math, 
R\"amistrasse 101, 
8092 Z\"urich, 
Switzerland 
\newline 
Universit\`{a} di Trento,
Dipartimento di Matematica,
via Sommarive 14,
38123 Povo di Trento,
Italy
\newline
\textit{E-mail address:} 
\texttt{alessandro.carlotto@math.ethz.ch, alessandro.carlotto@unitn.it}
\par\medskip
{\scshape Mario B. Schulz}
\newline 
University of M\"unster, 
Mathematisches Institut, 
Einsteinstrasse 62,
48149 M\"unster,
Germany
\newline
\textit{E-mail address:} 
\texttt{mario.schulz@uni-muenster.de}
\par\medskip
{\scshape David Wiygul}
\newline ETH D-Math, 
R\"amistrasse 101, 
8092 Z\"urich, 
Switzerland 
\newline 
Universit\`{a} di Trento,
Dipartimento di Matematica,
via Sommarive 14,
38123 Povo di Trento,
Italy
\newline
\textit{E-mail address:} 
\texttt{david.wiygul@math.ethz.ch, davidjames.wiygul@unitn.it}
\par
}
}}

\begin{document}

\maketitle

\begin{abstract} 
The topology and symmetry group of a free boundary minimal surface in the three-dimensional Euclidean unit ball do not determine the surface uniquely. 
We provide pairs of non-isometric free boundary minimal surfaces having any sufficiently large genus $g$, three boundary components and antiprismatic symmetry group of order $4(g+1)$.  
\end{abstract}

\tableofcontents

\section{Introduction}\label{sec:Intro} 

In the last decade, we have witnessed a striking development of the theory of free boundary minimal surfaces, to an extent that cannot be properly accounted for here. We refer the reader to the recent survey \cite{LiSurvey}, to the lecture notes \cite{FraserLectureNotes} and to the introduction of the PhD thesis \cite{FranzThesis} -- among others -- for an overview of some of the most significant advances in the field. Nevertheless, one cannot but note the abundance of open problems, including a few that have proven to be very elusive. In this smaller circle, there remains the question whether it is possible to realize any compact (orientable) surface as a properly embedded, free boundary minimal surface in the Euclidean unit ball (henceforth denoted $\B^3$): some years ago, the first two authors of the present article proved -- in joint work with Franz \cite{CarlottoFranzSchulz20} -- that the answer is affirmative for the infinite subclass of surfaces having connected boundary and any genus. Here, we shall instead be concerned with the related question whether such embeddings -- when they exist -- are in fact \emph{unique} modulo ambient isometry. In contrast with the case of low topological complexity -- as in Nitsche's theorem about discs \cite{Nitsche1985} and for the conjectural uniqueness of the critical catenoid in the class of topological annuli -- we answer here in the strongest negative terms:

\begin{theorem}\label{thm:Main}
For any sufficiently large integer $g$ there exist in the unit ball of Euclidean $\R^3$ two distinct (non-isometric), properly embedded, free boundary minimal surfaces having genus $g$, three boundary components and symmetry group coinciding with the antiprismatic group of order $4(g+1)$.
\end{theorem}

Some comments are appropriate. Firstly, the surfaces mentioned in the statement come in two infinite families, each of them being parametrized by the integer $g$. 
One of the two families of free boundary minimal surfaces in question has been constructed by Kapouleas and Li in \cite{KapouleasLiDiscCCdesing} (and by variational methods in \cite{KetoverFB}, but see Appendix \ref{app:convergence_large_g}), while we shall be concerned here with the construction of a second, new infinite family of surfaces displaying a different asymptotic behavior (as one can detect e.\,g. by looking at the corresponding varifold limit). 
Secondly, the symmetry group of a surface is here defined, in the setting of the theorem, as the subgroup of $\Ogroup(3)$ preserving it as a set; we note that the antiprismatic group $\apr_m$ of order $4m$ is (intrinsically) isomorphic to the more familiar dihedral group $\dih_{2m}$, although their standard actions on $\R^3$ are \emph{not} conjugate, so that our terminology has been chosen to provide a more accurate account of this matter (see Section \ref{sec:Nota} for further details). 

One important motivation for the present work (that lies behind the precise formulation of Theorem~\ref{thm:Main}) is the recent result, obtained by the third-named author and Kapouleas (see \cite{KapouleasWiygulUnique}), asserting the uniqueness of each Lawson surface, in the round three-dimensional sphere, given its topology \emph{and symmetry group}. So, our statement above should indeed be viewed in that perspective, and contrasted with such a theorem in the context of the comparative study of closed minimal surfaces in $\Sp^3$ and of free boundary minimal surfaces in $\B^3$.

We also wish to remark that it is still unclear whether the phenomenon described in Theorem \ref{thm:Main} also happens for complete, embedded, minimal surfaces in $\R^3$ having, say, finite total curvature; indeed, it seems hard to tell whether one can construct a non-isometric twin for each surface belonging to the Costa--Hoffman--Meeks family. 
In this sense, the main result of the present paper suggests some sort of additional \emph{flexibility} of free boundary minimal surfaces in the three-dimensional Euclidean unit ball.

\paragraph{Outline of the main construction.}

These important clarifications being made, let us proceed with a synthetic description of the geometric idea lying behind the construction that is the object of most of the present paper. Such an idea leads to obtaining a genuinely new family of free boundary minimal surfaces, as encoded in the statement of Theorem \ref{thm:ConstructFirst}, by means of a \emph{gluing procedure at the free boundary}.

We set $\B^2 \vcentcolon= \B^3 \cap \{z=0\}$, the closed equatorial disc.
Away from the equatorial circle $\partial \B^2$ the surfaces we construct will approximate the union of $\B^2$ with two catenoidal annuli, mirror images of one another by means of the reflection across $z=0$, each of which has $\partial \B^2$ as a boundary component and along its other component meets $\partial \B^3$ orthogonally. 
In the course of the construction we will need to consider also pairs of catenoidal annuli, henceforth denoted by $\pm\K_b$, close to these two and lying at height $\pm b$ respectively (we refer the reader to Section \ref{sec:Initial} for the precise definition).
We will construct sequences of free boundary minimal surfaces converging to
$\K_0 \cup \B^2 \cup -\K_0$ in two steps.
In the first step we replace $\K_0 \cup \B^2 \cup -\K_0$ by a smooth surface $\Sigma$, properly embedded in $\B^3$, which is nearly minimal, satisfies the free boundary condition, and approximates $\K_0 \cup \B^2 \cup -\K_0$ away from the latter's singular set, the equator.
In the second step $\Sigma$ is perturbed to exact minimality without sacrificing embeddedness or the free boundary condition.
In fact, we will need to construct not just one surface $\Sigma$ as above but a family of such surfaces, each called an \emph{initial surface} for the construction.
The family of initial surfaces will be indexed by the genus of its members and by one real parameter (that is roughly equivalent to the height parameter $b$, so that the preliminary configuration is basically given by the disjoint union $\K_b \cup \B^2 \cup -\K_b$).
The second step, deformation to exact minimality, will be possible only for certain members of the family, as we are about to explain.

First, we will need high genus, because -- as will become apparent from the construction of the initial surfaces and the estimates to follow -- the initial surfaces will be assembled by deforming and gluing together well-understood model surfaces, and only by taking the genus large will we be able to ensure that the initial surfaces well approximate these constituent models and in particular are approximately minimal.
Second, we will need the aforementioned parameter in order to properly handle the approximate cokernel we will encounter, at the linear level, when tackling the problem of perturbation to an exact free boundary minimal surface. The reader is referred to the preamble of Section \ref{sec:Nonlin} for a more detailed, technical account of this issue and an indication of how it is overcome.

We have already introduced the model surfaces $\B^2$ and $\K_b$, but are still missing
the other family of models that will be needed to glue the former together to produce a smooth surface.
Specifically, each initial surface will be obtained from some $\K_b \cup \B^2 \cup -\K_b$ by cutting out a neighborhood of the equator and smoothly gluing in -- essentially by an accurate interpolation procedure -- a suitably truncated and deformed surface from a subfamily of the singly periodic minimal surfaces having $2k$ asymptotic half planes that Karcher presented in \cite{KarcherScherk} as generalizations of Scherk's classical $k=2$ example in \cite{Scherk}. We will refer to any member of this family (with $k \geq 2$) as a \emph{Karcher--Scherk tower}.

The first gluing constructions utilizing the $k=2$ Karcher--Scherk towers were performed, independently, by Kapouleas in \cite{KapouleasEuclideanDesing} and by Traizet in \cite{TraizetPlanes}.
In this article we follow the gluing methodology of Kapouleas, which originated in \cite{KapouleasDelaunay};
in particular we apply the framework developed in \cite{KapouleasWenteTori} to manage the approximate cokernel confronted in the linearized problem (see also \cite{KapouleasClay} for a pedagogical overview and further references).
Higher-order ($k \geq 3$) Karcher--Scherk towers have featured in the gluing constructions (by differing techniques) \cites{KapouleasWiygulTordesing, HellerHellerTraizetLawsonArea, ChenTraizetKSI}.
For what specifically pertains to free boundary minimal surfaces in $\B^3$, Kapouleas and Li carried out the first construction via $k=2$ Karcher--Scherk towers in \cite{KapouleasLiDiscCCdesing}; the singular locus they need to deal with is the intersection of the horizontal disc $\B^2$ and the standard critical catenoid, so it lies at positive distance from the boundary of $\B^3$.
More recently, in fact very recently, Kapouleas and Zou have performed in \cite{KapouleasZouCloseToBdy} a construction of genus zero free boundary minimal surfaces in $\B^3$ using $k=2$ Karcher--Scherk towers at the boundary.
It is also appropriate to mention that Kapouleas and Li have further proposed
(as described in Section 3.3 of \cite{LiSurvey})
constructions of free boundary minimal surfaces
in $\B^3$ by desingularizing maximally symmetric unions of $k \geq 2$ free boundary discs
intersecting along a diameter of the ball, using maximally symmetric Karcher--Scherk towers.
That being said, the present article employs higher-order Karcher--Scherk towers in the free boundary setting for the first time and is also the first desingularization construction to apply the Kapouleas approach to $k \geq 3$ Karcher--Scherk towers having wings that do not contain lines of symmetry (unlike the examples in \cite{KapouleasWiygulTordesing}).

\begin{figure}%
\includegraphics[width=\fbmsscale\textwidth,page=1]{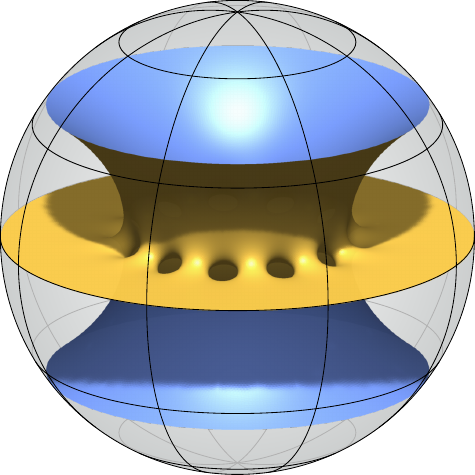}\hfill
\includegraphics[width=\fbmsscale\textwidth,page=2]{figures-fmbs}%
\caption{A pair of free boundary minimal surfaces with the same topology and symmetry group, here depicted in the case of genus $g=11$.}%
\label{fig:genus11}%
\end{figure}

\paragraph{Contents and structure of the paper.}

Besides Section \ref{sec:Nota}, that is devoted to defining the notation employed in the rest of the paper, the proof of Theorem \ref{thm:ConstructFirst} (which, in view of the above discussion, implies at once Theorem \ref{thm:Main}), is developed in Section \ref{sec:Initial} for what concerns the construction of the initial surfaces, in Section \ref{sec:Linear} for what concerns the linearized perturbation theory, and finally in Section \ref{sec:Nonlin} for the nonlinear iteration scheme which completes the construction in question.  

In fact, our work allows us -- with relatively straightforward modifications -- to obtain yet another infinite family of free boundary minimal surfaces in $\B^3$, again parametrizable by an integer $n\geq n_0$ but this time converging in the sense of varifolds (as one lets $n\to\infty$) to the union $\K_0\cup-\K_0$. Such surfaces have genus \emph{zero}, exactly $n+2$ boundary components and symmetry group coinciding with the prismatic group of order $4n$; see the statement of Theorem \ref{thm:ConstructSecond}. The corresponding proof is outlined in Section \ref{sec:OnlyCatsDesing}; here we shall limit ourselves to note how the $k=3$ Karcher--Scherk towers need to be replaced by the more familiar towers corresponding to $k=2$ (also known just as \emph{singly periodic Scherk surfaces} in $\R^3$). The surfaces we provide there should then be compared e.\,g. to those constructed in by Folha--Pacard--Zolotareva in \cite{FolPacZol17}, that converge in the sense of varifolds to a double equatorial disc, and to those recently devised by Kapouleas and Zou, as we just explained.

Setting aside for a moment Section \ref{sec:Index}, which deserves a more detailed presentation (purposefully postponed to the very end of this introduction), the core of the paper is then complemented by four important appendices, whose contents we are about to describe. 

Firstly, Appendix \ref{app:Karcher--Scherk} contains some ancillary results -- of independent interest -- about the geometry of the Karcher--Scherk towers: in particular, we give a detailed proof, appealing solely to the Enneper--Weierstrass representation, of the (proper) embeddedness of these minimal surfaces, whose fine properties are then carefully investigated. The reader is referred to Proposition \ref{prop:towerEWsummary} and to (the resulting) Proposition \ref{karcher-scherk} for more accurate statements, as well as for a discussion about the novel aspects of what we obtain here. 

Secondly, Appendix \ref{app:CylAnalysis} collects a series of results concerning the solvability (in weighted H\"older spaces) of linear elliptic problems on asymptotically cylindrical ends (such as the actual \emph{wings} of the Karcher--Scherk towers when quotiented by the action of period translations along the axis of periodicity). Although probably well-known to the experts, these results are typically scattered through the literature and often presented with rather non-self-contained proofs (or frequently in a different functional setting, that is not what we need here). These ancillary results are then repeatedly employed in the arguments we present to prove Theorem \ref{thm:ConstructFirst}, in particular for the linear analysis we perform in Section \ref{sec:Linear}.

Thirdly, Appendix \ref{app:graphs} is devoted some preparatory results and fairly technical material pertaining the deformation of immersions through normal graphs, possibly dealing with situations when different Riemannian metrics are actually into play. One (most important) reason for such a discussion is the fact that the perturbations of the aforementioned initial surfaces are conveniently phrased with respect to an auxiliary metric (that is \emph{not} the Euclidean one), so to avoid obvious well-posedness issues when working at boundary points. Those aspects are described in more detail at the beginning of Section \ref{sec:Nonlin}.

The motivation behind the content of Appendix \ref{app:convergence_large_g} is instead as follows. 
As already mentioned in the beginning of this introduction, Ketover \cite{KetoverFB} used equivariant min-max theory to construct a family of free boundary minimal surfaces in $\B^3$ whose varifold limit he claimed to be equal to the union of the equatorial disc $\B^2$ and the critical catenoid. 
In the proof of this statement however (cf. section~4.3 of \cite{KetoverFB}) the possibility that the asymptotic behavior is the same as for our surfaces, namely convergence to $\K_0\cup\B^2\cup -\K_0$, is actually not excluded.  
The scope of our Appendix \ref{app:convergence_large_g} is to identify the argument which seems to be missing in \cite{KetoverFB} and to fill that gap; 
in the same context, this also gives a complete proof that such a variational construction also produces free boundary minimal surfaces with exactly three boundary components and any sufficiently large genus. We wish to stress that our argument relates to the analysis of the catenoidal annuli in Section \ref{subsec:building_blocks} of the present article (cf. Corollary \ref{cor:area_K}). 
We further note, parenthetically, that it is in fact reasonable to conjecture that the surfaces produced by Ketover should actually coincide with those constructed via gluing-desingularization methods by Kapouleas--Li, cf. Open Question 4 of \cite{LiSurvey}.

Let us then conclude by describing the content of Section \ref{sec:Index}. There, we collect some results, data and conjectures about the Morse index of the free boundary minimal surfaces we produce. To begin with, we prove in Proposition \ref{prop:DistiguishByEquivIndex} that one can distinguish the surfaces in our family from Ketover's (and thus, conjecturally, from the elements of the Kapouleas--Li family) by their \emph{equivariant} Morse index as defined in \cite{FranzIndex}: while the Ketover minimal surfaces  have equivariant index equal to one, we prove that the elements in our surfaces have equivariant index \emph{at least two}. This is a fascinating result, which indicates (among other things) that the surfaces in our family \emph{cannot} possibly be obtained by means of a one-parameter min-max scheme, but would rather need the use of $p$-sweepouts for some $p\geq 2$ (with numerical evidence indicating that in fact $p=2$), modulo the very delicate problem of fully controlling the topology in the procedure.
We then proceed and present some data, based on numerical simulations, about what we expect to be the values of the (standard, i.\,e. non-equivariant) Morse index of both the Kapouleas--Li surfaces as well as ours. Very surprisingly, while the former values lie on an affine line (in analogy with what was proven by the third-named author and Kapouleas in \cite{KapouleasWiygulIndex} for the Lawson surfaces $\xi_{g,1}$, $g\geq 2$), the latter ones display a seemingly mysterious (non-periodic) pattern, which (to the best of our knowledge) had not been previously observed for minimal surfaces in \emph{any} framework. Such data motivate a series of open questions and conjectures that, we hope, have the potential of generating further advances in the field.

%===== THANKS & FUNDING =====================================
\paragraph{Acknowledgments.} 

This project has received funding from the European Research Council (ERC) under the European Union’s Horizon 2020 research and innovation programme (grant agreement No. 947923). 
The research of M.\,S. was partly funded by the EPSRC grant EP/S012907/1 
and by the Deutsche Forschungsgemeinschaft (DFG, German Research Foundation) under Germany's Excellence Strategy EXC 2044 -- 390685587, Mathematics M\"unster: Dynamics--Geometry--Structure, and the Collaborative Research Centre CRC 1442, Geometry: Deformations and Rigidity.

\section{Notation and preliminaries}\label{sec:Nota}

\paragraph{Basic notation for $\R^3$ and distinguished subsets.}
We set
\label{gls:ball_etc}
\begin{align*}
\glsuseri{ball}&\vcentcolon=\{(x,y,z)\in\R^3\st x^2+y^2+z^2\leq 1\}, & 
\glsuserii{ball}&\vcentcolon=\partial\B^3, \\
\glsuseriii{ball}&\vcentcolon=\B^3\cap\{z=0\}, &
\glsuseriv{ball}&\vcentcolon=\partial\B^2=\Sp^2\cap\{z=0\}.
\end{align*}
Given vectors $U, V \in \R^3$, we write $U \cdot V$ for their standard, Euclidean inner product.
For Cartesian coordinates $(x, y, z)$ on $\R^3$
we denote by $\gls{axis}$ \label{gls:axis}
the corresponding positively oriented coordinate axes.
Given a set $S \subset \R^3$,
we define the distance function
$d_S\colon\R^3 \to \Interval{0,\infty}$ by
\begin{equation}\label{eqn:definition_distancefunction}
\dist_S(x)\vcentcolon=\inf_{y \in S} \abs{x-y}.
\end{equation}
Given additionally a real number $s \geq 0$
and $\sim$ a binary relation on $\R$,
we define
\begin{equation}\label{tubularnbhd}
\glsuseri{tube} %S_{\sim s}
\vcentcolon=\bigl\{x \in \R^3 \st \inf_{y \in S} \abs{x-y} \sim s\bigr\}=\{\dist_S \sim s\},
\end{equation}
so that for example
$\axis{z}_{\geq R}
  =\{x=y=0\}_{\geq R}
  =\{(x,y,z)\in \R^3 \st x^2+y^2 \geq R^2\}
$.
Finally, if $X\colon S \to \R^3$ is a vector field on $S$,
we define the graph of $X$ over $S$
\begin{equation}
\label{eqn:definition_graph}
\graph(X)\vcentcolon=\{x+X(x) \st x \in S\}.
\end{equation}
Of course, in general $\graph(X)$ is nothing more than a set (i.\,e. it does not naturally come with additional structures, such as that of submanifold) unless we place additional assumptions on the set $S$ and the vector field $X$.

\paragraph{Homotheties and isometries of $\R^3$.}
For any set $S \subset \R^3$ and any $\lambda \in \R$
we define
$\lambda S\vcentcolon=\{(\lambda x, \lambda y, \lambda z) \st (x,y,z) \in S\}$;
for example $\lambda \B^3$ is the closed origin-centered ball of radius $\abs{\lambda}$.
For any affine subspace $V$ of $\R^3$ we write
\glsuseri{refl}\label{gls:refl_etc}
for reflection through $V$.
For any directed line $\hat{\ell}$ in $\R^3$
and parameter $t \in \R$
we write \glsuseriii{refl}
for translation by $t$
in the positive direction of $\hat{\ell}$
and \glsuserii{refl}
for counterclockwise rotation about $\hat{\ell}$.
Thus, for example, $\rot_{\axis{z}}^\phi$
stands for the rotation in $\R^3$ about the $z$-axis through angle $\phi$
in the conventional positive sense,
and we have
\[
\refl_{\axis{z}}
=\rot_{\axis{z}}^\pi
=\rot_{\axis{z}}^{-\pi}.
\]
Given a Riemannian manifold $M$
(possibly with boundary and having unnamed but understood metric)
and $S \subseteq M$,
we write \label{gls:aut}\gls{Aut} for the group of all isometries of $M$
that take $S$ to itself (globally, i.\,e. as a set). In this case we will sometimes say that $M$ is invariant under any such isometry.
Thus $\Aut_{\R^3}(S)$, with $S \subseteq \R^3$,
consists of all rigid motions  $\mathsf{M}$ of $\R^3$
(not necessarily preserving orientation)
such that $\mathsf{M}S=S$,
while $\Aut_{\B^3}(S)$, now with $S \subseteq \B^3$,
consists of all $\mathsf{R} \in\Ogroup(3)$
such that $\mathsf{R}S=S$. In particular, the symmetry group of a properly embedded surface $S\subseteq\B^3$ is precisely defined as $\Aut_{\B^3}(S)$.

For any group $G$
and elements
$g_1, \ldots, g_n \in G$
we write 
$\sk{g_1, \ldots, g_n}$
for the subgroup of $G$
generated by $g_1, \ldots, g_n$. 
Throughout this article
we shall conveniently identify $\Ogroup(2)$ with the subgroup of $\Ogroup(3)$ defined by
\begin{equation*}
\Ogroup(2) \vcentcolon= \Aut_{\B^3}(\{z \geq 0\} \cap \B^3),
\end{equation*}
and, for any integer $m \geq 2$,
we define the following concrete subgroups of $\Ogroup(3)$ (see Figures \ref{fig:symmetries1}--\ref{fig:symmetries2}):
\begin{align}
\label{cycdef}
 \text{cyclic group }&&
\cyc_m&\vcentcolon=
\sk[\Big]{\rot_{\axis{z}}^{2\pi/m}} < \Ogroup(2) 
&&\text{ of order $m$,}
\\
\label{dihdef}
\text{dihedral group }&&
\dih_m&\vcentcolon=
\sk[\Big]{\rot_{\axis{x}}^\pi,~
    \rot_{\axis{z}}^{2\pi/m}} < \SOgroup(3) 
&&\text{ of order $2m$,}
\\
\label{pyrdef}
\text{pyramidal group }&&
\glsuseriii{sym}&\vcentcolon=
\sk[\Big]{\rot_{\axis{z}}^{2\pi/m},~\refl_{\{y=x \tan(\pi/(2m))\}}} < \Ogroup(2) 
&&\text{ of order $2m$,}
\\
\label{aprdef}
\text{antiprismatic group }&&
\glsuseri{sym}&\vcentcolon=
\sk[\Big]{\rot_{\axis{x}}^\pi,~
    \refl_{\{y=x \tan(\pi/(2m))\}}} < \Ogroup(3)  
&&\text{ of order $4m$,}
\\
\text{prismatic group }&&
\glsuserii{sym}&\vcentcolon=\sk[\Big]{\rot_{\axis{x}}^\pi,~\rot_{\axis{z}}^{2\pi/m},~\refl_{\{y=0\}}} < \Ogroup(3)  
&&\text{ of order $4m$.}
\end{align}
Setting
\[ 
Q_m\vcentcolon=\Bigl\{\Bigl(
  \cos(j\tfrac{2\pi}{m}+\tfrac{\pi}{2m}\bigr),~
  \sin\bigl(j\tfrac{2\pi}{m}+\tfrac{\pi}{2m}\bigr),~
  0\Bigr)\st j\in\{0,\ldots,m-1\}\Bigr\},
\]
we can interpret
$\dih_m$ as the intersection
$
\SOgroup(3)
\cap
\Aut_{\R^3}(\rot_{\axis{z}}^{-\pi/(2m)}Q_m)
$
of $\SOgroup(3)$ with the symmetry group
of the regular $m$-gon inscribed in $\Sp^1$
and having $(1,0,0)$ as a vertex,
$\pyr_m$ as the (full $\Ogroup(3)$) symmetry group
of the pyramid having vertices $Q_m \cup \{(0,0,5)\}$,
and $\apr_m$ as the (full $\Ogroup(3)$) symmetry group
of the antiprism having vertices
$\trans^{\axis{z}}_5 Q_m\cup\rot_{\axis{x}}^\pi\trans^{\axis{z}}_5 Q_m$.
Note that $\dih_m < \apr_m$ and $\pyr_m < \apr_m$  (see Figure~\ref{fig:symmetries2}). 
However, it is to be remarked -- as a standard fact in basic group theory -- that there exist isomorphisms $\pyr_m \cong \dih_m$ and $\apr_m \cong \dih_{2m}$ as well as $\pri_m\cong\dih_m\times\cyc_2$.

\begin{figure}%
\pgfmathsetmacro{\thetaO}{72}
\pgfmathsetmacro{\phiO}{108}
\tdplotsetmaincoords{\thetaO}{\phiO}
\pgfmathsetmacro{\hoehe}{3/4}
\begin{tikzpicture}[scale=\unitscale,tdplot_main_coords,line cap=round,line join=round,thick,baseline={(0,0,0)}] 
\foreach\k in {0,...,5}{
\coordinate (T\k) at ({cos((\k-1/8)*360/5)},{sin((\k-1/8)*360/5)},{\hoehe});
\coordinate (B\k) at ({cos((\k-1/8)*360/5)},{sin((\k-1/8)*360/5)},{-\hoehe});
}
\foreach[count=\i]\k in {0,...,4}{
\path(T\k)--(T\i)coordinate[pos=0.5](t\i);
\path(B\k)--(B\i)coordinate[pos=0.5](b\i); 
}
\tdplotsetthetaplanecoords{(5-1/8)*360/5}
\filldraw[fill=black!30](B1)--(B2)--(B3)--(B4)--(B5)--cycle;
\begin{scope}[fill=black!20]
\filldraw(T1)--(T2)--(B2)--(B1)--cycle;
\filldraw(T4)--(T5)--(B5)--(B4)--cycle;
\filldraw(T5)--(T1)--(B1)--(B5)--cycle;
\end{scope}
\filldraw[fill=black!10](T1)--(T2)--(T3)--(T4)--(T5)--cycle;
\begin{scope}[dotted]
\draw(T3)--(T4)--(B4)--(B3)--cycle (B3)--(B2);
\end{scope}
\draw[dashed,color={cmyk,1:magenta,0.5;yellow,1}](T5)--(t3)--(b3)--(B5);
\draw(T1)--(T2)--(T3)--(T4)--(T5)--cycle;
\path(T5)--(B5)coordinate[midway](x);
\path(B3)--(T2)coordinate[midway](-x);
\draw[dashed](0,0,0)--(-x)(0,0,0)--(x)coordinate[pos=1.4](x->);
\draw[-latex](x)--(x->)node[left,inner sep=1pt]{$\widehat{x}$};
\draw plot[bullet](x);
\draw plot[bullet](-x);
\draw[dashed](0,0,0)--(0,0,-\hoehe)coordinate(-z)(0,0,0)--(0,0,\hoehe)coordinate(z);
\draw[-latex](z)--++(0,0,0.2)node[right=1.3pt,circle,fill=black!10,inner sep=0.2pt]{$\widehat{z}$};
\draw plot[bullet](z);
\draw plot[bullet](-z);
\end{tikzpicture}
%%%%%%%%%%%%%%
\hfill
\begin{tikzpicture}[baseline={(0,0)}]
\matrix[matrix of math nodes,row sep=3ex]
{
|(P)|\pri_m &   & |(A)|\apr_m \\
& |(Z)|\cyc_m & \\
};
\path(P)--node[midway,sloped]{$>$}(Z);
\path(Z)--node[midway,sloped]{$<$}(A);
\end{tikzpicture}
\hfill
%%%%%%%%%%%%%%
\begin{tikzpicture}[scale=\unitscale,tdplot_main_coords,line cap=round,line join=round,thick,baseline={(0,0,0)}] 
\foreach\k in {0,...,5}{
\coordinate (T\k) at ({cos((\k-1/4)*360/5)},{sin((\k-1/4)*360/5)},{\hoehe});
\coordinate (B\k) at ({cos((\k+1/4)*360/5)},{sin((\k+1/4)*360/5)},{-\hoehe});
}
\foreach[count=\i]\k in {0,...,4}{
\path(T\k)--(T\i)coordinate[midway](t\i);
\path(B\k)--(B\i)coordinate[midway](b\i); 
}
\filldraw[fill=black!30](B1)--(B2)--(B3)--(B4)--(B5)--cycle;
\begin{scope}[fill=black!20]
\filldraw(T1)--(B1)--(B5)--cycle;
\filldraw(T1)--(T2)--(B1)--cycle;
\filldraw(T4)--(T5)--(B4)--cycle;
\filldraw(T5)--(B5)--(B4)--cycle;
\filldraw(T5)--(T1)--(B5)--cycle;
\end{scope}
\filldraw[fill=black!10](T1)--(T2)--(T3)--(T4)--(T5)--cycle;
\begin{scope}[dotted]
\draw(T2)--(B2)--(B1)--cycle;
\draw(T3)--(B3)--(B2)--cycle;
\draw(T4)--(B4)--(B3)--cycle;
\end{scope}
\draw[dashed,color={cmyk,1:magenta,0.5;yellow,1}](T5)--(t3)--(B2)--(b5)--cycle;
\path(T5)--(b5)coordinate[midway](p)(t3)--(B2)coordinate[midway](-p);
\path(T5)--(B4)coordinate[midway](x)(B2)--(T2)coordinate[midway](-x);
\draw[dashed](0,0,0)--(x)coordinate[pos=1.4](x->);
\draw[dashed,color={cmyk,1:magenta,0.5;yellow,1}](0,0,0)--(-p);
\draw[dashed](0,0,0)--(0,0,-\hoehe)coordinate(-z)(0,0,0)--(0,0,\hoehe)coordinate(z);
\draw[dashed](0,0,0)--(-x);
\draw[dashed,color={cmyk,1:magenta,0.5;yellow,1}](0,0,0)--(p);
\draw(B5)--(T5)--(B4)--cycle(B5)--(T1)(T1)--(T2)--(T3)--(T4)--(T5)--cycle;
\draw[-latex](x)--(x->)node[below]{$\widehat{x}$};
\draw plot[bullet](x);
\draw plot[bullet](-x);
\draw[-latex](z)--++(0,0,0.2)node[right=1.3pt,circle,fill=black!10,inner sep=0.2pt]{$\widehat{z}$};
\draw plot[bullet](z);
\draw plot[bullet](-z);
\pgfresetboundingbox
\path(T1)--(T2)--(T3)--(T4)--(T5)--cycle(B1)--(B2)--(B3)--(B4)--(B5)--cycle;
\end{tikzpicture}%
\caption{Difference between prismatic (left) and antiprismatic (right) symmetry in the case $m=5$. 
\\[1ex]
\textbf{(Anti-) Prismatic group} of order $4m$. 
Both the prism and the antiprism (unicolored) are symmetric with respect to the rotation $\rot_{\axis{x}}^\pi$ by angle $\pi$ around the axis $\widehat{x}$ and with respect to reflection across any vertical plane containing the axis $\widehat{z}$ and one of the vertices. 
In the case of the prism, the axis $\widehat{x}$ is also \emph{contained} in such a plane of symmetry, but in the case of the antiprism $\widehat{x}$ is \emph{in between} two such planes. 
}%
\label{fig:symmetries1}%
\bigskip
\bigskip
\begin{tikzpicture}[scale=\unitscale,tdplot_main_coords,line cap=round,line join=round,thick,baseline={(0,0,0)}] 
\foreach\k in {0,...,5}{
\coordinate (T\k) at ({cos((\k-1/4)*360/5)},{sin((\k-1/4)*360/5)},{\hoehe});
\coordinate (B\k) at ({cos((\k+1/4)*360/5)},{sin((\k+1/4)*360/5)},{-\hoehe});
}
\foreach[count=\i]\k in {0,...,4}{
\path(T\k)--(T\i)coordinate[midway](t\i);
\path(B\k)--(B\i)coordinate[midway](b\i); 
}
\filldraw[fill=black!30](B1)--(B2)--(B3)--(B4)--(B5)--cycle;
\begin{scope}[fill=black!20]
\fill(T1)--(B1)--(B5)--cycle;
\fill(T1)--(T2)--(B1)--cycle;
\fill(T4)--(T5)--(B4)--cycle;
\fill(T5)--(B5)--(B4)--cycle;
\fill(T5)--(T1)--(B5)--cycle;
\end{scope}
\begin{scope}[fill=black!60]
\fill(t1)--(T5)--(B5)--cycle;
\fill(t2)--(T1)--(B1)--cycle;
\fill(t5)--(T4)--(B4)--cycle;
\fill(T1)--(b1)--(B1)--cycle;
\fill(T5)--(b5)--(B5)--cycle;
\end{scope}
\draw(T1)--(B1)--(B5)--cycle;
\draw(T1)--(T2)--(B1)--cycle;
\draw(T4)--(T5)--(B4)--cycle;
\draw(T5)--(B5)--(B4)--cycle;
\draw(T5)--(T1)--(B5)--cycle;
\filldraw[fill=black!10] (T1)--(T2)--(T3)--(T4)--(T5)--cycle;
\begin{scope}[dotted]
\draw(T2)--(B2)--(B1)--cycle;
\draw(T3)--(B3)--(B2)--cycle;
\draw(T4)--(B4)--(B3)--cycle;
\end{scope} 
\end{tikzpicture} 
%%%%%%%%%%%%%%
\hfill
\begin{tikzpicture}[baseline={(0,0)}]
\matrix[matrix of math nodes,row sep=3ex]
{&|(A)|\apr_{m} & \\
|(D)|\dih_{m} && |(Y)|\pyr_{m} \\
& |(Z)|\cyc_{m} & \\
};
\path(D)--node[midway,sloped]{$\cong$}(Y);
\path(D)--node[midway,sloped]{$<$}(A);
\path(Y)--node[midway,sloped]{$>$}(A);
\path(D)--node[midway,sloped]{$>$}(Z);
\path(Y)--node[midway,sloped]{$<$}(Z);
\end{tikzpicture}
\hfill
%%%%%%%%%%%%%%
\begin{tikzpicture}[scale=\unitscale,tdplot_main_coords,line cap=round,line join=round,thick,baseline={(0,0,0)}] 
\foreach\k in {0,...,5}{
\coordinate (T\k) at ({cos((\k-1/4)*360/5)},{sin((\k-1/4)*360/5)},{\hoehe});
\coordinate (B\k) at ({cos((\k+1/4)*360/5)},{sin((\k+1/4)*360/5)},{-\hoehe});
}
\foreach[count=\i]\k in {0,...,4}{
\path(T\k)--(T\i)coordinate[midway](t\i);
\path(B\k)--(B\i)coordinate[midway](b\i); 
}
\filldraw[fill=black!30](B1)--(B2)--(B3)--(B4)--(B5)--cycle;
\begin{scope}[fill=black!20]
\filldraw(T1)--(B1)--(B5)--cycle;
\filldraw(T1)--(T2)--(B1)--cycle;
\filldraw(T4)--(T5)--(B4)--cycle;
\filldraw(T5)--(B5)--(B4)--cycle;
\filldraw(T5)--(T1)--(B5)--cycle;
\end{scope}
\begin{scope}[fill=black!60]
\filldraw(T1)--(T2)--(B1)--cycle;
\filldraw(T4)--(T5)--(B4)--cycle;
\filldraw(T5)--(T1)--(B5)--cycle;
\end{scope}
\filldraw[fill=black!10](T1)--(T2)--(T3)--(T4)--(T5)--cycle;
\begin{scope}[dotted]
\draw(T2)--(B2)--(B1)--cycle;
\draw(T3)--(B3)--(B2)--cycle;
\draw(T4)--(B4)--(B3)--cycle;
\end{scope}
\end{tikzpicture}%
\caption{Difference between dihedral (left) and pyramidal (right) symmetry in the case $m=5$. 
\\[1ex]
\textbf{Dihedral subgroup} $\dih_m$ of order $2m$. 
Recoloring half of each triangle of an antiprism as shown on the left suppresses the reflection symmetries but preserves the $\rot_{\axis{x}}^\pi$-invariance. 
\\[1ex]
\textbf{Pyramidal subgroup} $\pyr_m$ of order $2m$.  
Recoloring every triangle adjacent to the top $m$-gonal face of an antiprism preserves the reflection symmetries but suppresses  the $\rot_{\axis{x}}^\pi$-invariance. 
}%
\label{fig:symmetries2}%
\end{figure}

\begin{remark}
The literature employs
multiple common systems of notation
(with no clear single standard)
to refer to finite subgroups
of $\Ogroup(3)$.
In Table \ref{table:group-notation}
we note some
alternative names,
according to three different schemes,
for the three groups
which play the most prominent roles
in the sequel of this article.
\begin{table} 
\centering\renewcommand{\arraystretch}{1.2}%
\begin{tabular}{
>{\centering\arraybackslash}p{0.096\textwidth-2\tabcolsep}|
>{\centering\arraybackslash}p{0.3\textwidth-2\tabcolsep}|
>{\centering\arraybackslash}p{0.3\textwidth-2\tabcolsep}|
>{\centering\arraybackslash}p{0.3\textwidth-2\tabcolsep}}
& {Schoenflies} & {Coxeter} & {Conway--Thurston} \\
\hline
$\pyr_m$ & $C_{mv}$ & $[m]$ & $(*mm)$ \\
\hline
$\pri_m$ & $D_{mh}$ & $[m,2]$ & $(*22m)$ \\
\hline
$\apr_m$ & $D_{md}$ & $[2m,2^+]$ & $(2*m)$ \\
%\hline
\end{tabular}
\caption{Alternative notation for the pyramidal, prismatic, and antiprismatic groups.}
\label{table:group-notation}
\end{table}
\end{remark}

\paragraph{Tubular coordinates.}
Let $(M,g)$ be a complete smooth Riemannian manifold, and
$\Sigma \subset M$ be
a two-sided embedded smooth hypersurface. (In fact, throughout this article, we agree all hypersurfaces to be smooth unless otherwise explicitly stated.)
Denoting by
$\nu$ a unit normal to $\Sigma$, 
we define the map
\begin{equation}\label{tubularcoordscodim1}
\begin{aligned}
\tubularexp_{(M,g),(\Sigma,\nu)} \colon \Sigma \times \R &\to M \\
(p,t) &\mapsto \exp^{(M,g)}_p t\nu(p),
\end{aligned}
\end{equation}
where $\exp^{(M,g)}\colon TM \to M$
is the exponential map on $(M,g)$.
Observe that if $(M,g)$
is Euclidean $\R^3$
and $u$ is a function on $\Sigma$,
then the image of $p \mapsto \tubularexp_{(M,g),(\Sigma,\nu)}(p,u(p))$
is precisely what we had previously denoted
$\graph(u\nu)$.
If $\Gamma \subset \Sigma$ 
satisfies the above assumptions with
$(\Sigma,\iota^*g, \Gamma, \eta)$
in place of $(M, g, \Sigma, \nu)$,
with $\iota\colon \Sigma \to M$ the inclusion map,
then we further define
\begin{equation}\label{tubularcoordscodim2}
\begin{aligned}
\tubularexp_{(M,g),(\Sigma,\nu),(\Gamma,\eta)}
 \colon  \Gamma \times \R \times \R &\to M \\
(p,s,t) &\mapsto \tubularexp_{(M,g),(\Sigma,\nu)}
  (\tubularexp_{(\Sigma, \iota^*g), (\Gamma, \eta)}(p,s),t).
\end{aligned}
\end{equation}
Note that in situations where completeness does not hold
the above maps may still be defined but with restricted domains.

\paragraph {Mean curvature and Jacobi operators.}
Let $M$ be a Riemannian manifold
(as above: possibly with boundary and having unnamed but understood metric), which for our purposes we can assume to have dimension three, and let $\Sigma$ be a properly embedded, two-sided, surface in $M$. We further stipulate that the boundary of $\Sigma$, if not empty, is contained in the boundary of the ambient manifold $M$ (namely: $\partial \Sigma \subset \partial M$).
 We agree to \emph{choose a side} of $\Sigma$ i.\,e. to perform a choice of a unit normal vector field
 \label{gls:nu-A}\glsuseri{nu}. 
 Hence, there is a well-defined notion of \emph{mean curvature} of $\Sigma$, which we shall denote by 
\glsuseriii{nu}, 
that is the trace (with respect to the background Riemannian metric in question) of the \emph{(scalar-valued) second fundamental form} 
\glsuserii{nu};  
we adopt the convention that it equals the first fundamental form in the case of $\Sp^2$ the unit sphere of Euclidean $\R^3$. 
That said -- as it is customary -- we call \emph{minimal} those surfaces for which the mean curvature function vanishes at all points. 

Given any surface as above (not necessarily minimal) one defines the so-called \emph{Jacobi operator}
\begin{equation}\label{eq:JacobiOp}
\gls{Jacobi} u \vcentcolon= \Delta_\Sigma u+(\abs{A_\Sigma}^2+\text{Ric}(\nu_{\Sigma},\nu_{\Sigma}))u
\end{equation}
where $\Delta_\Sigma$ denotes the Laplace--Beltrami operator of $\Sigma$, and $\text{Ric}(\cdot,\cdot)$ is the \emph{Ricci tensor} of the ambient Riemannian manifold.
As it is well-known, this operator relates to the second variation formula for the area functional (see below); equivalently, and significantly for our scopes, one has the pointwise equation
\begin{equation}\label{eq:PointwiseJacobi}
    \left[\frac{dH(t)}{dt}\right]_{t=0}=-J_{\Sigma} u
\end{equation}
where $H(t)=H_{\Sigma(t)}$ denotes the mean curvature of the surface $\Sigma(t)$ obtained by flowing $\Sigma$ according to the vector field $u\nu_{\Sigma}$ for (suitably small) time $t$. With slight abuse of language, we shall refer to any element in the kernel of $J_{\Sigma}$ as a Jacobi field of $\Sigma$ (although that is rather the projection along the normal $\nu_{\Sigma}$ of an actual vector field, which is in fact not uniquely determined).

\paragraph{Robin boundary operators.}\label{gls:conormal-Robin}
In the setting above, suppose further that $\Sigma$ has boundary $\partial\Sigma$ with outward unit conormal
\glsuseri{conormal} along $\partial \Sigma$,
and that $\eta^{\vphantom{|}}_\Sigma$ is everywhere orthogonal to $\partial M$.
Consistently with the convention we just stipulated, let $A_{\partial M}$ be the second fundamental form of $\partial M$.
Let further \glsuseri{nu} be a 
choice of unit normal vector field on $\Sigma$.
Then $A_{\partial M}(\nu^{\vphantom{|}}_\Sigma,\nu^{\vphantom{|}}_\Sigma)$
is independent of the choice of $\nu^{\vphantom{|}}_\Sigma$ and smooth.
Writing $\cdot$ for the ambient inner product,
we define the first-order differential operator
\begin{equation}
\label{Robin_op_general_def}
\glsuserii{conormal}
\vcentcolon=
\eta^{\vphantom{|}}_\Sigma \cdot \nabla_{\Sigma} 
  -A_{\partial M}(\nu^{\vphantom{|}}_\Sigma,\nu^{\vphantom{|}}_\Sigma)
\end{equation}
taking differentiable functions on a neighborhood, in $\Sigma$, of $\partial \Sigma$
to functions on $\partial \Sigma$.
In the special case that $M=\B^3$
we have $B^{\mathrm{Robin}}_\Sigma = \eta^{\vphantom{|}}_\Sigma \cdot \nabla_{\Sigma} - 1$.

\paragraph {Jacobi quadratic form and Morse index.} 

Still in the setting of the previous two paragraphs, 
we may then consider the quadratic form, henceforth named \emph{Jacobi quadratic form} (or, sometimes, just \emph{index form})
\begin{align}\label{eq:JacobiQuadraticForm}
\begin{aligned}
\glsuseri{indexform}(u,u)
&=\int_{\Sigma}\Bigl(\abs{\nabla_{\Sigma} u}^2-\bigl(\abs{A_\Sigma}^2+\text{Ric}(\nu_{\Sigma},\nu_{\Sigma})\bigr)u^2\Bigr)-\int_{\partial\Sigma}A_{\partial M}(\nu^{\vphantom{|}}_\Sigma,\nu^{\vphantom{|}}_\Sigma) u^2
\\
&=\int_{\Sigma}\Bigl(-u\Delta_{\Sigma} u-\bigl(\abs{A_\Sigma}^2+\text{Ric}(\nu_{\Sigma},\nu_{\Sigma})\bigr)u^2\Bigr)+\int_{\partial\Sigma} u\, B^{\mathrm{Robin}}_\Sigma u,
\end{aligned}
\end{align}
which, in the minimal case, arises when considering the second variation of the area functional along the normal variation generated by the function $u$, exactly as we explained above. Of course, in general one should assume the function $u$ in question to be compactly supported for such a relation to make sense.

For the purposes of the present work, when proceeding further to the definition of Morse index it is convenient to just focus on three special cases.
Let first $\Sigma$ be a compact, properly embedded, free boundary minimal surface in the unit ball $\B^3$; we remark that any such surface is necessarily two-sided. (To avoid ambiguities, let us stress that a compact surface $\Sigma$, minimal or otherwise, is said to be properly embedded in $\B^3$ if $\Sigma\cap \partial \B^3=\partial\Sigma$.)
The \emph{Morse index} of $\Sigma$ is defined as the maximal dimension of a vector space of smooth functions on $\Sigma$ on which the form $Q_{\Sigma}$ is negative definite; equivalently, it is the number of \emph{negative} eigenvalues $\lambda$ of the Robin eigenvalue problem 
\begin{equation}\label{eq:RobinEingenvProblem}
\left\{\begin{aligned}
\gls{Jacobi}u&=-\lambda u   &&\text{ in }\Sigma,\\ 
\glsuserii{conormal} u&=0 &&\text{ on }\partial\Sigma.
\end{aligned}\right.
\end{equation}
In a partly similar fashion, let us now consider instead the second case that is relevant for our discussion: 
let $\Sigma$ be a complete, properly embedded, two-sided boundaryless minimal surface in $\R^3$. 
We may then again consider the same quadratic form $Q_{\Sigma}(\cdot,\cdot)$, with no boundary term, which corresponds to the second variation of the area functional under, say, compactly supported deformations. 
If $\Sigma$ has finite total curvature, or (most relevant to our discussion) is a quotient under an isometric action and has then finite total curvature (which happens e.\,g. when we deal with quotiented Karcher--Scherk towers), then it has been shown in \cite{Fis85} that one can equivalently define the Morse index either by exhaustion or by simply looking at the spectrum in the space of square-integrable Sobolev functions. The latter perspective is patently more convenient, for indeed the standard spectral theorem applies, providing again a discrete spectrum for the Jacobi operator $J_{\Sigma}$ hence a diagonalization of $Q_{\Sigma}$.

Lastly, we will also deal with a complete, properly embedded, two-sided free boundary minimal surface in the half space $\R^3_{+}$: in this case $A_{\partial M}=0$ and thus  $B^{\mathrm{Robin}}_\Sigma$ reduces to a homogeneous normal derivative operator. It is possible to extend to this setting the results in \cite{Fis85}, and so define the Morse index either by exhaustion or -- perhaps more simply -- by looking at the spectrum of the Jacobi operator $J_{\Sigma}$ on global Sobolev functions on $\Sigma$ subject to Neumann boundary conditions.

\paragraph{H\"{o}lder norms and spaces of functions.}
Suppose $\Sigma \subset \R^3$ is a smooth,
properly embedded surface or curve.
Given an integer $k \geq 0$,
a real number $\alpha \in \Interval{0, 1}$,
and functions
$u\colon \Sigma \to \R$
and
$f\colon \Sigma \to \interval{0,\infty}$,
we define the weighted norm
\begin{equation}
\label{eqn:definition_weighted_norm}
\nm{u: C^{k,\alpha}(\Sigma, f)}
\vcentcolon=
\sum_{j=0}^k \sup_{x \in \Sigma} 
\frac{\abs{D^j\overline{u}(x)}}{f(x)}
  + \sup_{x \neq y \in \Sigma}
      \frac{
        \abs{D^k\overline{u}(x)-D^k\overline{u}(y)}
      }
      {
        \abs{x-y}^\alpha \min \{f(x),f(y)\}
      },
\end{equation}
where $\abs{{}\cdot{}}$ is the standard Euclidean norm on tensors
and $D^j\overline{u}$ is the $j$\textsuperscript{th}
Euclidean covariant derivative of 
any extension $\overline{u}$ of $u$
to a tubular neighborhood of $\Sigma$ in $\R^3$
such that $\overline{u}$ is constant on line segments
intersecting $\Sigma$ orthogonally.
We also agree that
$\nm{{}\cdot{}: C^k(\Sigma,f)}\vcentcolon
=
\nm{{}\cdot{}: C^{k,0}(\Sigma, f)}
$.

We wish to highlight two special cases, for which we convene to employ a somewhat lighter notation whenever ambiguity is unlikely to arise:
\begin{itemize}
    \item When the weight function is $f=1$, we define the norm
$\nm{{}\cdot{}}_{k,\alpha}\vcentcolon=\nm{{}\cdot{}: C^{k,\alpha}(\Sigma,1)}$
(equivalent to the usual H\"{o}lder norm)
along with the corresponding Banach spaces
 $C^{k,\alpha}(\Sigma)$
of functions on $S$ with finite respective norm;
in this same case
we write $[u]_\alpha$
for the last term of \eqref{eqn:definition_weighted_norm} i.\,e. the H\"older seminorm;
\item When the weight function is $f=e^{-\beta |z|}$, where $\beta\in\R$ is a fixed number and $z$ is the standard third coordinate in $\R^3$ we define the norm
\label{gls:cyl_norms}
$
\gls{norm}
 \vcentcolon=
 \nm{{}\cdot{}: C^{k,\alpha}(\Sigma,e^{-\beta |z|})}
$
along with the corresponding Banach spaces
$
 C^{k,\alpha,\beta}(\Sigma)
 \vcentcolon=
C^{k,\alpha}(\Sigma,e^{-\beta |z|})
$ of functions on $\Sigma$ with finite respective norm; we can equally well adopt this notation -- with obvious changes -- in the case when $|z|$ is replaced e.\, g. by the distance function from a suitable submanifold of $\R^3$.
\end{itemize}

We will make occasional use
of the more general notions of
(possibly weighted)
$C^k$ and $C^{k,\alpha}$
norms on tensor fields
defined on an open subset
of a complete Riemannian manifold,
for which,
referring to definition
\eqref{eqn:definition_weighted_norm},
we dispense with the extension of $u$
(now a tensor field),
interpret each $\abs{\cdot}$ and $D$
as the (intriniscally defined)
norm and connection induced
by the given Riemannian metric,
and appropriately reinterpret
the second term using parallel transport
along uniquely minimizing geodesics.
In all cases of interest to us
this broader definition
coincides with the narrower one,
whenever both are applicable,
up to equivalence of norms.

Assume now that $\Sigma$ is a two-sided hypersurface.
For any rigid motion $\mathsf{M}$
of $\R^3$ that preserves $\Sigma$ as set, we define
\begin{equation}
\label{symsign}
\gls{symsign}
\vcentcolon=
\begin{cases}
\hphantom{-}1
  &
  \text{ if $\mathsf{M}$ preserves each side of $\Sigma$,}
\\
-1
  &
  \text{ if $\mathsf{M}$ exchanges the sides of $\Sigma$.}  
\end{cases}  
\end{equation}
For any group $G$ of rigid motions of $\R^3$ preserving
$\Sigma$ as a set
(that is $G \leq \Aut_{\R^3}(\Sigma)$)
we call a function $u\colon \Sigma \to \R$
\emph{$G$-equivariant} if
$u \circ \mathsf{M}=(\sgn_\Sigma \mathsf{M})u$
for all $\mathsf{M} \in G$
and we define
\label{gls:Gequiv_Hoelder}
\begin{equation*}
\gls{CG}
\vcentcolon=
\{u \in C^{k,\alpha}(\Sigma) \st
  \mbox{$u$ is $G$-equivariant}\}.
\end{equation*}
If instead
$u \circ \mathsf{M}=u$
for all
$\mathsf{M} \in G$,
then $u$ is said to be \emph{$G$-invariant}.
Note that $G$-equivariance is preserved
under multiplication by any $G$-invariant function. Since any rigid motion preserving $\Sigma$
must also preserve $\partial \Sigma$,
by replacing in the definitions just made
each $u \colon \Sigma \to \R$
by $v \colon \partial \Sigma \to \R$,
we define
$G$-equivariance (and $G$-invariance)
for functions on $\partial \Sigma$
as well as
the spaces $C^{k,\alpha}_G(\partial \Sigma)$ in the very same fashion.
More generally, we may append a suffix
of $G$ to any space of functions
defined on $\Sigma$ or $\partial \Sigma$
to designate the corresponding
$G$-equivariant subspace.

Note that if $G \leq \Aut_{\R^3}(\Sigma)$,
then $\abs{A_\Sigma}^2$ is $G$-invariant
and $\Delta_\Sigma$ commutes with every element of $G$,
so $J_{\Sigma}$ preserves $G$-equivariance.
Likewise, for the purpose
of considering boundary value problems
when $\Sigma$ has boundary,
it is important to observe that conormal differentiation also preserves
$G$-equivariance.

The reader is also referred to Section 3 of \cite{FranzIndex} for a broader discussion of the equivariance constraints.

\paragraph{Spaces of square integrable functions.}
Let $(M,g)$ be a Riemannian manifold.
We write $\dvol{g}$ (regardless of the dimension of $M$)
for the volume measure
induced by $g$;
we write $L^2(M,g)$ for the corresponding space
of (equivalence classes of)
real-valued square integrable functions on $M$;
and we write
$\sk{{}\cdot{},{}\cdot{}}_{L^2(M,g)}$
for the corresponding inner product.
(Although we write $U \cdot V$ for the inner product of vectors
$U,V \in \R^3$,
we reserve $u \cdot v$ for the pointwise product
of two real-valued functions $u$ and $v$ on a common domain,
even when an $L^2$ inner product $\sk{u,v}_{L^2(g)}$ is also defined.)
When $M=\Sigma$ is a two-sided surface in $\R^3$
and $G \leq \Aut_{\R^3}(\Sigma)$,
much as for H\"{o}lder spaces
we write $L^2_G(M,g)$ for the subspace
of $L^2(M,g)$ consisting of those elements
which have a $G$-equivariant representative.
Context permitting, we may replace $L^2(M,g)$
by $L^2(M)$ with Riemannian metric tacitly understood.

\paragraph{Cutoff functions.}
We first fix a function $\Psi\colon \R \to \R$
such that
\begin{enumerate}[label={\normalfont(\roman*)}]
  \item $\Psi$ is $C^\infty$ and monotonically non-decreasing,
  \item $\Psi-\frac{1}{2}$ is odd,
  \item $\Psi$ is constantly $0$ 
        on $\interval{-\infty,-1}$, and
  \item $\Psi$ is constantly $1$ on $\interval{1,\infty}$.
\end{enumerate}
Given any $a \neq b \in \R$, we let $L_{a,b}\colon \R \to \R$ be the unique affine function such that $L_{a,b}(a)=-2$ and $L_{a,b}(b)=2$, and we define the function $\gls{cutoff}\colon \R \to \R$ by 
\begin{equation}
\label{eqn:definition_cutoff}
\gls{cutoff} \vcentcolon= \Psi \circ L_{a,b}. 
\end{equation}
Note that when $a>b$ the function $\cutoff{a}{b}$ is monotonically non-increasing.

\paragraph{Use of constants.}

Throughout this article, we shall typically employ the letter $C$ to denote any positive constant that appears in our estimates; when we wish to stress the functional dependence of such a constant in terms of some parameters, we will explicitly indicate them in brackets, so to obtain expressions such as $C(k)$ or $C(m,\xi)$ and so on. Within a given proof, the exact value of a constant is allowed to vary from line to line or even in the same line. 
In those rare cases when one needs to keep track of the specific value of a constant, often just for expository convenience within a given proof, we do so by numbering the constants in question.

\section{Initial surfaces}\label{sec:Initial}

\subsection{Building blocks for the initial surfaces}
\label{subsec:building_blocks}

In the following lemmata we work with coordinates $(z,r)$ on $\R^2$; a posteriori these are to be understood as standard cylindrical coordinates in the Euclidean space.
In equation \eqref{eqn:r_a,b} (and throughout this section) we use the positive branch of the inverse of the hyperbolic cosine which is defined for all $x\geq1$ by 
\[
\cosh^{-1}(x)\vcentcolon=\log\bigl(x+\sqrt{x^2-1}\bigr).
\]

\begin{lemma}\label{lem:catenoids}
For any $b\in\Interval{0,1}$ and any 
\begin{align}\label{eqn:condition_a,b}
a>\frac{1}{1-b^2}
\end{align}
the graph of the function $r_{a,b}\colon[0,1]\to\interval{0,\infty}$ given by 
\begin{align}\label{eqn:r_a,b}
r_{a,b}(z)&=\frac{1}{a}\cosh\bigl(a z-s_{a,b}\bigr)
, & \text{where } 
s_{a,b}&\vcentcolon=a b+\cosh^{-1}\Bigl(a\sqrt{1-b^2}\Bigr)
\end{align}
intersects the unit circle around the origin at $z=b$ and again at $z=h_{a,b}>b$.
Moreover, $h_{a,b}$ depends smoothly on $a,b$. 
\end{lemma}

\begin{figure}
\flushright
\pgfmathsetmacro{\globalscale}{2}
\pgfmathsetmacro{\xmax}{(\textwidth-2.5em)/\globalscale/1cm}
\begin{tikzpicture}[line cap=round,line join=round,baseline={(0,0)},scale=\globalscale]
\begin{scope}
\clip(0,0)rectangle(\xmax,1);
\fill[black!30,domain=0:sqrt(1-0.9*\xmax^(-2)),variable=\b,samples=50]plot({1/(1-\b*\b)},\b)|-cycle;  
\draw[thick,dotted] plot coordinates {
( 2.3328,     0 )( 2.3883,0.0276 )( 2.4491,0.0552 )( 2.5156,0.0828 )( 2.5884,0.1103 )( 2.6680,0.1379 )( 2.7550,0.1655 )( 2.8502,0.1931 )( 2.9545,0.2207 )( 3.0689,0.2483 )( 3.1945,0.2759 )( 3.3329,0.3034 )( 3.4855,0.3310 )( 3.6544,0.358 )( 3.8419,0.3862 )( 4.0508,0.4138 )( 4.2845,0.4414 )( 4.5471,0.4690 )( 4.8437,0.4966 )( 5.1806,0.5241 )( 5.5659,0.551 )( 6.0097,0.5793 )( 6.5252,0.6069 )( 7.1300,0.6345 )( 7.8473,0.6621 )( 8.7094,0.6897 )( 9.7617,0.7172 )(11.0702,0.7448 )(12.7344,0.7724 )(14.9119,0.8000 )};
\end{scope}
\draw[->](0,0)--(\xmax,0)node[below left]{$a$};
\draw[->](0,0)--(0,1.15)node[right]{$b$}; 
\draw[dashed](0,1)--(\xmax,1);
\draw plot[plus](0,0)node[below]{$0$};
\draw plot[hdash](0,1)node[left]{$1$};
\draw plot[vdash](1,0)node[below]{$1$};
\draw plot[vdash](2.3328,0)node[below]{$a_0$};
\end{tikzpicture}
\caption{If $(a,b)$ is in the shaded region then condition \eqref{eqn:condition_a,b} is satisfied.}%
\label{fig:condition_a,b}%
\end{figure}
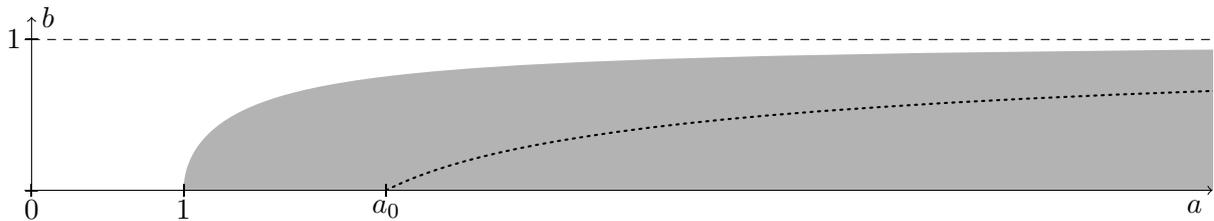

\begin{remark}
The surface of revolution given by the rotation of the curve $(z,r_{a,b}(z))$ around the vertical axis is a vertically shifted and rescaled catenoid which intersects the unit sphere $\partial\B^3$ along the circle at height $z=b$ and along another circle at height $z=h_{a,b}$. 
\end{remark}

\begin{proof}
The upper half of the unit circle in the coordinates $(z,r)$ is the graph of the function $f\colon[-1,1]\to[0,1]$ given by $f(z)=\sqrt{1-z^2}$. 
By construction $r_{a,b}(b)=f(b)$. 
We introduce the function $u=r_{a,b}^2-f^2$ defined on $[0,1]$ and differentiate with respect to $z$ twice:
\begin{align*}
u'&=2r_{a,b}\,r'_{a,b}+2z,
\\
u''&=2r_{a,b}'\,r'_{a,b}+2r_{a,b}\,r''_{a,b}+2 
\\
&=2\sinh^2(a z-s_{a,b})
+2\cosh^2(a z-s_{a,b})+2>0.
\end{align*}
Hence, the function $u$ is strictly convex with $u(b)=0$
and $u(1)=r_{a,b}^2(1)\geq a^{-2}> 0$. 
Consequently, the function $u$ has a unique second zero at $z=h_{a,b}\in\interval{b,1}$ if and only if $u'(b)<0$. 
Since $r_{a,b}(b)=\sqrt{1-b^2}$ and  $r_{a,b}'(b)=-\sinh\bigl(\cosh(a\sqrt{1-b^2})\bigr)=-\sqrt{(1-b^2)a^2-1} $, we obtain
\begin{align*}
u'(b)&=2b-2\sqrt{1-b^2}\sqrt{(1-b^2)a^2-1}<0 
\end{align*}
precisely if $a,b$ satisfy \eqref{eqn:condition_a,b}; see Figure \ref{fig:condition_a,b} for a visualization of this condition. 
\end{proof}

\begin{lemma}\label{lem:orthogonal-0}
A suitable rescaling and vertical translation of a catenoid along its axis of rotation intersects the unit sphere along the equator and orthogonally along another circle of latitude. 
\end{lemma}

\begin{proof}
Let $r_{a,0}\colon[0,1]\to\interval{0,\infty}$ and $h_{a,0}$ be as in Lemma \ref{lem:catenoids} for $b=0$. 
In this case, condition \eqref{eqn:condition_a,b} reads $a>1$. 
The surface of revolution given by the rotation of the curve $(r_{a,0}(z),z)$ around the vertical axis is a vertically shifted and rescaled catenoid which intersects the unit sphere $\partial\B^3$ along the equator and along another circle at height $z=h_{a,0}\in\interval{0,1}$. 
The unit sphere itself is a surface of revolution with profile function $f\colon[-1,1]\to[0,1]$ given by $f(z)=\sqrt{1-z^2}$. 
Note that $f'(h_{a,0})<0$ is well-defined since $0<h_{a,0}<1$. 
Therefore, it is sufficient to find $a_0>1$ such that 
\begin{align}\label{eqn:orthogonality_condition}
r_{a_0,0}'(h_{a_0,0})=-\frac{1}{f'(h_{a_0,0})}.
\end{align} 
\emph{Claim 1} (see Figure \ref{fig:r_a-plot}, left image). If $1<a<\sqrt{2}$, then $r_{a,0}'(h_{a,0})<-1/f'(h_{a,0})$. 

\begin{figure}%
\centering
\pgfmathsetmacro{\a}{sqrt(2)}
\begin{tikzpicture}[line cap=round,line join=round,baseline={(0,0)},scale=\unitscale]
\draw[->](-1.1,0)--(1.1,0)node[right]{$z$};
\draw[->](0,-0)--(0,1.1)node[left]{$r$};
\draw plot[vdash](0,0)node[below]{$0$};
\draw plot[vdash](1,0)node[below]{$1$};
\draw[name path=sphere](1,0)arc(0:180:1)node[pos=3/4,anchor=135]{$f$}; 
\draw[dashed,domain=0.8703:0.001,variable=\z,samples=50,name path=catenoid]plot({\z},{cosh(2.3328*\z-acosh(2.3328))/2.3328}); 
\pgfmathsetmacro{\zmax}{min(2.01*acosh(\a)/\a,1)}
\draw[semithick,domain=0.001:\zmax,variable=\z,samples=50,name path=catenoid]plot({\z},{cosh(\a*\z-acosh(\a))/\a})node[right]{$r_{a,0}$}; 
\draw [name intersections={of=catenoid and sphere}]
plot[bullet](intersection-1);
\draw plot[bullet](0,1);
\draw[dotted](intersection-1)--($(0,0)!(intersection-1)!(1,0)$)coordinate(h);
\draw plot[vdash](h)node[below]{$h_{a,0}$};  
\end{tikzpicture}
\hfill
\pgfmathsetmacro{\a}{25}
\begin{tikzpicture}[line cap=round,line join=round,baseline={(0,0)},scale=\unitscale]
\draw[->](-1.1,0)--(1.1,0)node[right]{$z$};
\draw[->](0,-0)--(0,1.1)node[left]{$r$};
\draw plot[vdash](0,0)node[below]{$0$};
\draw plot[vdash](1,0)node[below]{$1$};
\draw[name path=sphere](1,0)arc(0:180:1)node[pos=3/4,anchor=135]{$f$}; 
\draw[dashed,domain=0.8703:0.001,variable=\z,samples=50,name path=catenoid]plot({\z},{cosh(2.3328*\z-acosh(2.3328))/2.3328});  
\pgfmathsetmacro{\zmax}{min(2.01*acosh(\a)/\a,1)}
\draw[semithick,domain=0.001:\zmax,variable=\z,samples=50,name path=catenoid]plot({\z},{cosh(\a*\z-acosh(\a))/\a})node[right]{$r_{a,0}$}; 
\draw [name intersections={of=catenoid and sphere}]
plot[bullet](intersection-1);
\draw plot[bullet](0,1);
\draw[dotted](intersection-1)--($(0,0)!(intersection-1)!(1,0)$)coordinate(h);
\draw plot[vdash](h)node[below]{$h_{a,0}$};  
\end{tikzpicture}
\caption{Plot of $r_{a,0}$ for $a=\sqrt{2}$ (left image) and $a=\a$ (right image). We seek the dashed curve.}%
\label{fig:r_a-plot}%
\end{figure}
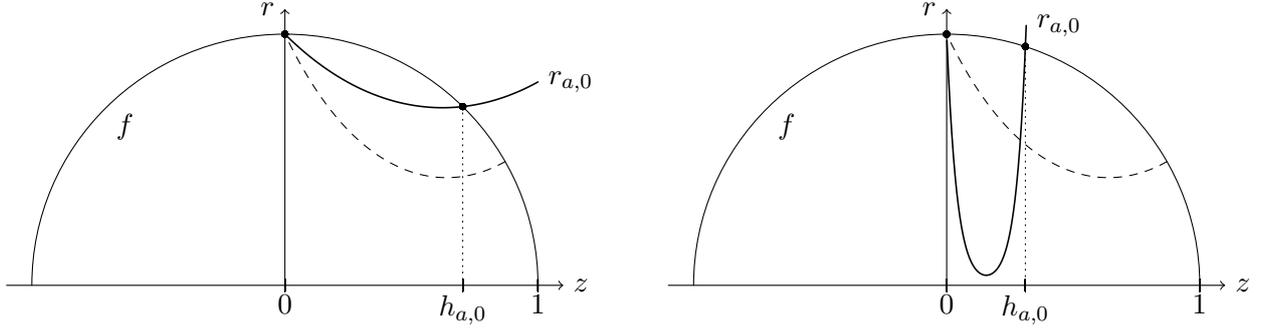

\begin{proof}[Proof of Claim 1.] 
For any $a>1$ and all $z\geq(2/a)\cosh^{-1}(a)$ we have 
\[
r_{a,0}(z)=\frac{1}{a}\cosh\bigl(a z-\cosh^{-1}(a)\bigr)
\geq 1.
\]
Therefore, $h_{a,0}<(2/a)\cosh^{-1}(a)$.  
Since $r_{a,0}'$ is increasing (by convexity, as we saw above), we obtain  
\begin{align*}
r_{a,0}'(h_{a,0})< r_{a,0}'\Bigl(\frac{2}{a}\cosh^{-1}(a)\Bigr)=\sinh\bigl(\cosh^{-1}(a)\bigr)=\sqrt{a^2-1}.
\end{align*}
Moreover, since $r_{a,0}(z)\geq1/a$ we have $h_{a,0}^2=1-\bigl(r_{a,0}(h_{a,0})\bigr)^2\leq1-a^{-2}$ 
which  implies for all $z\in[0,1]$
\begin{align*}
0\geq f'(h_{a,0})\geq f'\Bigl(\sqrt{1-a^{-2}}\Bigr)=-\sqrt{a^2-1}
\end{align*}
since $f'$ is decreasing. 
Consequently, $-f'(h_{a,0})\,r_{a,0}'(h_{a,0})\leq a^2-1<1$ for $a\in\interval{1,\sqrt{2}}$.
\end{proof}

\emph{Claim 2} (see Figure \ref{fig:r_a-plot}, right image). If $a>1$ is chosen sufficiently large, then $r_{a,0}'(h_{a,0})>-1/f'(h_{a,0})$. 

\begin{proof}[Proof of Claim 2.]
Assume $a>1$.
By Lemma \ref{lem:catenoids}
the equation
$r_{a,0}^2(z)+z^2=1$
has two solutions:
$z=0$ and $z=h_{a,0}>0$.
Since
$r_{a,0}(a^{-1} \cosh^{-1}(a))=a^{-1}$
and $\lim_{a \to \infty} a^{-1} \cosh^{-1}(a) = 0$,
we know that
for $a$ sufficiently large
the waist of the corresponding catenoid,
which lies at height $z=a^{-1} \cosh^{-1}(a)$,
is contained in the interior of $\B^3$.
It follows that
for $a$ sufficiently large
$h_{a,0}>a^{-1} \cosh^{-1}(a)$,
and so
$a\,h_{a,0}-\cosh^{-1}(a)>0$.

Next we claim that $\lim_{a \to \infty} h_{a,0} = 0$. This is geometrically evident, but let us see a formal justification.
If not, there would exist $\epsilon>0$
such that we have $h_{a,0} \geq \epsilon$
for a sequence of values of $a$ diverging to $\infty$.
However,
from the definition of $r_{a,b}$
we find
$\lim_{a \to \infty} r_{a,0}(z)=\infty$
for any $z>a^{-1} \cosh^{-1}(a)$,
so in particular for $z=\epsilon$.
Thus, using also the monotonicity of $r_{a,0}$ in $z$,
we can make $r_{a,0}(z)$ arbitrarily large
for all $z \geq \epsilon$
by taking $a$ large.
On the other hand,
by definition of $h_{a,b}$,
we also have $h_{a,0}^2+r^2_{a,0}(h_{a,0})=1$,
establishing a contradiction
and thus proving the asserted limit.
The last identity then also implies that
$\lim_{a \to \infty} r_{a,0}(h_{a,0})=1$,
In particular,
for sufficiently large $a$
we have
$r_{a,0}(h_{a,0})>1/2$.

The two inequalities terminating
the last two paragraphs
then imply (still keeping in mind the very definition of the function $r_{a,b}$)
\begin{align}\label{eqn:20211117}
a\,h_{a,0}-\cosh^{-1}(a)&>\cosh^{-1}\Bigl(\frac{a}{2}\Bigr).
\end{align}
In particular, \eqref{eqn:20211117} yields 
\begin{align*}
r_{a,0}'(h_{a,0}) 
&>\sinh\Bigl(\cosh^{-1}\Bigl(\frac{a}{2}\Bigr)\Bigr)
=\sqrt{\frac{a^2}{4}-1},
\\
-f'(h_{a,0})&= \frac{h_{a,0}}{\sqrt{1-h_{a,0}^2}}
>h_{a,0}>\frac{1}{a}\cosh^{-1}(a).
\end{align*}
Since $\cosh^{-1}(a)\to\infty$ as $a\to\infty$ we have 
$-f'(h_{a,0})\,r_{a,0}'(h_{a,0})\to\infty$ as $a\to\infty$ and Claim 2 follows. 
\end{proof}
The statement of Lemma \ref{lem:orthogonal-0} follows from Claims 1 and 2 by means of a straightforward continuity argument. 
\end{proof}

\begin{lemma}\label{lem:orthogonal-b}
Given $0\leq b<1<a$ satisfying \eqref{eqn:condition_a,b} let $r_{a,b}\colon[0,1]\to\R$ and $h_{a,b}$ be as in Lemma \ref{lem:catenoids}. 
The graph of $z\mapsto r_{a,b}(z)$ intersects the unit circle orthogonally at $z=h_{a,b}$ if and only if  
\begin{align*}
r_{a,b}'(h_{a,b})>0\quad\text{ and }\quad
(1-h_{a,b}^2)h_{a,b}^2=a^{-2}.
\end{align*}
\end{lemma}

\begin{proof}
By Lemma \ref{lem:catenoids} the graph of $z\mapsto r_{a,b}(z)$ intersects the unit circle at $z=h_{a,b}$. 
In particular, 
\begin{align}\label{eqn:circle}
r_{a,b}^2(h_{a,b})+h_{a,b}^2&=1.
\end{align}
This intersection is orthogonal if and only if the vectors $\bigl(1,r_{a,b}'(h_{a,b})\bigr)$ and $\bigl(h_{a,b},\sqrt{1-h_{a,b}^2}\bigr)$ are parallel or, equivalently, 
\begin{align}\label{eqn:orthogonality_condition2}
h_{a,b}\, r_{a,b}'(h_{a,b})=\sqrt{1-h_{a,b}^2}. 
\end{align}
Equation \eqref{eqn:orthogonality_condition2} directly implies $r_{a,b}'(h_{a,b})>0$. 
To ease notation, we omit the indices $a,b$ in the following computation and recall $r'(h)=\sinh(ah-s)$. 
Squaring \eqref{eqn:orthogonality_condition2} therefore yields 
\begin{align*}
1=h^2\sinh^2(ah-s)+h^2
=h^2\cosh^2(ah-s)
=h^2a^2r^2(h)
=h^2a^2(1-h^2).
\end{align*}
Conversely, given any $a,b$ such that \eqref{eqn:condition_a,b} and $(1-h_{a,b}^2)h_{a,b}^2=a^{-2}$ are satisfied, we obtain with \eqref{eqn:circle}
\begin{align*}
1=h^2 a^2 r^2(h)=
h^2\cosh^2(ah-s)=h^2+h^2\sinh^2(ah-s)
\end{align*}
which implies $h^2(r'(h))^2=1-h^2$. 
Assuming $r'(h)>0$, equation \eqref{eqn:orthogonality_condition2} follows. 
\end{proof}

\begin{corollary}\label{cor:area_K}
A rotationally symmetric minimal annulus intersecting the unit sphere along the equator and again orthogonally at some positive height has area greater than $\pi$.
\end{corollary}

\begin{proof}
It is well-known that the rotationally symmetric minimal annulus in question must be  catenoidal, i.\,e. it  is a surface of revolution (say $\K$) with profile function $r_{a,0}\colon[0,h_{a,0}]\to\R$ for some $a>1$ as given in Lemma \ref{lem:catenoids}. 
Hence,  
\begin{align*}
\area(\K)&=2\pi\int^{h_{a,0}}_{0}r_{a,0}\sqrt{(r_{a,0}')^2+1}\,dz
=\frac{2\pi}{a}\int^{h_{a,0}}_{0}\!\cosh^2(a z-s_{a,0}\bigr)\,dz
=\frac{2\pi}{a^2}\int^{ah_{a,0}-s_{a,0}}_{-s_{a,0}}\!\cosh^2(x)\,dx.
\end{align*} 
A primitive for $f(x)=\cosh^2(x)$ is $F(x)=\frac{1}{2}\bigl(x+\sinh(x)\cosh(x)\bigr)$. 
Recalling $s_{a,0}=\cosh^{-1}(a)$ from definition \eqref{eqn:r_a,b}, we also note that $\cosh(-s_{a,0})=a$ and $\sinh(-s_{a,0})=-\sqrt{a^2-1}$. Hence, 
\begin{align}\label{eqn:area_K}
\area(\K)&=\frac{\pi}{a^2} \Bigl(a h_{a,0}+\sinh\bigl(a h_{a,0}-s_{a,0}\bigr)\cosh\bigl(a h_{a,0}-s_{a,0}\bigr)+a\sqrt{a^2-1}\Bigr) 
\\\notag
&>\Bigl(\frac{h_{a,0}}{a}+\sqrt{1-a^{-2}}\Bigr)\pi
=\Bigl(h_{a,0}^2\sqrt{1-h_{a,0}^2}+\sqrt{1-(1-h_{a,0}^2)h_{a,0}^2}\Bigr)\pi
\end{align} 
by Lemma~\ref{lem:orthogonal-b}, and in particular using $\sinh\bigl(a h_{a,0}-s_{a,0}\bigr)=r_{a,0}'(h_{a,0})>0$.  
Since $h_{a,0}^2\leq1$ and $\sqrt{x}\geq x$ for any $x\in[0,1]$ the claim follows. 
\end{proof}

The previous corollary will be employed in Appendix \ref{app:convergence_large_g} when discussing the limit behavior of the free boundary minimal surfaces constructed by Ketover in \cite{KetoverFB}.

\begin{lemma}\label{lem:catenoid-K_b}
One can find $\beta>0$ such that for each $b\in\Interval{0,\beta}$ there exists a catenoid $\gls{Kb}$ intersecting $\partial \B^3$ exactly along the circle at height $z=b$ and orthogonally along a higher circle of latitude in the upper hemisphere. 
Moreover, $\K_b$ depends smoothly on $b$. 
\end{lemma}
 
\begin{proof}
Let $r_{a,b}(z)=(1/a)\cosh(a z-s_{a,b})$ and $h_{a,b}$ be as in Lemma \ref{lem:catenoids}. 
Lemmata \ref{lem:orthogonal-0} and \ref{lem:orthogonal-b} imply that there exists $a_0>1$ such that $r_{a_0,0}'(h_{a_0,0})>0$ and such that the function 
\begin{align*}
F(a,b)&\vcentcolon=(1-h^2_{a,b})h_{a,b}^2-a^{-2}
\end{align*}
vanishes at $(a,b)=(a_0,0)$. 
Moreover, by Lemma \ref{lem:catenoids}, $F$ inherits the smooth dependence on $(a,b)$ in a neighborhood of $(a_0,0)$. 
We claim 
\begin{align}\label{eqn:20211125}
\frac{\partial F}{\partial a}(a_0,0)>0.
\end{align}
If the claim is true, then the implicit function theorem yields $\beta>0$ and a uniquely defined, differentiable function $\alpha\colon\Interval{0,\beta}\to\R$ such that $\alpha(0)=a_0$ and $F(\alpha(b),b)=0$ for all $b\in\Interval{0,\beta}$. 
Moreover, since $r_{a,b}'(h_{a,b})$ depends continuously on $a,b$ and is positive at $(a_0,0)$ we may assume $r_{\alpha(b),b}'(h_{\alpha(b),b})>0$ for all $b\in\Interval{0,\beta}$ by reducing $\beta>0$ if necessary. 
The existence of $\K_b$ then follows by virtue of the characterization given in Lemma \ref{lem:orthogonal-b}. 

\begin{remark}
In fact
$\alpha(b)$ can be defined for all $b\in\Interval{0,1}$ (see Figure \ref{fig:condition_a,b}, dotted curve),
as follows from
\cite[Proposition 2.10]{KapouleasZouCloseToBdy}.
\end{remark}

A proof of \eqref{eqn:20211125} requires control on $\partial h_{a,b}/\partial a$. 
To ease notation, we omit the indices $a,b$ during the following computations keeping in mind that $r(h),h,s$ all depend on $a$ and $b$.   
Differentiating the identity $h^2=1-r^2(h)$ with respect to $a$, we obtain 
\begin{align*}
h\frac{\partial h}{\partial a}
&=-r(h)\frac{\partial}{\partial a}\biggl(\frac{1}{a}\cosh(ah-s)\biggr)
=-r(h)\biggl(-\frac{r(h)}{a}+\frac{r'(h)}{a}\Bigl(h+a\frac{\partial h}{\partial a}-\frac{\partial s}{\partial a}\Bigr)\biggr)
\end{align*}
or equivalently, since $r(h)=\sqrt{1-h^2}$,
\begin{align}\label{eqn:20211124-2}
\Bigl(h+\sqrt{1-h^2}\,r'(h)\Bigr)\frac{\partial h}{\partial a}&=\frac{1-h^2}{a}
-\frac{\sqrt{1-h^2}\,r'(h)}{a}\Bigl(h-\frac{\partial s}{\partial a}\Bigr).
\end{align}
Recalling that $s=a b+\cosh^{-1}\bigl(a\sqrt{1-b^2}\bigr)$ we have 
\begin{equation}\label{eqn:ds/da}
\frac{\partial s}{\partial a}=b+\sqrt{\frac{1-b^2}{ (1-b^2)a^2-1}}.
\end{equation}
After setting, for notational convenience, $h_0\vcentcolon=h_{a_0,0}$, the orthogonality condition \eqref{eqn:orthogonality_condition2} implies 
\begin{align}\label{eqn:20211124-3}
r_{a_0,0}'(h_0)&=\frac{\sqrt{1-h_0^2}}{h_0}. 
\end{align}
and equation \eqref{eqn:20211124-2} combined with \eqref{eqn:ds/da} thus reads 
\begin{align}\label{eqn:dh/da}
 \frac{1}{h_0} \frac{\partial h}{\partial a}(a_0,0)
&= \frac{1-h_0^2}{a_0 h_0\sqrt{a_0^2-1}}. 
\end{align} 
Recalling $F(a,b)=h_{a,b}^2-h_{a,b}^4-a^{-2}$ equation \eqref{eqn:dh/da} implies
\begin{align*}
\frac{\partial F}{\partial a} (a_0,0)  
%&=(2h_0-4h_0^3)\frac{\partial h}{\partial a}(a_0,0)+\frac{2}{a_0^3}
=\frac{(2h_0-4h_0^3)(1-h_0^2)}{a_0\sqrt{a_0^2-1}}
+\frac{2}{a_0^3}.
\end{align*}   
At this stage, we
use the identity $a_0^{-2}=(1-h_0^2)h_0^2$ (which is $F(a_0,0)=0$) to compute 
\begin{align*}
a_0^3\sqrt{a_0^2-1}\frac{\partial F}{\partial a} (a_0,0)
&=(2h_0-4h_0^3)(1-h_0^2)a_0^2
+2\sqrt{a_0^2-1}
\\
&=\frac{(2-4h_0^2)}{h_0 }
+\frac{2}{h_0}\sqrt{\frac{1-h_0^2+h_0^4}{1-h_0^2} }
\geq\frac{(4-4h_0^2)}{h_0}.
\end{align*}
Therefore, since $h_0\in\interval{0,1}$ claim \eqref{eqn:20211125} follows. 
\end{proof}

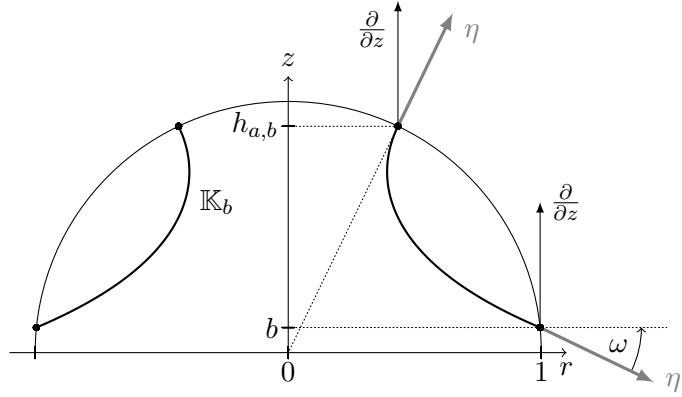
\begin{figure}%
\centering
\begin{tikzpicture}[line cap=round,line join=round,baseline={(0,0)},scale=\unitscale]
\pgfmathsetmacro{\b}{0.1}
\pgfmathsetmacro{\a}{2.5603}
\pgfmathsetmacro{\h}{0.9012}
\pgfmathsetmacro{\w}{asin(sqrt((1-\h*\h)/(1-\b*\b)))}
\draw[thick,domain=\b:\h,variable=\z,samples=50]plot({cosh(\a*\z-\a*\b-acosh(\a*sqrt(1-\b*\b)))/\a},{\z});
\draw[thick,domain=\b:\h,variable=\z,samples=50]plot({-cosh(\a*\z-\a*\b-acosh(\a*sqrt(1-\b*\b)))/\a},{\z});
\draw(-1/\a,0.6)node[right]{$\K_b$};
\draw[densely dotted](0,\h)--({sqrt(1-\h*\h)},\h)coordinate(h)--(0,0)
(0,\b)--({sqrt(1-\b*\b)},\b)coordinate(b)--++(0.5,0);
\draw plot[hdash](0,\h)node[left]{$h_{a,b}$};
\draw plot[hdash](0,\b)node[left]{$b$}; 
\draw[<-](b)+(0.4,0)arc(0:-\w:0.4)node[anchor=-\w/2,midway]{$\omega$};
\draw[-latex,very thick,black!50](b)--++(-\w:0.5)node[right]{$\eta$};
\draw[-latex,very thick,black!50](h)--++({sqrt(1-\h*\h)/2},\h/2)node[below right]{$\eta$};
\draw[-latex](b)--++(0,0.5)node[right]{$\frac{\partial}{\partial z}$};
\draw[-latex](h)--++(0,0.5)node[below left]{$\frac{\partial}{\partial z}$};
\draw plot[bullet](b);
\draw plot[bullet](h);
\draw plot[bullet]({-sqrt(1-\b*\b)},\b);
\draw plot[bullet]({-sqrt(1-\h*\h)},\h);
\draw[->](-1.1,0)--(1.1,0)node[below]{$r$};
\draw[->](0,-0)--(0,1.1)node[above]{$z$};
\draw plot[vdash](0,0)node[below]{$0$};
\draw plot[vdash](1,0)node[below]{$1$};
\draw plot[vdash](-1,0);
\draw(1,0)arc(0:180:1); 
\end{tikzpicture}
\caption{Angle $\omega$ between the conormal $\eta$ to $\K_b$ and the horizontal plane of equation $z=b$.}%
\label{fig:omega}%
\end{figure}

Here and throughout the article we let
\begin{align}\label{eqn:defininiton_omega}
\omega = \gls{omegab} \in\interval{0,\pi/2}
\end{align}
denote the angle between the outward unit conormal to $\K_b$ and the horizontal plane of equation $z=b$ as shown in Figure~\ref{fig:omega}. 
 
\begin{remark}
\label{catangleatb}
By applying an elementary balancing principle, i.\,e. considering the so-called \emph{flux homomorphism} associated to the vector field $\partial/\partial z$ in $\R^3$ (cf. Corollary 1.8 in \cite{CM11}) we find at once the equation
\begin{equation}\label{eq:BalanceKb}
    \sqrt{1-b^2} \sin(\omega_b) = h\sqrt{1-h^2},
\end{equation}
which then allows us, at least implicitly, to express $\omega=\omega_b$ in terms of the parameter $b\ll 1$ only, by means of 
the implicit functions for $a=\alpha(b)$ and $h_{a,b}$ (cf. Lemmata \ref{lem:orthogonal-b} and \ref{lem:catenoid-K_b}). 
\end{remark}
 
\begin{remark}\label{rem:values_for_b=0}
In the case $b=0$, equation \eqref{eq:BalanceKb} immediately implies $\omega_0\in\interval{0,\pi/6}$ since $h\sqrt{1-h^2}\in\interval{0, 1/2}$ for $h\in \interval{0,1}$.
Solving equations \eqref{eqn:circle}, \eqref{eqn:orthogonality_condition2}, \eqref{eq:BalanceKb} numerically for $a$, $h_{a,b}$ and $\omega$ we obtain in the case $b=0$ the approximate values
\begin{align}\label{eqn:values_a0_h0}
a_0&\approx2.3328, & 
h_{a_0,0}&\approx0.8703, & 
\omega_0&\approx 0.141\,\pi\approx 25.38^{\circ}.
\end{align}
Plugging in the values \eqref{eqn:values_a0_h0} in equation \eqref{eqn:area_K} we obtain  
\(\area(\K_0)\approx 1.3960\,\pi\), to be compared with the general lower bound proven in Corollary \ref{cor:area_K} above.
\end{remark} 
 
We now switch to another building block for our construction, i.\,e. the Karcher--Scherk towers that we alluded to in the introduction.
The following lemma describes
a certain one-parameter subfamily
of the singly periodic surfaces
discovered by Karcher \cite{KarcherScherk}
and generalizing the classical example
of Scherk \cite{Scherk}.
Actually our construction will eventually employ
just a single member
of this subfamily,
but we will make that specialization
after the proof of the proposition.

\begin{proposition}
[Definition and basic properties
 of the Karcher--Scherk tower $\tow_\vartheta$]
\label{karcher-scherk}
For each $\vartheta \in \interval{0, \pi/2}$
there exists $\glsuseri{tow} \subset \R^3$
with the following properties. 
\begin{enumerate}[label={\normalfont(\roman*)}]
\item\label{item:karcher-i}
    $\tow_\vartheta$ is a complete, connected,
   properly embedded minimal surface.
\item\label{towperiodicityandquotient} 
$\tow_\vartheta$ is $2\pi$-periodic in the vertical direction, i.\,e. $\trans^{\axis{z}}_{2\pi} \in \Aut_{\R^3}(\tow_\vartheta)$, and
     the quotient
     $\tow_\vartheta / \sk{\trans^{\axis{z}}_{2\pi}}$
     has genus $0$ and $6$ ends.
\item\label{towasymp} 
      $\tow_\vartheta$ has six ends,
      each asymptotically planar
      in the sense that,
      outside of a compact set,
      $\tow_\vartheta$ consists of six normal graphs
      over their respective
      asymptotic planes
      and the corresponding defining functions
      converge to zero exponentially with rate one
      (together with their derivatives of all orders). 
      Two of the asymptotic planes are contained in $\{y=0\}$ 
      and each of the other four is parallel to the vertical axis $\axis{z}$
      and makes an angle of $\vartheta$
      with the plane $\{y=0\}$.
      (See Remark \ref{towasymptotics}
      for a more detailed statement,
      with the particular value of $\vartheta$
      we need in our construction.)
\item\label{item:towhorizontalplanesofsym}
  $\tow_\vartheta$ is invariant under reflection
  $\refl_{\{z=\pi/2\}}$ through the plane $\{z=\pi/2\}$.
\item\label{item:towverticalplaneofsym}
  $\tow_\vartheta$ is invariant under reflection
  $\refl_{\{x=0\}}$ through the plane $\{x=0\}$.
\item\label{towchoice}
       $\tow_\vartheta
        \cap \{x=0\} 
        \cap \{z=\pi/2\}
       $
       consists of only one point
       and is contained in
       $\{y>0\}$.
\item\label{towsymgroup}
      $
        \Aut_{\R^3}(\tow_\vartheta)
        \geq
        \sk{
          \refl_{\{x=0\}},~
          \refl_{\axis{x}},~
          \refl_{\{z=\pi/2\}}
        }
      $,
      with equality provided $\vartheta \neq \pi/3$.
\item\label{towsymsigns}
      Recalling \eqref{symsign},
      $\sgn_{\tow_\vartheta} \refl_{\{x=0\}}
        =
        \sgn_{\tow_\vartheta} \refl_{\{z=\pi/2\}}
        =
        -\sgn_{\tow_\vartheta} \refl_{\axis{x}}
        =1
      $.
\item\label{towcontainslines}
  $\tow_\vartheta$
  contains every line
  $
   \trans^{\axis{z}}_{n\pi}\axis{x}
   =
   \{y=0\} \cap \{z=n\pi\}
  $
  with $n \in \Z$;
  $\tow_\vartheta$ contains no other lines
  when $\vartheta \neq \pi/3$.
\item\label{towcapxeq0}
     The half tower
     $\tow_\vartheta^+\vcentcolon=\tow_\vartheta \cap \{x \geq 0\}$
     is a connected free boundary minimal surface in the half space
     $\{x \geq 0\}$ and has connected, smooth, embedded boundary.
\item\label{genusofnontrivialquotient}
      For any integer $m \geq 1$ the quotient
      $\tow_\vartheta / \sk{\trans^{\axis{z}}_{2m\pi}}$
      has genus $2(m-1)$ and $6$ ends.
\item\label{towEWrep}
      $\tow_\vartheta/\sk{\trans^{\axis{z}}_{2\pi}}$
      is parametrized by \eqref{standardtowerparam}
      (itself based on \eqref{eqn:original_parametrisation}
      and \eqref{eqn:rescaled_parametrization}).
\end{enumerate}
\end{proposition}

\begin{proof}
First we observe that the final few items,
excepting \ref{towEWrep},
follow easily from the earlier ones.
Items \ref{towsymsigns} and \ref{towcontainslines}
follow from items \ref{towasymp} and \ref{towsymgroup},
the line $\axis{x}$ lying on $\tow_\vartheta$
since reflection through it is a symmetry and
$\{y=0\}$ is an asymptotic plane,
and containment of the other lines then following
by the symmetry $\trans^{\axis{z}}_{2\pi}$ and \ref{item:towhorizontalplanesofsym};
by the reflection principle for minimal surfaces
any line contained in $\tow_\vartheta$
is a line of reflectional symmetry,
so item \ref{towsymgroup} precludes
the possibility of other lines
when $\vartheta\neq\pi/3$.
Item \ref{genusofnontrivialquotient}
follows from the Gauss--Bonnet theorem,
using items \ref{towperiodicityandquotient}
and \ref{towasymp}.
Aside from the connectedness of $\partial \tow_\vartheta^+$,
all other claims of item \ref{towcapxeq0}
follow from
just items \ref{item:karcher-i} and \ref{towsymgroup}.
To see the connectedness of $\partial \tow_\vartheta^+ = \tow_\vartheta \cap \{x=0\}$
first note $\partial \tow_\vartheta^+/\sk{\trans^{\axis{z}}_{2\pi}}$
is connected, since $\tow_\vartheta/\sk{\trans^{\axis{z}}_{2\pi}}$
has genus $0$.
It follows that $\partial \tow_\vartheta^+ \cap \{\abs{z} \leq \pi/2\}$
is also connected
and moreover an arc,
with endpoints on $\{\abs{z}=\pi/2\}$.
Since $\refl_{\axis{x}}$ is a symmetry,
in fact precisely one endpoint lies in $\{z=\pi/2\}$,
and since $\refl_{\{z=n\pi/2\}}$ is a symmetry for each odd $n$,
this proves the connectedness of $\partial \tow_\vartheta^+$.
Note also that
\ref{towsymgroup}
subsumes items
\ref{item:towhorizontalplanesofsym}
and \ref{item:towverticalplaneofsym}
(which we have stated separately
merely to facilitate
the incidental Remark \ref{towuniqueness} below).

To prove the outstanding items
we refer to Appendix \ref{app:Karcher--Scherk}
(as well as Appendix \ref{app:CylAnalysis}
for the detailed asymptotics),
where we proceed directly
from Enneper-Weierstrass data
presented by Karcher in \cite{KarcherScherk}
to construct a family of surfaces
satisfying all the asserted properties.
In the remainder of the proof
we link these properties
to the results of Appendix \ref{app:Karcher--Scherk}
(and Appendix \ref{app:CylAnalysis}).
Specifically,
given $\vartheta \in \interval{0,\pi/2}$,
we define $\phi \in \interval{0,\pi/2}$
by equation \eqref{eqn:theta_phi},
and we set
$\tow_\vartheta \vcentcolon= \Gamma_\phi$,
where $\Gamma_\phi$ is constructed as in the proof
of Proposition \ref{prop:towerEWsummary}.
Items
\ref{towEWrep}
and
\ref{item:karcher-i}--\ref{towchoice}
are now immediate consequences of
Proposition \ref{prop:towerEWsummary}
and its proof,
except for the exponential decay
asserted in item \ref{towasymp},
which follows from Lemma \ref{lem:CylMSE}
in conjunction with Proposition \ref{prop:towerEWsummary}.
That $\refl_{\axis{x}} \in \Aut_{\R^3}(\tow_\vartheta)$
follows from Lemma~\ref{lem:Karcher-Scherk_symmetries}
(simply by composing the two symmetries in its statement,
and bearing in mind the exchange of the $x$ and $y$
coordinates in definitions
\eqref{standardtowerparamPRELIM}
and \eqref{standardtowerparam}).
Thus we have established containment in one direction
in item \ref{towsymgroup}.

It remains only to upgrade the containment
just discussed to the equality asserted
in item \ref{towsymgroup}.
To this end note
that the genus $0$ condition on the quotient
implies that $\tow_\vartheta$ does not 
also have a period smaller than $2\pi$.
Indeed, let $\tau$ be the infimum of
all $c > 0$ such that $\trans^{\axis{z}}_c$
is a symmetry of $\tow_\vartheta$.
Then $\trans^{\axis{z}}_\tau$ is itself a symmetry,
and of course $\tau \leq 2\pi$.
Since $\tow_\vartheta$ is not a plane,
we also have $\tau>0$.
If $2\pi$ were not an integer multiple of $\tau$,
then there would be a positive integer $n$
such that $n\tau<2\pi<(n+1)\tau$,
but then
$\tau' \vcentcolon= (n+1)\tau-2\pi=\tau-(2\pi-n\tau)$
would lie in $\interval{0,\tau}$
and $\trans^{\axis{z}}_{\tau'}$ would be a symmetry
of $\tow_\vartheta$,
contradicting the definition of $\tau$.
Thus $2\pi$ is an integer multiple of $\tau$,
but then items
\ref{towperiodicityandquotient}
and
\ref{towasymp}
together with the Gauss--Bonnet theorem
force $2\pi/\tau=2/(2+\gamma)$,
where $\gamma \geq 0$ is the genus of
$\tow_\vartheta/\sk{\trans^{\axis{z}}_\tau}$,
and so we conclude that $\tau=2\pi$.

In particular, since the composition of two reflections through parallel planes at distance $\delta$ equals a translation by an orthogonal vector of length $2\delta$, $\tow_\vartheta$ 
contains infinitely many horizontal planes of symmetry
and the distance between any two closest horizontal planes
of symmetry is $\pi$.
Now suppose that
$\mathsf{S} \in \Aut_{\R^3}(\tow_\vartheta)$
and assume $\vartheta \neq \pi/3$;
we will show that $\mathsf{S}$
is generated by the reflections appearing in
item \ref{towsymgroup}.
First, $\mathsf{S}$
must permute the ends,
all asymptotic to vertical planes,
and so in particular must also
permute the above horizontal planes of symmetry.
It follows that $\mathsf{S}\axis{z}=\axis{z}$
and (now using $\vartheta \neq \pi/3$)
$\mathsf{S}\{y=0\}=\{y=0\}$.
By composing if necessary with reflections
we have already identified as symmetries
we may further assume
that $\mathsf{S}$
preserves each of $\{z \geq 0\}$
and $\{x \geq 0\}$.
Thus we need only rule out the possibility
that $\mathsf{S}=\refl_{\{y=0\}}$.
However, if $\refl_{\{y=0\}}$
were a symmetry of $\tow_\vartheta$,
then $\refl_{\{z=0\}}$ would be too
(since $\refl_{\axis{x}}$ is)
and in turn $\trans^{\axis{z}}_\pi$
would be another symmetry
(since $\refl_{\{z=\pi/2\}}$ is),
in contradiction to the conclusion of
the previous paragraph.
This completes the proof.
\end{proof}

\begin{remark}[Uniqueness of $\tow_\vartheta$]
\label{towuniqueness}
In \cite{PerezTraizetClassificationSP}
P\'{e}rez and Traizet obtained a classification
of complete, embedded,
singly periodic minimal surfaces
with genus $0$ in the quotient (by the period translation)
and finitely many ends,
all asymptotically planar.
One can make use of their classification
to prove that $\tow_\vartheta$
is in fact \emph{characterized} by the properties
enumerated in Proposition \ref{karcher-scherk}.
More precisely,
$\tow_\vartheta$ is uniquely determined by
items \ref{item:karcher-i}--\ref{towchoice}
of Proposition \ref{karcher-scherk};
any surface satisfying
\ref{item:karcher-i}--\ref{towasymp}
is a translate of $\tow_\vartheta$;
any surface satisfying
\ref{item:karcher-i}--\ref{towasymp}
and \ref{item:towverticalplaneofsym}
is a translate of $\tow_\vartheta$
in the $\axis{z}$ direction;
and
there are precisely two surfaces satisfying
\ref{item:karcher-i}--\ref{item:towverticalplaneofsym}, namely $\tow_\vartheta$
and $\trans^{\axis{z}}_\pi\tow_\vartheta$
(the latter of which satisfies \ref{towchoice}
with $\{y<0\}$ in place of $\{y>0\}$).
Since, however,
we make no use of such characterizations
in this article,
we omit the proof.
\end{remark}

\begin{remark}[Alternative construction of $\tow_\vartheta$]
In \cite{KarcherScherk}
Karcher constructs the family $\tow_\vartheta$
(up to congruence) in two different ways
(whose equivalence,
though not needed in this article,
can be confirmed
with the aid of Remark \ref{towuniqueness}).
Instead of starting with Enneper-Weierstrass data
as we do in Appendix \ref{app:Karcher--Scherk},
one can follow the so-called conjugate construction,
whereby (in this particular application)
one
starts with the Dirichlet problem
for minimal graphs 
with infinite boundary data
on a certain family of convex equilateral hexagons
having data alternatingly $\pm\infty$
from side to side
(see Figure \ref{fig:hexagon}),
then takes the solution (unique up to vertical translation)
guaranteed by a result
of Jenkins and Serrin \cite{JenkinsSerrinII},
next passes to the conjugate minimal surface of this last graph,
and finally from this conjugate produces a complete surface
by repeated reflection.
We refer the reader to \cite{KarcherScherk}
or the lecture notes \cite{KarcherTokyo}
for details.
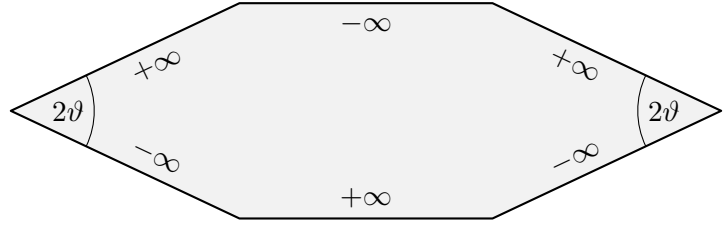
\begin{figure}%
\centering
\pgfmathsetmacro{\thetapar}{25.38}
\begin{tikzpicture}[line cap=round,line join=round,scale=\unitscale]
\draw[thick,fill=black!5](-0.5,0)
--++(1,0)node[midway,sloped,above]{$+\infty$}
--++(\thetapar:1)coordinate(vertex1)node[midway,sloped,above,pos=0.4]{$-\infty$}
--++(180-\thetapar:1)node[pos=0.6,sloped,below]{$+\infty$}
--++(-1,0)node[midway,sloped,below]{$-\infty$}
--++(180+\thetapar:1)coordinate(vertex2)node[pos=0.4,sloped,below]{$+\infty$}
--(-0.5,0)node[midway,sloped,above,pos=0.6]{$-\infty$}
--cycle;
\draw
(vertex1)++(180-\thetapar:0.33)arc(180-\thetapar:180+\thetapar:0.33)node[midway,right]{$2\vartheta$}
(vertex2)++(-\thetapar:0.33)arc(-\thetapar:\thetapar:0.33)node[midway,left]{$2\vartheta$};
\end{tikzpicture}
\caption{Equilateral hexagon with possible choice of Jenkins--Serrin boundary data.}%
\label{fig:hexagon}%
\end{figure}
\end{remark}

\begin{remark}[Maximally symmetric case]
It is not difficult to prove that
the special case $\vartheta=\pi/3$
admits additional symmetries.
Indeed we have
$
        \Aut_{\R^3}(\tow_{\pi/3})
        =
        \sk[\big]{
          \refl_{\{x=0\}},\,
          \refl_{\axis{x}},\,
          \refl_{\{z=\pi/2\}},\,
          \refl_{\{y = \sqrt{3}x\} \cap \{z=0\}}
        }
$,
but we omit the proof,
as the surface $\tow_{\pi/3}$
plays no role in our construction.
\end{remark}

In fact, recalling the definition of $\omega_0$ from Remark \ref{rem:values_for_b=0}, in this article we will only need to work with the tower 
\label{gls:tow_0}$\glsuserii{tow}\vcentcolon= \tow_{\omega_0}$ and the corresponding half tower (see Figure \ref{fig:halftower}) 
\begin{equation*}
\glsuseriii{tow} \vcentcolon= \tow \cap \{x \geq 0\}.
\end{equation*}

\begin{figure}%
\centering
\pgfmathsetmacro{\globalscale}{0.85}
\pgfmathsetmacro{\thetapar}{25.38}
\pgfmathsetmacro{\btow}{(sin(\thetapar)-tan(\thetapar)/2)*ln(2*sec(\thetapar)-1)-2*sin(\thetapar)*ln(sec(\thetapar)-1)}
\pgfmathsetmacro{\Rtow}{3.4}
\pgfmathsetmacro{\xmax}{(\textwidth-1pt)/2cm/\globalscale}
\pgfmathsetmacro{\ymax}{4.4}
\pgfmathsetmacro{\xint}{(sqrt((tan(\thetapar)^2+1)*\Rtow*\Rtow-\btow*\btow)-\btow*tan(\thetapar))/(tan(\thetapar)^2+1)}
\begin{tikzpicture}[line cap=round,line join=round,semithick,scale=\globalscale]
\clip(-\xmax,-\ymax)rectangle(\xmax,\ymax); 
\draw[line width=\globalscale*2.02cm,black]
(\Rtow,0)--(\xmax-0.1,0)
(\xint,{\btow+\xint*tan(\thetapar)})--(\xmax,{\btow+\xmax*tan(\thetapar)});
\draw[line width=\globalscale*1.98cm,black!8]
(\Rtow,0)--(\xmax-0.1,0)
(\xint,{\btow+\xint*tan(\thetapar)})--(\xmax,{\btow+\xmax*tan(\thetapar)}); 
\node[scale=\globalscale]{\includegraphics[page=5]{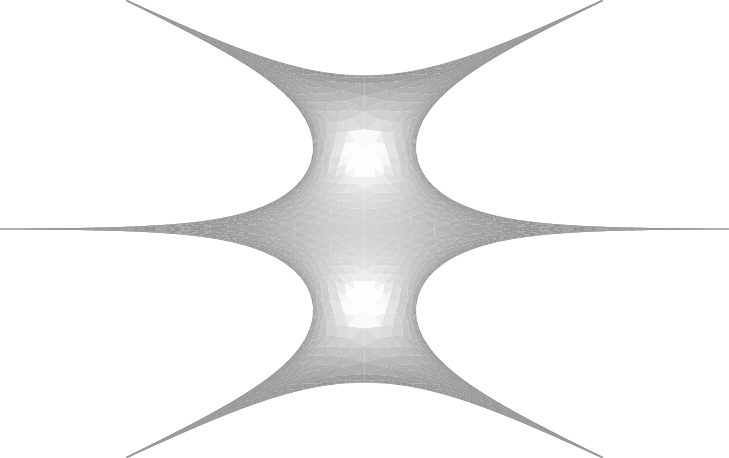}};
\draw[->](-\xmax,0)--(\xmax,0)node[below left]{$x$};
\draw[->](0,-\ymax)--(0,0)(0,\btow)--(0,\ymax)node[below right]{$y$};
\begin{scope}[color={cmyk,1:magenta,1;yellow,0.5},very thick]
\path(0,0)--(0,\btow)node[pos=0.62,inner sep=2pt](b){$b^{\mathrlap{\mathrm{tow}}}_{\mathrlap{\omega_0}}$};
\draw[stealth-](0,0)to(b);
\draw[stealth-](0,\btow)to(b);
\end{scope}
\draw[dashed]
(0, \btow)--({ (\ymax-\btow)/tan(\thetapar)}, \ymax)coordinate[pos=0.62](wing1)
(0, \btow)--({-(\ymax-\btow)/tan(\thetapar)}, \ymax)coordinate[pos=0.62](wing2)
(0,-\btow)--({ (\ymax-\btow)/tan(\thetapar)},-\ymax)coordinate[pos=0.62](wing5)
(0,-\btow)--({-(\ymax-\btow)/tan(\thetapar)},-\ymax)coordinate[pos=0.62](wing4);
\path
(0,0)--(\xmax,0)coordinate[pos=0.6](wing0)
(0,0)--(-\xmax,0)coordinate[pos=0.6](wing3);
\begin{scope}[black!50,very thick]
\draw[-latex](wing2)--++(-\thetapar-90:-1);
\draw[-latex](wing3)--++(90:-1);
\draw[-latex](wing4)--++(\thetapar-90:-1);
\draw[-latex](wing5)--++(-\thetapar-90:-1);
\end{scope}
\draw[ultra thick,color={cmyk,1:magenta,0.5;cyan,1}]
(\Rtow,0)--(\xmax-0.1,0)node[above,midway](H0){$\Pi^0$}
(\xint,{\btow+\xint*tan(\thetapar)})--(\xmax,{\btow+\xmax*tan(\thetapar)})node[above,sloped,pos=0.2](H1){$\Pi^1$};
\path(\xmax-1,0)--++(0,1)node[midway,inner sep=2pt](1m){$1$}coordinate[pos=0](1s)coordinate[pos=1](1e);
\draw[-stealth,very thin](1m)--(1s);
\draw[-stealth,very thin](1m)--(1e);
\draw[black!80](H0)++(0,-1)node{$\Pi^0_{\leq1}$};
\draw[black!80](H1)++(\thetapar-90:1)node[rotate=\thetapar]{$\Pi^1_{\leq1}$};
\draw[dotted](0,\btow)--++(\xmax,0);
\draw(1.75,\btow)arc(0:\thetapar:1.75)node[anchor=\thetapar/2,midway]{$\omega_0$};
\draw[dotted](\Rtow,0)arc(0:360:\Rtow); 	
\path(0,0)--(-40:\Rtow)node[pos=0.66,inner sep=2pt](Rtow){$\frac{1}{2}R_{\mathrm{tow}}$}coordinate[pos=1](Rtow1);
\draw[-stealth](Rtow)--(0,0);
\draw[-stealth](Rtow)--(Rtow1);
\draw(b)++(-1,0)node[left]{$\tow$};
\end{tikzpicture}
\caption{Asymptotics of $\tow$ in plan view and visualization of Remark \ref{towasymptotics}.}%
\label{fig:tow_asymptotics}%
%
%%%%%%%%%%%%%%%%%%%%%%%%%%%%%%%%%%%%%%%%%%%%%%%%%%%%
\medskip
%%%%%%%%%%%%%%%%%%%%%%%%%%%%%%%%%%%%%%%%%%%%%%%%%%%%
%
\pgfmathsetmacro{\thetaO}{45}
\pgfmathsetmacro{\phiO}{-123}
\pgfmathsetmacro{\radius}{9}
\tdplotsetmaincoords{\thetaO}{\phiO}
\begin{tikzpicture}[tdplot_main_coords,semithick,line cap=round,line join=round,scale=\globalscale]
\draw(0,0,-2*pi)
node[scale=\globalscale]{\includegraphics[page=7]{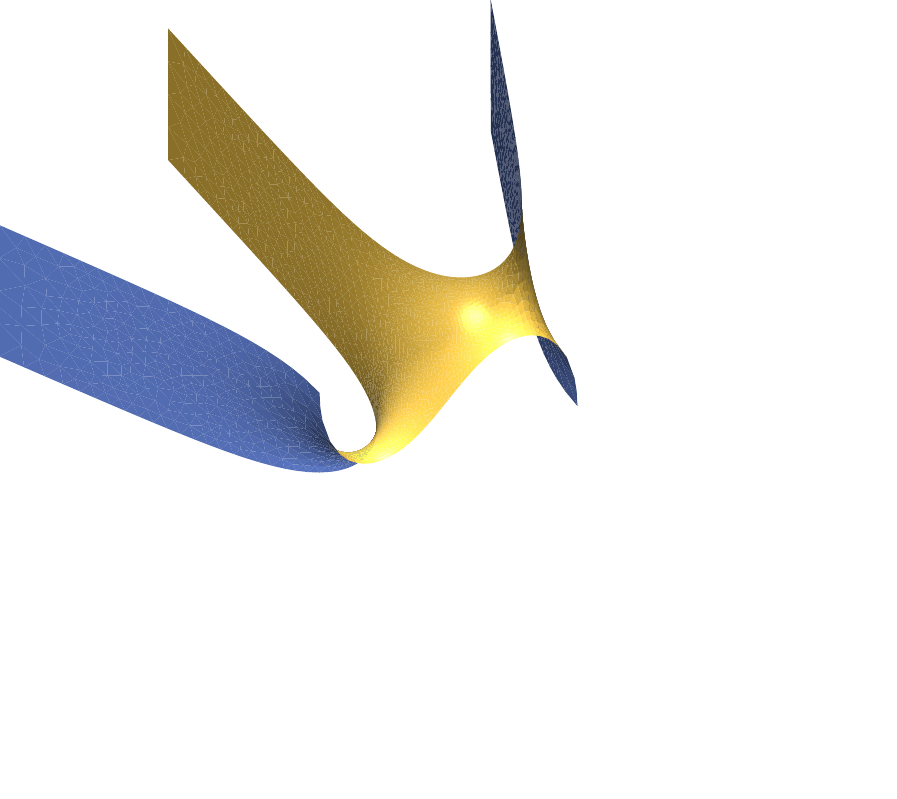}};
\draw(0,0,0)
node[scale=\globalscale]{\includegraphics[page=6]{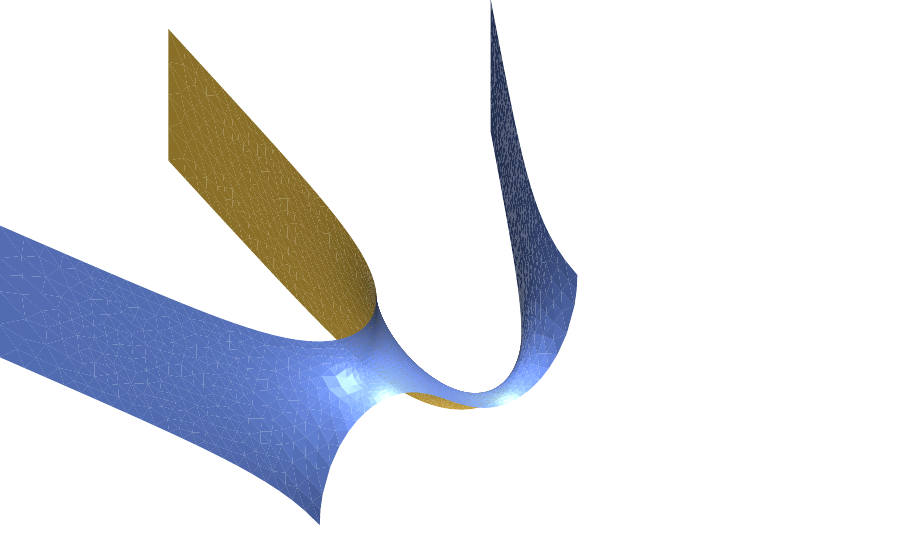}}
node[scale=\globalscale]{\includegraphics[page=7]{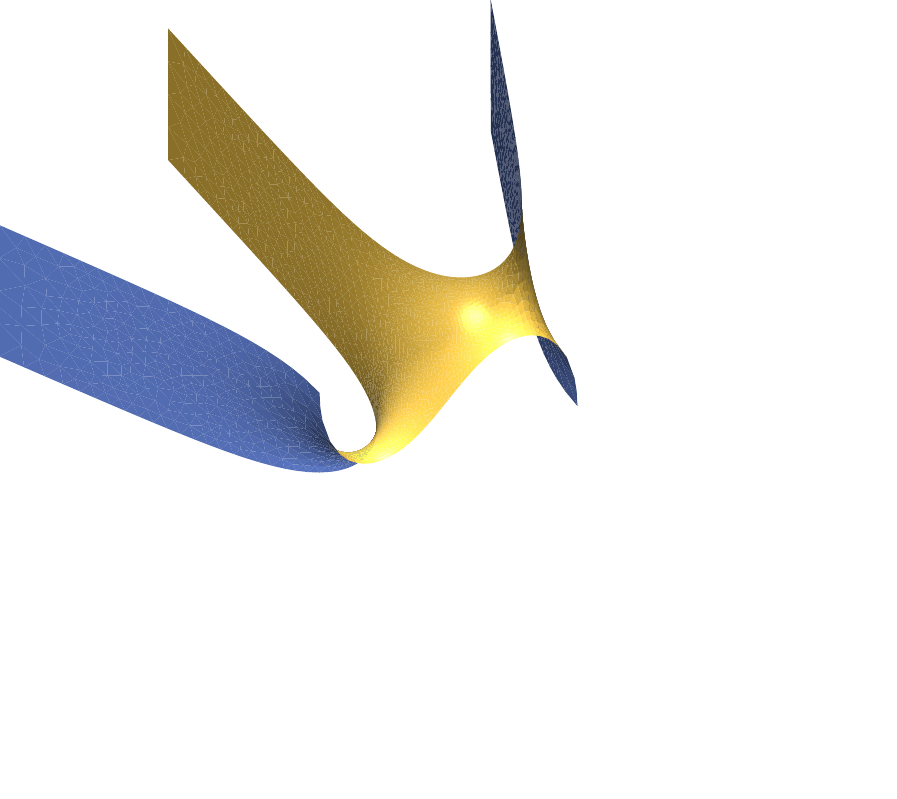}};
\draw[-latex](0,0,0)--(\radius,0,0)node[above left,inner sep=2pt,circle]{$x$};
\draw(0,0,2*pi)
node[scale=\globalscale]{\includegraphics[page=6]{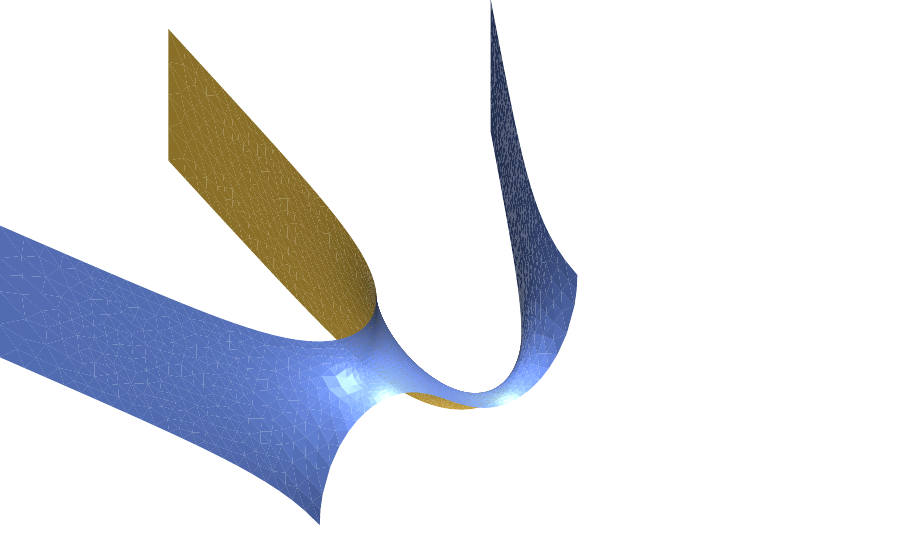}}
;
\pgfmathsetmacro{\radius}{\radius*0.77}
\draw[densely dashed]
(0,\btow,3*pi/2)--++({\radius*cos(\thetapar)},{\radius*sin(\thetapar)},0)
(0,-\btow,3*pi/2)--++({\radius*cos(\thetapar)},{-\radius*sin(\thetapar)},0)
(0,0,3*pi/2)--(0,-\btow,3*pi/2)
;
\tdplottransformmainscreen{0}{-1}{0}
\pgfmathsetmacro{\VecHx}{\tdplotresx}
\pgfmathsetmacro{\VecHy}{\tdplotresy}
\tdplottransformmainscreen{0}{0}{1}
\pgfmathsetmacro{\VecVx}{\tdplotresx}
\pgfmathsetmacro{\VecVy}{\tdplotresy}
\path[color={cmyk,1:magenta,1;yellow,0.5}] (0,0,3*pi/2)--++(0,\btow,0)node[pos=0.5,cm={\VecHx ,\VecHy ,\VecVx ,\VecVy ,(0,0)},inner sep=1pt](b){\large$b^{\mathrm{tow}}_{\omega_0}$};
\draw[stealth-,color={cmyk,1:magenta,1;yellow,0.5},very thick](0,0,3*pi/2)to(b);
\draw[stealth-,color={cmyk,1:magenta,1;yellow,0.5},very thick](0,\btow,3*pi/2)to(b);
\draw[dotted](0,2.5965,pi/2)--(0,0,pi/2);
\draw[](0,0.08,pi/2)--(0,-0.08,pi/2)node[inner sep=1pt,right,cm={\VecHx ,\VecHy ,\VecVx ,\VecVy ,(0,0)}]{$\dfrac{\pi}{2}$};
\tdplottransformmainscreen{0}{-1}{0}
\pgfmathsetmacro{\VecHx}{\tdplotresx}
\pgfmathsetmacro{\VecHy}{\tdplotresy}
\tdplottransformmainscreen{1}{0}{0}
\pgfmathsetmacro{\VecVx}{\tdplotresx}
\pgfmathsetmacro{\VecVy}{\tdplotresy}
\tdplotdrawarc[]{(0,-\btow,3*pi/2)}{1/2*\radius}{0}{0-\thetapar}{cm={\VecHx ,\VecHy ,\VecVx ,\VecVy ,(0,0)},anchor=60}{\large$\omega_0$}
\draw[densely dashed](0,-\btow,3*pi/2)--++(\radius,0,0);
\coordinate(eckpunkt)at (current bounding box.north west);
\pgfresetboundingbox
\path(eckpunkt);
\draw[-latex](0,0,-3*pi+pi/2)--(0,0,2*pi)node[above]{$z$};
\draw[-latex](0,-2/3*\radius,0)--(0,2/3*\radius,0)node[below left,inner sep=2pt,circle]{$y$};
\end{tikzpicture}
\caption{Two vertical periods of the half tower $\tow^+$. }%
\label{fig:halftower}%
\end{figure}

\begin{remark}[Asymptotics of $\tow$]
\label{towasymptotics}
As promised, we now elaborate on
item \ref{towasymp} of Proposition \ref{karcher-scherk},
in the special case of $\vartheta=\omega_0$.
In the following we refer to
Lemma \ref{lem:btow_theta}
for the definition of $\gls{btow}$
and we observe that
$b^{\mathrm{tow}}_\vartheta$
is strictly positive
for $\vartheta \in \interval{0,\pi/3}$,
so that
by Remark \ref{rem:values_for_b=0}
we have in particular
$b^{\mathrm{tow}}_{\omega_0}>0$.
In fact, we may compute $b^{\text{\text{tow}}}_{\omega_0} \approx 1.95$ numerically, using formula \eqref{eqn:btow_theta} from Appendix \ref{app:Karcher--Scherk}. 

Item \ref{towEWrep} of Proposition \ref{karcher-scherk},
Proposition \ref{prop:towerEWsummary},
and Lemma \ref{lem:CylMSE}
now imply the existence of
$R_{\mathrm{tow}}>1$
such that $\tow$ has the following asymptotic
description.
First,
recalling the notation introduced in \eqref{tubularnbhd},
we have
\begin{equation*}
\tow \cap \axis{z}_{\geq R_{\mathrm{tow}}/2}
\cap \{x \geq 0\} \cap \{y \geq -1\}
\subset(\Pi^0 \cup \Pi^1)_{\leq 1},
\end{equation*}
for which we define the half planes
\begin{align*}
\Pi^0
  &\vcentcolon=
  \{y=0\} \cap \{x \geq 0\}
  \cap \axis{z}_{\geq R_{\mathrm{tow}}/2},
\\
\Pi^1&\vcentcolon=
\{y=b^{\mathrm{tow}}_{\omega_0}+x \tan \omega_0\}
\cap \{x \geq 0\}
\cap \axis{z}_{\geq R_{\mathrm{tow}}/2}.
\end{align*}
In more detail, for each $i=0,1$ there is a smooth $\Aut_{\R^3}(\Pi^i)\cap \Aut_{\R^3}(\tow)$-equivariant function $w^i\colon \Pi^i \to \R$
satisfying, for any integer $k \geq 0$, the estimate 
\begin{equation}\label{eq:ExpDecayDefFunctions}
\nm{w^i}_{k,0,1} \leq C(k),
\end{equation}
for some $C(k)>0$
(and for which we recall the notation
introduced below \eqref{eqn:definition_weighted_norm}),
and such that
(recalling the notation \eqref{eqn:definition_graph})
each of the graphs
\begin{equation*}
W^i \vcentcolon= \graph(w^i \nu^i)
\end{equation*}
for $i\in\{0,1\}$ with $\nu^0\vcentcolon =-\partial_y$ and
$\nu^1 \vcentcolon=
 \cos \omega_0 \, \partial_y
 -\sin \omega_0 \, \partial_x$, 
is contained in $\tow$,
and conversely
\begin{equation*}
\tow \cap \axis{z}_{\geq R_{\mathrm{tow}}} \cap \{x \geq 0\}
\subset W^0 \cup W^1 \cup \refl_{\axis{x}}W^1.
\end{equation*}
We call $W^0$, $W^1$, $\refl_{\axis{x}}W^1$ and the three corresponding images under $\refl_{\{x=0\}}$ the \emph{wings} of $\tow$.
For the sake of notational convenience we agree, from now onwards, to write $W^{-1}$ in lieu of $\refl_{\axis{x}}W^1$ for the ``lower'' wing of $\tow^+$.
\end{remark}

Given any positive integer $n\in\N^{\ast}$ we define also the canonical projection 
\begin{equation}
\label{quotientbytrans} 
\glsuseri{varpi}
\colon
\R^3 \to \R^3/\sk{\trans^{\axis{z}}_{2n \pi}}
\end{equation}
and the quotients
\begin{equation}
\label{towquotient}
\glsuseri{towquot} %\towquot_{(n)}
\vcentcolon=\varpi_{(n)}(\tow)
\quad \mbox{ and } \quad
\glsuserii{towquot} %\towquot_{(n)}^+
\vcentcolon=\varpi_{(n)}(\tow^+).
\end{equation}
Note that the translational symmetries
of $\tow$ descend (trivially when $n=1$)
to $\towquot_{(n)}$ and $\towquot_{(n)}^+$.
In the simplest and most important case $n=1$ we agree to simply write $\glsuserii{varpi}$ in lieu of $\varpi_{(1)}$ and 
\glsuseriii{towquot} %$\towquot$
in lieu of $\towquot_{(1)}$.
Focusing on that case, we further note that the reflectional symmetries also descend to the quotient, 
in the sense that there exist unique isometries
$\mathsf{H}$, $\mathsf{L}$, and $\mathsf{V}$
of $\R^3/\sk{\trans^{\axis{z}}_{2\pi}}$ such that
\begin{equation}
\label{HLV}
\begin{aligned}
\mathsf{H} \circ \varpi
&=
  \varpi \circ \refl_{\{z=n\pi/2\}}
  \quad \forall n \in 2\Z+1, \\
\mathsf{V} \circ \varpi
&=
  \varpi \circ \refl_{\{x=0\}}, \\
\mathsf{L} \circ \varpi
&=
  \varpi \circ \refl_{\{y=0\} \cap \{z=n\pi\}}
  \quad \forall n \in \Z.  \\
\end{aligned}
\end{equation}
Note also that $\mathsf{L}$ and $\mathsf{H}$ restrict to isometries of 
\glsuseriv{towquot} %$\towquot^+$,
while $\towquot=\towquot^+ \cup \mathsf{V}\towquot^+$
with
$
 \towquot^+ \cap \mathsf{V}\towquot^+
 =
 \partial \towquot^+
$,
that is the fixed point set of $\mathsf{V}$.
Given any function $u$ on $\towquot$ and for each $\mathsf{S} \in \{\mathsf{H},\mathsf{L},\mathsf{V}\}$
we further define the projection
\begin{equation*}
\pi_{\mathsf{S}}^\pm u
\vcentcolon=
\frac{1}{2}\left(u\pm \bigl( u \circ \mathsf{S}|_{\towquot} \bigr)\right);
\end{equation*}
when $u$ is instead a function on $\towquot^+$
or $\partial \towquot^+$, the same equation also makes sense provided
$\mathsf{S} \in \{\mathsf{L}, \mathsf{H}\}$. 
Since $\mathsf{H}$, $\mathsf{L}$, and $\mathsf{V}$
commute pairwise,
the operators
\begin{equation}
\label{equivariantprojectorsonquotient}
\glsuserii{piAut}%\pi_{\Aut(\towquot^+)}
  \vcentcolon=
  \pi_{\mathsf{L}}^- \pi_{\mathsf{H}}^+
\quad \mbox{and} \quad
\glsuseri{piAut}%\pi_{\Aut(\towquot)}
  \vcentcolon=
  \pi_{\mathsf{V}}^+\pi_{\Aut(\towquot^+)}
\end{equation}
are also projectors, i.\,e. they are idempotent operators.

\begin{remark}[Equivariance in the quotient]
\label{rem:equiv_in_quot}
Note that if $u$ is a function on $\towquot$
(or $\towquot^+$ or $\partial \towquot^+$),
then $\varpi|_{\towquot}^*u$
(or $\varpi|_{\tow^+}^*u$
or $\varpi|_{\partial \tow^+}^*u$)
is $\Aut_{\R^3}(\tow)$-equivariant
(or $\Aut_{\R^3}(\tow^+)$-equivariant)
if and only if
$u=\pi_{\Aut(\towquot)}u$
(or $u=\pi_{\Aut(\towquot^+)}u$).
\end{remark}

Note also that
if $u$ belongs to the kernel of $\pi_{\Aut(\towquot)}$
and $v$ to its image,
then the product $uv$ is odd under at least
one of $\mathsf{H}, \mathsf{L}, \mathsf{V}$.
In particular the kernel and image
of $\pi_{\Aut(\towquot)}|_{L^2(\towquot)}$
are $L^2(\towquot)$-orthogonal;
equivalently
$\pi_{\Aut(\towquot)}|_{L^2(\towquot)}$
is self-adjoint: 
\begin{equation}
\label{equivariant_projector_is_self_adjoint}
\sk[\Big]{\pi_{\Aut(\towquot)}u,~ v}_{L^2(\towquot)}
=\sk [\Big]{u,~ \pi_{\Aut(\towquot)}v}_{L^2(\towquot)} 
\text{ for any $u,v \in L^2(\towquot)$.}
\end{equation}

Finally, for any integer $k \geq 0$,
any $\alpha \in \Interval{0,1}$,
and any $\beta \in \R$
we define on $\tow$ the norm
\begin{equation}
\label{weighted_norm_on_tow}
%\nm{{}\cdot{}}_{k,\alpha,\beta}
\gls{norm}
\vcentcolon=
\nm[\Big]{{}\cdot{}
    \colon C^{k,\alpha}\bigl(\tow, e^{-\beta \dist_{\axis{z}}}\bigr)},
\end{equation}
recalling \eqref{eqn:definition_weighted_norm}.
For domains of $\tow$
(and for $\tow^+$ in particular)
we apply exactly the same
notation with the same weight function.
One could equivalently employ on each wing of the tower a definition in the spirit of that introduced in the second bullet following \eqref{eqn:definition_weighted_norm} (e.\,g. using the coordinate $x$, hence the weight $|x|$, along the two horizontal asymptotic half planes, and similarly for the other four half planes) and design a global norm on $\tow$ using cutoff functions.
Of course, it is readily checked that the two norms in question
are equivalent.
Note last that
the function $\dist_{\axis{z}}$ descends to
each quotient $\towquot_{(n)}$ and $\towquot_{(n)}^+$,
so we apply the notation
\eqref{weighted_norm_on_tow}
on these surfaces as well.

\subsection{Construction of the initial surfaces}

\paragraph{Half towers with broken and straightened wings.}
In a neighborhood of the equator the initial surfaces will be modeled on the half tower $\tow^+$ visualized in Figure \ref{fig:halftower}.
Below we will define maps to scale down
and transplant $\tow^+$ from $\R^3$
to $\B^3$, wrapping its axis of periodicity around the equator.
Obviously, such maps will need to deform the tower substantially,
but actually it will be convenient to perform two
preliminary deformations now, in the preimage.
First, we introduce a prescribed \emph{dislocation}
on the wing $W^1$,
breaking it at some distance from $\axis{z}$,
translating the resulting noncompact component
in the $y$-direction as desired,
and then rejoining the two pieces by smooth interpolation.
This operation adjusts the mean curvature
of the initial surface, near the dislocation site,
in a way that will be needed to control the approximate
cokernel encountered in the linearized problem
on the towers.
Second, we simply straighten all the wings
to coincide exactly with their respective asymptotic planes
far away from $\axis{z}$.

We will define the towers thus deformed as graphs over $\tow^+$.
To localize the modifications as described
we will make use  of the smooth $\Aut_{\R^3}(\tow^+)$-equivariant
cutoff function (see Figure \ref{fig:Psi-dislocate}) 
\begin{equation}
\label{Psidislocate}
\Psi^{\mathrm{dislocate}}
  \vcentcolon=
  \left.
  \frac{y}{\abs{y}}
    \cdot
    \left(
      \cutoff{0}{1} \circ \abs{y}
    \right)
    \cdot
    \left(
      \cutoff{R_{\mathrm{tow}}}{R_{\mathrm{tow}}+1}
      \circ \dist_{\axis{z}}
    \right)
    \right|_{\tow^+}
\end{equation}
as well as, for any $\glsuseri{m}>0$, the smooth $\Aut_{\R^3}(\tow^+)$-invariant
cutoff function
\begin{equation}
\Psi^{\mathrm{straighten}}_m
  \vcentcolon=
      \cutoff{m^{3/4}}{m^{3/4}+1}
      \circ \dist_{\axis{z}}|_{\tow^+}.
\end{equation}
Write $\nu_{\tow^+}$ for the global unit normal on $\tow^+$
having positive inner product on $W^1$ with $\partial_y$ (which, in particular, is consistent with the choice of $\nu^0$ on $W^0$ and $\nu^1$ on $W^1$ as given in Remark \ref{towasymptotics}).
As a consequence of Remark \ref{towasymptotics},
for $m$ sufficiently large there exists a function $f$
on $\tow^+ \cap \axis{z}_{\geq m^{3/4}}$
such that $\graph(f\nu_{\tow^+}|_{\dom(f)}) \subset \Pi^i$
for $i=0,1$
and for each integer $k \geq 0$ there holds an estimate of the form
$\nm{f}_{k,0,1} \leq C(k)$
for some positive constant $C(k)$.
(Here, and throughout the paper, we write $\dom(F)$
for the domain of a given function $F$.)

\begin{figure}%
\centering
\pgfmathsetmacro{\globalscale}{0.5}
\pgfmathsetmacro{\thetapar}{25.38}
\pgfmathsetmacro{\btow}{(sin(\thetapar)-tan(\thetapar)/2)*ln(2*sec(\thetapar)-1)-2*sin(\thetapar)*ln(sec(\thetapar)-1)}
\pgfmathsetmacro{\Rtow}{2*3.4}
\pgfmathsetmacro{\xmax}{(\textwidth-1pt)/2cm/\globalscale}
\pgfmathsetmacro{\ymax}{\Rtow+2}
\pgfmathsetmacro{\xint}{(sqrt((tan(\thetapar)^2+1)*\Rtow*\Rtow-\btow*\btow)-\btow*tan(\thetapar))/(tan(\thetapar)^2+1)}
\begin{tikzpicture}[line cap=round,line join=round,semithick,scale=\globalscale]
\draw[->](0,0)--(\xmax,0)node[below left]{$x$};
\draw 
(0, \btow)--({ (\Rtow-\btow)/tan(\thetapar)}, \Rtow) 
(0,-\btow)--({ (\Rtow-\btow)/tan(\thetapar)},-\Rtow) 
; 
\draw[dotted](0,\Rtow)arc(90:-90:\Rtow); 	
\draw[dotted](0,\Rtow+1)arc(90:-90:\Rtow+1); 	
\draw(\Rtow*2/3,\btow)arc(0:\thetapar:\Rtow*2/3)node[anchor=\thetapar/2,midway]{$\omega_0$}
;
\begin{scope}[ultra thick,black!50]
\draw(0,0)--(\xmax-0.2,0)node[below=0.0pt,pos=0.4,fill=white]{$\Psi^{\mathrm{dislocate}}=0$};
\clip(-0,-\ymax)rectangle(\xmax,\ymax);  
\clip(0,0)circle(\Rtow);
\draw
(0, \btow)--({ (\ymax-\btow)/tan(\thetapar)}, \ymax)
(0,-\btow)--({ (\ymax-\btow)/tan(\thetapar)},-\ymax)
;
\end{scope}
\begin{scope}
\clip(-0,-\ymax)rectangle(\xmax,\ymax);  
\node[scale=\globalscale]{\includegraphics[page=5]{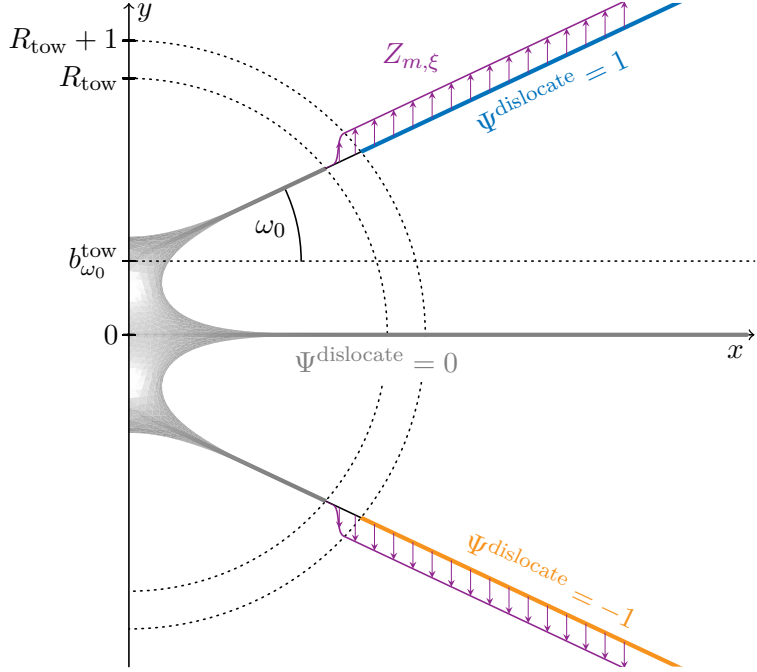}};
\clip(0,\Rtow+1)arc(90:-90:\Rtow+1)--(0,-\ymax)-|(\xmax,\ymax)-|cycle;
\draw[ultra thick,color={cmyk,1:magenta,0.5;cyan,1}]
(\xint,{\btow+\xint*tan(\thetapar)})--(\xmax,{\btow+\xmax*tan(\thetapar)})node[below,sloped,pos=0.5]{$\Psi^{\mathrm{dislocate}}=1$};
\draw[ultra thick,color={cmyk,1:magenta,0.5;yellow,1}]
(\xint,{-\btow-\xint*tan(\thetapar)})--(\xmax,{-\btow-\xmax*tan(\thetapar)})node[above,sloped,pos=0.5]{$\Psi^{\mathrm{dislocate}}=-1$}
;
\end{scope}
\draw[->](0,-\ymax)--(0,\ymax);
\draw(0.2,\ymax)node[below right,inner sep=0]{$y$}; 
\draw[dotted](0,\btow)--++(\xmax,0);
\draw plot[hdash](0,\Rtow)node[left]{$R_{\mathrm{tow}}$};
\draw plot[hdash](0,\Rtow+1)node[left]{$R_{\mathrm{tow}}+1$};
\draw plot[hdash](0,\btow)node[left]{$b^{\mathrm{tow}}_{\omega_0}$};
\draw plot[hdash](0,0)node[left]{$0$}; 
\begin{scope}[color={cmyk,1:magenta,1;cyan,0.5}]
\clip(-0,-\ymax)rectangle(\xmax,\ymax);  
\clip (0, \btow)--({ (\xmax-\btow)/tan(\thetapar)}, \xmax)-|cycle 
;
\draw(\xint,{\btow+\xint*tan(\thetapar)})--++(\thetapar:0.2)..controls+(\thetapar:0.25) and+(\thetapar:-0.25)..++(\thetapar+45:0.9)coordinate[pos=0.7](p0)--++(\thetapar:\xmax)node foreach[count =\i] \t in {0.02,0.054,...,1}[pos=\t](p\i){};
\foreach\i in {0,...,15}{
\draw[stealth-,very thin](p\i.center)--++(0,-1); 
}
\draw(p6)node[above left]{$Z_{m,\xi}$};
\end{scope}
\begin{scope}[color={cmyk,1:magenta,1;cyan,0.5},yscale=-1]
\clip(-0,-\ymax)rectangle(\xmax,\ymax);  
\clip (0, \btow)--({ (\xmax-\btow)/tan(\thetapar)}, \xmax)-|cycle 
;
\draw(\xint,{\btow+\xint*tan(\thetapar)})--++(\thetapar:0.2)..controls+(\thetapar:0.25) and+(\thetapar:-0.25)..++(\thetapar+45:0.9)coordinate[pos=0.7](p0)--++(\thetapar:\xmax)node foreach[count =\i] \t in {0.02,0.054,...,1}[pos=\t](p\i){};
\foreach\i in {0,...,15}{
\draw[stealth-,very thin](p\i.center)--++(0,-1); 
}
\end{scope}
\end{tikzpicture}
\caption{Plan view of the cutoff function $\Psi^{\mathrm{dislocate}}$ and the vector field $Z_{m,\xi}$ on the half tower $\tow^+$. }%
\label{fig:Psi-dislocate}%
\end{figure}

Given also $\gls{xi}\in \R$,
we define on $\tow^+$ the
$\Aut_{\R^3}(\tow^+)$-equivariant vector field (see Figure \ref{fig:Psi-dislocate}) 
\begin{equation}
\label{vectorfieldefiningbenttower}
Z_{m,\xi}
\vcentcolon=
\frac{\xi}{m}\Psi^{\mathrm{dislocate}} \, \partial_y
  + f \Psi^{\mathrm{straighten}}_m
    \, \nu_{\tow^+} 
\end{equation}
and, in turn, 
the deformed half tower 
$\towbent_{m,\xi}^+$
and parametrization
$\widehat{\phi}_{m,\xi}\colon\tow^+\to\towbent_{m,\xi}^+$ by
\begin{equation}
\label{towbent}
\gls{towbent}%\towbent_{m,\xi}^+
  \vcentcolon=
  \graph(Z_{m,\xi})
\quad \mbox{and} \quad
\widehat{\phi}_{m,\xi}(p) 
  \vcentcolon=
  p + Z_{m,\xi}(p).
\end{equation}
Finally we also define
\begin{equation}
\label{towcokerandgen}
\glsuseri{towcoker}%\towcokergen
  \vcentcolon=
  (\nu_{{\tow^+}} \cdot \partial_y)\Psi^{\mathrm{dislocate}}
\quad \mbox{and} \quad
\glsuserii{towcoker}%\towcoker
  \vcentcolon=-
  J_{\tow^+} \towcokergen.
\end{equation}
The function $\nu_{{\tow^+}} \cdot \partial_y$
is of course the Jacobi field on $\tow^+$ generating translations
in the $y$-direction,
so $\towcokergen$ can be interpreted as the generator of
\emph{dislocations} (following the terminology of Kapouleas)
of the wings $W^1$ and $W^{-1}=\refl_{\axis{x}}W^1$ -- fixing
$W^0$ and the core of $\tow^+$ but translating most of
these other two wings in opposite directions -- and
$\towcoker$ as the correspondingly induced mean curvature, to first order,
by virtue of the minimality of $\tow^+$.

In the following lemma we summarize the properties of $\towbent_{m,\xi}^+$
that are important for our construction.

\begin{lemma}[Dislocation and straightening estimates]
\label{lemmatowbent}
For each $c>0$ there exists $m_0(c)>0$
such that the following hold
for every $m>m_0(c)$
and every $\xi \in \IntervaL{-c,c}$:
\begin{enumerate}[label={\normalfont(\roman*)}]
\item $\towbent_{m,\xi}^+$
      is a connected, smooth,
      properly embedded
      surface meeting $\{x=0\}$ orthogonally
      along a smooth connected curve;
\item $\Aut_{\R^3}(\towbent_{m,\xi}^+)
       =\Aut_{\R^3}(\tow^+)$;
\item \label{towbent_action_of_phi_hat}
      $\widehat{\phi}_{m,\xi}$ is a diffeomorphism
      restricting to the identity on
      $\tow^+
       \setminus 
       \bigl((W^0 \cap \{\dist_{\axis{z}}>m^{3/4}\}) \cup W^1 \cup W^{-1}\bigr),
      $
      to $\trans^{\axis{y}}_{\xi/m}$ on
      $
       W^1 \cap 
        \{
         R_{\mathrm{tow}}+1<\dist_{\axis{z}}<m^{3/4}
        \}
      $ and to $\trans^{\axis{y}}_{-\xi/m}$ on
      $
       W^{-1} \cap 
        \{
         R_{\mathrm{tow}}+1<\dist_{\axis{z}}<m^{3/4}
        \}
      $;
\item the map
      $
       \iota_{{\towbent_{m,\xi}^+}}
       \circ
       \widehat{\phi}_{m,\xi}
       \colon
       \tow^+ \to \R^3
      $
      is smooth in $\xi$,
      $
       \iota_{{\towbent_{m,\xi}^+}}
       \colon
       \towbent_{m,\xi}^+ \to \R^3
      $ being the inclusion;
\item\label{lemmatowbent-wings} the wing
      $
       \widehat{W}^0_{m,\xi}
       =\widehat{W}^0_m
       \vcentcolon=
       \widehat{\phi}_{m,\xi}(W^0)
       \subset \towbent_{m,\xi}^+
      $
      is asymptotic to (and eventually coincides with) the half plane
      \[\widehat{\Pi}^0_{m,\xi}
       =\widehat{\Pi}^0
       \vcentcolon=
       \{y=0\}
        \cap 
        \{x \geq 0\} \cap \axis{z}_{\geq R_{\mathrm{tow}}/2},
      \]
      the wing
      $
       \widehat{W}^1_{m,\xi}
       \vcentcolon=
       \widehat{\phi}_{m,\xi}(W^1)
       \subset \towbent_{m,\xi}^+
      $
      is asymptotic to (and eventually coincides with) the half plane
      \[\widehat{\Pi}^1_{m,\xi}
       \vcentcolon=
       \{y=b_{\omega_0}^{\mathrm{tow}} +m^{-1}\xi +x \tan \omega_0\}
        \cap 
        \{x \geq 0\} \cap \axis{z}_{\geq R_{\mathrm{tow}}/2};
      \]
      for each $i=0,1$ there exists a function $\widehat{w}^i_{m,\xi}\in C^\infty(\widehat{\Pi}_{m,\xi}^i)$
      such that $\widehat{W}_{m,\xi}^i = \graph(\widehat{w}^i_{m,\xi} \widehat{\nu}^i)$,
      with
      $\widehat{\nu}^0 \vcentcolon=-\partial_y$
      and
      $
       \widehat{\nu}^1
       \vcentcolon=
       (\cos \omega_0 \, \partial_y
        - \sin \omega_0 \, \partial_x)
      $;  
\item
\label{towbent_quantitative}
    for each integer $k\geq0$
    there exists a constant $C(k)>0$
    (independent of $m$, $c$, and $\xi$)
    such that we have the estimates
\begin{enumerate}[label={\normalfont(\roman{enumi}.\roman*)}]
\item
\label{towbent_wing_defining_func_est}
$
   \nm{\widehat{w}^i_{m,\xi}}_{k,0,1}
   \leq C(k)
  $
  for each defining function
  $\widehat{w}^0_{m,\xi}$, $\widehat{w}^1_{m,\xi}$
  of the wings as in \ref{lemmatowbent-wings},
\item
\label{towbent_metric_est}
$
   \nm[\big]{
         \widehat{\phi}_{m,\xi}^*
           g^{\vphantom{|}}_{\towbent_{m,\xi}^+}
           -g^{\vphantom{|}}_{\tow^+}
    }_{k,0,0}
  \leq
  C(k)\bigl(m^{-1}\abs{\xi}+e^{-m^{3/4}}\bigr)
  $
  for the metrics
  $g^{\vphantom{|}}_{\towbent_{m,\xi}^+}$
  and $g^{\vphantom{|}}_{\tow^+}$
  induced on $\towbent_{m,\xi}^+$
  and $\tow^+$ respectively by the ambient
  Euclidean metric,
\item
\label{towbent_mc_disloc_est}
  $
  \nm[\Big]{
    \bigl(\widehat{\phi}_{m,\xi}^*H_{\towbent_{m,\xi}^+}
    -m^{-1}\xi\towcoker\bigr)
      \big\vert_{\axis{z}_{\leq 3R_{\mathrm{tow}}}}
  }_{k,0,0}
  \leq C(k)m^{-2}\xi^2$
  and
\item
\label{towbent_mc_straighten_est}
  $
  \nm[\Big]{
    \bigl(\widehat{\phi}_{m,\xi}^*
    H_{\towbent_{m,\xi}^+}\bigr)
      \big\vert_{\axis{z}_{\geq 2R_{\mathrm{tow}}}}
  }_{k,0,0}
  \leq C(k)e^{-m^{3/4}}
  $
  for the mean curvature
  $H_{\towbent_{m,\xi}^+}$
  of $\towbent_{m,\xi}^+$;
  and 
\end{enumerate}
\item
 $\widehat{\phi}_{m,\xi}^*H_{\towbent_{m,\xi}^+}$
      has support contained in
\[
 S_m
 \vcentcolon=
 \left\{\dist_{\axis{z}}|_{\tow^+}
     \in
      \interval{m^{3/4},m^{3/4}+1}
  \right\}
  \cup
  \left(
    (W^1\cup W^{-1})
    \cap
     \{
       \dist_{\axis{z}}
       \in
       \interval{R_{\mathrm{tow}}, R_{\mathrm{tow}}+1}
     \}
  \right).
  \]
\end{enumerate}
\end{lemma}

\begin{proof}
All claims,
except for item \ref{towbent_quantitative},
are immediate consequences of the definition \eqref{towbent}, based on the results proven in Proposition \ref{karcher-scherk} and Remark \ref{towasymptotics}.
Item \ref{towbent_wing_defining_func_est}
we know by virtue of the corresponding estimate
for the defining functions in Remark \ref{towasymptotics};
the distortion introduced by the second term of
\eqref{vectorfieldefiningbenttower}
is of order $e^{-m^{3/4}}$
while that introduced by the first term
has support contained in $\axis{z}_{\leq 2R_{\mathrm{tow}}}$
and is of order $\xi/m$
(and our assumptions allow us to take $m$ large in terms of $\xi$).
For items
\ref{towbent_metric_est}--\ref{towbent_mc_straighten_est}
we use the facts
that the induced metric and mean curvature
of the graph of a vector field
(in the sense of definition \eqref{eqn:definition_graph})
depend smoothly on the vector field
and that the background geometry
(the geometry of $\tow^+$)
is bounded.
For example, item \ref{towbent_metric_est}
can be obtained from the identity
\begin{equation*}
\widehat{\phi}_{m,\xi}^*g^{\vphantom{|}}_{\towbent_{m,\xi}^+}
  -g^{\vphantom{|}}_{\tow^+}
=
\int_0^1 \partial_t
  \widehat{\phi}_{m,t\xi}^*g^{\vphantom{|}}_{\towbent_{m,t\xi}^+}
  \, dt,
\end{equation*}
because the integrand is controlled by $Z_{m,\xi}$,
which we can estimate directly from
\eqref{vectorfieldefiningbenttower} and the supporting definitions.
The mean curvature estimates can established
in similar fashion;
the estimate \ref{towbent_mc_disloc_est}
is quadratic in $\xi/m$
because we have subtracted the first-order term
on the left-hand side.
\end{proof}

As previewed above,
the initial surfaces
will be constructed,
in part,
by scaling down
the deformed towers
$\towbent_{m,\xi}^+$
and wrapping them
(as described in detail below)
around the equator.
In fact the parameter $m$,
up to now a sufficiently large real number,
will also play the role of scale factor
-- that is we will work with
$\frac{1}{m}\towbent_{m,\xi}^+$ --
and to accommodate the wrapping
we will henceforth restrict to positive integral values of $m$.
For future reference,
we recall definitions \eqref{quotientbytrans}
and \eqref{towquotient}
and we define the diffeomorphism
$
 \widetilde{\phi}_{m,\xi}
 \colon
 \towquot^+_{(m)}
 \to
 \varpi(\frac{1}{m}\towbent_{m,\xi}^+)
$
by
\begin{equation}
\label{scaledphibentquot}
\widetilde{\phi}_{m,\xi}(x,y,z+2\pi m\Z)
\vcentcolon=
\frac{1}{m}\widehat{\phi}_{m,\xi}(x,y,z)
  + (0,0,2\pi\Z).
\end{equation}

\paragraph{Equatorial coordinates.}
In order to import and deform the towers as needed
from $\R^3$ to $\B^3$
we define a map
$\gls{Phi}\colon \R^3 \to \R^3$ as follows.
First of all, to help distinguish the domain and target 
we reserve the coordinate labels $(x,y,z)$ for the target
and relabel to $(\sigma, \psi, \theta)$ for the domain.
The  map $\Phi$
will then furnish (local) coordinates on the target,
whereby $\theta$ and $\psi$ are respectively
longitudinal and latitudinal angles
and $\sigma$ is directed distance from $\partial \B^3$,
increasing toward the origin
(see Figure \ref{fig:equatorial_coordinates}). 
To emphasize the general idea behind the definition
we first formulate it somewhat abstractly as
\begin{equation}
\label{PhiR3def}
  \Phi(\sigma, \psi, \theta)
  \vcentcolon=
  \tubularexp_{(\R^3,g_{\mathrm{euc}}),(\Sp^2,\nu),(\Sp^1,\eta)}
    (\gamma(\theta),\psi,\sigma)
  =
  \exp^{\R^3}_{\exp^{\Sp^2}_{\gamma(\theta)} \psi \eta(\theta)}
    \sigma \nu,
\end{equation}
where $\gamma$ is the unit-speed, positively directed parametrization of the equator,
$\eta$ is the upward unit conormal along $\gamma$ in $\Sp^2$,
$\nu$ is the inward unit normal to $\Sp^2$ in $\R^3$,
each $\exp$ is the exponential map on the indicated manifold with its standard metric,
and we recall definition \eqref{tubularcoordscodim2}.
In particular, we remark that $\Phi$ (suitably restricted as below) descends to a diffeomorphism
\begin{equation}
\label{PhiOnQuotient}
\Phi\colon
\set[\big]
{
  (\sigma, \psi, \theta) \in
  \interval{-\infty,1}
  \times \interval{-\tfrac{\pi}{2},\tfrac{\pi}{2}} \times \R
} \big/ \sk[\big]{\trans^{\axis{\theta}}_{2\pi}}
\to
\R^3 \setminus \axis{z}.
\end{equation}%
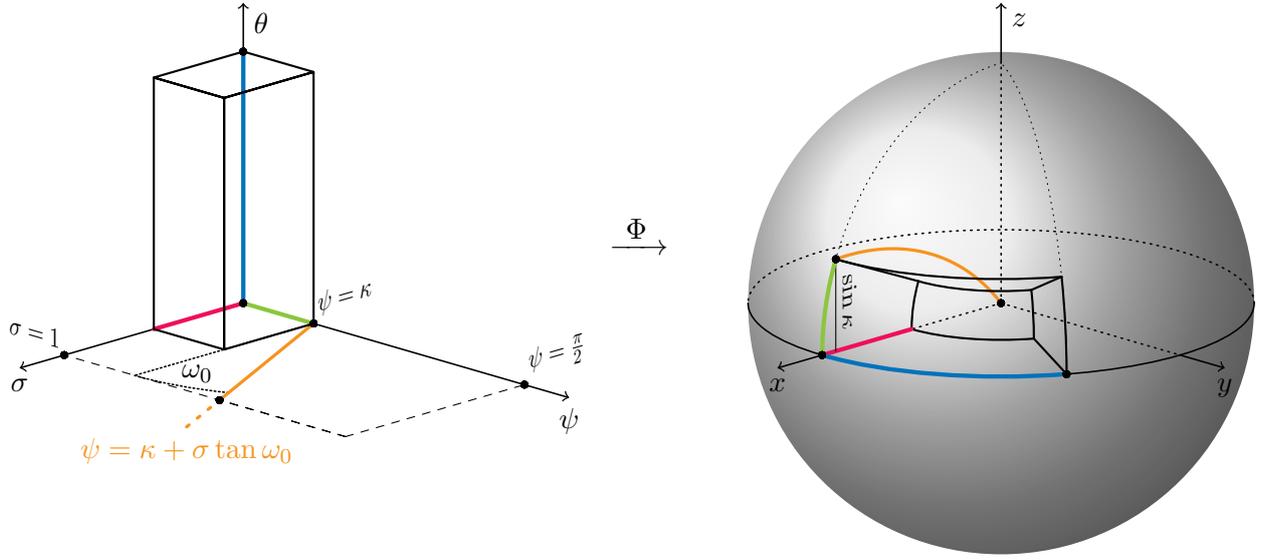
\begin{figure}%
\pgfmathsetmacro{\thetaO}{73}
\pgfmathsetmacro{\phiO}{135}
\tdplotsetmaincoords{\thetaO}{\phiO}
\pgfmathsetmacro{\axis}{1.25}
\pgfmathsetmacro{\sigmapar}{0.5}
\pgfmathsetmacro{\psipar}{pi/8}
\pgfmathsetmacro{\thetapar}{pi/3}
\pgfmathsetmacro{\alphapar}{25.38}
\pgfmathsetmacro{\bpar}{\psipar}
\pgfmathsetmacro{\psistart}{deg(\psipar)}
\pgfmathsetmacro{\psiend}{deg(\bpar+tan(\alphapar))}
\definecolor{xcol}{cmyk}{0,1,0.5,0}
\definecolor{ycol}{cmyk}{0.5,0,1,0}	
\definecolor{zcol}{cmyk}{1,0.5,0,0}	
\definecolor{Pcol}{cmyk}{0,0.5,1,0}	
\begin{tikzpicture}[line cap=round,scale=\unitscale, line join=round,tdplot_main_coords,semithick,baseline={(0,0,-\unitscale/2)}]
\tdplotdrawarc[tdplot_main_coords,densely dotted]{(0,\bpar,0)}{1}{0}{\alphapar}{anchor=south west,inner sep=1pt}{$\omega_0$}
\draw[dashed,thin]
(1,pi/2,0)--(1,0,0)
(1,pi/2,0)--(0,pi/2,0)
(1,pi/2,0)--(1,{rad(\psiend)},0)
;
\tdplottransformmainscreen{0}{1}{0}
\pgfmathsetmacro{\VecHx}{\tdplotresx}
\pgfmathsetmacro{\VecHy}{\tdplotresy}
\tdplottransformmainscreen{0}{0}{1}
\pgfmathsetmacro{\VecVx}{\tdplotresx}
\pgfmathsetmacro{\VecVy}{\tdplotresy}
\draw plot[bullet](1,0,0)node[inner sep=2pt,above left,cm={\VecHx ,\VecHy ,\VecVx ,\VecVy ,(0,0)}]{$\sigma=1$};
\draw[Pcol,very thick]({-\bpar*cot(\alphapar)},0,0)(0,\psipar,0)--(1,{rad(\psiend)},0)coordinate(P1)coordinate[pos=1.35](P2);
\draw[Pcol,very thick,loosely dotted](P2)node[below]{$\psi=\kappa+\sigma\tan\omega_0$}--(P1);
\draw[densely dotted](0,\bpar,0)--(1,\bpar,0);
\draw plot[bullet](1,{rad(\psiend)},0);
\pgfresetboundingbox
\draw[->](0,0,0)--(\axis,0,0)node[below]{$\sigma$};
\draw[->](0,0,0)--(0,pi/2+0.25,0)node[below]{$\psi$};
\draw[->](0,0,0)--(0,0,\axis)node[below right]{$\theta$};
\coordinate(000)at(0,0,0);
\coordinate(100)at(\sigmapar,0,0);
\coordinate(010)at(0,\psipar,0);
\coordinate(001)at(0,0,\thetapar);
\coordinate(110)at(\sigmapar,\psipar,0);
\coordinate(011)at(0,\psipar,\thetapar);
\coordinate(101)at(\sigmapar,0,\thetapar);
\coordinate(111)at(\sigmapar,\psipar,\thetapar);
\draw[ultra thick,xcol](000)--(100);
\draw[ultra thick,ycol](000)--(010);
\draw[ultra thick,zcol](000)--(001); 
\draw[thick]
(111)--(011)--(001)--(101)--cycle
(111)--(101)--(100)--(110)--cycle
(111)--(110)--(010)--(011)--cycle
;
\draw plot[bullet](000) plot[bullet](001); 
\tdplottransformmainscreen{-1}{0}{0}
\pgfmathsetmacro{\VecHx}{\tdplotresx}
\pgfmathsetmacro{\VecHy}{\tdplotresy}
\tdplottransformmainscreen{0}{0}{1}
\pgfmathsetmacro{\VecVx}{\tdplotresx}
\pgfmathsetmacro{\VecVy}{\tdplotresy}
\draw[] plot[bullet](010)node[inner sep=2pt,above right,cm={\VecHx ,\VecHy ,\VecVx ,\VecVy ,(0,0)}]{$\psi=\kappa$};
\draw[] plot[bullet](0,pi/2,0)node[inner sep=2pt,above right,cm={\VecHx ,\VecHy ,\VecVx ,\VecVy ,(0,0)}]{$\psi=\frac{\pi}{2}$};
\end{tikzpicture}
\hfill
$\mathllap{\xrightarrow{~\displaystyle\Phi~}}$
\hfill
\begin{tikzpicture}[scale=\unitscale,line cap=round,line join=round,tdplot_main_coords,semithick,baseline={(0,0,-\unitscale/2)}] 
\shade[tdplot_screen_coords,ball color=black!10](0,0,0)circle(1);
\draw[dotted]
(0,0,0)--(1,0,0)
(0,0,0)--(0,1,0)
(0,0,0)--(0,0,1);
\draw[->](1,0,0)--(\axis,0,0)node[below]{$x$};
\pgfresetboundingbox\path[tdplot_screen_coords](0,0,0)circle(1);
\draw[->](0,1,0)--(0,\axis,0)node[below]{$y$};
\draw[->](0,0,1)--(0,0,\axis)node[below right]{$z$};
\tdplotdrawarc[dotted]{(0,0,0)}{1}{\phiO}{\phiO+180}{}{} 
\tdplotdrawarc{(0,0,0)}{1}{\phiO}{\phiO-180}{}{} 
\draw[domain=\psistart:\psiend,smooth,variable=\psipar,samples=30,Pcol,very thick]
plot({((\bpar-rad(\psipar))*cot(\alphapar)+1)*cos(\psipar)},0,{((\bpar-rad(\psipar))*cot(\alphapar)+1)*sin(\psipar)});
\tdplotsetcoord{000}{1}{90-0}{0}
\tdplotsetcoord{001}{1}{90-0}{deg(\thetapar)}
\tdplotsetcoord{010}{1}{90-deg(\psipar)}{0}
\tdplotsetcoord{011}{1}{90-deg(\psipar)}{deg(\thetapar)}
\tdplotsetcoord{100}{1-\sigmapar}{90-0}{0}
\tdplotsetcoord{101}{1-\sigmapar}{90-0}{deg(\thetapar)}
\tdplotsetcoord{110}{1-\sigmapar}{90-deg(\psipar)}{0}
\tdplotsetcoord{111}{1-\sigmapar}{90-deg(\psipar)}{deg(\thetapar)}
\begin{scope}[thick]
\tdplotdrawarc{(0,0,0)}{1-\sigmapar}{0}{deg(\thetapar)}{}{} 
\tdplotdrawarc{(0,0,{(1-\sigmapar)*sin(deg(\psipar))})}{{(1-\sigmapar)*cos(deg(\psipar))}}{0}{deg(\thetapar)}{}{}
\tdplotsetthetaplanecoords{0} 
\tdplotdrawarc[tdplot_rotated_coords]
{(0,0,0)}{1-\sigmapar}{90}{90-deg(\psipar)}{}{}
\tdplotsetthetaplanecoords{deg(\thetapar)} 
\tdplotdrawarc[tdplot_rotated_coords]
{(0,0,0)}{1-\sigmapar}{90}{90-deg(\psipar)}{}{}
\draw[ultra thick,xcol](000)--(100);
\draw(001)--(101)(010)--(110)(011)--(111);
\tdplottransformmainscreen{0}{0}{-1}
\pgfmathsetmacro{\VecHx}{\tdplotresx}
\pgfmathsetmacro{\VecHy}{\tdplotresy}
\tdplottransformmainscreen{-1}{0}{0}
\pgfmathsetmacro{\VecVx}{\tdplotresx}
\pgfmathsetmacro{\VecVy}{\tdplotresy}
\draw[thin]({cos(deg(\psipar))},0,{sin(deg(\psipar))})--({cos(deg(\psipar))},0,0)node[midway,above,cm={\VecHx ,\VecHy ,\VecVx ,\VecVy ,(0,0)}]{$\sin\kappa$};
\tdplotdrawarc[ultra thick,zcol]{(0,0,0)}{1}{0}{deg(\thetapar)}{}{} 
\tdplotdrawarc[]{(0,0,{sin(deg(\psipar))})}{{cos(deg(\psipar))}}{0}{deg(\thetapar)}{}{} 
\tdplotsetthetaplanecoords{0} 
\tdplotdrawarc[tdplot_rotated_coords,dotted,thin]{(0,0,0)}{1}{90}{0}{}{}
\tdplotdrawarc[tdplot_rotated_coords,ultra thick,ycol]
{(0,0,0)}{1}{90}{90-deg(\psipar)}{}{}
\tdplotsetthetaplanecoords{deg(\thetapar)} 
\tdplotdrawarc[tdplot_rotated_coords,dotted,thin]{(0,0,0)}{1}{90}{0}{}{}
\tdplotdrawarc[tdplot_rotated_coords]
{(0,0,0)}{1}{90}{90-deg(\psipar)}{}{}
\end{scope}
\draw plot[bullet](1,0,0);
\draw plot[bullet](0,0,0);
\draw plot[bullet]({cos(deg(\psipar))},0,{sin(deg(\psipar))});
\draw plot[bullet]({cos(deg(\thetapar))},{sin(deg(\thetapar))},0)
;
\end{tikzpicture}
\caption{Visualization of equatorial coordinates.}%
\label{fig:equatorial_coordinates}%
\end{figure}%
More explicitly,   
\[
\Phi(\sigma, \psi, \theta)
=(1-\sigma)
\bigl(
  \cos \theta \cos \psi, \;
  \sin \theta \cos \psi, \;
  \sin \psi
\bigr) 
\]
so that
\begin{equation}
\label{Phullback}
\Phi^*g_{\mathrm{euc}}
=
d\sigma^2 + (1-\sigma)^2 \, d\psi^2
  + (1-\sigma)^2(\cos\psi)^2 \, d\theta^2,
  \end{equation}
where of course
\begin{equation}
\label{eqn_definition_geuc}
\gls{geuc} %g_{\mathrm{euc}}
=dx^2+dy^2+dz^2.
\end{equation}
Note also that
$\Phi(\{\sigma \in [0,1[\})=\B^3\setminus\axis{z}$,
with
$\Phi(\{\sigma=0\}) = \partial\B^3\setminus\axis{z}$
and $\Phi$ taking lines parallel to the $\sigma$-axis
to radial segments in $\R^3$ (meeting $\partial \B^3$ orthogonally).
Additionally
$\Phi(\{\theta=\theta_0\}) = \{y=x \tan \theta_0\} \setminus \axis{z}$
for each $\theta_0 \in \R$ and
$\Phi(\{\psi=0\}) = \{z=0\} \setminus \axis{z} \supset \B^2 \setminus \axis{z}$.
In particular then $\Phi$ takes lines of constant $\theta$ in $\{\psi=0\}$
to horizontal lines through the origin
(with the origin deleted).
Moreover $\Phi$ intertwines
the corresponding symmetries in the domain and in the target, in the sense that
\begin{equation}
\label{Phintertwine}
\begin{aligned}
\Phi \circ \refl_{\{\psi=0\}}
&=\refl_{\{z=0\}} {}\circ \Phi,
\\
\Phi \circ \trans^{\axis{\theta}}_t
&=\rot_{\axis{z}}^t {}\circ \Phi,
\\
\Phi \circ \refl_{\axis{\sigma}}
&=\refl_{\axis{x}} \circ \Phi,
\\
\Phi \circ \refl_{\{\theta=\theta_0\}}
&=\refl_{\{y=x \tan \theta_0\}} {}\circ \Phi.
\end{aligned}
\end{equation}
For later reference, we shall convene to set
\[
\dom(\Phi)\vcentcolon=\set[\big]
{(\sigma, \psi, \theta) \in
  \R
  \times \R \times \R
} \big/ \sk[\big]{\trans^{\axis{\theta}}_{2\pi}}.
\]

\paragraph{Wings over the catenoids and the disc.}
As above, let $m \geq 1$ be an integer and let $\xi \in \R$.
Set
\begin{equation}
\label{bmxi}
\begin{aligned}
\kappa_{m,\xi}&\vcentcolon=
  \frac{b^{\mathrm{tow}}_{\omega_0}}{m} +\frac{\xi}{m^2}, 
  \\
\gls{bmxi}%b_{m,\xi}
&\vcentcolon=\sin \kappa_{m,\xi}, 
  \\[.5ex]
P^0&\vcentcolon=\set[\big]
{(\sigma, \psi, \theta) \in
  \R
  \times \R \times \R \st \psi=0
} \big/ \sk[\big]{\trans^{\axis{\theta}}_{2\pi}}, \quad \mbox{and} 
  \\[.5ex]
P^1=P^1_{m,\xi}
  &\vcentcolon=\set[\big]
{(\sigma, \psi, \theta) \in
  \R
  \times \R \times \R \st \psi=\kappa_{m,\xi} + \sigma \tan \omega_0
} \big/ \sk[\big]{\trans^{\axis{\theta}}_{2\pi}}.
  \end{aligned}
\end{equation}
Thus, recalling Lemma \ref{lemmatowbent},
the asymptotic half planes
$\widehat{\Pi}^0_{m,\xi}$ and $\frac{1}{m}\widehat{\Pi}^1_{m,\xi}$
of $\frac{1}{m}\towbent_{m,\xi}$
have images in the quotient contained in $P^0$ and $P^1$ respectively.
Note also that $\Phi(P^1 \cap \{\sigma=0\})$ is the lower component of
$\partial \K_{b_{m,\xi}}$. 
As visualized in Figure \ref{fig:tubular-parametrization} we then set
\begin{align*}
\K=\K_{m,\xi}
  &\vcentcolon=
  \mbox{the complete catenoid containing } \K_{b_{m,\xi}}, \\
\nu^{\vphantom{|}}_{\K}
  &\vcentcolon=
  \mbox{the upward unit normal on } \K, \\
\check\eta^{\vphantom{|}}_{\K}
  &\vcentcolon=
  \mbox{the inward unit conormal on }
    \Phi(P^1 \cap \{\sigma=0\})
      \subset \partial \K_{b_{m,\xi}}, \\
\nu^{\vphantom{|}}_{{P^1}}
  &\vcentcolon=
  (\cos \omega_0)  \partial_\psi
    -(\sin \omega_0)  \partial_\sigma, \\
\check\eta^{\vphantom{|}}_{{P^1}}
  &\vcentcolon=
  (\cos \omega_0)  \partial_\sigma
    + (\sin \omega_0)  \partial_\psi.
\end{align*}
(To avoid confusion, we specify that the requirement on $\nu^{\vphantom{|}}_{\K}$ of being \emph{upward} pointing is understood, say, about the nearly-equatorial connected component of $\partial\K_{b_{m,\xi}}$.)
Recalling \eqref{tubularcoordscodim2} and \eqref{eqn_definition_geuc}, and analogously letting $\widehat{g}_{\mathrm{euc}}
\vcentcolon=
d\sigma^2 + d\psi^2 + d\theta^2$,
we also define the maps
\begin{align*}
E_{\B^2}
  &\vcentcolon=
  \tubularexp_{
     (\R^3,g_{\mathrm{euc}}),~
     (\B^2,\partial_z),~
     (\Sp^1,d\Phi \, \partial_\sigma)
     }, \\[.3ex]
E_{\K}
   &\vcentcolon=
  \tubularexp_{
    (\R^3,g_{\mathrm{euc}}),~
    (\K, \nu^{\vphantom{|}}_{{\K}}),~
    (\Phi(P^1 \cap \{\sigma=0\}),\check\eta^{\vphantom{|}}_{\K})
    }, 
\\
E_{{P^1}}
  &\vcentcolon=
  \tubularexp_{
    (\dom(\Phi),\widehat{g}_{\mathrm{euc}}),~
    (P^1,\nu_{{P^1}}),~
    (\partial P^1,\check\eta_{{P^1}})
    },
\intertext{as well as the maps $\varphi_{P^1},
~\varphi^{\vphantom{|}}_{\K}\colon\R^3 \to \R^3$ by} 
\varphi^{\vphantom{1}}_{P^1}(\theta,s,t)
  &\vcentcolon=
  E_{{P^1}}\bigl((0,\kappa_{m,\xi},\theta),s,t\bigr), 
\\
\varphi^{\vphantom{|}}_{\K}(\theta,s,t)
  &\vcentcolon=
  E_{{\K}}\bigl(\Phi(0,\kappa_{m,\xi},\theta),s,t\bigr),
\end{align*}
and, finally, corresponding respectively to the wings $W^0$, $W^1$, the maps
\begin{align}\label{initsurfphidef}
&\begin{aligned}
&\varphi^0\colon\dom(\Phi) \cap \{\sigma < 1\}\to\R^3
\\
&\varphi^0(\sigma,\psi,\theta)
\vcentcolon=
E_{{\B^2}}\bigl(\Phi(0,0,\theta),\sigma,\psi\bigr),
\end{aligned}
&
&\begin{aligned}
&\varphi^1\colon\dom(\Phi)\to\R^3
\\
&\varphi^1=\varphi^1_{m,\xi}
\vcentcolon=
\varphi^{\vphantom{|}}_{\K} \circ \varphi_{P^1}^{-1}.
\end{aligned}
\end{align}
In short, the map $\varphi^0$ provides a natural parametrization of a tubular neighborhood of $\B^2$ over a tubular neighborhood of $P^0$, and similarly $\varphi^1$ provides a natural parametrization of a tubular neighborhood of $\K$ over a tubular neighborhood of $P^1$, as visualized in Figure \ref{fig:tubular-parametrization}.

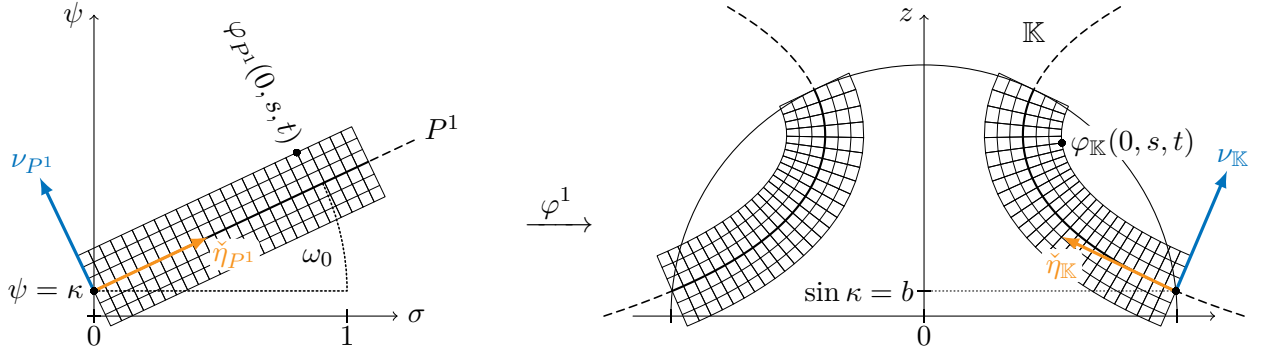
\begin{figure}%
\centering
\pgfmathsetmacro{\anz}{25}
\pgfmathsetmacro{\anzm}{\anz-1}	
\pgfmathsetmacro{\b}{0.1}
\pgfmathsetmacro{\k}{rad(asin(\b))}
\pgfmathsetmacro{\omegao}{25.38}
\definecolor{etacol}{cmyk}{0,0.5,1,0}	
\definecolor{nucol}{cmyk}{1,0.5,0,0}	
\begin{tikzpicture}[line cap=round,line join=round,baseline={(0,\unitscale/3)},scale=\unitscale]
\draw[->](0,0)--(1.2,0)node[right]{$\sigma$};
\draw[->](0,0)--(0,1.2)node[left]{$\psi$};
\draw plot[plus](0,0)node[below]{$0$};
\draw plot[vdash](1,0)node[below]{$1$};
\draw[densely dotted,semithick](0,\k)--++(1,0)arc(0:\omegao:1)node[pos=0.35,anchor=\omegao/2]{$\omega_0$};
\draw[densely dashed,semithick](0,\k)++(\omegao:1.4)node[anchor=\omegao+180]{$P^1$}--++(\omegao:-.2);
\draw[thick](0,\k)--++(\omegao:1.2)
node foreach \i in {0,...,\anz}[pos={(\i/\anz)},sloped,allow upside down,minimum size=0.1*\unitscale cm](A\i){}
node foreach \i in {0,...,\anz}[pos={(\i/\anz)},sloped,allow upside down,minimum size=0.2*\unitscale cm](B\i){}
node foreach \i in {0,...,\anz}[pos={(\i/\anz)},sloped,allow upside down,minimum size=0.3*\unitscale cm](C\i){}
;
\begin{scope}[very thin]
\foreach[count=\k]\i in {0,...,\anzm}{
\draw(A\i.center)--(A\i.north)--(A\k.north)--(A\k.center);
\draw(A\i.north)--(B\i.north)--(B\k.north)--(A\k.north);
\draw(B\i.north)--(C\i.north)--(C\k.north)--(B\k.north);
\draw(A\i.center)--(A\i.south)--(A\k.south)--(A\k.center);
\draw(A\i.south)--(B\i.south)--(B\k.south)--(A\k.south);
\draw(B\i.south)--(C\i.south)--(C\k.south)--(B\k.south);
}
\end{scope}
\draw[-latex,very thick,etacol](0,\k)--++(\omegao:1/2)node[below right,fill=white,inner sep=1pt,rectangle,rounded corners=.5ex]{$\check\eta_{P^1}$};
\draw[-latex,very thick,nucol](0,\k)--++(\omegao+90:1/2)node[anchor=-90+\omegao,inner sep=1pt]{$\nu_{P^1}$};
\draw plot[bullet](0,\k)node[left]{$\psi=\kappa$};
\draw plot[bullet](C20.north)node[left,rotate=\omegao-90]{$\varphi_{P^1}(0,s,t)$};
\end{tikzpicture}
\hfill
$\mathllap{\xrightarrow{~\displaystyle\varphi^1~}}$
\begin{tikzpicture}[line cap=round,line join=round,baseline={(0,\unitscale/3)},scale=\unitscale]
\pgfmathsetmacro{\dscale}{0.245}
\pgfmathsetmacro{\a}{2.5603}
\pgfmathsetmacro{\h}{0.9012}
\pgfmathsetmacro{\w}{asin(sqrt((1-\h*\h)/(1-\b*\b)))}
\draw[->](-1.15,0)--(1.15,0);
\draw[->](0,-0)--(0,1.2)node[left]{$z$};
\draw plot[vdash](0,0)node[below]{$0$};
\draw plot[vdash](1,0);
\draw plot[vdash](-1,0);
\draw(1,0)arc(0:180:1); 
\draw[thick,domain=\b:\h,variable=\z,samples=50]plot({cosh(\a*\z-\a*\b-acosh(\a*sqrt(1-\b*\b)))/\a},{\z});
\draw[semithick,densely dashed,domain=\h+0.33:\h,variable=\z,samples=10]plot({cosh(\a*\z-\a*\b-acosh(\a*sqrt(1-\b*\b)))/\a},{\z})
node[above=3ex]{$\K$};
\draw[semithick,densely dashed,domain=\b-0.1:\b,variable=\z,samples=5]plot({cosh(\a*\z-\a*\b-acosh(\a*sqrt(1-\b*\b)))/\a},{\z});
\draw[xscale=-1,thick,domain=\b:\h,variable=\z,samples=50]plot({cosh(\a*\z-\a*\b-acosh(\a*sqrt(1-\b*\b)))/\a},{\z});
\draw[xscale=-1,semithick,densely dashed,domain=\h+0.33:\h,variable=\z,samples=10]plot({cosh(\a*\z-\a*\b-acosh(\a*sqrt(1-\b*\b)))/\a},{\z});
\draw[xscale=-1,semithick,densely dashed,domain=\b-0.1:\b,variable=\z,samples=5]plot({cosh(\a*\z-\a*\b-acosh(\a*sqrt(1-\b*\b)))/\a},{\z});
\path
({ sqrt(1-\b*\b)},{\b})..controls+({-sinh(acosh(\a*sqrt(1-\b*\b)))*\dscale},\dscale)and+({-sinh(\a*\h-\a*\b-acosh(\a*sqrt(1-\b*\b)))*\dscale},-\dscale)..({ cosh(\a*\h-\a*\b-acosh(\a*sqrt(1-\b*\b)))/\a},{\h})
node foreach \i in {0,...,\anz}[pos={(1/3)*(\i/\anz)^3+(1-1/3)*(\i/\anz)(1/3)*(\i/\anz)^3+(1-1/3)*(\i/\anz)},sloped,allow upside down,minimum size=0.1*\unitscale cm](A\i){}
node foreach \i in {0,...,\anz}[pos={(1/3)*(\i/\anz)^3+(1-1/3)*(\i/\anz)(1/3)*(\i/\anz)^3+(1-1/3)*(\i/\anz)},sloped,allow upside down,minimum size=0.2*\unitscale cm](B\i){}
node foreach \i in {0,...,\anz}[pos={(1/3)*(\i/\anz)^3+(1-1/3)*(\i/\anz)(1/3)*(\i/\anz)^3+(1-1/3)*(\i/\anz)},sloped,allow upside down,minimum size=0.3*\unitscale cm](C\i){}
;
\begin{scope}[very thin]
\foreach[count=\k]\i in {0,...,\anzm}{
\draw(A\i.center)--(A\i.north)--(A\k.north)--(A\k.center);
\draw(A\i.north)--(B\i.north)--(B\k.north)--(A\k.north);
\draw(B\i.north)--(C\i.north)--(C\k.north)--(B\k.north);
\draw(A\i.center)--(A\i.south)--(A\k.south)--(A\k.center);
\draw(A\i.south)--(B\i.south)--(B\k.south)--(A\k.south);
\draw(B\i.south)--(C\i.south)--(C\k.south)--(B\k.south);
}
\end{scope}
\begin{scope}[shift={(A0.center)},x={(A0.east)}, y={(A0.south)}]
\draw[-latex,very thick,nucol](0,0)--(0,10)node[anchor=-90-\w,inner sep=1pt]{$\nu_{\K}$};
\end{scope}
\draw plot[bullet](C20.south)node[right=2pt,fill=white, inner sep=1pt,rounded corners=.5ex]{$\varphi_{\K}(0,s,t)$};
\path
({-sqrt(1-\b*\b)},{\b})..controls+({ sinh(acosh(\a*sqrt(1-\b*\b)))*\dscale},\dscale)and+({ sinh(\a*\h-\a*\b-acosh(\a*sqrt(1-\b*\b)))*\dscale},-\dscale)..({-cosh(\a*\h-\a*\b-acosh(\a*sqrt(1-\b*\b)))/\a},{\h})
node foreach \i in {0,...,\anz}[pos={(1/3)*(\i/\anz)^3+(1-1/3)*(\i/\anz)},sloped,allow upside down,minimum size=0.1*\unitscale cm](A\i){}
node foreach \i in {0,...,\anz}[pos={(1/3)*(\i/\anz)^3+(1-1/3)*(\i/\anz)},sloped,allow upside down,minimum size=0.2*\unitscale cm](B\i){}
node foreach \i in {0,...,\anz}[pos={(1/3)*(\i/\anz)^3+(1-1/3)*(\i/\anz)},sloped,allow upside down,minimum size=0.3*\unitscale cm](C\i){}
;
\begin{scope}[very thin]
\foreach[count=\k]\i in {0,...,\anzm}{
\draw(A\i.center)--(A\i.north)--(A\k.north)--(A\k.center);
\draw(A\i.north)--(B\i.north)--(B\k.north)--(A\k.north);
\draw(B\i.north)--(C\i.north)--(C\k.north)--(B\k.north);
\draw(A\i.center)--(A\i.south)--(A\k.south)--(A\k.center);
\draw(A\i.south)--(B\i.south)--(B\k.south)--(A\k.south);
\draw(B\i.south)--(C\i.south)--(C\k.south)--(B\k.south);
} 
\end{scope}
\draw[densely dotted](0,\b)--({sqrt(1-\b*\b)},\b);
\draw plot[hdash](0,\b)node[left]{$\sin\kappa=b$}; 
\draw[-latex,very thick,etacol]({ sqrt(1-\b*\b)},\b*0.98)--++(-\w:-0.5)node[below=4pt,fill=white,inner sep=1.5pt,rectangle,rounded corners=.5ex]{$\check\eta_{\K}$};
\draw plot[bullet]({ sqrt(1-\b*\b)},\b);
\end{tikzpicture}
\caption{Visualization of the maps $\varphi_{P^1}$, $\varphi_{\K}$ and $\varphi^1=\varphi_{\K}\circ\varphi_{P^1}^{-1}$.}%
\label{fig:tubular-parametrization}%
\end{figure}

We observe that all equations in
\eqref{Phintertwine}
hold with $\Phi$
replaced either throughout by $\varphi^0 \circ \varpi$
or throughout by $\varphi^1 \circ \varpi$
and that furthermore
\begin{equation}
\label{phivsPhi}
\begin{aligned}
&\varphi^0|_{\axis{\theta}}=\Phi|_{\axis{\theta}},
\quad
d\varphi^0|_{\axis{\theta}}=d\Phi|_{\axis{\theta}},
\quad
\varphi^1|_{P^1 \cap \{\sigma=0\}}
  =\Phi|_{P^1 \cap \{\sigma=0\}}, \quad \mbox{and} \\
&d(\Phi^{-1} \circ \varphi^1)|_{(0,\kappa_{m,\xi},0)}
  \mbox{ is a rotation through angle }
  \omega_{b_{m,\xi}} - \omega_0 + \kappa_{m,\xi}.
\end{aligned}
\end{equation}
Next we define
$\Phi^i_{m,\xi}\colon \dom(\varphi^i) \to \R^3$
for $i=0,1$ by
\begin{equation}
\label{Phii}
\Phi^i_{m,\xi}
\vcentcolon=
(\cutoff{3m^{-1}R_{\mathrm{tow}}}{2m^{-1}R_{\mathrm{tow}}}
    \circ \dist_{\axis{\theta}}) \Phi
  +(\cutoff{2m^{-1}R_{\mathrm{tow}}}{3m^{-1}R_{\mathrm{tow}}}
    \circ \dist_{\axis{\theta}}) \varphi^i,
\end{equation}
a convex interpolation between $\Phi$ and $\varphi^i$.
It follows, making use of \eqref{phivsPhi},
that there exists $\epsilon>0$,
independent of $m$ and $\xi$,
such that $\Phi^i_{m,\xi}|_{P^i_{<\epsilon}}$
is a diffeomorphism onto its image
(for each $i=0,1$) provided that $m$ is sufficiently large
in terms of $\xi$ and a universal constant.

Last, recalling the definition of 
$\widetilde{\phi}_{m,\xi}
 \colon
 \tow^+/\sk{\trans^{\axis{\theta}}_{2m\pi}}
 \to
 \frac{1}{m}\towbent_{m,\xi}^+ /\sk{\trans^{\axis{\theta}}_{2\pi}}$ 
from \eqref{scaledphibentquot}, we define the wings
\begin{equation}
\label{wingsinball}
\begin{aligned}
W^{\K}_{m,\xi}
  &\vcentcolon=
  \left(
    \Phi^1_{m,\xi} \circ \widetilde{\phi}_{m,\xi}
  \right)\bigl(W^1/\sk{\trans^{\axis{\theta}}_{2m\pi}}\bigr) \cap \B^3,
\\
W^{\B^2}_{m,\xi}
  &\vcentcolon=
  \{(0,0,0)\}
  \cup
  \left(\left(
    \Phi^0_{m,\xi} \circ \widetilde{\phi}_{m,\xi}
  \right)\bigl(W^0/\sk{\trans^{\axis{\theta}}_{2m\pi}}\bigr) \cap \B^3\right)
\end{aligned}
\end{equation}
in $\B^3$ (see Figure \ref{fig:initialsurface})
along with the corresponding identification maps
\begin{align}\notag
&\varpi_{W^{\B^2}_{m,\xi}}
 \colon
 W^{\B^2}_{m,\xi} \setminus \{(0,0,0)\}
 \to
 \mathring{W}^0/\sk{\trans^{\axis{\theta}}_{2m\pi}}, 
&
&\varpi_{W^{\K}_{m,\xi}}
 \colon W^{\K}_{m,\xi}
 \to
 \mathring{W}^1/\sk{\trans^{\axis{\theta}}_{2m\pi}},
\\
\label{projectionsontowings}
&\varpi_{W^{\B^2}_{m,\xi}}
\vcentcolon=\left(\Phi^0_{m,\xi} \circ \widetilde{\phi}_{m,\xi}\right)\Big|_{\mathring{W}^0/\sk{\trans^{\axis{\theta}}_{2m\pi}}}^{-1},
&
&\varpi_{W^{\K}_{m,\xi}}
\vcentcolon=\left(\Phi^1_{m,\xi} \circ \widetilde{\phi}_{m,\xi}\right)\Big|_{\mathring{W}^1/\sk{\trans^{\axis{\theta}}_{2m\pi}}}^{-1}. 
\end{align}
We explicitly remark that the domains $\mathring{W}^i \subset W^i$, for $i\in\{0,1\}$, are defined by the requirement that the maps in question (obtained, in turn, by restriction of $\Phi^i_{m,\xi} \circ \widetilde{\phi}_{m,\xi}$) be bijective.
We similarly define the core
\begin{align}
\label{coreinball}
M^{\mathrm{core}}_{m,\xi}
&\vcentcolon=
\left(\Phi \circ \widetilde{\phi}_{m,\xi}\right)
\Bigl((\tow^+ \cap \axis{\theta}_{\leq R_{\mathrm{tow}}})/\sk{\trans^{\axis{\theta}}_{2m\pi}}\Bigr) 
\end{align}
and its accompanying map $\varpi_{M^{\mathrm{core}}_{m,\xi}}\colon M^{\mathrm{core}}_{m,\xi}\to\towquot^+_{(m)}$ by
\begin{align*} 
\varpi^{\vphantom{1}}_{M^{\mathrm{core}}_{m,\xi}}
&\vcentcolon=
\left(\Phi \circ \widetilde{\phi}_{m,\xi}\right)
  \Big|_{
    (\tow^+ \cap \axis{\theta}_{
    \leq R_{\mathrm{tow}}}
    )/\sk{\trans^{\axis{\theta}}_{2m\pi}}
   }^{-1}. 
\end{align*}

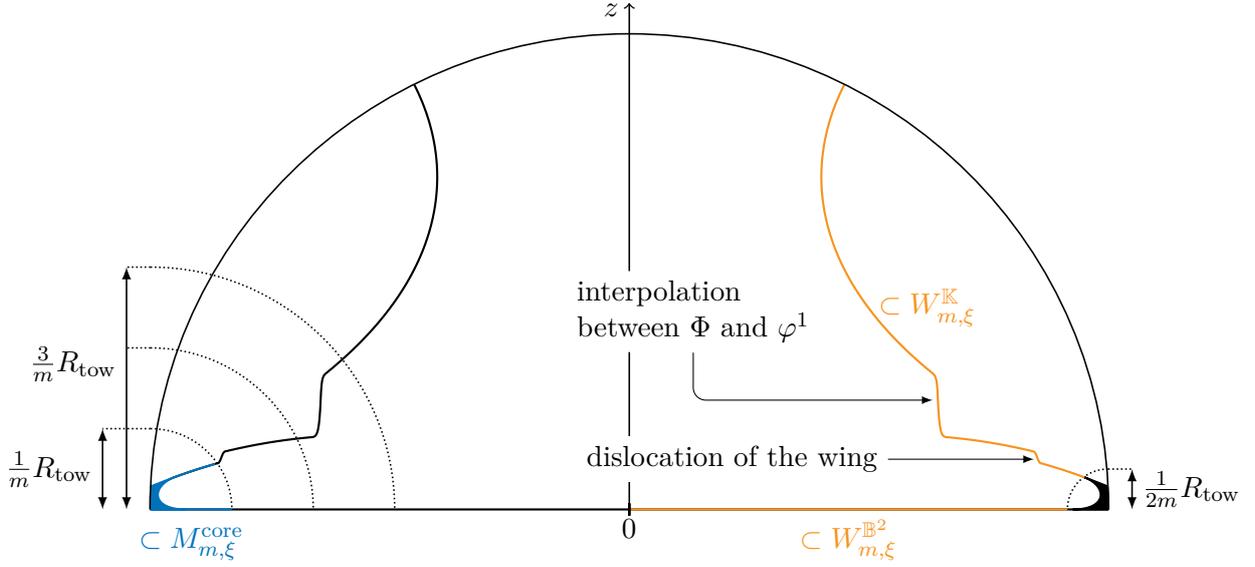
\begin{figure}\centering
\providecommand{\initialsurfacetikzcode}[1][]{
\draw[#1] plot[smooth] coordinates{
(0.9508,0.0676)(0.9475,0.0688)(0.9442,0.0700)(0.9409,0.0712)(0.9376,0.0724)(0.9343,0.0736)(0.9309,0.0747)(0.9276,0.0759)(0.9243,0.0770)(0.9210,0.0782)
(0.9177,0.0793)(0.9143,0.0804)(0.9110,0.0815)(0.9077,0.0826)(0.9044,0.0837)(0.9011,0.0848)(0.8977,0.0858)(0.8944,0.0869)(0.8911,0.0879)(0.8877,0.0890)
(0.8844,0.0900)(0.8811,0.0910)(0.8777,0.0920)(0.8744,0.0930)(0.8711,0.0940)(0.8677,0.0950)(0.8644,0.0960)(0.8611,0.0969)(0.8577,0.0979)(0.8536,0.1051)
(0.8488,0.1173)(0.8449,0.1218)(0.8416,0.1226)(0.8382,0.1235)(0.8348,0.1243)(0.8315,0.1251)(0.8281,0.1258)(0.8247,0.1266)(0.8214,0.1274)(0.8180,0.1281)
(0.8146,0.1289)(0.8113,0.1296)(0.8079,0.1303)(0.8045,0.1310)(0.8011,0.1318)(0.7978,0.1324)(0.7944,0.1331)(0.7910,0.1338)(0.7877,0.1345)(0.7843,0.1351)
(0.7809,0.1358)(0.7776,0.1364)(0.7742,0.1370)(0.7708,0.1376)(0.7674,0.1382)(0.7641,0.1388)(0.7607,0.1394)(0.7573,0.1400)(0.7540,0.1406)(0.7506,0.1411)
(0.7472,0.1417)(0.7438,0.1422)(0.7405,0.1427)(0.7371,0.1432)(0.7337,0.1438)(0.7304,0.1442)(0.7270,0.1447)(0.7236,0.1452)(0.7203,0.1457)(0.7169,0.1461)
(0.7135,0.1466)(0.7101,0.1470)(0.7068,0.1474)(0.7034,0.1479)(0.7000,0.1483)(0.6967,0.1487)(0.6933,0.1490)(0.6899,0.1494)(0.6866,0.1498)(0.6832,0.1501)
(0.6798,0.1505)(0.6765,0.1508)(0.6731,0.1512)(0.6698,0.1515)(0.6664,0.1518)(0.6630,0.1521)(0.6597,0.1527)(0.6567,0.1541)(0.6541,0.1571)(0.6519,0.1617)
(0.6502,0.1680)(0.6488,0.1756)(0.6477,0.1842)(0.6468,0.1935)(0.6461,0.2034)(0.6454,0.2136)(0.6449,0.2242)(0.6444,0.2348)(0.6439,0.2454)(0.6432,0.2556)
(0.6423,0.2648)(0.6411,0.2727)(0.6393,0.2787)(0.6371,0.2828)(0.6345,0.2856)(0.6317,0.2879)(0.6290,0.2902)(0.6263,0.2925)(0.6235,0.2948)(0.6208,0.2970)
(0.6181,0.2993)(0.6154,0.3017)(0.6127,0.3040)(0.6100,0.3063)(0.6074,0.3087)(0.6047,0.3110)(0.6020,0.3134)(0.5994,0.3157)(0.5968,0.3181)(0.5941,0.3205)
(0.5915,0.3229)(0.5889,0.3253)(0.5863,0.3278)(0.5837,0.3302)(0.5811,0.3326)(0.5786,0.3351)(0.5760,0.3376)(0.5735,0.3401)(0.5709,0.3426)(0.5684,0.3451)
(0.5659,0.3476)(0.5634,0.3501)(0.5609,0.3526)(0.5584,0.3552)(0.5560,0.3577)(0.5535,0.3603)(0.5511,0.3629)(0.5486,0.3655)(0.5462,0.3681)(0.5438,0.3707)
(0.5414,0.3733)(0.5390,0.3760)(0.5367,0.3786)(0.5343,0.3813)(0.5320,0.3840)(0.5296,0.3866)(0.5273,0.3893)(0.5250,0.3920)(0.5227,0.3948)(0.5205,0.3975)
(0.5182,0.4002)(0.5160,0.4030)(0.5137,0.4058)(0.5115,0.4086)(0.5093,0.4113)(0.5071,0.4142)(0.5050,0.4170)(0.5028,0.4198)(0.5007,0.4226)(0.4986,0.4255)
(0.4965,0.4284)(0.4944,0.4312)(0.4923,0.4341)(0.4903,0.4370)(0.4882,0.4399)(0.4862,0.4429)(0.4842,0.4458)(0.4822,0.4488)(0.4803,0.4517)(0.4783,0.4547)
(0.4764,0.4577)(0.4745,0.4607)(0.4726,0.4637)(0.4707,0.4667)(0.4689,0.4697)(0.4670,0.4728)(0.4652,0.4759)(0.4634,0.4789)(0.4617,0.4820)(0.4599,0.4851)
(0.4582,0.4882)(0.4565,0.4913)(0.4548,0.4945)(0.4531,0.4976)(0.4515,0.5007)(0.4499,0.5039)(0.4483,0.5071)(0.4467,0.5103)(0.4451,0.5135)(0.4436,0.5167)
(0.4421,0.5199)(0.4406,0.5231)(0.4392,0.5264)(0.4377,0.5296)(0.4363,0.5329)(0.4349,0.5361)(0.4336,0.5394)(0.4322,0.5427)(0.4309,0.5460)(0.4296,0.5493)
(0.4284,0.5527)(0.4271,0.5560)(0.4259,0.5593)(0.4248,0.5627)(0.4236,0.5661)(0.4225,0.5694)(0.4214,0.5728)(0.4203,0.5762)(0.4192,0.5796)(0.4182,0.5830)
(0.4172,0.5864)(0.4163,0.5898)(0.4153,0.5933)(0.4144,0.5967)(0.4135,0.6001)(0.4127,0.6036)(0.4119,0.6070)(0.4111,0.6105)(0.4103,0.6140)(0.4096,0.6175)
(0.4089,0.6209)(0.4082,0.6244)(0.4076,0.6279)(0.4069,0.6314)(0.4064,0.6349)(0.4058,0.6384)(0.4053,0.6420)(0.4048,0.6455)(0.4043,0.6490)(0.4039,0.6525)
(0.4035,0.6561)(0.4031,0.6596)(0.4028,0.6631)(0.4025,0.6667)(0.4022,0.6702)(0.4020,0.6738)(0.4017,0.6773)(0.4016,0.6809)(0.4014,0.6844)(0.4013,0.6880)
(0.4012,0.6915)(0.4011,0.6951)(0.4011,0.6986)(0.4011,0.7022)(0.4012,0.7057)(0.4012,0.7093)(0.4013,0.7128)(0.4015,0.7164)(0.4016,0.7199)(0.4018,0.7235)
(0.4020,0.7270)(0.4023,0.7306)(0.4026,0.7341)(0.4029,0.7377)(0.4033,0.7412)(0.4036,0.7447)(0.4041,0.7483)(0.4045,0.7518)(0.4050,0.7553)(0.4055,0.7588)
(0.4060,0.7624)(0.4066,0.7659)(0.4072,0.7694)(0.4078,0.7729)(0.4085,0.7764)(0.4091,0.7798)(0.4099,0.7833)(0.4106,0.7868)(0.4114,0.7903)(0.4122,0.7937)
(0.4130,0.7972)(0.4139,0.8006)(0.4148,0.8041)(0.4157,0.8075)(0.4166,0.8109)(0.4176,0.8144)(0.4186,0.8178)(0.4196,0.8212)(0.4207,0.8246)(0.4218,0.8279)
(0.4229,0.8313)(0.4240,0.8347)(0.4252,0.8380)(0.4264,0.8414)(0.4276,0.8447)(0.4289,0.8481)(0.4301,0.8514)(0.4314,0.8547)(0.4327,0.8580)(0.4341,0.8613)
(0.4355,0.8646)(0.4369,0.8678)(0.4383,0.8711)(0.4397,0.8743)(0.4412,0.8776)(0.4427,0.8808)(0.4442,0.8840)(0.4457,0.8872)(0.4473,0.8904)(0.4489,0.8936)
};
}
\pgfmathsetmacro{\m}{40}
\begin{tikzpicture}[line join=round,baseline={(0,0)},scale=6.3,semithick]
\pgfmathsetmacro{\Rtow}{2*3.4}
\pgfmathsetmacro{\bpar}{1.95/\m}
\draw[densely dotted,xscale=-1](1,1*\Rtow/\m)arc(90:180:1*\Rtow/\m);
\draw[densely dotted,xscale=-1](1,2*\Rtow/\m)arc(90:180:2*\Rtow/\m);
\draw[densely dotted,xscale=-1](1,3*\Rtow/\m)arc(90:180:3*\Rtow/\m);
\draw[densely dotted](1,0.5*\Rtow/\m)arc(90:180:0.5*\Rtow/\m);
\draw[densely dotted](1,0.5*\Rtow/\m)--++(0.05,0)coordinate(R/2); 
\draw[latex-latex](R/2)--++(0,-\Rtow/2/\m)node[pos=0.5,right]{$\frac{1}{2m}R_{\mathrm{tow}}$};

\draw[densely dotted,xscale=-1](1,\Rtow/\m)--++(0.1,0)coordinate(R);
\draw[densely dotted,xscale=-1](1,2*\Rtow/\m)--++(0.05,0)coordinate(2R);
\draw[densely dotted,xscale=-1](1,3*\Rtow/\m)--++(0.05,0)coordinate(3R);
\draw[latex-latex,xscale=-1](R)--++(0,-\Rtow/\m)node[midway,left]{$\frac{1}{m}R_{\mathrm{tow}}$};
\draw[latex-latex,xscale=-1](3R)--++(0,-3*\Rtow/\m)node[pos=0.4,left]{$\frac{3}{m}R_{\mathrm{tow}}$};
\begin{scope}[thick]
\draw(0,0)--(1,0);
\filldraw(0.9508,0.0676)..controls({sqrt(1-\bpar*\bpar)},\bpar)and+(1.5*\bpar,0)..(1-1.5*\bpar,0)
--(1,0)arc(0:{deg(\bpar)}:1)--cycle;
\initialsurfacetikzcode[color={cmyk,1:magenta,0.5;yellow,1}]
\draw[color={cmyk,1:magenta,0.5;yellow,1}](0,0)--(1-\Rtow/\m/2,0)coordinate[midway](WB); 
\draw[color={cmyk,1:magenta,0.5;yellow,1}](WB)++(0,-0.08)node[anchor=base]{$\subset W^{\B^2}_{m,\xi}$};
\draw[color={cmyk,1:magenta,0.5;yellow,1}](0.4986,0.4255)node[right]{$\subset W^{\K}_{m,\xi}$};
\end{scope}
\begin{scope}[xscale=-1,thick]
\draw(0,0)--(1,0);
\initialsurfacetikzcode
\draw[color={cmyk,1:magenta,0.5;cyan,1}] plot[smooth]coordinates{(0.9508,0.0676)(0.9224,0.0777)(0.8939,0.0871)(0.8653,0.0957)};
\draw[color={cmyk,1:magenta,0.5;cyan,1}] (1,0)--++(-\Rtow/\m,0)coordinate[midway](WK);
\draw[color={cmyk,1:magenta,0.5;cyan,1}](WK)++(0,-0.08)node[anchor=base]{$\subset M^{\mathrm{core}}_{m,\xi}$};
\filldraw[color={cmyk,1:magenta,0.5;cyan,1}](0.9508,0.0676)..controls({sqrt(1-\bpar*\bpar)},\bpar)and+(1.5*\bpar,0)..(1-1.5*\bpar,0)
--(1,0)arc(0:{deg(\bpar)}:1)--cycle;
\end{scope}
\draw[->](0,-0)--(0,1.066)node[below left,inner sep=0]{$z~$};
\draw plot[vdash](0,0)node[below]{$0$};
\draw(1+0.001,0)arc(0:180:1+0.001); 
\draw[latex-,thin](0.8653-0.02,0.0957+0.01)--++(-0.305,0)node[left,fill=white]{dislocation of the wing};
\draw[latex-,thin,rounded corners=1ex](0.6436-0.01,0.24-0.01)-|++(-0.5,0.1)node[above,fill=white,align=left]{interpolation \\ between $\Phi$ and $\varphi^1$};
\end{tikzpicture}
\caption{Vertical cut through the upper half of the initial surface $\Sigma_{m,\xi}$ for $m=40$ and $\xi=m$. 
(For $m\gg 40$, the ``interpolation error'' becomes small compared to the dislocation.)
} %
\label{fig:initialsurface}%
\end{figure}

\paragraph{Definition of the initial surfaces and regional projections to the models.}
Given $\xi \in \R$ and a sufficiently large integer $m$, we recall definitions \eqref{wingsinball} and \eqref{coreinball} and define the \emph{initial surface}
\begin{equation}
\label{initsurfdef}
\gls{Sigmamxi}%\Sigma_{m,\xi}
\vcentcolon=
  M^{\mathrm{core}}_{m,\xi} 
  \cup W^{\K}_{m,\xi}
    \cup W^{\B^2}_{m,\xi}
    \cup \refl_{\axis{x}}W^{\K}_{m,\xi}
\end{equation}
as visualized in Figure \ref{fig:initialsurface}. 
Recalling also the identification maps \eqref{projectionsontowings}, we define the regions
\begin{equation}
\begin{aligned}
\label{eqn:definition_regions}
\glsuseri{region} %\discr_{m,\xi}
    &\vcentcolon=
\{(0,0,0)\}
  \cup
  \varpi_{W^{\B^2}_{m,\xi}}^{-1}
    \left(
      (\mathring{W}^0 \cap \axis{\theta}_{\geq m^{1/4}})
      /\sk{\trans^{\axis{\theta}}_{2m\pi}}
    \right)
    \subset
    W^{\B^2}_{m,\xi}, \\
\glsuserii{region} %\catr_{m,\xi}
    &\vcentcolon=
 \varpi_{W^{\K}_{m,\xi}}^{-1}
    \left(
    (\mathring{W}^1 \cap \axis{\theta}_{\geq m^{1/4}})
      /\sk{\trans^{\axis{\theta}}_{2m\pi}}
    \right)
    \subset
    W^{\K}_{m,\xi},
\\
\glsuseriii{region} %\towr_{m,\xi}
    &\vcentcolon=
M^{\mathrm{core}}_{m,\xi}
    \cup
    \varpi_{W^{\B^2}_{m,\xi}}^{-1}
    \left(
      (\mathring{W}^0 \cap \axis{\theta}_{\leq m^{1/2}})
      /\sk{\trans^{\axis{\theta}}_{2m\pi}}
    \right) \\
    &\hphantom{{}\vcentcolon={}}\cup
    \bigcup_{j=0}^1
    \refl_{\axis{x}}^j
    \varpi_{W^{\K}_{m,\xi}}^{-1}
    \left(
      (\mathring{W}^1 \cap \axis{\theta}_{\leq m^{1/2}})
      /\sk{\trans^{\axis{\theta}}_{2m\pi}}
    \right),
\end{aligned}
\end{equation}
and corresponding maps
$
\varpi_{\discr_{m,\xi}}\colon 
  \discr_{m,\xi} \to \B^2
$,
$
\varpi_{\catr_{m,\xi}}\colon
  \catr_{m,\xi} \to \K_{b_{m,\xi}}
$,
and
$
\varpi_{\towr_{m,\xi}}\colon
  \towr_{m,\xi} \to \towquot^+_{(m)}
$
by
\begin{equation}
\label{regionalprojections}
\begin{aligned}
\glsuseri{regionalprojections} %\varpi_{\discr_{m,\xi}}
&\vcentcolon=
  \mbox{nearest-point projection to } \B^2,
\\
\glsuserii{regionalprojections} %\varpi_{\catr_{m,\xi}}
&\vcentcolon=
  \mbox{nearest-point projection to } \K_{b_{m,\xi}},
\\
\glsuseriii{regionalprojections} %\varpi_{\towr_{m,\xi}}
  &\vcentcolon=
  \begin{cases}
    \varpi_{M^{\mathrm{core}}_{m,\xi}}
      & \mbox{ on }  
      M^{\mathrm{core}}_{m,\xi}
    \\
    \varpi_{W^{\B^2}_{m,\xi}} 
      & \mbox{ on }  
      \towr_{m,\xi} \cap W^{\B^2}_{m,\xi}
    \\
    \varpi_{W^{\K}_{m,\xi}}
      & \mbox{ on }  
      \towr_{m,\xi} \cap W^{\K}_{m,\xi}
    \\
    \refl_{\axis{\sigma}}\varpi_{W^{\K}_{m,\xi}}\refl_{\axis{x}}
      & \mbox{ on }  
      \towr_{m,\xi} \cap \refl_{\axis{x}}W^{\K}_{m,\xi}.
  \end{cases}
\end{aligned}
\end{equation}
We also define the region
\begin{equation}
\label{towr1}
\towr_{m,\xi}^1
  \vcentcolon=
  \varpi_{\towr_{m,\xi}}^{-1}
    \left(\axis{\theta}_{\leq m^{1/2}-1}\right)
  \subset
  \towr_{m,\xi} 
\end{equation}
and supplement definition \eqref{eqn:definition_regions} with a perhaps more intuitive description: 
\begin{itemize}
\item $\discr_{m,\xi}$ is  close to a flat horizontal disc of radius $1-m^{-3/4}$ centered at the origin. 
\item $\catr_{m,\xi}$ is close to the catenoidal annulus $\K_{b_{m,\xi}}$ minus the tubular neighbourhood of radius $m^{-3/4}$ around the equator. 
If $m$ is sufficiently large, $\catr_{m,\xi}$ does \emph{not} contain the region of interpolation between $\Phi$ and $\varphi^1$ or the dislocation of the wing (cf. Figure \ref{fig:initialsurface}).
\item $\towr_{m,\xi}$ is (roughly speaking) the region of the initial surface inside the tubular neighbourhood of radius $m^{-1/2}$ around the equator. 
$\towr_{m,\xi}$ overlaps with $\discr_{m,\xi}$ and $\catr_{m,\xi}$, and contains both the region of interpolation between $\Phi$ and $\varphi^1$ and the dislocation of the wing if $m$ is sufficiently large.
\end{itemize}
Note that for $m$ sufficiently large
$\varpi_{\catr_{m,\xi}}$,
$\varpi_{\discr_{m,\xi}}$,
and
$\varpi_{\towr_{m,\xi}}$
are all diffeomorphisms onto their images, and, if we let
$
 \pi_{(m)}
 \colon
 \R^3/\sk{\trans^{\axis{z}}_{2m\pi}}
 \to 
 \R^3/\sk{\trans^{\axis{z}}_{2\pi}}
$
be the unique map such that
$\varpi_{(1)}=\pi_{(m)} \circ \varpi_{(m)}$,
then
$\pi_{(m)} \circ \varpi_{\towr_{m,\xi}}$
is an $m$-fold covering of its image.
Additionally,
$\varpi_{\discr_{m,\xi}}$
commutes with each element of $\apr_m$
and $\varpi_{\catr_{m,\xi}}$
commutes with each element of $\pyr_m$,
while,
recalling \eqref{HLV},
\begin{align*}
\pi_{(m)} \circ \varpi_{\towr_{m,\xi}} \circ \refl_{\axis{x}}|_{\towr_{m,\xi}}
  &=
  \mathsf{L} \circ \pi_{(m)} \circ \varpi_{\towr_{m,\xi}},
  %\quad \mbox{and} 
  \\
\pi_{(m)} \circ
 \varpi_{\towr_{m,\xi}} \circ \refl_{\{y = x \tan (\pi/2m)\}}|_{\towr_{m,\xi}}
  &=
  \mathsf{H} \circ \pi_{(m)} \circ \varpi_{\towr_{m,\xi}}.
\end{align*}

\subsection{Basic properties of the initial surfaces and comparison with the models}

\begin{proposition}
[Basic properties of the initial surfaces]
\label{initsurfbasicprops}
For each $c>0$
there exists $m_0=m_0(c)>0$
such that for every $\xi \in \IntervaL{-c,c}$
and every integer $\glsuseri{m}>m_0$
the initial surface 
\gls{Sigmamxi} %$\Sigma_{m,\xi}$
 defined in \eqref{initsurfdef} 
has the following properties.
\begin{enumerate}[label={\normalfont(\roman*)}]
  \item \label{initsurfproperembedding}
        $\Sigma_{m,\xi}$ is a connected
        smooth surface with boundary
        and is properly embedded in $\B^3$.
  \item \label{initsurfgenus}
        $\Sigma_{m,\xi}$ has genus $m-1$.
  \item \label{initsurfbdy}
        $\partial \Sigma_{m,\xi}$ has three connected components.
  \item \label{initsurforthoint}
        $\Sigma_{m,\xi}$ meets $\partial \B^3$ orthogonally.
  \item \label{initsurfsymgroup}
        Recalling \eqref{aprdef},
        $\Aut_{\B^3}(\Sigma_{m,\xi})=\apr_m$.
  \item \label{initsurfsymsigns}
        Recalling \eqref{symsign},
        $\sgn_{\Sigma_{m,\xi}} \refl_{\{y=x \tan(\pi/(2m))\}}
          =
          -\sgn_{\Sigma_{m,\xi}} \refl_{\axis{x}}
          =1
        $.
\item\label{initsurfuniformgeometry}
      For each integer $k \geq 0$
      the $k$-th covariant derivative of the second
      fundamental form of the dilated surface $m\Sigma_{m,\xi}$ (in $m\B^3$)
      has norm bounded by some $C(k)>0$
      independent of $m$, $c$, and $\xi$.
\end{enumerate}
\end{proposition}

\begin{proof}
Item \ref{initsurfproperembedding}
follows directly from the definition of $\Sigma_{m,\xi}$,
assuming $m$ sufficiently large.
Item \ref{initsurfbdy} is also straightforward from the definition
of $\Sigma_{m,\xi}$,
making use of
item \ref{towcapxeq0}
of Proposition \ref{karcher-scherk}.
The topological doubling of $\Sigma_{m,\xi}$
therefore has, 
using also
item \ref{genusofnontrivialquotient}
of Proposition \ref{karcher-scherk},
genus $2m$,
proving item \ref{initsurfgenus},
in view of item \ref{initsurfbdy}.
Item \ref{initsurforthoint}
is a consequence of the definition of $\Sigma_{m,\xi}$,
the definition of $\K_b$ (as per Lemma \ref{lem:catenoid-K_b}),
the orthogonality,
as in item \ref{towcapxeq0}
of Proposition \ref{karcher-scherk},
of the intersection of
$\tow$ with the symmetry plane containing the axis of periodicity,
and the fact, clear from \eqref{Phullback},
that this orthogonality is preserved
by $\frac{1}{m} \circ \Phi$.
Items \ref{initsurfsymgroup} 
and \ref{initsurfsymsigns}
follow readily from the definition of $\Sigma_{m,\xi}$,
the $\Ogroup(2)$-invariance of $\K_b$, equation
\eqref{Phintertwine},
and items \ref{towsymgroup} and \ref{towsymsigns}
of Proposition \ref{karcher-scherk}.
To help verify that $\Sigma_{m,\xi}$
has no other symmetries
note that any symmetry must preserve, as a set,
the component of $\partial \Sigma_{m,\xi}$
closest to the equator.
Actually, in the sequel (cf. end of Section \ref{subs:ProofMainThm}) we will need merely
the containment
$\Aut_{\B^3}(\Sigma_{m,\xi}) \geq \apr_m$, which is indeed clear from the previous remarks,
and so we leave the (straightforward)
details of verifying equality to the interested reader.
Finally, item \ref{initsurfuniformgeometry}
is obvious since each blown-up initial surface
$m\Sigma_{m,\xi}$
is covered by finitely many regions each of which
is for any integer $k \geq 0$
a $C^k$ perturbation,
uniformly bounded in $m$ and $\xi$,
of a region of $\tow^+$, $m\K_0$, or $m\B^2$.
\end{proof}

For use in the following lemma
and later use in Proposition \ref{globalsol}
we define
\begin{equation}
\label{equatorialbdy}
\partial^0\Sigma_{m,\xi}
\vcentcolon=
\towr_{m,\xi} \cap \partial \Sigma_{m,\xi},
\end{equation}
the middle
(that is: closest to the equator $\Sp^1$)
boundary component of $\Sigma_{m,\xi}$.
Recalling \eqref{towcokerandgen},
we also define on $\Sigma_{m,\xi}$
the smooth, compactly supported,
$\apr_m$-equivariant function
\begin{equation}
\label{coker}
\gls{coker} %\coker
\vcentcolon= \varpi_{\towr_{m,\xi}}^*\towcoker,
\end{equation}
extended to be constantly zero on
$\Sigma_{m,\xi}$ outside its support in $\towr_{m,\xi}$.
(Note that, by its $\Aut_{\R^3}(\tow^+)$-equivariance,
$\towcoker$ obviously descends to a function
on the quotient $\towquot_{(m)}^+$,
and it is really this function
we mean in place of $\towcoker$
in the above definition.)
Finally,
in the statement and proof of Proposition \ref{initsurfcomp},
given a surface $\Sigma$ embedded in $\R^3$
(or $\B^3$),
we agree to write
$g^{\vphantom{|}}_\Sigma$
for the metric on $\Sigma$
induced by the ambient Euclidean metric.

\begin{proposition}
[Regionwise comparison of the initial surfaces with the models]
\label{initsurfcomp}
There exists $C>0$
and
for each $c>0$
there exists $m_0=m_0(c)>0$
such that 
for every $\xi \in \IntervaL{-c,c}$,
every integer $m>m_0$,
and every $\alpha,\beta \in \interval{0,1}$
the following estimates hold.
\begin{enumerate}[label={\normalfont(\roman*)}]
\item \emph{Riemannian metric comparison.}
\begin{enumerate}[label={\normalfont(\roman{enumi}.\roman*)}]
\item\label{MetCompTow}
       $
        \nm[\Big]{
          m^2g^{\vphantom{|}}_{\Sigma_{m,\xi}}
          -\varpi_{\towr_{m,\xi}}^*
            g^{\vphantom{|}}_{\towquot_{(m)}^+}
        }_{2,\alpha}
        \leq
        Cm^{-1/4}
       $ 
\item\label{MetCompDiscAndCat}
       $
        \nm[\Big]{
          g^{\vphantom{|}}_{\Sigma_{m,\xi}}
          -\varpi_{\discr_{m,\xi}}^*
            g^{\vphantom{|}}_{\B^2}
        }_{2,\alpha}
        +
        \nm[\Big]{
          g^{\vphantom{|}}_{\Sigma_{m,\xi}}
          -\varpi_{\catr_{m,\xi}}^*
            g^{\vphantom{|}}_{\K_{b_m,\xi}}
        }_{2,\alpha}
        \leq
        Cm^{2+\alpha}e^{-m^{1/4}}
       $ 
\end{enumerate}   
\item \emph{Mean curvature comparison.}
\begin{enumerate}[label={\normalfont(\roman{enumi}.\roman*)}]
\item\label{HestTow}
       $
       \nm[\big]{
           (1+\dist_{\axis{\theta}})^{-1}
           \varpi_{\towr_{m,\xi}}^{-1*}
             (H_{\Sigma_{m,\xi}}-\xi \coker)
          }_{0,\alpha,1}
       \leq
       C
       $, and so
       
       $
       \nm[\big]{
           \varpi_{\towr_{m,\xi}}^{-1*}
             (H_{\Sigma_{m,\xi}}-\xi \coker)
          }_{0,\alpha,\beta}
       \leq
       C/(1-\beta)
      $ 
\item\label{HestOffTow} 
      $
       \nm[\Big]{
           (\varpi_{\catr_{m,\xi}}^{-1*}H_{\Sigma_{m,\xi}})|_{
             \varpi_{\catr_{m,\xi}}
               (
                 \Sigma_{m,\xi}
                 \setminus
                 \towr_{m,\xi}^1
                )
             }
          }_{0,\alpha}
      \!+\nm[\Big]{
           (\varpi_{\discr_{m,\xi}}^{-1*}H_{\Sigma_{m,\xi}})|_{
             \varpi_{\discr_{m,\xi}}
               (
                 \Sigma_{m,\xi}
                 \setminus
                 \towr_{m,\xi}^1
                )
             }
          }_{0,\alpha}
        \leq 
        Cm^2e^{-m^{1/2}}
      $ 
\end{enumerate}   
\item \emph{Pull-back action on H\"older norms.} For each $S_{m,\xi} \in \{\discr_{m,\xi}, \catr_{m,\xi}\}$ and $k=0,1,2$ we have 
\begin{enumerate}[label={\normalfont(\roman{enumi}.\roman*)}]
\item\label{TowrToNotTowr}
      $\nm[\big]{\varpi_{S_{m,\xi}}^{-1*}\varpi_{\towr_{m,\xi}}^*u}_{k,\alpha}\leq Cm^{k+\alpha}\nm{u}_{k,\alpha}$
       for all compactly supported $u \in C^{k,\alpha}(\varpi_{\towr_{m,\xi}}(S_{m,\xi} \cap \towr_{m,\xi}))$
\item\label{NotTowrToTowr}
      $\nm[\big]{\varpi_{\towr_{m,\xi}}^{-1*}\varpi_{S_{m,\xi}}^*u}_{k,\alpha}\leq C\nm{u}_{k,\alpha}$
      for all compactly supported $u \in C^{k,\alpha}(\varpi_{S_{m,\xi}}(S_{m,\xi} \cap \towr_{m,\xi}))$
\end{enumerate}      
\item \emph{Jacobi operators comparison.} Recall \eqref{eq:JacobiOp}. 
\begin{enumerate}[label={\normalfont(\roman{enumi}.\roman*)}]
\item\label{JacOpCompTow}
      $
       \nm[\Big]{
         \Big(
           m^{-2}\varpi_{\towr_{m,\xi}}^{-1*}J_{\Sigma_{m,\xi}}\varpi_{\towr_{m,\xi}}^*
           -J_{\towquot_{(m)}^+}
          \Big)u
       }_{0,\alpha,\beta}
       \leq
       Cm^{-1/4}\nm{u}_{2,\alpha,\beta}
      $ 
      $~\forall u \in C^{2,\alpha}\bigl(\varpi_{\towr_{m,\xi}}(\towr_{m,\xi})\bigr)
     $ 
\item\label{JacOpCompDisc} 
      $
       \nm[\Big]{
         \left(
           \varpi_{\discr_{m,\xi}}^{-1*}J_{\Sigma_{m,\xi}}\varpi_{\discr_{m,\xi}}^*
           -J_{\B^2}
          \right)u
       }_{0,\alpha}
       \leq
       Cm^2e^{-m^{1/4}}\nm{u}_{2,\alpha}
       $
     $~\forall u \in C^{2,\alpha}(\varpi_{\discr_{m,\xi}}(\discr_{m,\xi}))$ 
\item\label{JacOpCompCat}
      $
       \nm[\Big]{
         \left(
           \varpi_{\catr_{m,\xi}}^{-1*}J_{\Sigma_{m,\xi}}\varpi_{\catr_{m,\xi}}^*
           -J_{\K_{b_m,\xi}}
          \right)u
       }_{0,\alpha}
       \leq
       Cm^2e^{-m^{1/4}}\nm{u}_{2,\alpha}
       $
      $~\forall u \in C^{2,\alpha}(\varpi_{\catr_{m,\xi}}(\catr_{m,\xi}))$ 
    \end{enumerate}
 \item \emph{Robin operators comparison.} Recall \eqref{Robin_op_general_def}.    
 For all $u \in C^{2,\alpha}\bigl(\towquot_{(m)}^+\bigr)$  we have
\begin{enumerate}[label={\normalfont(\roman{enumi}.\roman*)}]
\item\label{BdyOpCompTow} 
      $
       \nm[\Big]{
         m^{-1}\varpi^{-1}_{\towr_{m,\xi}}|_{\partial \towquot_{(m)}^+}^{*}
           (B^{\mathrm{Robin}}_{\Sigma_{m,\xi}}
             \varpi_{\towr_{m,\xi}}^*u)\big|_{\partial^0\Sigma_{m,\xi}}
         -B^{\mathrm{Robin}}_{\towquot_{(m)}^+}u
       }_{1,\alpha}
       \leq
       Cm^{-1}\nm{u}_{2,\alpha}$

\item\label{BdyOpCompCat} 
      $
       B^{\mathrm{Robin}}_{\Sigma_{m,\xi}}
         \varpi_{\catr_{m,\xi}}^*
       =
       \varpi_{\catr_{m,\xi}}^*
         B^{\mathrm{Robin}}_{\K_{b_m,\xi}}
       $\quad
       on $\Sp^2 \cap \partial \catr_{m,\xi}$.
\end{enumerate}
\end{enumerate}
We emphasize that while $m_0$ depends on $c$,
the constant $C$ is independent of $c$, $\xi$, and $m$.
\end{proposition}

\begin{proof}
We start with the items
that follow with little computation
or that are immediate consequences of other items.
Item \ref{BdyOpCompCat}
is clear since
$\catr_{m,\xi}$ and $\K_{b_{m,\xi}}$ coincide
on a neighborhood of their upper boundary circle.
To check item \ref{BdyOpCompTow}
we merely observe that on $\partial^0\Sigma_{m,\xi}$
the blown-up derivative $m\,d\varpi_{\towr_{m,\xi}}$,
by virtue of \eqref{Phullback}
and the definition, \eqref{initsurfdef}, of $\Sigma_{m,\xi}$,
takes the outward unit conormal of $\Sigma_{m,\xi}$
to the outward unit conormal of $\tow^+_{(m)}$ 
and that the boundary of 
$\varpi_{(m)}(\{\sigma\geq0\})$
is totally geodesic
whilst the second fundamental form of the boundary of $m\B^3$
has norm $\sqrt{2}/m$.
Items \ref{TowrToNotTowr}
and \ref{NotTowrToTowr}
follow at once from
items \ref{MetCompTow}
and \ref{MetCompDiscAndCat}.

Next note that the various Jacobi operators
appearing in the statement
are of course local operators,
so the estimates for them can be proven by
neighborhoodwise comparison of the induced metrics
and second fundamental forms defining the operators
in question
(and the decay estimate in item \ref{JacOpCompTow}
is an immediate consequence of corresponding local
H\"{o}lder estimates).
Naturally items \ref{HestTow} and \ref{HestOffTow}
will be obtained by comparison of the induced mean curvatures
(identically zero of course for
$\towquot^+_{(m)}$,
$\B^2$,
and $\K_{b_{m,\xi}}$),
but then in fact items
\ref{JacOpCompTow}--\ref{JacOpCompCat}
follow from these same comparisons
supplemented by comparisons of the induced metrics,
since via the Gauss equation
we thereby obtain comparisons for the Schr\"odinger potentials
of the Jacobi operators.

Thus, as a result of this discussion, it only remains to prove items
\ref{MetCompTow}--\ref{HestOffTow}.
Let us write, as above, $g_{\mathrm{euc}}$
(respectively: $\widehat{g}_{\mathrm{euc}}$)
for the standard Euclidean metrics on $\R^3$ with coordinates $(x,y,z)$ (respectively: with coordinates $(\sigma,\psi,\theta)$)
and let us also agree not to modify the notation (understood as coordinates and metrics) when passing to quotients; in particular we will equally employ $\widehat{g}_{\mathrm{euc}}$ as a Riemannian metric on $\R^3/\sk{\trans^{\axis{\theta}}_{2m\pi}}$.
The homothety $m^{-1}\colon\R^3 \to \R^3$ 
descends to a map
$
 \R^3/\sk{\trans^{\axis{\theta}}_{2m\pi}}
 \to
 \R^3/\sk{\trans^{\axis{\theta}}_{2\pi}}
$,
which we will also call $m^{-1}$.
From \eqref{Phullback} we have
\begin{equation*}
\nm[\big]{
  (
    m \circ \Phi \circ m^{-1}
  )^*g_{\mathrm{euc}}
  -\widehat{g}_{\mathrm{euc}}
}_{C^k(\axis{\theta}_{\leq 3R_{\mathrm{tow}}})} 
\leq
\frac{C(k)}{m},
\end{equation*}
where
the $C^k$ norm is defined using
the metric $\widehat{g}_{\mathrm{euc}}$.
Recalling \eqref{bmxi}
and using \eqref{phivsPhi}
and the bound
\begin{equation*}
\abs{\omega_{b_{m,\xi}} - \omega_0 + \kappa_{m,\xi}}
\leq
C \kappa_{m,\xi}
\leq
\frac{C}{m}\Bigl(1+\frac{c}{m}\Bigr),
\end{equation*}
where $C>0$ is independent of $m$, $c$, and $\xi$,
we find, recalling also \eqref{Phii},
that 
\begin{equation*}
\nm[\big]{
  (m \circ \Phi^i_{m,\xi} \circ m^{-1})^*g_{\mathrm{euc}}
  -\widehat{g}_{\mathrm{euc}}
}_{
    C^k(mP^i_{<1} \cap \axis{\theta}_{
        \leq 2R})
  } 
\leq
\frac{C(k,R)}{m}
\end{equation*}
for each $i=0,1$,
every $R \geq 3R_{\mathrm{tow}}$,
and some $C(k,R)>0$ independent of $m$, $c$, and $\xi$
(assuming $m$ large enough in terms of $c$ and $R$).
We will fix $R$ later in the proof
(independently of $m$, $c$, and $\xi$).
From the above estimates, Lemma \ref{lemmatowbent}
(to relate $\towbent^+_{m,\xi}$ to $\tow^+_m$),
and the definition \eqref{initsurfdef} of the initial surfaces,
it then follows for each integer $k \geq 0$
\begin{align*}
\nm[\big]{
  (m \circ \varpi_{\towr_{m,\xi}}^{-1})^*
    g^{\vphantom{|}}_{m\Sigma_{m,\xi}}
  -g^{\vphantom{|}}_{\towquot^+_{(m)}}
}_{C^k(\axis{\theta}_{\leq 2R})}
&\leq
\frac{C(k,R)}{m}, 
\\[1ex]
\nm[\Big]{
  (m \circ \varpi_{\towr_{m,\xi}}^{-1})^*H_{m\Sigma_{m,\xi}}
  -\frac{\xi}{m} \towcoker
}_{C^k(\axis{\theta}_{\leq 2R})}
&\leq 
\frac{C(k,R)}{m},
\end{align*}
where the $C^k$ norms are defined using the metric $\widehat{g}_{\mathrm{euc}}$,
each constant $C(k,R)>0$ is independent of $c$, $\xi$, and $m$,
and we assume $m$ sufficiently large in terms of $c$ and $R$.
This proves \ref{MetCompTow} and \ref{HestTow} appropriately restricted to
$\varpi_{\towr_{m,\xi}}^{-1} (\axis{\theta}_{\leq 2R})$.
Clearly items \ref{MetCompDiscAndCat} and \ref{HestOffTow} hold trivially when correspondingly restricted
to the regions where $\Sigma_{m,\xi}$ agrees
exactly with $\K_{b_{m,\xi}}$ or $\B^2$.

To complete the estimates on the remainder of $\Sigma_{m,\xi}$
we will appeal to Lemma \ref{lem:mc_and_met_of_graphs_euc}.
Away from $\Sp^1$
and modulo the symmetries,
$\Sigma_{m,\xi}$
has been constructed by transferring
the defining functions
of the wings of $\towbent_{m,\xi}$
over their asymptotic half planes
to (subsets of) $\B^2$ and $\K_{b_{m,\xi}}$
to generate the corresponding graphs
over the latter surfaces.
We will therefore use Lemma \ref{lem:mc_and_met_of_graphs_euc}
to compare (subsets of) the initial surfaces
to $\B^2$ and $\K_{b_{m,\xi}}$
and also to the wings of $\towbent_{m,\xi}$.
Specifically,
we will apply Lemma \ref{lem:mc_and_met_of_graphs_euc}
multiple times
with
\begin{equation*}
\begin{aligned}
\Sigma
  &=
  \Lambda^i
  \vcentcolon=
  \widehat{\Pi}^i
    \cap
    \{
      R \leq \dist_{\axis{\theta}} \leq m^{3/4}+2
    \},
\\
\phi_1
  &=
  \phi^i_1
  \vcentcolon=
  \mbox{the inclusion map of $\Lambda^i$ in $\R^3$},
\\
\phi_2
  &=
  \phi^i_2
  \vcentcolon=
  m \circ \varphi^i 
    \circ \frac{1}{m}
    \circ \varpi_{(m)}|_{\Lambda^i},
\\
u
  &=
  u^i
  \vcentcolon=
  \widehat{w}^i|_{\Lambda^i},
\end{aligned}
\end{equation*}
for each $i=0,1$,
and also in some instances with $u=0$
or with $\phi_1=\phi_2^i$;
for the preceding
we recall
from \eqref{initsurfphidef}
the definitions of the maps $\varphi^i$
and
from Lemma \ref{lemmatowbent}
of the asymptotic half planes
$\widehat{\Pi}^0$ and 
$\widehat{\Pi}^1=\widehat{\Pi}^1_{m,\xi}$
to $\towbent^+_{m,\xi}$
and on each of these
the defining function
$\widehat{w}^i=\widehat{w}^i_{m,\xi}$
of the corresponding wing
$\widehat{W}^i$
of $\towbent^+_{m,\xi}$.
Thus $\phi_2^0$
is a parametrization over $\Lambda^0$
of a subset of $m\B^2$,
while $\phi_2^1$
is a parametrization over $\Lambda^1$
of a subset of $m\K_{b_{m,\xi}}$.

We continue to assume
that $R \geq 3R_{\mathrm{tow}}$,
and, by taking $m$ large enough,
we can without loss of generality
assume also
$R<m^{3/4}$.
To prepare for the application of Lemma \ref{lem:mc_and_met_of_graphs_euc}
we first interpret each map $\phi[u]$
(corresponding to the various choices of $\phi$ and $u$ above)
and observe some preliminary, supporting estimates.

For each $i=0,1$ let
$\varpi_{\widehat{W}^i}\colon \widehat{W}^i \to \widehat{\Pi}^i$
be nearest-point projection in Euclidean $\R^3$.
Then
\begin{equation}
\label{phi1w}
\phi^i_1[\widehat{w}^i]
 =
  \varpi_{\widehat{W}^i}^{-1}
\qquad
\mbox{for } i=0,1,
\end{equation}
and moreover,
referring to
\eqref{towbent}
and
\eqref{regionalprojections},
\begin{equation}
\label{phi2w}
\phi^i_2[\widehat{w}^i]
=
\begin{cases}
  m \circ \varpi_{\towr_{m,\xi}}^{-1}
    \circ \varpi_{(m)}
    \circ \widehat{\phi}_{m,\xi}^{-1}
    \circ \phi^i_1[\widehat{w}^i]
  &\mbox{on }
    \phi_2^i[\widehat{w}^i]^{-1}(m\towr_{m,\xi})
\\
  m \circ \varpi_{\discr_{m,\xi}}^{-1} 
    \circ \frac{1}{m} \circ \phi^0_2
  &\mbox{on }
    \phi^0_2[\widehat{w}^0]^{-1}(m\discr_{m,\xi})
    \quad
    \mbox{when }
    i=0
\\
  m \circ \varpi_{\catr_{m,\xi}}^{-1}
    \circ \frac{1}{m} \circ \phi^1_2
   &\mbox{on }
    \phi^1_2[\widehat{w}^1]^{-1}(m\catr_{m,\xi})
    \quad
    \mbox{when }
    i=1. 
\end{cases}
\end{equation}
Thus,
recalling also \eqref{wingsinball},
$\phi_1^0[\widehat{w}^0]$
and $\phi_2^0[\widehat{w}^0]$
are parametrizations
over $\Lambda^0$
of subsets of
$\widehat{W}^0$
and $mW^{\B^2}_{m,\xi}$ respectively,
while
$\phi_1^1[\widehat{w}^1]$
and $\phi_2^1[\widehat{w}^1]$
are parametrizations
over $\Lambda^1$
of subsets of
$\widehat{W}^1$
and $mW^{\K}_{m,\xi}$ respectively.
In particular we have
\begin{equation}
\label{ests_for_phi_1_phi_2_and_phi_1_w}
\begin{aligned}
&A_{\phi^0_1}=A_{\phi^0_2}=A_{\phi^1_1}=0,  
\\[1ex]
&H[\phi^i_j,0]
=0
\qquad \mbox{for } i=0,1, \quad j=1,2,
\\[1ex]
&H[\phi^i_1,\widehat{w}^i]
=(\varpi_{\widehat{W}^i}^{-1*}
    H_{\towbent^+_{m,\xi}}
  )|_{\Lambda^i}
\quad \mbox{for } i=0,1.
\end{aligned}
\end{equation}
We also observe
the pointwise bounds,
for $i=0,1$ and every integer $k \geq 0$,
\begin{align}
\label{phi_2_vs_phi_1}
&\abs{
  g_{\phi^i_2}-g_{\phi^i_1}
  }_{g_{\phi^i_1}}
  \leq
  C\frac{1+\dist_{\axis{\theta}} \circ \phi^i_1}{m},
&\abs{
    D_{g_{\phi^i_1}}^{k+1}
    g_{\phi^i_2}
  }_{g_{\phi^i_1}}
  +
  \abs{
    D_{g_{\phi^1_2}}^k
    A_{\phi^1_2}
  }_{g_{\phi^1_2}}
  \leq
  \frac{C(k)}{m^{k+1}},
\end{align}
which follow easily from the definitions
of $\phi^i_1$ and $\phi^i_2$.
Given the extent of $\Lambda^i$,
the bounds of \eqref{phi_2_vs_phi_1}
obviously imply
\begin{equation*}
\nm{
    g[\phi^i_2,0]
    -g[\phi^i_1,0]
  }_3
  \leq
  Cm^{-1/4}
\qquad
\mbox{for } i=0,1,
\end{equation*}
where each norm is the $C^3$ norm defined by $g_{\phi^i_1}$;
in particular, for the estimates below
we have equivalence
(through constants independent of $m$, $c$, and $\xi$)
of the norms induced by $g_{\phi^i_1}$ and $g_{\phi^i_2}$. 

We now apply
Lemma \ref{lem:mc_and_met_of_graphs_euc}
in conjunction with
\eqref{ests_for_phi_1_phi_2_and_phi_1_w},
\eqref{phi_2_vs_phi_1},
and the estimates of $\widehat{w}^i$
and $H_{\towbent^+}$ from Lemma~\ref{lemmatowbent}
(and also the fact that
$\dist_{\axis{\theta}} \leq m^{3/4}+2$
on both $\Lambda^0$ and $\Lambda^1$)
to obtain the following further estimates,
in which each norm is defined
using the metric $g_{\phi^i_1}$.
For these applications of Lemma \ref{lem:mc_and_met_of_graphs_euc}
we take $k=2$, $n=3$ and
$R$ large enough
(but independent of $m$, $c$, and $\xi$)
that $\nm{\widehat{w}^i|_{\Lambda^i}}_{4}$ is so small as required by the assumption of such lemma,
as ensured by item \ref{towbent_wing_defining_func_est}
of Lemma \ref{lemmatowbent};
we also take $m$ large enough
(but independent of $c$ and $\xi$)
that
$
 \nm{g_{\phi^i_2}-g_{\phi^i_1}}_3
 +\nm{A_{\phi^i_2}-A_{\phi^i_1}}_3
$
is likewise
small for the application of the same lemma,
as \eqref{phi_2_vs_phi_1} ensures we can do.
Then, in particular:
\begin{itemize}
\item{by item \ref{met_comp_euc_same_def_funcs}
of Lemma \ref{lem:mc_and_met_of_graphs_euc}
we have
\begin{equation}
\label{gphi1w_vs_gphi2w}
\nm{
    g[\phi^i_2,\widehat{w}^i]
    -g[\phi^i_1,\widehat{w}^i]
  }_3
  \leq
  Cm^{-1/4}
\qquad
\mbox{for } i=0,1;
\end{equation}
}
\item{
by item \ref{met_comp_euc_same_background}
of Lemma \ref{lem:mc_and_met_of_graphs_euc}
(actually with
$\phi_1\vcentcolon=\phi_j^i$
for both cases of $j=1,2$)
\begin{equation}
\label{gphiw_vs_gphi}
\nm{
    g[\phi^i_j,\widehat{w}^i]
    -g[\phi^i_j,0]
  }_{3,0,1}
  \leq
  C
\qquad
\mbox{for } i=0,1,
  \quad j=1,2;
\end{equation}
}
\item{
by item \ref{mc_comp_euc_same_background}
of Lemma \ref{lem:mc_and_met_of_graphs_euc}
(actually with
$\phi_1\vcentcolon=\phi^i_2$)
\begin{equation}
\label{Hphi2w}
\nm{H[\phi_2^i,\widehat{w}^i]}_{1,0,1}
\leq
C
\qquad
\mbox{for } i=0,1;
\end{equation}
}
\item{by item \ref{mc_quad_comp_euc}
of Lemma \ref{lem:mc_and_met_of_graphs_euc}
\begin{equation}
\label{Hphi2w_quad}
\nm{
    (1+\dist_{\axis{\theta}} \circ \phi^i_1)^{-1}
    \cdot 
    H[\phi^i_2,\widehat{w}^i]
  }_{1,0,1}
  \leq
  Cm^{-1}
\qquad
\mbox{for } i=0,1.
\end{equation}
}
\end{itemize}

We now complete the verification of items
\ref{MetCompDiscAndCat}
and \ref{HestOffTow}.
First, making use of the final two lines of
\eqref{phi2w},
we observe the equalities
\begin{align*}
g[\phi^0_2,0]
&=g_{\phi^0_2}
={\phi^0_2}^*m^{-1*}g^{\vphantom{|}}_{\B^2},
&
g[\phi^0_2,\widehat{w}^0]
&={\phi^0_2}^*m^{-1*}\varpi_{\discr_{m,\xi}}^{-1*}
g^{\vphantom{|}}_{\Sigma_{m,\xi}},
\shortintertext{as well as}
g[\phi^1_2,0]
&=g_{\phi^1_2}
={\phi^1_2}^*m^{-1*}g^{\vphantom{|}}_{\K_{b_{m,\xi}}},
&
g[\phi^1_2,\widehat{w}^1]
&={\phi^1_2}^*m^{-1*}\varpi_{\catr_{m,\xi}}^{-1*}
g^{\vphantom{|}}_{\Sigma_{m,\xi}}.
\end{align*}
By applying these in the estimate
\eqref{gphiw_vs_gphi}
with $j=2$
(for both cases $i=0$ and $i=1$)
the proof of item
\ref{MetCompDiscAndCat}
is completed.
To pass from the weighted estimate
of \eqref{gphiw_vs_gphi}
to the unweighted estimate
of item \ref{MetCompDiscAndCat}
we have also made use
of the definition
of the disc and catenoidal regions
in \eqref{eqn:definition_regions}.
Similarly we complete the proof of item \ref{HestOffTow}
by applying the equalities
\begin{align*}
H[\phi^0_2,\widehat{w}^0]
&=m^{-1}
  {\phi^0_2}^*
      m^{-1*} 
      \varpi_{\discr_{m,\xi}}^{-1*}
  H_{\Sigma_{m,\xi}},
&
H[\phi^1_2,\widehat{w}^1]
&=m^{-1}
  {\phi^1_2}^*
      m^{-1*} 
      \varpi_{\catr_{m,\xi}}^{-1*}
  H_{\Sigma_{m,\xi}}
\end{align*}
in the estimate \eqref{Hphi2w}.

To complete the verification of item \ref{MetCompTow} we apply in estimate \eqref{gphi1w_vs_gphi2w} the equalities
\begin{align*}
g[\phi^i_1,\widehat{w}^i]
&=\varpi_{\widehat{W}^i}^{-1*}
  g^{\vphantom{|}}_{\towbent^+_{m,\xi}},
&
g[\phi^i_2,\widehat{w}^i]
&=\varpi_{\widehat{W}^i}^{-1*}
      \widehat{\phi}_{m,\xi}^{-1*}
      \varpi_{(m)}^*
      \varpi_{\towr_{m,\xi}}^{-1*}
    g^{\vphantom{|}}_{\Sigma_{m,\xi}}
\end{align*}
(referring to \eqref{phi1w}
and the first line of \eqref{phi2w});
in fact,
by virtue of
\eqref{gphiw_vs_gphi}
with $j=1$,
we can take $R$ large enough
(but independent of $m$, $c$, and $\xi$)
so that
\eqref{gphi1w_vs_gphi2w}
continues to hold
(with possibly different choices of $C$,
still independent of $m$, $c$, and $\xi$)
if the norm there is defined
by $g[\phi^i_1,\widehat{w}^i]$
instead of $g[\phi^i_1,0]=g_{\phi^i_1}$.
From this estimate
(again referring to \eqref{eqn:definition_regions}
to pass from the weighted estimate
to the unweighted one)
we obtain
item \ref{MetCompTow}
(on the complement of a suitably large
neighborhood of $\Sp^1$,
where we have already established the estimate) but with
$g_{\towquot_{(m)}^+}$
replaced by
$\widehat{\phi}_{m,\xi}^*g_{\towbent_{m,\xi}^+}$
(understood on the quotient $\towquot_{(m)}^+$).
The proof of item \ref{MetCompTow}
is then completed by invoking
the triangle inequality
along with items
\ref{towbent_action_of_phi_hat}
and
\ref{towbent_metric_est}
of Lemma \ref{lemmatowbent}.
Similarly
we complete the proof of
item \ref{HestTow}
by relying upon
item \ref{towbent_mc_straighten_est}
of Lemma \ref{lemmatowbent}
in conjunction with
the estimate \eqref{Hphi2w_quad}
and the equalities
\begin{align*}
H[\phi^i_1,\widehat{w}^i]
&=\varpi_{\widehat{W}^i}^{-1*} 
   H^{\vphantom{|}}_{\towbent^+_{m,\xi}},
&
H[\phi^i_2,\widehat{w}^i]
&=m^{-1}
 \varpi_{\widehat{W}^i}^{-1*}
    \widehat{\phi}_{m,\xi}^{-1*}
    \varpi_{(m)}^*
    \varpi_{\towr_{m,\xi}}^{-1*} 
  H^{\vphantom{|}}_{\Sigma_{m,\xi}}.
\qedhere
\end{align*}
\end{proof}

\section{Linearized problem}\label{sec:Linear}

We will start our discussion by recalling a few relatively standard facts that will be employed both in Subsection \ref{subs:LinDisc} and then later in the article.
 
Let $\Sigma \subset \B^3$ be a properly embedded surface (thus with smooth boundary contained in $\Sp^2$).
Suppose also that the boundary of $\Sigma$ is partitioned as 
$\partial \Sigma = \partial_D \Sigma \cup \partial_R \Sigma$,
where $\partial_D \Sigma \cap \partial_R \Sigma = \emptyset$
and each of $\partial_D \Sigma$, $\partial_R \Sigma$
is a (possibly empty) union of connected components of $\partial \Sigma$. 
We also allow for a (possibly trivial) finite subgroup of isometries $G< \Aut_{\B^3}(\Sigma)$ and we tacitly assume the surface $\Sigma$, as well as each of its boundary components, to be invariant under the action of the elements of $G$. 
(Note that, a priori, there could be isometries that leave $\Sigma$ invariant while interchanging some of its boundary components; we only place the tacit restriction that whenever such a phenomenon happens, then the action on the connected components of the boundary is just a permutation of the components of $\partial_D\Sigma$, and of those of $\partial_R\Sigma$ so without ever changing the type of the boundary condition in question.) 
One can then consider the linear map
\begin{equation}\label{eq:MapT}
\begin{aligned}
T\colon C^{2,\alpha}_{G}(\Sigma)
&\to
  C^{0,\alpha}_{G}(\Sigma)
    \oplus C^{2,\alpha}_{G}(\partial_D \Sigma)
    \oplus C^{1,\alpha}_{G}(\partial_R \Sigma) \\
u
&\mapsto
  \left(
    J_{\Sigma}u,~
    u|_{\partial_D \Sigma},~
    B_{\Sigma}^{\text{Robin}}u|_{\partial_R \Sigma}
  \right)
\end{aligned}
\end{equation} 
where $B^{\text{Robin}}_\Sigma \vcentcolon= \eta_{\Sigma}\cdot\nabla_\Sigma-1$, the Robin boundary operator,
and it is agreed that, in case either $\partial_D\Sigma$ or $\partial_R\Sigma$ is empty,
we simply omit the corresponding slot in the above equations. 

Firstly, the operator in question has a discrete spectrum, as encoded in the following statement (for its proof see e.\,g. Appendix A in \cite{FranzIndex}).

\begin{lemma}\label{lem:BasicLinTheory-Spectrum}
In the setting above
there exists a Hilbertian basis $(\psi_k)_{k\geq 0}\subset C^{\infty}_G(\Sigma)$ of $L^2_G(\Sigma)$ and a non-decreasing sequence $(\lambda_k)_{k\geq 0}\subset \R$ diverging to $+\infty$ such that
\[
\left\{
\begin{aligned}
J_{\Sigma}\psi_k &= -\lambda_k\psi_k &&\text{ on }\Sigma,
\\
\psi_k&=0 &&\text{ on }\partial_D\Sigma,
\\
B_{\Sigma}^{\mathrm{Robin}}\psi_k&=0 &&\text{ on }\partial_R\Sigma.
\end{aligned}
\right.
\]
\end{lemma}

We recall, parenthetically, that if $\Sigma$ is minimal, $\partial_D \Sigma=\emptyset$ and $G$ is the trivial group, then the number of (strictly) negative eigenvalues is precisely the standard \emph{Morse index} of $\Sigma$. If instead one considers non-trivial symmetry groups then one defines the \emph{equivariant} Morse index, cf. \cite{FranzIndex}. We will get back to these notions in Section \ref{sec:Index}.

Secondly, we can also rely upon the basic $C^{2,\alpha}$ Schauder estimates (cf. Section 6.7 of \cite{GilTru2001}); for our purposes we need this (special) result.

\begin{lemma}\label{lem:BasicLinTheory-Schauder}
In the setting above, if $T$ defined in \eqref{eq:MapT} is injective then it is an isomorphism and there exists $C>0$ such that for all $u \in C^{2,\alpha}(\Sigma)$
\begin{equation}\label{eq:AprioriLinEstimate}
\nm{u}_{2,\alpha}\leq C \nm{Tu}
\vcentcolon=C\left[\nm{J_{\Sigma} u}_{0,\alpha}+\nm{u|_{\partial_D \Sigma}}_{2,\alpha}+\nm{B_{\Sigma}^{\mathrm{Robin}}u|_{\partial_R \Sigma}}_{1,\alpha} \right].
\end{equation}
\end{lemma}

The scaling behavior of the Jacobi operator and Robin boundary operator also plays a critical role in our construction.
For any $m>0$ let $m\colon\R^3 \to \R^3$ be the homothety sending $x \in \R^3$ to $mx$.
Then
\begin{equation*}
m^* J_{m\Sigma} = m^{-2}J_\Sigma m^*,
\end{equation*}
and if $\partial \Sigma$ is nonempty
with outward unit conormal $\eta_{\Sigma}$,
interpreted as the Neumann boundary operator,
so that $\partial(m\Sigma)=m\partial \Sigma$
has outward unit conormal $\eta_{{m\Sigma}}=m^{-1}\, dm \, \eta_{\Sigma}$,
then we also have
\begin{equation*}
m^* \eta_{{m\Sigma}} = m^{-1}\eta_{\Sigma} m^*.
\end{equation*}
In particular if $\partial \Sigma$
has a component contained in $\Sp^2$
with Robin boundary operator
$B^{\text{Robin}}_\Sigma$,
then the corresponding Robin boundary operator
on the corresponding component of $m\partial \Sigma$ in $m\Sp^2$ is
$B^{\text{Robin}}_{m\Sigma}= \eta_{m\Sigma}\cdot\nabla_{m\Sigma}-m^{-1}$,
and we have
\begin{equation}\label{BRobin_scaling}
m^*B^{\text{Robin}}_{m\Sigma} = m^{-1}B^{\text{Robin}}_\Sigma m^*.
\end{equation}
Note that, as a result of these simple facts, $B^{\text{Robin}}_{m\Sigma}$
is a small perturbation of the standard Neumann boundary operator
when $m$ is large
(as the second fundamental form of $m\Sp^2$ tends to $0$).

\subsection{Linearized problem on the disc and on catenoidal annuli}\label{subs:LinDisc}

Those general preliminaries being given, we start here our discussion of the linear analysis with the case of the simplest ``block'' in our construction, i.\,e. the central disc $\B^2$, and  recall the definition of the antiprismatic group $\apr_m$ from \eqref{aprdef}. 

\begin{lemma}\label{lem:KeyLinDisc}
For any $m\geq 1$ the map
\begin{equation*}
\begin{aligned}
T_m\colon C^{2,\alpha}_{\apr_m}(\B^2)
&\to 
  C^{0,\alpha}_{\apr_m}(\B^2) 
  \oplus C^{2,\alpha}_{\apr_m}(\partial \B^2) \\
u &\mapsto \left(J_{\B^2}u, u|_{\partial \B^2}\right)
\end{aligned}
\end{equation*}
is invertible, and considered -- as above -- the product Banach norm on the target given by $\nm{(f,\varphi)}=\nm{f}_{0,\alpha}+\nm{\varphi}_{2,\alpha}$ there holds the estimate 
\begin{equation}\label{eq:DiscBound}
\nm{u}_{2,\alpha}\leq C \nm{T_m u}
\end{equation}
for a constant $C>0$ that is independent of $m$.
\end{lemma}

\begin{proof}
Firstly, we note that the Jacobi operator of $\B^2$ (that is just the Laplace operator) acting on the space of smooth functions vanishing on $\partial \B^2$ has trivial kernel: indeed, by the maximum principle for harmonic functions, we have $\Delta_{\B^2} u=0$ in $\B^2$ and $u=0$ on $\partial \B^2$ if and only if $u$ vanishes identically.  Of course, the same conclusion holds true for the domain $C^{2,\alpha}(\B^2)$, to greater extent if we impose additional symmetries (thus \emph{restricting} the domain in question).
Thus, Lemma \ref{lem:BasicLinTheory-Schauder}, applied for $\Sigma=\B^2$ and $G=\apr_m$, gives the desired conclusion and appropriate estimate.
\end{proof}

In the sequel of this article, we let 
\label{gls:PB2}
$\glsuseri{modelresol} %P^{m}_{\B^2}
\colon C^{0,\alpha}_{\apr_m}(\B^2)\oplus C^{2,\alpha}_{\apr_m}(\partial \B^2)\to C^{2,\alpha}_{\apr_m}(\B^2)$ 
denote the inverse of the operator $T_m$, i.\,e. the resolvent operator for the associated elliptic problem; in Section \ref{subs:LinConclusion} we will employ the corresponding continuity estimate in the special case of zero boundary data: 
\begin{equation}\label{eq:ResolveDisc}
\nm{P^{m}_{\B^2}E}_{2,\alpha}\leq C\nm{E}_{0,\alpha}.
\end{equation}

The discussion for catenoidal annuli is similar, at the level of outcome, although somewhat more elaborate and relying on the imposed symmetry group. 
Recalling the definition \eqref{pyrdef} of the pyramidal group $\pyr_m$ we consider the Jacobi operator on $\pyr_m$-equivariant functions on $\K_0$ with Dirichlet data $u|_{C_0}=0$ on the lower boundary circle $C_0$ and Robin data $B^{\mathrm{Robin}}_{\K_{0}}u=0$ as defined in \eqref{Robin_op_general_def} on the upper boundary circle $C_\perp$. 
Note that $\pyr_m=\Aut_{\B^3}(\K_0) \cap \apr_m$ is the largest subgroup of $\apr_m$ preserving $\K_0$.
 
\begin{lemma}[Kernel on $\K_0$]
\label{lem:catker}
There exist $m_0$ and $b_0\in (0,\beta)$ (where $\beta>0$ is provided by Lemma~\ref{lem:catenoid-K_b}), independent of $m_0$, such that for any $m\geq m_0$ and any $0\leq b\leq b_0$ 
the map
\begin{equation}\label{eqn:definition_JacBdyOpOnK0}
\begin{aligned}
T_{m,b}\colon C^{2,\alpha}_{\pyr_m}(\K_{b})
&\to 
  C^{0,\alpha}_{\pyr_m}(\K_{b})
  \oplus C^{2,\alpha}_{\pyr_m}(C_0) \oplus C^{1,\alpha}_{\pyr_m}(C_{\perp})
  \\
u &\mapsto 
  \left(
    J_{\K_{b}}u,~
    u|_{C_0},~
    B^{\mathrm{Robin}}_{\K_{b}}u|_{C_\perp}
  \right)
\end{aligned}
\end{equation}
is invertible, and considered the product Banach norm on the target given by $\nm{(f,\varphi,\psi)}\vcentcolon=\nm{f}_{0,\alpha}+\nm{\varphi}_{2,\alpha}+\nm{\psi}_{1,\alpha}$ there holds the estimate 
\begin{equation}\label{eq:Bound}
\nm{u}_{2,\alpha}\leq C \nm{T_{m,b} u}    
\end{equation}
for a constant $C>0$ that is independent of $m$ and $b$.
\end{lemma}

\begin{proof}
As a first step we prove that if $b=0$ and $m\in\N$ is sufficiently large, then any $\pyr_m$-equivariant eigenfunction of $J_{\K_0}$ with eigenvalue $0$ must be rotationally symmetric. 

Let $u \in C^\infty_{\pyr_m}(\K_0)$ be $L^2(\K_0)$-orthogonal (with respect to the metric on $\K_0$ induced by the ambient Euclidean metric) to the $\Ogroup(2)$-invariant functions on $\K_0$. 
Then, the restriction of $u$ to any circle of constant height $z$ has at least $m$ zeroes. 
If $u$ is an eigenfunction of $J_{\K_0}$, it then has at least $m$ nodal domains, so by the Courant nodal domain theorem the eigenvalue corresponding to $u$ is at least the $m$\textsuperscript{th} eigenvalue of $J_{\K_0}$ (without imposing symmetries and counting with multiplicity as usual). 
Since the eigenvalues of $J_{\K_0}$ tend to infinity, by taking $m$ sufficiently large we conclude that any $\pyr_m$-equivariant eigenfunction of $J_{\K_0}$ with eigenvalue $0$ must be $\Ogroup(2)$-invariant. 

Let $v_{\mathrm{t}},v_{\mathrm{d}}\in C^\infty(\K_0)$ be the (functions associated, by taking the normal component to the surface, to) Jacobi fields of $\K_0$ generated respectively by vertical translation along the axis of $\K_0$, and by varying the waist parameter $a>0$ with $b$ constantly $0$ (as described in Lemma \ref{lem:catenoids}).
Then it is readily checked that $\{v_{\mathrm{t}},v_{\mathrm{d}}\}$ is a linearly independent subset of $C^\infty_{\Ogroup(2)}(\K_0)$. 
On the other hand, since $J_{\K_0}$ is second-order, it follows from an elementary ODE argument that $v_{\mathrm{t}}$ and $v_{\mathrm{d}}$ span the space $V\vcentcolon=\{v \in C^\infty_{\Ogroup(2)}(\K_0) \st J_{\K_0} v=0\}$. 
The subspace $V_0\vcentcolon=\{v \in V \st v|_{C_0}=0\}$ has dimension $1$ since, for example, it is clear from the definition
of $\K_0$ that $v_{\mathrm{t}}$ is nonzero on $C_0$.  On the other hand, it is clear that $v_d$ belongs to $V_0$ and thus it is in fact a generator for $V_0$.
However, equation \eqref{eqn:20211125} implies that $v_d$ does not satisfy the Robin condition $B^{\mathrm{Robin}}_{\K_{0}} u= 0$ on $C_\perp$.
We conclude that the operator defined in \eqref{eqn:definition_JacBdyOpOnK0} has trivial $\Ogroup(2)$-invariant kernel.

At this stage, the invertibility of the operator $T_{m,b}$ and the Schauder estimate, in the case $b=0$, come at once via application of Lemma \ref{lem:BasicLinTheory-Schauder}. 
Then, one needs to note that, for $0\leq b\leq \beta$ (where $\beta$ is provided by Lemma \ref{lem:catenoid-K_b}) the catenoidal annulus $\K_{b}$ becomes an arbitrarily small perturbation of $\K_0$ and so (by openness) it is standard to derive that $T_{m,b}$ is also an isomorphism provided we take $b\leq b(m)$; hence it follows that the claimed inequality is still true in such a range for a marginally larger multiplicative constant on the right-hand side. 
We still need to check that $b(m)$ admits a positive (uniform) lower bound, say $b_0$ as one varies $m\geq m_0$, and jointly that the constant $C$ in \eqref{eq:Bound} can be chosen uniformly for $m\geq m_0$ and $0\leq b\leq b_0$. 
For that purpose, assume towards a contradiction the existence of a sequence of functions $(u_m)_{m\geq m_0}$ in $C^{2,\alpha}_{\pyr_m}(\K_{0})$ such that $\nm{u_m}_{2,\alpha}=1$ for all $m\geq m_0$ but $\nm{T_{m,0}u_m}\to 0$ as $m\to\infty$. 
Appealing to the Arzel\`a--Ascoli compactness theorem we could extract a subsequence converging in $C^2(\K_{0})$ to a non-trivial limit function $u_{\infty}$ (the non-triviality following from the standard Schauder estimates); on the other hand, an elementary argument gives at once that $u_\infty$ must in fact be rotationally symmetric, i.\,e. an $\Ogroup(2)$-invariant functions on $\K_0$. Hence, we appeal to the proof given above to conclude that necessarily $u_{\infty}=0$, a contradiction. 
This completes the proof.
\end{proof}

In analogy with what we did above, we let 
\label{gls:PK}
$\glsuserii{modelresol} %P^{m}_{\K_b}
\colon C^{0,\alpha}_{\pyr_m}(\K_{b})\oplus C^{2,\alpha}_{\pyr_m}(C_0) \oplus C^{1,\alpha}_{\pyr_m}(C_{\perp}) \to C^{2,\alpha}_{\pyr_m}(\K_{b})$
denote the resolvent operator; in Section \ref{subs:LinConclusion} we will employ the corresponding continuity estimate in the special case of homogeneous Dirichlet condition on the lower component of $\K_b$ and homogeneous Robin condition on the upper component of $\partial \K_b$, namely: 
\begin{equation}\label{eq:ResolveCat}
\nm{P^{m}_{\K_b}E}_{2,\alpha}\leq C \nm{E}_{0,\alpha}.
\end{equation}

\subsection{Linearized problem on the towers}\label{subs:LinTow}
We now analyze the Jacobi operator $J_{\tow}$ on the tower $\tow$.
In fact, we explicitly note that the proofs below do not at all depend on the particular value $\vartheta$,
so the main results of this subsection, Lemma~\ref{towker} and Lemma~\ref{towsol}, hold with $\tow$ replaced by $\tow_\vartheta$ for any $\vartheta \in \interval{0,\pi/2}$. 
Although the arguments to follow could be reformulated avoiding use of the Enneper--Weierstrass representation of $\tow$, 
we opted for that approach as it allows for rather simple arguments. 
The reader is referred to the discussion presented in Appendix \ref{app:Karcher--Scherk}.

For the purposes of this discussion we identify $\Sp^2$ with the Riemann sphere $\C \cup \{\infty\}$ via stereographic projection. 
We recall the notation $\towquot\vcentcolon=\tow/\sk{\trans^{\axis{z}}_{2\pi}}$, and (sticking for convenience to Appendix \ref{app:Karcher--Scherk})
we write $N\colon\tow\to\Sp^2$ for the corresponding Gauss map; we observe that $N$ descends to a map,
which we will also denote by $N$, on the quotient $\towquot$.
The Enneper--Weierstrass representation of $\tow$
(see in particular equation \eqref{standardtowerparam}, in turn based on \eqref{eqn:original_parametrisation}
      and \eqref{eqn:rescaled_parametrization}, and the statement of Proposition~\ref{prop:towerEWsummary})
yields a conformal diffeomorphism
\begin{equation}
\label{Gdec}
G\colon \Sp^2 \setminus \{\pm i, \pm e^{i\phi}, \pm e^{-i\phi}\}
 \xrightarrow{\cong}
 \towquot
\end{equation}
where $\phi \in \R$
is determined by \eqref{eqn:theta_phi}
with $\vartheta=\omega_0$.
Moreover, the map
$N \circ G$ extends to a smooth map
$\overline{N}$ on all of $\Sp^2$, which map is moreover
surjective and conformal
but with branch points $0$ and $\infty$;
in fact, this extension coincides with the map
 $v$ specified in \eqref{eqn:Karcher-data},
up to an orthogonal transformation in the target.
It is established in Appendix \ref{app:Karcher--Scherk} that
\begin{align}
\label{images_under_G}
G(\{\abs{w}=1\}) &= \{z = \pm \pi/2\}, &
G(\{\Re w = 0\}) &= \axis{x},  &
G(\{\Im w = 0\}) &= \{x = 0\},
\end{align}
where it is to be stressed that, by \eqref{Gdec}, the target of the map $G$ is $\towquot$ (not $\R^3)$ and so equations like $\{x = 0\}$ need to be interpreted accordingly, with (quotiented) coordinates on $\R^3/\sk{\trans^{\axis{z}}_{2\pi}}$, then restricted to $\towquot$.

We write $g_{{\Sp^2}}$ and $g_{{\tow}}$ respectively for the round metric on $\Sp^2$
and the metric on $\tow$ induced by the ambient Euclidean metric.
Of course $g_{{\tow}}$ and $\abs{A_{\tow}}^2$ descend to $\towquot$,
where we refer to the corresponding objects by these same names.
As it is well-known, the minimality of $\tow$ implies that
\begin{equation}\label{eq:ConfGaussMap}
N^*g_{{\Sp^2}}=\tfrac{1}{2}\abs{A_{\tow}}^2g_{{\tow}},
\end{equation}
while the conformality of $G$ means that there exists
a smooth function $\rho>0$ on $\dom(G)$
such that
\begin{equation}
\label{rhodef}
G^*g_{{\tow}}=\rho^{-2}g_{{\Sp^2}}.
\end{equation}
In fact, it is readily checked that $\rho$ extends to a smooth function
(which we give the same name)
on $\Sp^2$ vanishing on $\Sp^2 \setminus \dom(G)$, the six points which correspond under $G$ to the ends of $\tow$.
Note that (cf. Lemma \ref{lem:CylMSE}) $\abs{A_{\tow}}^2$ also tends to zero exponentially along the ends of $\tow$;
$G^*\abs{A_{\tow}}^2$ vanishes on $\{0,\infty\}$ only
(extending smoothly to zero on $\Sp^2 \setminus \dom(G)$).

As observed above, the extended Gauss map
$\overline{N}\colon\Sp^2 \to \Sp^2$ (pulled back via $G$)
is conformal, with two branch points,
and we can now express the conformal factor
in terms of $\rho$ and $\abs{A_\tow}^2$:
\begin{equation}
\label{extended_N_conformality}
\overline{N}^*g_{\Sp^2}
=
\frac{G^*\abs{A_\tow}^2}{2\rho^2}g_{\Sp^2}.
\end{equation}
Since $\overline{N}$ and $v$ in \eqref{eqn:Karcher-data}
agree up to an isometry of $\Sp^2\subset \R^3$,
the form of $v$ confirms that the conformal factor
$\rho^{-2}G^*\abs{A_{\tow}}^2$ is bounded
on all of $\Sp^2$ and vanishes on $\{0,\infty\}$ exactly.

Last we define on $\Sp^2$ the Schr\"{o}dinger operators
\begin{align}
\label{LGLNdef}
L_G&\vcentcolon=\Delta_{\Sp^2}+\rho^{-2}G^*\abs{A_{\tow}}^2,
&
L_N&\vcentcolon=\Delta_{\Sp^2}+2,
\end{align}
where we stress that the latter equals the Jacobi operator of a totally equatorial $2$-sphere in round $\Sp^3$, for which of course the spectrum is well-known.

Now, we recall two basic facts: if $F\colon(M_1,g_1)\to (M_2, g_2)$ is a smooth map then $F^* \Delta_{g_2}=\Delta_{F^* g_2}F^*$ and, if $(M,g)$ is a two-dimensional Riemannian manifold then $\Delta_{e^{2f}g}=e^{-2f}\Delta_g$. 
As a result, appealing to equations \eqref{rhodef}, \eqref{eq:ConfGaussMap} and \eqref{extended_N_conformality} respectively, one can easily check that 
\begin{align}
\nonumber
L_G G^*
  &=
  \rho^{-2}
    G^*J_{\tow}, \\
\nonumber
J_{\tow}N^*
  &=
  \frac{1}{2}\abs{A_\tow}^2
    N^*L_N,
  \\
\label{LG_LN_conformality}
L_G\overline{N}^*
  &=
  \frac{G^*\abs{A_\tow}^2}{2\rho^2}
    \overline{N}^*L_N
\end{align}
where (consistently with the other conventions we have so far adopted in this discussion) by $J_\tow$ we really mean the corresponding
operator on the quotient $\towquot$,
equivalently the Jacobi operator $J_{\towquot}$.

With this notation in place we are ready to investigate
the bounded equivariant Jacobi fields on $\tow$.
In \cite{KapouleasEuclideanDesing}
Kapouleas analyzed the Jacobi operator on any Karcher--Scherk tower with two wings, acting on the space of functions invariant
under reflection through every $n$\textsuperscript{th}
symmetry plane orthogonal to the axis of periodicity,
for any strictly positive integer $n$;
using arguments from \cite{MontielRos} of Montiel and Ros
he showed that the space of bounded Jacobi fields
with such symmetries
is generated by translations orthogonal to the axis of periodicity.
A later result of Cos\'{i}n and Ros,
Theorem 4.2 in \cite{CosinRos},
can be applied to all Karcher--Scherk towers,
and implies that the space of
$2\pi$-periodic bounded Jacobi fields on any $2\pi$-periodic tower
is generated by translations;
in our construction we impose enough symmetry
to exclude all these functions.
In fact, because of this high symmetry,
it is not difficult to give a brief and direct proof,
without invoking either of the above approaches,
though partly in the spirit of both.

\begin{lemma}[Kernel on $\tow$] 
\label{towker}
There are no nontrivial $\Aut_{\R^3}(\tow)$-equivariant bounded Jacobi fields on $\tow$.
\end{lemma}

\begin{proof}
Suppose $u$ is a bounded $\Aut_{\R^3}(\tow)$-equivariant
Jacobi field on $\tow$.
Then,
$u$ descends to a function $\widetilde{u}$
on $\towquot$
and,
thanks to the first equation in \eqref{LG_LN_conformality} and appealing to the analysis contained in Appendix \ref{app:CylAnalysis} (as discussed in Remark \ref{rem:J_on_wings}),
$G^*\widetilde{u}$
extends to a smooth function $v$ on $\Sp^2$
that satisfies $L_G v=0$,
vanishes on the imaginary line,
and has conormal derivative vanishing along
the real line and unit circle.
Here we have used the conformality of $G$
and \eqref{images_under_G}.
Fix a component $T$ of
$\Sp^2
 \setminus (\{\Re w=0\} \cup \{\Im w=0\} \cup \{\abs{w}=1\})
$,
so that $\partial T$ is a geodesic triangle (of course understood in unit round metric),
and let $\alpha,\beta,\gamma$ be the sides obtained
by taking the intersection of $\partial T$
with respectively the imaginary axis, real axis, and unit circle.

Thus, if $v$ is nontrivial,
then $v|_T$ is an eigenfunction, with eigenvalue $0$, of $L_G$
with Dirichlet condition on $\alpha$
and Neumann condition on $\beta \cup \gamma$.
We will show, however, that $0$ cannot possibly be an eigenvalue
of this boundary value problem.
In doing so, it will be convenient to use the following notation: we shall write
$(L_G,\partial_D T,\partial_N T)$
to denote the operator $L_G$ on $T$ with Dirichlet condition on $\partial_D T$
and Neumann condition on $\partial_N T$,
where
$\partial_D T$ and $\partial_N T$ are
unions of sides of $\partial T$
that themselves
have union $\partial T$ but disjoint interiors.
When we refer to the eigenvalues
of $(L_G,\partial_D T,\partial_N T)$,
we follow the sign convention that
the eigenvalues are bounded below.

To begin with, note that the least eigenvalue of
$(L_G, \alpha, \beta \cup \gamma)$
is strictly less than the least eigenvalue of 
$(L_G, \alpha \cup \beta, \gamma)$.
Recalling the extended Gauss map $\overline{N}$, defined just below \eqref{Gdec},
and using equation \eqref{eqn:Karcher-data},
we find that $\overline{N}(T)$ is a quarter sphere
bounded by a half equator $\overline{N}(\gamma)$
and an orthogonally intersecting
meridian (another half great circle)
$\overline{N}(\alpha \cup \beta)$.
In particular if $V$ is a nontrivial
constant vector field on $\R^3$ orthogonal
to the plane containing $\overline{N}(\alpha \cup \beta)$,
then the function $\overline{N} \cdot V$
is an eigenfunction of eigenvalue $0$
for $(L_G,\alpha \cup \beta,\gamma)$
that has constant sign on $T$.
(The Jacobi field $N \cdot V$
on $\tow$ generates translations along $\axis{x}$.)
Consequently $0$ is the least eigenvalue of
$(L_G,\alpha \cup \beta, \gamma)$,
and so the least eigenvalue
of $(L_G, \alpha, \beta \cup \gamma)$
is strictly negative.

Next, note that -- by the same comparison principle -- the second least eigenvalue
of $(L_G,\alpha, \beta \cup \gamma)$
is strictly greater than the second least eigenvalue
of $(L_G,\emptyset, \partial T)$.
We claim that
the index and nullity of $(L_G,\emptyset,\partial T)$
coincide with the index and nullity (respectively)
of $L_N$
(recalling \eqref{LGLNdef})
on the space $W$ of Sobolev $H^1$ functions
on $\Sp^2$ which are even with respect to each
reflection through a pair of orthogonal great circles,
namely the circles containing
$\overline{N}(\alpha \cup \beta)$
and $\overline{N}(\gamma)$.
From standard results on spherical
harmonics it then follows
that $(L_G,\emptyset,\partial T)$
has index and nullity both equal to $1$.
Hence, the second least eigenvalue
of $(L_G,\alpha,\beta \cup \gamma)$ is strictly positive.
Since its least eigenvalue is strictly negative,
we conclude that
$(L_G,\alpha, \beta \cup \gamma)$ has nullity $0$,
so the only bounded $\Aut_{\R^3}(\tow)$-equivariant
Jacobi field $u$ on $\tow$ is $u=0$.

It remains to verify the previous claim, namely that $(L_G,\emptyset,\partial T)$ has the same index and nullity as $L_N$ on $W$.
In one direction note that if $f \in W$, then,
by \eqref{extended_N_conformality}
and the third equation of \eqref{LG_LN_conformality},
the $L_N$-Rayleigh quotient for $f$
and the $L_G$-Rayleigh quotient for $\overline{N}^*f|_T$
have the same sign
(either strictly or else they are both zero).
This implies that
the index and nullity of
$(L_G,\emptyset,\partial T)$
are respectively at least the index and nullity
of $L_N$ on $W$.
In the other direction,
first note that the restriction of $\overline{N}$
to the closure of $T$ is injective,
as follows from the form of $v$ in \eqref{eqn:Karcher-data}
(or otherwise establishing that the map $\overline{N}$
has degree $2$ and exploiting the symmetries of the problem).
As a result, an eigenfunction
of $(L_G,\emptyset,\partial T)$,
can be transplanted to a unique function
$f$ on the closure of $\overline{N}(T)$
which extends by even reflection
to a function $\overline{f} \in W$.
Again the conformality ensures
that the $L_N$-Rayleigh quotient of $\overline{f}$
on $\Sp^2$
has the same sign as the $L_G$-Rayleigh quotient of the eigenfunction on $T$ we had started with.
In view of the conclusion of the preceding paragraph,
this completes the proof.
\end{proof}

\begin{remark}\label{rem:NegativeDirection}
For later reference, we explicitly note that the argument above (in particular: its third paragraph) imply the existence of a smooth function, say $w$, on the geodesic triangle $T$ solving $L_G w=-\lambda w$ for some $\lambda<0$ and satisfying Dirichlet boundary conditions on $\alpha$ as well as Neumann boundary conditions on $\beta\cup\gamma$. In fact, we may take it to be the first such eigenfunction, so that (among other things) it does not change its sign in the domain in question. Hence, by suitably reflecting across the sides of such a geodesic triangle we obtain a (smooth) function on the round sphere that is an eigenfunction for the operator $L_G$ and we still denote by $w$. At that stage, the function $u\vcentcolon =(G^{-1})^{\ast}w \in C^{\infty}(\towquot)$ is bounded, $\Aut(\towquot)$-equivariant and, based on the first equation in \eqref{LG_LN_conformality} and the conformal invariance of the Jacobi quadratic form (we are working with surfaces), it also satisfies 
$Q_{\towquot}(u,u)<0$.

Because of the exponential decay of the second fundamental form along the six wings of $\towquot$ and the fact that (now by virtue of its very definition) $u$ has a finite limit along each wing, we immediately conclude that both summands 
\[
\int_{\towquot}\abs{A_{\tow}}^2 u^2  \text{ and } \int_{\towquot}|\nabla_{\tow}u|^2
\]
are finite. 
As a result, a standard cutoff argument allows us to construct compactly supported functions where the quadratic form $Q_{\towquot}(\cdot,\cdot)$ is negative. 
This fact will be crucially employed in Section~\ref{sec:Index}, when discussing about the equivariant Morse index of the free boundary minimal surfaces we construct in Theorem \ref{thm:ConstructFirst}. 
\end{remark}

\begin{corollary}
\label{orthogonality_to_ker_LG}
Recall the definitions
\eqref{equivariantprojectorsonquotient},
\eqref{Gdec},
\eqref{rhodef}, and \eqref{LGLNdef}
of the projection $\pi_{\Aut(\towquot)}$,
the map $G$, 
the conformal factor $\rho$,
and the Schr\"{o}dinger operator $L_G$.
Let $E$ be a continuous compactly supported
function on $\towquot$
in the image of $\pi_{\Aut(\towquot)}$
(so that the pullback $\varpi^*E$ under the canonical projection is $\Aut_{\R^3}(\tow)$-equivariant).
Then the function $\rho^{-2}G^*E$
belongs to $L^2(\Sp^2)$
and is $L^2(\Sp^2)$-orthogonal
to the kernel of $L_G$.
\end{corollary}

\begin{proof}
Since $\rho^{-2} \in C^\infty(\dom(G))$
and $E$ is compactly supported and continuous,
we indeed have $\rho^{-2}G^*E \in L^2(\Sp^2)$.
Now suppose $K$ belongs to the kernel of $L_G$.
Then, recalling that $\rho^{-2}G^*|A_{\tow}|^2$ is actually a smooth function on $\Sp^2$, we get that $K \in C^\infty(\Sp^2)$,
so in particular $K$ is bounded
and
$K_{\towquot} \vcentcolon= (G^{-1})^* K|_{\dom(G)}$
is also smooth and bounded,
and by the first equation in \eqref{LG_LN_conformality}
satisfies 
\[
J_{\towquot} K_{\towquot} = 0.
\]
Thus $K_{\towquot}$ is a bounded Jacobi field on $\towquot$.
Then $\varpi^*\pi_{\Aut(\towquot)}K_{\towquot}$
(recalling \eqref{quotientbytrans}
and \eqref{equivariantprojectorsonquotient})
is an $\Aut_{\R^3}(\tow)$-equivariant Jacobi field on $\tow$,
and so by Lemma \ref{towker}
we must have 
$\pi_{\Aut(\towquot)}K_{\towquot}=0$. 
Equivalently, 
\[
K_{\towquot}=(I-\pi_{\Aut(\towquot)})K_{\towquot},
\]
where $I\colon L^2(\towquot) \to L^2(\towquot)$ is the identity map,
while on the other hand by assumption 
\[
E=\pi_{\Aut(\towquot)}E.
\]
By \eqref{equivariant_projector_is_self_adjoint}
the images of the operators 
\[
I-\pi_{\Aut(\towquot)}
\quad\text{ and }\quad
\pi_{\Aut(\towquot)},
\]
restricted to $L^2(\towquot)$, are $L^2(\towquot)$-orthogonal.
By the conformality equation \eqref{rhodef} of $g_{\towquot}$ and $g_{\Sp^2}$ 
(and the two-dimensionality of $\towquot$ and $\Sp^2$)
we thus have
\[
\sk[\Big]{\rho^{-2}G^*E,~ K}_{L^2(\Sp^2)}
=\sk[\Big]{E,~ K_{\towquot}}_{L^2(\towquot)} 
=\sk[\Big]{\pi_{\Aut(\towquot)}E,~
  \bigl(I-\pi_{\Aut(\towquot)}\bigr)K_{\towquot}} 
  =0. 
  \qedhere
\]
\end{proof}

\begin{remark}[Jacobi equation on the wings]
\label{rem:J_on_wings}
We observe that
Corollary \ref{cor:wingsol} holds with $\Delta_W$ replaced by the Jacobi operator $J_W$ of $W$.
Indeed, exactly the same proof
of item \ref{cor:wingsol-i}
goes through with 
$\Delta_W$ replaced by $J_W$,
$\Delta_\cyl$ replaced by $\Delta_W$,
the map $\varphi$ replaced by the identity map on $W$,
Lemma \ref{lem:cylsol} replaced by Corollary \ref{cor:wingsol},
and
the estimate
\eqref{eqn:20220405-1}
replaced by
\begin{equation*}
\nm{(J_W-\Delta_W)(v+\mu)}_{0,\alpha,\beta}
\leq
C
  \left(
    e^{(\beta-\gamma-2)R}\nm{v}_{2,\alpha,\gamma}
      + e^{(\beta-2)R}\abs{\mu}
  \right),
\end{equation*}
a consequence of \eqref{eqn:20220405-1}
and the exponential decay of $\abs{A_W}$,
itself ensured by the exponential decay
of the defining functions
of the wings over their asymptotic planes.
Note that, as in Corollary \ref{cor:wingsol},
we assume $\beta \in \interval{0,1}$,
so in particular we have $\beta-2<0$.
The $J_W$-analogue of item \ref{cor:wingsol-ii}
we obtain not as a corollary
of the preceding $J_W$-analogue of item \ref{cor:wingsol-i}
but rather by a variation of the same argument, as follows.
By item \ref{cor:wingsol-i} of Corollary \ref{cor:wingsol}
the map
\begin{align*}
T_W
  \colon
  \R
    \oplus C^{2,\alpha,\beta}(W)
  &\to
  C^{0,\alpha,\beta}(W)
    \oplus  C^{1,\alpha}(\partial W)
\\
(\mu, v)
  &\mapsto
  \left(
    \Delta_W v,~
    (v+\mu)|_{\partial W}
  \right)
\end{align*}
is invertible.
(Of course $\Delta_W v = \Delta_W (v+\mu)$.)
Using the same estimate as above,
now with $\gamma=\beta$ and $\mu=0$,
and taking $R$ sufficiently large,
we therefore obtain
invertibility of the map $S_W$
having the same definition as $T_W$
but with $J_W$ in place of $\Delta_W$.
For the $J_W$-analogue of item \ref{cor:wingsol-ii}
we can then take $P_W \vcentcolon= S_W^{-1}(E,0)$.
\end{remark}

The following lemma is obtained by combining the ancillary result above (Corollary \ref{orthogonality_to_ker_LG}) with Remark \ref{rem:J_on_wings}.

Recalling \eqref{towcokerandgen}, we are about to show that
 $\towcoker$
spans the cokernel
(the \emph{extended substitute kernel}
in the terminology of Kapouleas)
of the linearized operator
on the weighted spaces we choose on the towers.
Note that $\nu_{\tow^+} \cdot \partial_y$
is $\Aut_{\R^3}(\tow^+)$-invariant,
while $\Psi^{\mathrm{dislocate}}$
is $\Aut_{\R^3}(\tow^+)$-equivariant,
so that $\towcokergen$ is also $\Aut_{\R^3}(\tow^+)$-equivariant.
Furthermore, 
$\towcokergen$ 
is smooth
and has support contained in
$W^1 \cup W^{-1}$
and,
as follows from Remark \ref{towasymptotics},
is asymptotically a nonzero constant
on $W^1$ (solely determined by the angle $\vartheta$), with exponential convergence,
while $\towcoker$ is of course compactly supported
and also smooth.
In particular
\[
\towcoker \in C^{0,\alpha,\beta}_{\Aut_{\R^3}(\tow^+)}(\tow^+)
\]
for all $\alpha,\beta \in \interval{0,1}$,
but $\towcokergen$, though bounded and smooth,
does not belong to $C^{2,\alpha,\beta}(\tow^+)$
for any $\beta \in \interval{0,1}$.

\begin{lemma}[Solutions on the tower modulo cokernel]
\label{towsol}
Let $\alpha, \beta \in \interval{0,1}$ and let $\eta^{\vphantom{|}}_{\tow^+}$ denote the outward unit conormal to $\tow^+$.
Then:
\begin{enumerate}[label={\normalfont(\roman*)}]
\item\label{lem:towsol-i} 
For any $E \in C_{\Aut_{\R^3}(\tow^+)}^{0,\alpha,\beta}(\tow^+)$ and
$f \in C_{\Aut_{\R^3}(\tow^+)}^{1,\alpha}(\partial \tow^+)$
there exists a unique (bounded) function
$u \in C_{\Aut_{\R^3}(\tow^+)}^{2,\alpha}(\tow^+)$
such that
\begin{equation*}
J_{\tow^+}u = E 
\quad \mbox{ and } \quad
B^{\mathrm{Robin}}_{\tow^+}u
=\eta^{\vphantom{|}}_{\tow^+}\cdot\nabla_{\tow^+} u
=f;
\end{equation*}
moreover there is a (unique) $\mu \in \R$ such that
\begin{equation*}
\abs{\mu}
  +\nm[\big]{u|_{W^1}-\mu}_{2,\alpha,\beta}
  +\nm[\big]{u|_{\tow^+ \setminus
    (W^1 \cup W^{-1})}}_{2,\alpha,\beta}
\leq C\bigl(\nm{E}_{0,\alpha,\beta}+\nm{f}_{1,\alpha}\bigr)
\end{equation*}
for some constant $C=C(\alpha,\beta)>0$ independent of the data.
\item\label{lem:towsol-ii}
The image of the map
\begin{equation*}
\bigl(J_{\tow+},B^{\mathrm{Robin}}_{\tow^+}\bigr)\colon
  C^{2,\alpha,\beta}_{\Aut_{\R^3}(\tow^+)}(\tow^+) 
  \to
  C^{0,\alpha,\beta}_{\Aut_{\R^3}(\tow^+)}(\tow^+)
    \oplus C^{1,\alpha}_{\Aut_{\R^3}(\tow^+)}(\partial \tow^+)
\end{equation*}
is the kernel of the surjective map $(E,f) \mapsto \mu$ given in \ref{lem:towsol-i},
so the map $\bigl(J_{\tow^+},B^{\mathrm{Robin}}_{\tow^+}\bigr)$, with domain and target as above, is Fredholm with Fredholm index $-1$.
\item\label{lem:towsol-iii}
There exists a bounded linear map
\begin{equation*}
\glsuseriii{modelresol} %P_{\tow^+}
\colon
C^{0,\alpha,\beta}_{\Aut_{\R^3}(\tow^+)}(\tow^+)\oplus C^{1,\alpha}_{\Aut_{\R^3}(\tow^+)}(\partial \tow^+)
\to C^{2,\alpha,\beta}_{\Aut_{\R^3}(\tow^+)}(\tow^+)\oplus\R
\end{equation*}
such that for any $(E,f)$ in the domain of $P_{\tow^+}$ the pair $(u,\mu)\vcentcolon=P_{\tow^+}(E,f)$ satisfies
\begin{equation*}
J_{\tow^+}u = E + \mu \towcoker
\quad \mbox{ and } \quad
B^{\mathrm{Robin}}_{\tow^+}u = f.
\end{equation*}
\end{enumerate}
\end{lemma}

\begin{proof}
Assuming item \ref{lem:towsol-i}, the surjectivity of the map $(E,f) \mapsto \mu$ is clear from the observation that the image
of $(\towcoker,0)$ under this map can only be $\towcokergen|_{W^1}(\infty) \neq 0$.
Item \ref{lem:towsol-iii} and the remainder of item \ref{lem:towsol-ii} then follow immediately from item \ref{lem:towsol-i}. 
To verify the uniqueness assertion in \ref{lem:towsol-i} note that if we have two bounded solutions with the same data,
then we can extend their difference by even reflection to a bounded $\Aut_{\R^3}(\tow)$-equivariant $C^2$ (so in fact $C^\infty$) Jacobi field on $\tow$, but Lemma \ref{towker} then implies that this difference vanishes identically.

Turning to existence, we recall definitions
\eqref{towquotient}
and \eqref{equivariantprojectorsonquotient}.
Since the data and operators are
$\trans^{\axis{z}}_{2\pi}$-invariant,
the problem descends to the quotient
$\towquot^+$,
with data denoted by $(\widetilde{E},\widetilde{f})$,
interior operator $J_{\towquot^+}$
(simply the Jacobi operator on the quotient surface
 in $\R^3/\sk{\trans^{\axis{z}}_{2\pi}}$),
and boundary operator $B^{\mathrm{Robin}}_{\towquot^+}$
(simply the outward conormal derivative on the quotient).
Thus we assume
\begin{alignat*}{2}
\widetilde{E}&=\pi_{\Aut(\towquot^+)}\widetilde{E}, &
\widetilde{f}&=\pi_{\Aut(\towquot^+)}\widetilde{f}
\intertext{and we seek a function $\widetilde{u}$ on $\towquot^+$ such that }
J_{\towquot^+}\widetilde{u}&=\widetilde{E}, \quad &
B^{\mathrm{Robin}}_{\towquot^+}\widetilde{u}&=\widetilde{f}, 
\end{alignat*}
$\widetilde{u}$ is bounded, satisfies the appropriate estimates, 
and $\widetilde{u}=\pi_{\Aut(\towquot^+)}\widetilde{u}$
(because of Remark \ref{rem:equiv_in_quot}).
Then we can take
$u\vcentcolon= \varpi|_{\tow^+}^{\ast}\widetilde{u}$ 
as the solution to the given problem on $\tow^+$.

Before continuing, we point out that although it would be easy to enforce the necessary symmetries throughout the construction of the solution, it is unnecessary to do so.
In fact, since %$J_{\towquot^+}$ and $B^{\mathrm{Robin}}_{\towquot^+}$
both opearators on $\towquot^+$ commute with $\pi_{\Aut(\towquot^+)}$,
if $\widetilde{F}$ and $\widetilde{u}$ are functions on $\towquot^+$
satisfying
\begin{alignat*}{3}
\pi_{\Aut(\towquot^+)}\widetilde{F}&=\widetilde{E}, &
J_{\towquot^+}\widetilde{u}&=\widetilde{F}, & 
B^{\mathrm{Robin}}_{\towquot^+}\widetilde{u}&=\widetilde{f},
\shortintertext{then}
&&
J_{\towquot^+}\pi_{\Aut(\towquot^+)}\widetilde{u}&=\widetilde{E}, \qquad &
B^{\mathrm{Robin}}_{\towquot^+}\pi_{\Aut(\towquot^+)}\widetilde{u}
 &=\widetilde{f}.
\end{alignat*}
We further reduce to the case of homogeneous Neumann boundary data by taking a function $v \in C^{2,\alpha}(\towquot^+)$
which has support contained in $\{x=0\}_{\leq 1}$ and satisfies 
\[
B^{\mathrm{Robin}}_{\towquot^+}v=\widetilde{f} \quad\text{ and } \quad \nm{v}_{2,\alpha} \leq C\nm{\widetilde{f}}_{1,\alpha}
\]
for some $C>0$ independent of the data.
(In light of the comments in the preceding paragraph we need not bother to take $v$ equivariant.)
By replacing $\widetilde{E}$ with $\widetilde{E}-J_{\towquot^+}v$ we may thus assume $\widetilde{f}=0$, so $f=0$ in the statement of the lemma.
Next we extend (the updated) $\widetilde{E}$ by even reflection to $\overline{E} \in C^{0,\alpha}(\towquot)$.
We now seek a function $\overline{u}$ on $\towquot$ such that $J_{\towquot}\overline{u}=\overline{E}$;
in fact, by the same considerations as in the preceding paragraph, in the following construction we may -- whenever convenient -- replace $\overline{E}$ by another function $F$ satisfying $\pi_{\Aut(\towquot)}F=\overline{E}$
and we will eventually conclude by taking (with respect to the task of proving statement \ref{lem:towsol-i} of the present lemma)
\[
u\vcentcolon=
 \varpi|_{\towquot}^*(
  \pi_{\Aut(\towquot)}
  \overline{u})|_{\tow^+}.
\]
The function $\dist_{\axis{z}}$ measuring the distance from the $z$-axis descends to $\towquot$ (which we shall freely employ without renaming),
so we can define the spaces
$C^{k,\alpha,\beta}(\towquot)$
and corresponding norms $\nm{\cdot}_{k,\alpha,\beta}$
in the obvious way.
Thus $\overline{E} \in C^{0,\alpha,\beta}(\towquot)$
and, in view of the previous discussion,
it only remains to find $\overline{u}$
as above so that (in addition) the desired estimates hold.

We will next decompose $\overline{E}$ into a part supported on the core and six more parts, each supported on a wing.
Referring to the content (and notation) of Remark~\ref{towasymptotics} we have six well-defined wings (each being a normal graph over a cylindrical base), whose union we shall henceforth denote by $\towquot'$. Furthermore, for any $R\geq R_{\mathrm{tow}}$ we have that $\towquot \cap \{\dist_{\axis{z}} \geq R\}\subset \towquot'$. 
For each component $W$,
by appealing to Remark \ref{rem:J_on_wings} (and possibly increasing $R_{\mathrm{tow}}$, and redefining the wings accordingly) 
we obtain a bounded right inverse
$P_W$ to
$
 J_{W}\colon C^{2,\alpha,\beta}(W)
 \to
 C^{0,\alpha,\beta}(W)
$.
The number $R_{\mathrm{tow}}$ is understood as fixed from now onwards, and we allow the constants in the estimates we are about to present to possibly depend on such a value.

%(Here and in the sequel of this proof,
%with slight abuse of notation we shall simply write $J_{\towquot}$
%to refer indiscriminately to the Jacobi operator
%of different subsets of $\towquot$, 
%such as e.\,g. the wing $W$.)

Having defined $P_W$ in this way
for each component $W$
of $\towquot'$,
we in turn define
the linear map
\begin{equation*}
P_{\mathrm{wings}}
\colon
C^{0,\alpha,\beta}\bigl(\towquot'\bigr)
\to
C^{2,\alpha,\beta}\bigl(\towquot'\bigr)
\end{equation*}
such that for any $F$ in its domain
and any component $W$ as above
we have
$(P_{\mathrm{wings}}F)|_W = P_W(F|_W)$.
We also define
\begin{equation*}
\overline{u}_{\mathrm{wings}}
\vcentcolon=
  (\cutoff{R}{R+1} \circ \dist_{\axis{z}})
  \cdot
  P_{\mathrm{wings}}
  \bigl((\cutoff{R}{R+1} \circ \dist_{\axis{z}})
    \cdot \overline{E}\bigr)|_{\towquot'},
\end{equation*}
extended to be equal to the constant zero inside the set $\{\dist_{\axis{z}} \leq R \}$, thereby defining
a function in $C^{2,\alpha,\beta}(\towquot)$.
Then
\begin{equation}
\label{uwingsest}
\nm{\overline{u}_{\mathrm{wings}}}_{2,\alpha,\beta}
\leq
C\nm{\overline{E}}_{0,\alpha,\beta}
\end{equation}
and the function
\begin{equation*}
\overline{E}_{\mathrm{core}}
\vcentcolon=
\pi_{\Aut(\towquot)}
\left(
  \overline{E}-J_{\towquot}\overline{u}_{\mathrm{wings}}
\right)
\end{equation*}
has support contained in
$\{\dist_{\axis{z}}\leq R+1\}$
and satisfies
\begin{equation}
\label{Ecore_decay_est}
\nm{\overline{E}_{\mathrm{core}}}_{0,\alpha,\beta}
\leq
C\nm{\overline{E}}_{0,\alpha,\beta}.
\end{equation}
To proceed, we shall recall the conformal diffeomorphism $G$, the
associated conformal factor $\rho$ 
and Schr\"{o}dinger operator $L_G$ defined at the beginning of this section (see, in particular, equations \eqref{rhodef}, \eqref{LGLNdef}).
The function
$\rho^{-2}G^*\overline{E}_{\mathrm{core}}$
has a unique continuous extension
$E_{\Sp^2}$
to all of $\Sp^2$
(vanishing around the punctures in $\dom(G)$
corresponding to the ends of $\towquot$),
which satisfies
\begin{equation}
\label{ES2est}
\nm{E_{\Sp^2}}_{0,\alpha}
\leq
C\nm{\overline{E}}_{0,\alpha,\beta}.
\end{equation}
Furthermore $\overline{E}_{\mathrm{core}}$
by construction lies in the image of
$\pi_{\Aut(\towquot)}$, so that
Corollary \ref{orthogonality_to_ker_LG}
implies that $E_{\Sp^2}$
is $L^2(\Sp^2)$-orthogonal to the kernel of $L_G$.
Hence, by the standard Fredholm alternative, the equation
\begin{equation*}
L_G u_{\Sp^2}
=
E_{\Sp^2}
\end{equation*}
has a unique solution $u_{\Sp^2}$
orthogonal (in $L^2(\Sp^2)$) to the kernel of $L_G$,
which solution satisfies the Schauder estimate
\begin{equation}
\label{uS2est}
\nm{u_{\Sp^2}}_{2,\alpha}
\leq
C\nm{E_{\Sp^2}}_{0,\alpha}.
\end{equation}
As a result, the function
\begin{equation*}
\overline{u}_{\mathrm{core}}
\vcentcolon=
\pi_{\Aut(\towquot)}
\left((G^{-1})^* u_{\Sp^2}|_{\dom(G)}\right)
\end{equation*}
satisfies, thanks to the first equation in \eqref{LG_LN_conformality}
\begin{equation}
\label{ucore_eqn}
J_{\towquot}\overline{u}_{\mathrm{core}}
=
\overline{E}_{\mathrm{core}}.
\end{equation}
By the standard local Schauder estimates on $\towquot$
(which has bounded geometry)
and \eqref{Ecore_decay_est} we further have
\[
\nm{\overline{u}_{\mathrm{core}}}_{2,\alpha}
\leq C\left(\nm{\overline{u}_{\mathrm{core}}}_0 + \nm{\overline{E}}_{0,\alpha,\beta}\right),
\]
but by the estimates \eqref{ES2est} and \eqref{uS2est} we also have in particular
\[
\nm{\overline{u}_{\mathrm{core}}}_0 \leq C\nm{\overline{E}}_{0,\alpha,\beta},
\]
so that in the end we actually obtain
\begin{equation}
\label{corest}
\nm{\overline{u}_{\mathrm{core}}}_{2,\alpha} \leq C\nm{\overline{E}}_{0,\alpha,\beta}.
\end{equation}
In view of the considerations concerning equivariance at the beginning of the proof we may take
$
 \overline{u}
 \vcentcolon=
 \overline{u}_{\mathrm{core}}+\overline{u}_{\mathrm{wings}},
$
and it remains only to verify the asserted asymptotics
for $\overline{u}_{\mathrm{core}}$.
To this end let $W$ be a wing as defined above.
By appealing to 
Remark \ref{rem:J_on_wings}
it then follows from
\eqref{ucore_eqn}
that there exists $\mu_{W} \in \R$ such that
\begin{equation*}
\nm{\overline{u}_{\mathrm{core}}|_W - \mu_{W}}_{2,\alpha,\beta}
  +\abs{\mu_{W}}
\leq
C\nm{\overline{E}}_{0,\alpha,\beta},
\end{equation*}
where
we have made use of
\eqref{Ecore_decay_est}
and
\eqref{corest}
to obtain the bound.
Recall that we already have the estimate \eqref{uwingsest}.
By the symmetries we need only consider the cases that
$W=W^0$ and $W=W^1$,
but in the former case the symmetries imply
that $\mu_{W}=0$
(since $W^0 \supset \axis{x}$, for example,
and $\refl_{\axis{x}}$ is a symmetry of $\tow$).
Taking $\mu\vcentcolon=\mu_{W}$
for $W=W^1$
completes the proof.
\end{proof}

\subsection{Global solutions on the initial surfaces modulo approximate cokernel}
\label{subs:LinConclusion}

We will now exploit the resolvents
$P_{\B^2}^m$, $P_{\K_{b}}^m$ and $P_{\tow^+}$
obtained earlier on the model surfaces
to construct approximate solutions
to the linearized problem on the initial surfaces.
Later, in Section \ref{sec:Nonlin}, by perturbation (or iteration)
we will obtain exact solutions,
modulo a one-dimensional subspace
inherited from the cokernel
confronted on $\tow^+$
in Lemma \ref{towsol}.
We will refer to this subspace as
the approximate cokernel to the linearized problem
on the initial surfaces
(since, in a sense which we do not attempt to make precise here,
it converges in the large-$m$ limit to the cokernel on $\tow^+$,
which by item \ref{lem:towsol-ii} of Lemma \ref{towsol}
has dimension $1$).

We do not claim that the
(exact) cokernel of the
linearized problem on $\Sigma_{m,\xi}$
is nontrivial; rather, the approximate cokernel is an inevitable
consequence of the strategy we follow
(involving a comparison
of the region $\towr_{m,\xi} \subset \Sigma_{m,\xi}$
to $\tow^+$) to construct solutions,
and in fact in the next section
we will show that at the nonlinear level
we can solve
in the direction of the approximate cokernel
by variation of the parameter $\xi$.

It would be possible to consider
the effect of the variation of $\xi$
at the linear level.
The complication is that the function
generating the family $\Sigma_{m,\xi+t}$ (for any $\xi$ and for $\abs{t}$ small), i.\,e.
the normal projection of the velocity,
pulled back to $\tow^+$,
does not vanish at infinity.
Since we will exploit
exponential decay along the wings
to ensure convergence
in constructing solutions
both to the linear and nonlinear problems,
it is necessary to isolate the effect
of the dislocations somehow,
and we have chosen one particular scheme to do so.

In order to obtain bounds uniform in $m$
(and in particular to bring $\towr_{m,\xi}$
to the scale of its model $\tow^+$)
it is natural to consider our problem
(at both the nonlinear and linearized levels)
on the rescaled initial surfaces $m\Sigma_{m,\xi}$.
Equivalently, at the linearized level,
we can consider the rescaled operators
$m^{-2}J_{\Sigma_{m,\xi}}$
and $m^{-1}B^{\mathrm{Robin}}_{\Sigma_{m,\xi}}$;
see \eqref{BRobin_scaling} and the discussion right before it.
In fact, for our purposes it suffices to consider
the case of homogeneous boundary data,
though our proof will entail the more general
situation of inhomogeneous data on the boundary
component closest to the equator $\Sp^1$,
namely $\partial^0\Sigma_{m,\xi}$
defined in \eqref{equatorialbdy}.

In view of Lemma \ref{towsol},
for the purposes of solving the linearized problem
on $\Sigma_{m,\xi}$ it is also natural to equip the space
of data with norms having exponentially decaying weights
on $\towr_{m,\xi}$,
and to consider such norms
on the space of candidate solutions too.
To this end,
for any $\alpha, \beta \in \interval{0,1}$,
any integer $k \geq 0$,
and any 
$\apr_m$-equivariant function $u$ on $\Sigma_{m,\xi}$
we define the norm
\begin{equation}
\label{globalnorms}
\begin{aligned}
\glsuseri{norm} %\nm{u}_{k,\alpha,\beta}
&\vcentcolon=
\nm[\Big]{
    \left(
      \varpi_{\towr_{m,\xi}}^{-1*}u
    \right)\Big|_{\varpi_{\towr_{m,\xi}}(\towr_{m,\xi})}
  }_{k,\alpha,\beta} \\
&\hphantom{{}\vcentcolon={}}+e^{\beta m^{1/2}}
  \nm[\Big]{
    \left(
      \varpi_{\catr_{m,\xi}}^{-1*}u
    \right)\Big|_{
      \varpi_{\catr_{m,\xi}}
        (
          \catr_{m,\xi}
          \setminus
          \towr_{m,\xi}^1
        )
    }
  }_{k,\alpha} \\
&\hphantom{{}\vcentcolon={}}+e^{\beta m^{1/2}}
  \nm[\Big]{
    \left(
      \varpi_{\discr_{m,\xi}}^{-1*}u
    \right)\Big|_{
      \varpi_{\discr_{m,\xi}}
        (
          \discr_{m,\xi}
          \setminus
          \towr_{m,\xi}^1
        )
    }
  }_{k,\alpha} \\
\end{aligned}
\end{equation}
recalling the regions \eqref{eqn:definition_regions}
and \eqref{towr1}
and regional projections \eqref{regionalprojections}.
Note that on the right-hand side of definition \eqref{globalnorms}
$\nm{\cdot}_{k,\alpha,\beta}$
refers to the standard definition made in
\eqref{weighted_norm_on_tow};
for a function $u$ on an initial surface $\Sigma_{m,\xi}$
by $\nm{u}_{k,\alpha,\beta}$
we will always mean the norm defined in \eqref{globalnorms}.

Of course, since $\Sigma_{m,\xi}$ is compact,
the norm $\nm{\cdot}_{k,\alpha,\beta}$
is equivalent (though not uniformly in $m$)
to the usual H\"{o}lder norm $\nm{\cdot}_{k,\alpha}$
on the space of $\apr_m$-equivariant functions on $\Sigma_{m,\xi}$.
Nevertheless,
the weighted norms will be indispensable in the sequel,
since for large $m$, in a neighborhood of the equator
our initial surfaces tend, after rescaling,
to the complete half tower $\tow^+$,
on which our analysis of the linearized problem
required the corresponding weighted spaces.
Recalling the ``equatorial'' boundary component
$\partial^0\Sigma_{m,\xi}$ from \eqref{equatorialbdy},
we also define for any $\apr_m$-equivariant functions
$E$ and $f$ on $\Sigma_{m,\xi}$ and $\partial^0\Sigma_{m,\xi}$
respectively
the norm
\begin{equation}
\label{global_data_norm}
\nm{(E,f)}_{\alpha,\beta}
\vcentcolon=
\nm{E}_{0,\alpha,\beta}
  +\nm[\big]{
    \varpi_{\towr_{m,\xi}}|_{\partial^0\Sigma_{m,\xi}}^{-1*}f
   }_{1,\alpha}.
\end{equation}
In the following proposition
we sacrifice some decay in the construction of the solution
in order to account
for the discrepancy of scale ($m^2$ in ratio)
between the tower region on the one hand
and the disc and catenoidal regions on the other.
This loss, though suboptimal, is entirely acceptable
because the first-order correction
to the initial surfaces will still be as small as needed
and because in the nonlinear problem
the solution operator of the proposition
will be applied to quadratic terms,
for which the faster decay is anyway recovered.

Finally we recall definition \eqref{coker} and write 
$\iota^{\vphantom{|}}_{\Sigma_{m,\xi}}$ for the inclusion map
of $\Sigma_{m,\xi}$ in $\B^3$,
and for each $\xi \in \R$ and sufficiently large integer $m$ we choose a diffeomorphism 
$\gls{varsigma} %\varsigma_{m,\xi}
\colon \Sigma_{m,0} \to \Sigma_{m,\xi}$ such that
\begin{equation}
\begin{aligned}
\label{Sigma_0_to_Sigma_xi}
%&\varsigma_{m,\xi}\colon \Sigma_{m,0} \to \Sigma_{m,\xi}\mbox{ such that} \\
&\varsigma_{m,0} \mbox{ is the identity}, \\
&(\xi,p) \mapsto \iota^{\vphantom{|}}_{\Sigma_{m,\xi}} \circ \varsigma_{m,\xi}(p)
  \mbox{ is smooth,} \\
&\varsigma_{m,\xi}
  \mbox{ is $\apr_m$-equivariant, and} \\
&\varsigma_{m,\xi}^*\coker = \cokerval{0}.
\end{aligned}
\end{equation}
In the following section we will require
further properties of $\varsigma_{m,\xi}$.
Namely, we select $\varsigma_{m,\xi}$
satisfying not only \eqref{Sigma_0_to_Sigma_xi}
but also this additional requirement: for any real $c>0$
and any integer $k \geq 0$
and any $\alpha,\beta \in \interval{0,1}$
there exist $C(c,k,\alpha,\beta)>0$
and $m_0=m_0(c,k,\alpha,\beta)>0$
such that for every $\xi \in \IntervaL{-c,c}$
and every integer $m>m_0$
and any functions
$u \in C^{k,\alpha}_{\apr_m}(\Sigma_{m,\xi})$
and $v \in C^{k,\alpha}_{\apr_m}(\Sigma_{m,0})$
we have the estimates
\begin{equation}
\label{weighted_ests_for_Sigma_0_to_Sigma_xi}
\begin{aligned}
\nm{\varsigma_{m,\xi}^*u}_{k,\alpha,\beta}
  &\leq
  C(c,k,\alpha,\beta)\nm{u}_{k,\alpha,\beta}
  \quad \mbox{and}
\\
\nm{\varsigma_{m,\xi}^{-1*}v}_{k,\alpha,\beta}
  &\leq
  C(c,k,\alpha,\beta)\nm{v}_{k,\alpha,\beta},
\end{aligned}
\end{equation}
where, we emphasize,
the constant $C(c,k,\alpha,\beta)$
is independent of $m$.
One way to achieve all of the conditions
in \eqref{Sigma_0_to_Sigma_xi}
and \eqref{weighted_ests_for_Sigma_0_to_Sigma_xi}
is
to choose a suitable family of diffeomorphisms
$\lambda_{m,\xi}\colon\K_{b_{m,0}}\to\K_{b_{m,\xi}}$
and to set
\begin{equation*}
\varsigma_{m,\xi}
\vcentcolon=
\begin{cases}
\varpi_{\discr_{m,\xi}}^{-1}\circ\varpi_{\discr_{m,0}}
  & \mbox{on} \quad
  \discr_{m,0}\setminus \towr_{m,0}
\\
\varpi_{\towr_{m,\xi}}^{-1}\circ \varpi_{\towr_{m,0}}
& \mbox{on} \quad
  \towr_{m,0}^1
\\
\varpi_{\catr_{m,\xi}}^{-1}
  \circ \lambda_{m,\xi}
  \circ \varpi_{\catr_{m,0}}
  & \mbox{on} \quad
  \catr_{m,0}\setminus \towr_{m,0}
\end{cases}
\end{equation*}
and then to complete the definition by smooth interpolation using cutoff functions, enforcing $\apr_m$-equivariance.

\begin{proposition}
[Solutions on the initial surface modulo
approximate cokernel]
\label{prop:OutputLinearTheory}
\label{globalsol}
Recall \eqref{equatorialbdy},
\eqref{globalnorms},
\eqref{global_data_norm},
and \eqref{Sigma_0_to_Sigma_xi}.
Assume $0<\alpha<1$, $0<\beta<\gamma<1$, and $c>0$.
There exists $m_0=m_0(c)>0$ such that
for any integer $m>m_0$ and any $\xi \in \IntervaL{-c,c}$
there is a linear map
\begin{equation*}
\glsuseri{initialresol} %P_{\Sigma_{m,\xi}}
\colon
C^{0,\alpha}_{\apr_m}(\Sigma_{m,\xi})
  \oplus
  C^{1,\alpha}_{\apr_m}(\partial^0\Sigma_{m,\xi})
\to
C^{2,\alpha}_{\apr_m}(\Sigma_{m,\xi}) \oplus \R
\end{equation*}
such that if $(E,f) \in \dom(P_{\Sigma_{m,\xi}})$
and $(u,\mu)=P_{\Sigma_{m,\xi}}(E,f)$,
then
\begin{enumerate}[label={\normalfont(\roman*)}]
\item
\(
\nm{u}_{2,\alpha,\beta}+\abs{\mu}
  \leq 
  C\nm{(E,f)}_{\alpha,\gamma}
\)
for some constant $C>0$ independent of
$c$, $m$, $m_0$, $\xi$, and the data $(E,f)$;
\item
\(\displaystyle
\left\{
\begin{aligned}
m^{-2}J_{\Sigma_{m,\xi}}u &=E+\mu\coker 
&&\text{ in }\Sigma_{m,\xi},
\\
B_{\Sigma_{m,\xi}}^{\mathrm{Robin}}u &=0 &&\text{ on }\partial\Sigma_{m,\xi}\setminus\partial^0\Sigma_{m,\xi}, 
\\
m^{-1}B_{\Sigma_{m,\xi}}^{\mathrm{Robin}}u &=f &&\text{ on }\partial^0\Sigma_{m,\xi};
\end{aligned}\right.
\)
\item
the map
\begin{align*}
  \R
  \oplus C^{0,\alpha}_{\apr_m}(\Sigma_{m,0})
  \oplus C^{1,\alpha}_{\apr_m}(\partial^0 \Sigma_{m,0})
&\to 
C^{2,\alpha}_{\apr_m}(\Sigma_{m,0})
  \oplus \R
\\
(
  \xi,
  E_0,
  f_0
)
&\mapsto
(
  \varsigma_{m,\xi}^*u_\xi,
  \mu_\xi
),
\end{align*}
where $(u_\xi,\mu_\xi)\vcentcolon=P_{\Sigma_{m,\xi}}\bigl(\varsigma_{m,\xi}^{-1*}E_0,~\varsigma_{m,\xi}|_{\partial^0\Sigma_{m,\xi}}^{-1*}f_0\bigr)$, is continuous.
\end{enumerate}
\end{proposition}

\begin{proof}
Recall the quantity $b_{m,\xi}$
from \eqref{bmxi}.
Recall also that
$
 \towquot_{(m)}^+
 \vcentcolon=
 \tow^+/\sk{\trans^{\axis{\theta}}_{2m\pi}}
$
and for each nonnegative integer $k$ let 
\begin{alignat*}{3}
&C^{k,\alpha,\gamma}_{\Aut(\towquot_{(m)}^+)}\bigl(\towquot_{(m)}^+\bigr) &
\quad&\text{ and }\quad &
&C^{k,\alpha}_{\Aut(\towquot_{(m)}^+)}\bigl(\partial \towquot_{(m)}^+\bigr)
\intertext{be the Banach spaces, equipped with the obvious norms, consisting of all functions 
%on $\towquot_{(m)}^+$ and $\partial \towquot_{(m)}^+$ respectively 
whose pullbacks under the canonical projection belong to}
&C^{k,\alpha,\gamma}_{\Aut_{\R^3}(\tow^+)}(\tow^+) &
&\text{ and } &
&C^{k,\alpha}_{\Aut_{\R^3}(\tow^+)}(\partial \tow^+)
\end{alignat*}
respectively.
Lemmata \ref{lem:KeyLinDisc}, \ref{lem:catker}, and \ref{towsol} imply the existence of linear maps
\begin{equation*}
\begin{aligned}
P^m_{\B^2}\colon 
  C^{0,\alpha}_{\apr_m}(\B^2)
  &\to C^{2,\alpha}_{\apr_m}(\B^2), \\[1ex]
P^m_{\K_{b_{m,\xi}}}\colon
  C^{0,\alpha}_{\pyr_m}(\K_{b_{m,\xi}})
  &\to 
  C^{2,\alpha}_{\pyr_m}(\K_{b_{m,\xi}}), \quad \mbox{and} \\[1ex]
P^m_{\tow}\colon
  C^{0,\alpha,\gamma}_{\Aut(\towquot_{(m)}^+)}\bigl(\towquot_{(m)}^+\bigr)
    \oplus 
    C^{1,\alpha}_{\Aut(\towquot_{(m)}^+)}\bigl(\partial \towquot_{(m)}^+\bigr)
  &\to 
  C^{2,\alpha,\gamma}_{\Aut(\towquot_{(m)}^+)}\bigl(\towquot_{(m)}^+\bigr)
    \oplus 
    \R
\end{aligned}
\end{equation*}
that have operator norms bounded by a constant independent of $m$ and such that
\begin{enumerate}[label={(\arabic*)},wide]
\item $P^m_{\B^2}$ is the inverse of $J_{\B^2}$
with homogeneous Dirichlet condition on $\partial \B^2$,
\item $P^m_{\K_{b_{m,\xi}}}$ is the inverse of $J_{\K_{b_{m,\xi}}}$
with homogeneous Dirichlet condition on
the lower component of $\partial \K_{b_{m,\xi}}$
and homogeneous Robin condition
on the upper component of $\partial \K_{b_{m,\xi}}$,
and
\item for any $(E,f) \in \dom(P^m_\tow)$, if $(u,\mu)= P^m_\tow (E,f)$,
then
\begin{equation*}
\left\{
\begin{aligned}
J_{\towquot_{(m)}^+}u
  &=
  E+\mu H^{\mathrm{dislocate}}_{\towquot_{(m)}^+},
\\
B^{\mathrm{Robin}}_{\towquot_m^+}u
  &=
  f,
\end{aligned}
\right.
\end{equation*}
where $H^{\mathrm{dislocate}}_{\towquot_m^+}$
is the unique function whose pullback under the canonical projection
is $\towcoker$.
\end{enumerate}
Roughly speaking, we will ``paste together'' these three operators to obtain $P_{\Sigma_{m,\xi}}$.

Now suppose
$
 (E,f)
 \in
 C^{0,\alpha}(\Sigma_{m,\xi})
   \oplus
   C^{1,\alpha}(\partial^0\Sigma_{m,\xi})
$.
Throughout the proof we will often tacitly
extend, without renaming,
a function $q\colon A \to \R$ on a set $A \subset B$ in a given manifold
to all of $B$
by decreeing $q|_{B \setminus A}=0$;
in all such instances the extension will
be smooth, preserving $C^{k,\alpha}$ regularity,
because $q$ will vanish identically on a neighborhood
of $\partial A \subset B$.
With this understanding,
and bearing in mind the extents
of the regions \eqref{eqn:definition_regions},
we start by setting
\begin{equation*}
\begin{aligned}
\Psi_\tow
  &\vcentcolon=
  (\cutoff{m^{1/2}}{m^{1/2}-1} \circ \dist_{\axis{\theta}}
  )|_{\towquot_{(m)}^+},
  \\
E_\tow
  &\vcentcolon=
  (\varpi_{\towr_{m,\xi}}^{-1*}E|_{\towr_{m,\xi}})
    \Psi_\tow, \\
f_\tow
  &\vcentcolon=
  \varpi_{\towr_{m,\xi}}|_{\partial^0\Sigma_{m,\xi}}^{-1*}f, \\
(u^{\vphantom{|}}_\tow,\mu)
  &\vcentcolon=
  P^m_{\tow}(E_\tow, f_\tow), \mbox{ and} \\
u^{\vphantom{|}}_\towr
  &\vcentcolon=
  \varpi_{\towr_{m,\xi}}^*(u_\tow\Psi_\tow).
\end{aligned}
\end{equation*}
Then
\(\nm{\Psi_\tow}_k\leq C(k)\) and 
\(\abs{\mu} + \nm{u^{\vphantom{|}}_\tow}_{2,\alpha,\gamma}
\leq C\bigl(\nm{E_\tow}_{0,\alpha,\gamma}+\nm{f_\tow}_{1,\alpha}\bigr)\), 
so by \eqref{globalnorms} and \eqref{global_data_norm}
\begin{equation}
\label{mu_u_towr_bound}
\abs{\mu} + \nm{u^{\vphantom{|}}_\towr}_{2,\alpha,\gamma}
\leq
\nm{(E,f)}_{\alpha,\gamma}.
\end{equation}
By item \ref{BdyOpCompTow}
of Proposition \ref{initsurfcomp}
we also have
\begin{equation}
\label{approx_bdy_cond}
\nm[\big]{
    f
    -m^{-1}B_{\Sigma_{m,\xi}}^{\mathrm{Robin}}
      u^{\vphantom{|}}_\towr
    }_{1,\alpha}
  \leq
Cm^{-1}\nm{(E,f)}_{\alpha,\gamma},
\end{equation}
so that $u_\towr$ \emph{approximately}
satisfies the desired
boundary condition, with the error controlled by the
discrepancy between the region $\towr_{m,\xi}$
and its model $\towquot_{(m)}^+$
(compared via $\varpi_{\towr_{m,\xi}}$).
In the same way $u_\towr$
is also an approximate solution
to the problem posed on the interior, appropriately restricted,
except that there is an additional source of error
originating from the cutoff applied in defining
$u_\towr$ from $u_\tow$.
More precisely, we shall conveniently define
\begin{equation*}
E_\towr\vcentcolon=\varpi_{\towr_{m,\xi}}^*
  \left(
    \Psi_\tow^2 \varpi_{\towr_{m,\xi}}^{-1*}E
  \right)
  +\varpi_{\towr_{m,\xi}}^*
    \left( 
      \left[ 
        J_{\towquot_{(m)}^+},
        \Psi_\tow
      \right]
      u^{\vphantom{|}}_\tow
    \right),
\end{equation*}
where $[J_{\towquot_{(m)}^+}, \Psi_\tow]$
is the commutator of $J_{\towquot_{(m)}^+}$
with the operator that multiplies its argument by $\Psi_\tow$;
using the above definitions, item \ref{JacOpCompTow} of Proposition \ref{initsurfcomp}, and definitions
\eqref{eqn:definition_regions},
\eqref{towr1},
\eqref{globalnorms}, and \eqref{global_data_norm}, we then have 
$\spt E_\towr \subseteq \towr_{m,\xi}$ and $\spt (E-E_\towr)\subset\Sigma_{m,\xi} \setminus \towr_{m,\xi}^1$, where $\spt(\cdot)$ denotes the support of its argument, as well as the estimates 
\begin{align}
\nm{E_\towr}_{0,\alpha,\gamma}
&\leq C\nm{(E,f)}_{\alpha,\gamma}, 
\nonumber
\\\label{u_towr_approx_soln}
\nm[\big]{E_\towr + \mu \coker -m^{-2}J_{\Sigma_{m,\xi}} u^{\vphantom{|}}_\towr}_{0,\alpha,\gamma}
&\leq Cm^{-1/4}\nm{(E,f)}_{\alpha,\gamma}. 
\end{align}
Next we set
\(
E_{\K}\vcentcolon=\varpi_{\catr_{m,\xi}}^{-1*}(E-E_\towr)|_{\catr_{m,\xi}}
\) and \(
E_{\B^2}\vcentcolon=\varpi_{\discr_{m,\xi}}^{-1*}(E-E_\towr)|_{\discr_{m,\xi}}
\). 
From these facts, item \ref{TowrToNotTowr} of Proposition \ref{initsurfcomp}, and definitions \eqref{globalnorms} and \eqref{global_data_norm} we then have
\begin{equation*}
\nm{E_\K}_{0,\alpha} + \nm{E_{\B^2}}_{0,\alpha}
\leq
C m^\alpha e^{-\gamma m^{1/2}} \nm{(E,f)}_{\alpha,\gamma}.
\end{equation*}
Then
\(
u^{\vphantom{|}}_\K\vcentcolon= m^2P_{\K_{b_{m,\xi}}}^m E_\K
\) and \(
u_{\B^2}\vcentcolon=m^2P_{\K_{b_{m,\xi}}}^m E_{\B^2}
\) 
satisfy 
\begin{equation}
\label{u_cat_and_u_disc_bounds}
\nm{u^{\vphantom{|}}_{\K}}_{2,\alpha}
  + \nm{u_{\B^2}}_{2,\alpha}
\leq 
Cm^{2+\alpha}e^{-\gamma m^{1/2}}\nm{(E,f)}_{\alpha,\gamma}.
\end{equation}
Defining in turn
\begin{align*}
\Psi_{\Sigma \setminus \towr}
&\vcentcolon=1-\varpi_{\towr_{m,\xi}}^*
    (\cutoff{m^{1/4}}{m^{1/4}-1} \circ \dist_{\axis{\theta}}),
&
u^{\vphantom{|}}_\catr
&\vcentcolon=\Psi_{\Sigma \setminus \towr}\varpi_{\catr_{m,\xi}}^* u^{\vphantom{|}}_\K,
&
u^{\vphantom{|}}_\discr
&\vcentcolon= \Psi_{\Sigma \setminus \towr} \varpi_{\discr_{m,\xi}}^* u_{\B^2},
\end{align*}
it follows, also using
item \ref{NotTowrToTowr} of Proposition \ref{initsurfcomp}
and $\nm{\Psi_{\Sigma \setminus \towr}}_k\leq C(k)$, that
\begin{align}
\label{u_catr_and_u_discr_gamma_bounds}
\nm{u^{\vphantom{|}}_\catr}_{2,\alpha,\gamma}
  +\nm{u^{\vphantom{|}}_\discr}_{2,\alpha,\gamma}
&\leq 
Cm^{2+\alpha}\nm{(E,f)}_{\alpha,\gamma},
\\
\label{u_catr_and_u_discr_beta_bounds}
\nm{u^{\vphantom{|}}_\catr}_{2,\alpha,\beta}
  +\nm{u^{\vphantom{|}}_\discr}_{2,\alpha,\beta}
&\leq 
Cm^{2+\alpha}e^{(\beta-\gamma)m^{1/2}}
  \nm{(E,f)}_{\alpha,\gamma},
\end{align}
where the second line is a consequence of the first
and the definitions
\eqref{globalnorms} and \eqref{global_data_norm}
of the relevant norms.
Moreover, $u^{\vphantom{|}}_\catr$
and $u^{\vphantom{|}}_\discr$
are approximate solutions to
the interior problem
\[m^{-2}J_{\Sigma_{m,\xi}} u = E-E_\towr,\]
appropriately restricted,
where, just as for $u^{\vphantom{|}}_\towr$ above,
the error has two components,
one driven by the deviation
of the regions $\catr_{m,\xi}$ and $\discr_{m,\xi}$
from their models ${\K_{b_{m,\xi}}}$ and $\B^2$
and the other the cutoff error introduced
by $\Psi_{\Sigma \setminus \towr}$.
Indeed,
setting here
\begin{align*}
E_{\Sigma \setminus \towr}^{\mathrm{geometric}}
&\vcentcolon=
m^{-2}
  \left(
    J_{\Sigma_{m,\xi}} - \varpi_{\discr_{m,\xi}}^* J_{\B^2} \varpi_{\discr_{m,\xi}}^{-1*}
  \right)
  u^{\vphantom{|}}_\discr
+
m^{-2}
  \left(
    J_{\Sigma_{m,\xi}} - \varpi_{\catr_{m,\xi}}^* J_{\K_{b_{m,\xi}}} \varpi_{\catr_{m,\xi}}^{-1*}
  \right)
  u^{\vphantom{|}}_\catr
\quad \mbox{and} \quad
\\
E_{\Sigma \setminus \towr}^{\mathrm{cutoff}}
&\vcentcolon= 
m^{-2}\varpi_{\discr_{m,\xi}}^*
  \left(
    \left[
      J_{\B^2},~
      \varpi_{\discr_{m,\xi}}^{-1*} \Psi_{\Sigma \setminus \towr} 
    \right]
    u_{\B^2}
  \right)
+m^{-2}\varpi_{\catr_{m,\xi}}^*
  \left(
    \left[
      J_{\K_{b_{m,\xi}}},~
      \varpi_{\catr_{m,\xi}}^{-1*} \Psi_{\Sigma \setminus \towr} 
    \right]
    u^{\vphantom{|}}_\K
  \right),
\end{align*}
we have
\begin{equation}
\label{u_discr_and_catr_approx_soln}
m^{-2}J_{\Sigma_{m,\xi}}(u^{\vphantom{|}}_\discr+u^{\vphantom{|}}_\catr)
=
E-E_\towr
  +E_{\Sigma \setminus \towr}^{\mathrm{geometric}}
  +E_{\Sigma \setminus \towr}^{\mathrm{cutoff}},
\end{equation}
having observed that
$(E-E_\towr)\Psi^{\Sigma \setminus \towr}=E-E_\towr$.
Using \ref{JacOpCompCat}, \ref{JacOpCompDisc} and \ref{NotTowrToTowr} of Proposition~\ref{initsurfcomp},
definition \eqref{globalnorms},
and the estimates
\eqref{u_catr_and_u_discr_gamma_bounds},
we obtain as well
\begin{equation}
\label{E_geometric_est}
\nm[\big]{E_{\Sigma \setminus \towr}^{\mathrm{geometric}}}_{0,\alpha,\gamma}
\leq
Cm^{2+\alpha}e^{-m^{1/4}}\nm{(E,f)}_{\alpha,\gamma}.
\end{equation}
To estimate $E_{\Sigma \setminus \towr}^{\mathrm{cutoff}}$
note that
\begin{equation}\label{eq:SupportAnnulus}
\spt 
  \left[
    J_{\B^2},~
    \varpi_{\discr_{m,\xi}}^{-1*} \Psi_{\Sigma \setminus \towr} 
  \right]
\subset
\varpi_{\discr_{m,\xi}}
  \left(
    \varpi_{\towr_{m,\xi}}^{-1}
      \left(
        \{m^{1/4}-1 \leq \dist_{\axis{\theta}} \leq m^{1/4}\}
      \right)
  \right) 
 \end{equation} 
 as well as 
 \begin{equation}\label{eq:CheapCommutatorEst}
\nm[\Big]{
  \left[
    J_{\B^2},~
    \varpi_{\discr_{m,\xi}}^{-1*}
      \Psi_{\Sigma \setminus \towr} 
  \right]
  v
}_{0,\alpha}
\leq
Cm^{2+\alpha}
  \nm{v}_{2,\alpha}
\quad
\forall v \in C^{2,\alpha}(\B^2),
\end{equation}
and likewise if we simultaneously replace $\discr_{m,\xi}$ by $\catr_{m,\xi}$ and $\B^2$ by $\K_{b_{m,\xi}}$.
Hence, using item \ref{NotTowrToTowr} of Proposition~\ref{initsurfcomp},
definition \eqref{globalnorms},
and the estimates \eqref{eq:CheapCommutatorEst} and \eqref{u_cat_and_u_disc_bounds},
we then get
\begin{equation}
\label{E_cutoff_est}
\nm{
  E_{\Sigma \setminus \towr}^{\mathrm{cutoff}}
}_{0,\alpha,\gamma}
\leq
Cm^{2+2\alpha}
  e^{\gamma(m^{1/4}-m^{1/2})}
  \nm{(E,f)}_{\alpha,\gamma}.
\end{equation}
Concerning this bound, we warn the reader that actually (by the very definition of the cutoff function $\Psi_{\Sigma \setminus \towr}$, which directly implies \eqref{eq:SupportAnnulus}) it is only the \emph{first} summand on the right-hand side of  \eqref{globalnorms} that actually contributes to $\nm{
  E_{\Sigma \setminus \towr}^{\mathrm{cutoff}}
}_{0,\alpha,\gamma}.$

By following the foregoing construction
of $\mu$, 
$u^{\vphantom{|}}_\towr$,
$u^{\vphantom{|}}_\discr$,
and $u^{\vphantom{|}}_\catr$
for arbitrary data $(E,f)$
we define
\begin{align*}
\widetilde{P}_{m,\xi}\colon
C^{0,\alpha}_{\apr_m}(\Sigma_{m,\xi})
    \oplus 
    C^{1,\alpha}_{\apr_m}(\partial^0\Sigma_{m,\xi})
  &\to
  C^{2,\alpha}_{\apr_m}(\Sigma_{m,\xi})
    \oplus \R \\
(E,f)
  &\mapsto
  \bigl(\underbrace{u^{\vphantom{|}}_\towr
  +u^{\vphantom{|}}_\discr
  +u^{\vphantom{|}}_\catr}_{\vphantom{|}\displaystyle u}
  ,~\mu\bigr).
\end{align*}
The map $\widetilde{P}_{m,\xi}$ is clearly linear by construction, and by \eqref{mu_u_towr_bound} and \eqref{u_catr_and_u_discr_beta_bounds} we have
\begin{equation*}
\abs{\mu}+\nm{u}_{2,\alpha,\beta}
\leq
C\nm{(E,f)}_{\alpha,\gamma}
\end{equation*}
for 
$(u,\mu)=\widetilde{P}_{m,\xi}(E,f)$ and
any data $(E,f)$.
Moreover, if we define the map
\begin{align*}
L_{m,\xi}\colon
C^{2,\alpha}_{\apr_m}(\Sigma_{m,\xi}) \oplus \R
  &\to
 C^{0,\alpha}_{\apr_m}(\Sigma_{m,\xi})
    \oplus 
    C^{1,\alpha}_{\apr_m}(\partial^0\Sigma_{m,\xi}) \\
(u,\mu)
  &\mapsto
  \left(
    m^{-2}J_{\Sigma_{m,\xi}}u-\mu\coker,~
    \left(m^{-1}B_{\Sigma_{m,\xi}}^{\mathrm{Robin}}u\right)
      \Big|_{\partial^0\Sigma_{m,\xi}}
  \right)
\end{align*}
and write $I_{\dom(\widetilde{P}_{m,\xi})}$
for the identity map
on the domain of $\widetilde{P}_{m,\xi}$
(coinciding with the target of $L_{m,\xi}$),
then we find from
\eqref{approx_bdy_cond},
\eqref{u_towr_approx_soln},
\eqref{u_discr_and_catr_approx_soln},
\eqref{E_geometric_est},
and \eqref{E_cutoff_est}
that
\begin{equation*}
\lim_{m \to \infty}
  \nm[\Big]{
    I_{\dom(\widetilde{P}_{m,\xi})}
    -L_{m,\xi}\widetilde{P}_{m,\xi}
  }_{
  \End(\dom(\widetilde{P}_{m,\xi}))}
=0,
\end{equation*}
where the norm is the operator norm on the space of linear
maps from the domain of $\widetilde{P}_{m,\xi}$ to itself
equipped with the data norm $\nm{\cdot}_{\alpha,\gamma}$
as defined by
\eqref{global_data_norm}.
As a result,
the composite $L_{m,\xi}\widetilde{P}_{m,\xi}$
is invertible, for $m$ sufficiently large,
with inverse bounded
(under the same norm)
independently of $m$.
We conclude by taking
$
 P_{\Sigma_{m,\xi}}
 \vcentcolon=
 \widetilde{P}_{m,\xi}(L_{m,\xi} \widetilde{P}_{m,\xi})^{-1}
$.
\end{proof}

\section{Solution to the nonlinear problem}\label{sec:Nonlin}

In this section we shall exploit all the results obtained above to prove the following statement, which -- as explained in the introduction -- immediately implies Theorem \ref{thm:Main}.

\begin{theorem}\label{thm:ConstructFirst}
There exists a sequence $\{\Sigma_g\}_{g\geq g_0}$ of  properly embedded, free boundary minimal surfaces in $\B^3$ such that:
\begin{enumerate}[label={\normalfont(\alph*)}]
    \item  $\Sigma_g$ has genus $g$, three boundary components and symmetry group coinciding with the antiprismatic group of order $4(g+1)$;
    \item as one lets $g\to\infty$ the surface $\Sigma_g$ converges, in the sense of varifolds to the union $\K_0 \cup \B^2 \cup -\K_0$; the convergence is smooth, with multiplicity one, away from the intersection $\K_0 \cap \B^2 \cap -\K_0$.
\end{enumerate}
\end{theorem}

\begin{remark}\label{rem:AreaConj}
The area of the limit varifold is $\approx3.7921\,\pi$ (cf. Corollary~\ref{cor:area_K} and Remark~\ref{rem:values_for_b=0}); 
in particular, such a varifold has larger area than the union of the critical catenoid 
\gls{Kcrit} %$\K_{\mathrm{crit}}$ 
with the horizontal disc $\B^2$ that is approximately $2.6671\,\pi$, which in turn is relevant for the convergence result stated in Appendix \ref{app:convergence_large_g}. Incidentally, we mention here that,
by the monotonicity formula, a symmetric portion of the Karcher--Scherk tower has less area than the corresponding portion of its asymptotic planes, and therefore we would expect the area our free boundary minimal surfaces to be increasing in $g$ and in particular to be uniformly bounded from above by the area of the limit varifold.
\end{remark}

We will first discuss in Subsection \ref{subs:NonlinSetup}
(referring as needed to Appendix \ref{app:graphs}
for supporting technical details)
how to conveniently set up the construction as a nonlinear elliptic problem with oblique boundary conditions,
and we will later
describe in Subsection \ref{subs:ProofMainThm}  how to solve the problem in question by means of a suitable iteration scheme,
based on the ancillary results in Section \ref{sec:Linear}
(in particular relying on Proposition \ref{prop:OutputLinearTheory}) and on the preliminary estimates in Subsection \ref{subs:FirstOrderCorr}.

\subsection{Graphical deformation under the auxiliary metric}\label{subs:NonlinSetup}
Each initial surface $\Sigma_{m,\xi}$
has been constructed so as to intersect
$\Sp^2=\partial \B^3$ 
precisely along $\partial \Sigma_{m,\xi}$
and at a constant, right angle.
We wish to deform $\Sigma_{m,\xi}$
to a minimal surface while maintaining
these last conditions on the boundary.
To do so with minimum effort,
as in \cites{KapouleasWiygul2017,KapouleasZouCloseToBdy}
we will make these deformations graphically,
in the normal direction to $\Sigma_{m,\xi}$,
but with respect to an alternative
ambient metric, to be called the auxiliary metric, 
designed to preserve the above boundary conditions. The reason for that is easy to explain: if one even takes a flat equatorial disc in Euclidean $\B^3$ then \emph{any} normal graph will be ill-defined (in the sense that it does not correspond to a surface in $\B^3$) unless the defining function vanishes along the boundary of the disc in question, which is not the natural geometric boundary condition we wish to impose.

\paragraph{Definition of the auxiliary metric.}
In view of items \ref{deform_stays_on_Gamma}
and \ref{deform_Robin_cond_implies_orthogonality}
of Lemma \ref{basic_MC_bdy_lemma}
we wish to define our auxiliary metric $h$ on $\R^3$
so that $h|_{\Sp^2}=g_{\mathrm{euc}}|_{\Sp^2}$
and so that $\Sp^2$ is totally geodesic under $h$.
This way we ensure that every (normal) graphical
deformation of $\Sigma_{m,\xi}$,
defined with respect to $h$,
keeps its boundary on $\partial \B^3$
and meets $\partial \B^3$
orthogonally,
provided the defining function
$u_{h,g_{\mathrm{euc}}}$
satisfies the homogeneous Neumann condition,
again defined with respect to $h$.
As explained in the general setting
of Appendix \ref{app:graphs}
and as we will shortly clarify
in our specific application below,
this last boundary condition on $u_{h,g_{\mathrm{euc}}}$
is equivalent to the homogeneous Robin condition,
now with respect to $g_{\mathrm{euc}}$,
on a function $u$
which can be recovered from $u_{h,g_{\mathrm{euc}}}$
(and vice versa).

We can achieve the above two conditions on $h$ by a simple conformal change.
Specifically, we recall the notation \eqref{eqn:definition_cutoff} for cutoff functions and choose the conformal factor
\begin{equation}
\label{aux_conf_fact}
\Omega^4
\vcentcolon=
  \left(
    \cutoff{\frac{1}{3}}{\frac{2}{3}}
    \circ \dist_{\Sp^2}
  \right)
  + \left(
      \cutoff{\frac{2}{3}}{\frac{1}{3}}
      \circ \dist_{\Sp^2}
    \right)  \dist_{\{(0,0,0)\}}^{-2},
\end{equation}
where, according to definition
\eqref{eqn:definition_distancefunction},
$\dist_{\{(0,0,0)\}}(x,y,z)
 =
 r(x,y,z)
 \vcentcolon=
 \sqrt{x^2+y^2+z^2}
$
and $\dist_{\Sp^2}=\abs{1-r}$.
Thus $\Omega^4$ is identically $1$
on $\Sp^2_{\geq 2/3}$
(that is: near the origin, as well as very far away from it)
and identically $r^{-2}$
on $\Sp^2_{\leq 1/3}$
(that is: near the boundary unit sphere)
.
We in turn define the auxiliary metric
\begin{equation}
\label{auxmet}
\glsuseri{auxmet} %h
\vcentcolon=
\Omega^4 g_{\mathrm{euc}}.
\end{equation}%
Since $\Omega^4|_{\Sp^2}=1$,
clearly $h|_{\Sp^2}=g|_{\Sp^2}$.
We recall that if one considers a conformally flat Riemannian metric of the form
$h=\Omega^{4}g_{\mathrm{euc}}$ on $\B^3$
then the scalar-valued second fundamental form $A_{\Gamma}^h$
(respectively: $A_{\Gamma}^{g_{\mathrm{euc}}}$)
of a surface $\Gamma$ in metric $h$
(respectively: in Euclidean metric $g_{\mathrm{euc}}$)
obeys the equation
\begin{equation*}
A_{\Gamma}^h
=
\Omega
\left(
  \Omega A_{\Gamma}^{g_{\mathrm{euc}}}
  + 2\frac{\partial \Omega}{\partial \nu}g_{\mathrm{euc}}
\right),
\end{equation*}
where $\nu$ denotes
a $g_{\mathrm{euc}}$ unit normal
to $\Gamma$
(its orientation being chosen consistently
with the sign of $A_{\Gamma}^{g_{\mathrm{euc}}}$
according to our sign convention).
It follows at once from \eqref{aux_conf_fact} that
the  boundary of our ball is totally geodesic
with respect to $h$.

\begin{remark}
The conformality is not really essential
for our purposes here.
One very natural alternative,
more readily generalized to settings
other than $\B^3$,
to the definition we chose above
for the auxiliary metric $h$,
would be a smooth interpolation
(using cutoff functions again)
between the Euclidean metric $g_{\mathrm{euc}}$
and the cylindrical metric
$
 h^{\mathrm{cyl}}_{\Sp^2}
 \vcentcolon=
 d\sigma^2 + g^{\vphantom{|}}_{\Sp^2},
$
on a neighborhood of $\Sp^2$,
where $\sigma$ is directed distance from $\Sp^2$
toward (say) the origin,
$g_{\Sp^2}$ is the standard metric on $\Sp^2$,
and we are identifying a neighborhood of $\Sp^2$
with a cylinder $\Sp^2 \times \interval{-\epsilon,\epsilon}$
via the map
$
 \tubularexp_{
 (\R^3,g_{\mathrm{euc}}),
 (\Sp^2, \nabla_{\R^3} \sigma)
 }
$,
recalling definition \eqref{tubularcoordscodim1}.
Then, getting back to the setting above (and the specific auxiliary metric we defined), we remark that $h$ and $h^{\mathrm{cyl}}_{\Sp^2}$
are related,
on suitable neighborhoods of $\Sp^2$,
by a diffeomorphism.
In a more general setting,
with $\Sp^2$ replaced by a hypersurface
$\Gamma$ in a Riemannian manifold,
the analogously constructed
$h^{\mathrm{cyl}}_\Gamma$
would,
unlike our $h$,
in general entail
the consideration of
nonzero
tangential components
$u^{\top}_{h^{\mathrm{cyl}}_\Gamma,g}$
in an application of Lemma \ref{basic_MC_bdy_lemma}.
An alternative (but fundamentally similar) framework
for constructing graphs with controlled boundary angle
and mean curvature
over surfaces in general Riemannian manifolds with boundary
is presented in detail in
\cite[Section~2]{KapouleasLiDiscCCdesing}.
\end{remark}

Before proceeding we point out that
since the conformal factor $\Omega^4$
is manifestly $\Ogroup(3)$-invariant,
in particular every member of $\apr_m$,
for every integer $m \geq 1$,
is an isometry of the auxiliary metric $h$.

\paragraph{Formulation of the nonlinear problem.}
With the above notation and definitions in place
we can precisely formulate the nonlinear boundary value
problem we will solve in order to construct our family
of free boundary minimal surfaces.
Let $\Sigma_{m,\xi}$ be an initial surface as constructed in \eqref{initsurfdef}, 
let $\nu_{m,\xi}$ denote a unit normal to $\Sigma_{m,\xi}$ 
and let $\iota^{\vphantom{|}}_{\Sigma_{m,\xi}}\colon \Sigma_{m,\xi} \to \B^3 \subset \R^3$ be the inclusion map. 
Then 
\begin{equation}
\label{normal_in_aux_met}
\nu^h_{m,\xi}\vcentcolon=\Omega^{-2}\nu_{m,\xi} 
\end{equation}
is a unit normal to $\Sigma_{m,\xi}$ with respect to the conformal metric $h=\Omega^4 g_{\mathrm{euc}}$ on $\B^3$ defined in \eqref{aux_conf_fact}--\eqref{auxmet}.  
Given a function $u\colon \Sigma_{m,\xi} \to \R$ and using the notation of Appendix \ref{app:graphs}
(see \eqref{u_h} in particular), 
we observe that
\begin{equation*}
u_{h,g_{\mathrm{euc}}}
=
\frac{u}{g_{\mathrm{euc}}(\nu^h_{m,\xi},\nu_{m,\xi})}
=
\Omega^2u
\quad \mbox{and} \quad
u_{h,g_{\mathrm{euc}}}^\top = 0.
\end{equation*}
Let then the map 
$\gls{iota}%\iota_{m,\xi,u}
\colon \Sigma_{m,\xi} \to \R^3$ be given by
\begin{equation}
\label{deformed_init_surf_inc}
\iota_{m,\xi,u}(p)
\vcentcolon=
\exp_{\iota^{\vphantom{|}}_{\Sigma_{m,\xi}}(p)}^{(\R^3,h)} u_{h,g_{\mathrm{euc}}}(p)\nu^h_{m,\xi}(p).
\end{equation}
For $u \in C^2(\Sigma_{m,\xi})$ sufficiently small, 
the mean curvature
\begin{equation}
\label{deformed_MC_def}
\glsuseri{deformedMC} %H_{m,\xi,u}
\end{equation}
of $\iota_{m,\xi,u}$
(as measured with respect to $g_{\mathrm{euc}}$)
is well-defined
as in item \ref{deform_well_defined}
of Lemma \ref{basic_MC_bdy_lemma}.
We seek $u$ such that
$H_{m,\xi,u}=0$ on $\Sigma_{m,\xi}$.
In this event,
because $\Sigma_{m,\xi}$ was constructed
to meet $\partial \B^3$ orthogonally
and precisely along $\partial \Sigma_{m,\xi}$,
item \ref{deform_stays_on_Gamma}
of Lemma \ref{basic_MC_bdy_lemma}
ensures that
$
 \iota_{m,\xi,u}(\partial \Sigma_{m,\xi})
 \subset
 \partial \B^3
$,
and the maximum principle then
guarantees that
$\iota_{m,\xi,u}(\Sigma_{m,\xi}) \subset \B^3$
and
$\iota_{m,\xi,u}^{-1}(\partial \B^3)=\partial\Sigma_{m,\xi}$.
Item \ref{deform_Robin_cond_implies_orthogonality}
of Lemma \ref{basic_MC_bdy_lemma}
moreover implies that
$\iota_{m,\xi,u}(\Sigma_{m,\xi})$
will intersect $\partial \B^3$ orthogonally
provided we impose the corresponding
homogeneous Robin condition on $u$.

Of course, we will enforce the symmetries of the construction.
Since $h$ and $\Omega^4$ are $\apr_m$-invariant,
$\iota_{m,\xi,u}$ will commute with every element
of $\apr_m$ whenever $u$ is $\apr_m$-equivariant,
and in this case $H_{m,\xi,u}$ will also be
$\apr_m$-equivariant.
We therefore require 
\begin{align}
\label{system_to_prove_main_theorem}
u &\in C^2_{\apr_m}(\Sigma_{m,\xi}),
& \hspace{-1cm}
\left\{\begin{aligned}
H_{m,\xi,u}&=0
&& \text{ on } \Sigma_{m,\xi}
\\
B^{\mathrm{Robin}}_{\Sigma_{m,\xi}}u&=0
&& \text{ on } \partial \Sigma_{m,\xi}.
\end{aligned}\right.
\end{align}
As previously emphasized,
it will be necessary to solve for the parameter $\xi$
as well and to take $m$ large.
By taking $m$ large enough
we will see that we can indeed ensure solvability
and further ensure that the solution
$u$ is small enough to guarantee
not only that $\iota_{m,\xi,u}$
is an immersion
with $H_{m,\xi,u}$ well-defined
(and zero for the solution)
but also that $\iota_{m,\xi,u}$
is in fact an embedding
and that its image has no symmetries
outside $\apr_m$.
We will return to these points
in Subsection \ref{subs:ProofMainThm},
after we have obtained our solutions
and estimates.
These will be obtained
via an iteration scheme,
applying our estimates
(Proposition \ref{initsurfcomp})
of the initial mean curvature
and our results
(Proposition \ref{globalsol})
on the linearized problem.
We now turn to the details of this iteration scheme.

\subsection{First-order correction and quadratic terms}\label{subs:FirstOrderCorr}

For each initial surface $\Sigma_{m,\xi}$,
with $m$ large enough
(in terms of an upper bound on $\xi$
and a universal constant)
to ensure the existence
of the operator $P_{\Sigma_{m,\xi}}$
of Proposition \ref{globalsol},
we define the linear operator
\begin{equation}
\label{global_resolvent_with_hom_bdy_data}
\begin{aligned}
P_{m,\xi}
  \colon
  C^{0,\alpha}_{\apr_m}(\Sigma_{m,\xi})
  &\to 
  C^{2,\alpha}_{\apr_m}(\Sigma_{m,\xi})
    \oplus \R \\
E 
  &\mapsto
  P_{\Sigma_{m,\xi}}(E,0)
\end{aligned}
\end{equation}
and,
keeping in mind the issues related to the non-trivial approximate cokernel,
the function
$u_{m,\xi}^{(1)} \in C^\infty_{\apr_m}(\Sigma_{m,\xi})$
and real number $\mu_{m,\xi}^{(1)}$
by
\begin{equation}
\label{sol_to_first_order}
\left(u_{m,\xi}^{(1)},\mu_{m,\xi}^{(1)}\right)
\vcentcolon=
m^{-2}P_{m,\xi}\left(H_{\Sigma_{m,\xi}}-\xi\coker\right).
\end{equation}
Here,
although we have no Hilbertian structure in place,
heuristically
one can think of the second summand on the right-hand side
as the projection of the initial mean curvature
onto the approximate cokernel.
We subtract this term,
because,
in the course of the construction of the solution,
at each application of $P_{m,\xi}$
we incur error in the direction
of the approximate cokernel,
and we need to reserve this term now
so that we can later control such error
by varying $\xi$.
If the second term were absent from \eqref{sol_to_first_order},
the corresponding bound in Lemma \ref{first-order_ests} below
would depend on $\xi$,
and we would then require
more detailed information on that dependence
in order to solve our problem
in the direction of the approximate cokernel.

Recalling \eqref{deformed_init_surf_inc} and \eqref{deformed_MC_def}
defining $H_{m,\xi,u}$ we of course have
\begin{equation*}
H_{\Sigma_{m,\xi}}=H_{m,\xi,0},
\end{equation*}
and for $u \in C^2(\Sigma_{m,\xi})$
sufficiently small we further define the function
\begin{equation}
\label{nonlinear_terms_def}
\glsuserii{deformedMC} %Q_{m,\xi,u}
\vcentcolon=
H_{m,\xi,u} - H_{m,\xi,0} + J_{\Sigma_{m,\xi}}u,
\end{equation}
so that, for each $m$ and $\xi$,
$u \mapsto Q_{m,\xi,u}$
is a nonlinear, second-order partial differential operator in $u$.
We denote by
\begin{equation*}
\dom(Q_m)
\end{equation*}
the collection of all
$(\xi,u) \in \R \oplus C^2(\Sigma_{m,\xi})$
such that $H_{m,\xi,u}$ and so also $Q_{m,\xi,u}$
are defined.

By Proposition \ref{globalsol} and definition \eqref{sol_to_first_order}, we have
\begin{align*}
\left\{\begin{aligned}
J_{\Sigma_{m,\xi}}u_{m,\xi}^{(1)}&=H_{m,\xi,0}-\xi\coker+m^2\mu_{m,\xi}^{(1)}\coker
\\[.5ex]
B^{\mathrm{Robin}}_{\Sigma_{m,\xi}}u_{m,\xi}^{(1)}
&=0.
\end{aligned}\right.
\end{align*}
In particular, using \eqref{nonlinear_terms_def}, 
\begin{align*}
H_{m,\xi,u_{m,\xi}^{(1)}}
&=\left(\xi-m^2\mu_{m,\xi}^{(1)}\right)\coker+Q_{m,\xi,u_{m,\xi}^{(1)}}.
\end{align*}
In view of \eqref{system_to_prove_main_theorem} we seek
a higher-order correction
$u_{m,\xi}^{(2)} \in C^\infty_{\apr_m}(\Sigma_{m,\xi})$
such that
\begin{equation}
\label{desired_correction}
\left\{\begin{aligned}
H_{m,\xi,u_{m,\xi}^{(1)}+u_{m,\xi}^{(2)}}
&=0
\\[.5ex]
B^{\mathrm{Robin}}_{\Sigma_{m,\xi}}u_{m,\xi}^{(2)}
&=0.
\end{aligned}\right.
\end{equation}
Appealing to \eqref{nonlinear_terms_def} with $u_{m,\xi}^{(1)}+u_{m,\xi}^{(2)}$ in lieu of $u$ we obtain 
\begin{equation}\label{eqn:Hu1+u2}
H_{m,\xi,u_{m,\xi}^{(1)}+u_{m,\xi}^{(2)}}
=
\left(\xi-m^2\mu_{m,\xi}^{(1)}\right)\coker
  -J_{\Sigma_{m,\xi}}u_{m,\xi}^{(2)}
  + Q_{m,\xi,u_{m,\xi}^{(1)}+u_{m,\xi}^{(2)}}. 
\end{equation}
Thus, we can secure
\eqref{desired_correction}
if we require the pair
\( 
\bigl(v_{m,\xi},\mu_{m,\xi}^{(2)}\bigr)
\vcentcolon=
m^{-2}P_{m,\xi}
  Q_{m,\xi,u_{m,\xi}^{(1)}+u_{m,\xi}^{(2)}}
\) 
to satisfy
\begin{equation}\label{eqn:conditions2}
v_{m,\xi}=u_{m,\xi}^{(2)}
\quad \mbox{and} \quad
\mu_{m,\xi}^{(1)}+\mu_{m,\xi}^{(2)} = m^{-2}\xi.
\end{equation}
Indeed, if \eqref{eqn:conditions2} holds, then again by Proposition \ref{globalsol}
\begin{align*}
J_{\Sigma_{m,\xi}}u_{m,\xi}^{(2)}&=   Q_{m,\xi,u_{m,\xi}^{(1)}+u_{m,\xi}^{(2)}}
+m^2\mu_{m,\xi}^{(2)}\coker
\shortintertext{which combined with \eqref{eqn:Hu1+u2} yields}
H_{m,\xi,u_{m,\xi}^{(1)}+u_{m,\xi}^{(2)}}
&=
\left(\xi-m^2\mu_{m,\xi}^{(1)}-m^2\mu_{m,\xi}^{(2)}\right)\coker
=0.
\end{align*}
In order to achieve the two conditions in \eqref{eqn:conditions2}
we must solve for
$\xi$ and $u_{m,\xi}^{(2)}$
simultaneously.
Since $\xi$ is also an unknown,
we will use the maps
$
 \varsigma_{m,\xi}
 \colon
 \Sigma_{m,0}
 \to
 \Sigma_{m,\xi}
$
defined in \eqref{Sigma_0_to_Sigma_xi}
to identify candidates for $u_{m,\xi}^{(2)}$
with functions on $\Sigma_{m,0}$.
Specifically,
we fix $\alpha \in \interval{0,1}$
and
for each sufficiently large $m$
we pursue a fixed point
(of small norm)
to the nonlinear map
\begin{align}
\label{map_to_prove_main_theorem}
F_{m,\alpha}
  \colon
  \dom(F_{m,\alpha})
  &\to
  C^{2,\alpha}_{\apr_m}(\Sigma_{m,0})
    \oplus
    \R
\\\notag
(v,~ \xi)
  &\mapsto
  \left(
    \varsigma_{m,\xi}^*u_{m,\xi,v},~
    m^2\mu_{m,\xi}^{(1)}
    +m^2\mu_{m,\xi,v}
  \right),
\shortintertext{where}
\notag
\left( u_{m,\xi,v},~ \mu_{m,\xi,v} \right)
  &\vcentcolon= 
  m^{-2}P_{m,\xi}Q_{
    m,\,\xi,\,u_{m,\xi}^{(1)}
      +\varsigma_{m,\xi}^{-1*}v
  }
\\\notag
\dom(F_{m,\alpha})
&\vcentcolon=
\left\{ 
  (v,\xi)
  \in 
  C^{2,\alpha}_{\apr_m}(\Sigma_{m,0})
    \oplus 
    \R
  \st
  \left(
    \xi,~
    u_{m,\xi}^{(1)}
      +\varsigma_{m,\xi}^{-1*}v
  \right)
  \in
  \dom Q_m
\right\}.
\end{align}
In the next subsection we will identify
such a fixed point by using
Schauder's fixed point theorem
(stated for example as Theorem 11.1 in \cite{GilTru2001}),
for whose application
we will require
the following estimates.

\begin{lemma}
[Estimates for the first-order correction]
\label{first-order_ests}
Let $c>0$ and $\alpha,\beta \in \interval{0,1}$.
There exist $C=C(\alpha,\beta)>0$
(independent of $c$)
and $m_0=m_0(c)>0$
such that for any $\xi \in \IntervaL{-c,c}$
and any integer $m>m_0$
the function
$
 u_{m,\xi}^{(1)}
 \in C^\infty_{\apr_m}(\Sigma_{m,\xi})
$
and real number $\mu_{m,\xi}^{(1)}$
are well-defined
by \eqref{sol_to_first_order}
and satisfy the estimate
\begin{equation*}
\nm[\Big]{u_{m,\xi}^{(1)}}_{2,\alpha,\beta} 
  + \abs[\Big]{\mu_{m,\xi}^{(1)}}
\leq 
Cm^{-2}.
\end{equation*}
Furthermore, for each $m>m_0$ the map
\begin{align*}
\IntervaL{-c,c}
  &\to
  C^{2,\alpha}(\Sigma_{m,0}) \\
\xi
  &\mapsto
  \varsigma_{m,\xi}^*u_{m,\xi}^{(1)}
\end{align*}
is continuous.
\end{lemma}

\begin{proof}
That $u_{m,\xi}^{(1)}$
and $\mu_{m,\xi}^{(1)}$
are well-defined
and bounded as stated
are corollaries of 
definition \eqref{sol_to_first_order}
(of $u_{m,\xi}^{(1)}$ and $\mu_{m,\xi}^{(1)}$),
definition \eqref{global_resolvent_with_hom_bdy_data}
(of $P_{m,\xi}$),
Proposition \ref{globalsol}
(providing existence and estimates for $P_{m,\xi}$),
definitions \eqref{globalnorms} and \eqref{global_data_norm}
(of the norms $\nm{\cdot}_{k,\alpha,\beta}$
and $\nm{(\cdot, \cdot)}_{\alpha,\beta}$),
and items \ref{HestTow} and \ref{HestOffTow}
of Proposition \ref{initsurfcomp}
(estimating the mean curvature
of the initial surfaces).
In particular,
in applying Proposition \ref{globalsol},
we may choose
any $\gamma \in \interval{\beta,1}$
(for example $\gamma=(1+\beta)/2$),
since items \ref{HestTow}
and \ref{HestOffTow} of Proposition \ref{initsurfcomp}
ensure the estimate
\begin{equation*}
\nm{H_{\Sigma_{m,\xi}}-\xi \coker}_{0,\alpha,\gamma}
\leq
C(\gamma).
\end{equation*}
The continuity claim
follows from the continuity statement
in Proposition \ref{globalsol}
for $P_{\Sigma_{m,\xi}}$
and
the properties \eqref{Sigma_0_to_Sigma_xi}
of $\varsigma_{m,\xi}$.
\end{proof}

\begin{lemma}
[Estimate of the nonlinear terms]
\label{quad_ests}
Let $C,c>0$,
$\alpha, \beta, \gamma \in \interval{0,1}$,
with
$\gamma < 2\beta$.
There exists
$m_0=m_0(C,c)$
such that for any
integer $m>m_0$,
any real $\xi \in \IntervaL{-c,c}$,
and any function
$u \in C^{2,\alpha}_{\apr_m}(\Sigma_{m,\xi})$
satisfying
$
 \nm{u}_{2,\alpha,\beta}
 \leq
 Cm^{-2}
$
the corresponding function
$H_{m,\xi,u} \in C^{0,\alpha}_{\apr_m}(\Sigma_{m,\xi})$
is well-defined by \eqref{deformed_init_surf_inc}, \eqref{deformed_MC_def}, the function
$Q_{m,\xi,u} \in C^{0,\alpha}_{\apr_m}(\Sigma_{m,\xi})$
is well-defined by \eqref{nonlinear_terms_def}
and moreover satisfies the estimate
\begin{equation*}
\nm{Q_{m,\xi,u}}_{0,\alpha,\gamma}
\leq
m^{-3/4}.
\end{equation*}
In addition, for each $m$,
the map
\begin{equation*}
\begin{aligned}
\{
  (
    \varsigma_{m,\xi}^*u,
    \xi
  )
  \st 
  (\xi,u) \in \dom(Q_m)
\}
&\to 
C^{0,\alpha}(\Sigma_{m,0})
\\
(v,\xi)
&\mapsto
\varsigma_{m,\xi}^*Q_{m,\xi,\varsigma_{m,\xi}^{-1*}v}
\end{aligned}
\end{equation*}
is continuous.
\end{lemma}

\begin{remark}
The proof of Lemma \ref{quad_ests}
will show that
in the stated estimate for $Q_{m,\xi,u}$
we could in fact
replace
$m^{-3/4}$
by $C(\alpha,\beta,\gamma)m^{-1}$,
but this improvement
is irrelevant to the proof of the main theorem.
\end{remark}

\begin{proof}
The continuity
(assuming existence of $Q_{m,\xi,u}$)
is clear from the definitions.
For the existence of $H_{m,\xi,u}$
(whence follows the existence of $Q_{m,\xi,u}$)
and for the estimate
we will appeal to items
\ref{deform_well_defined}
and
\ref{deform_quad_est_of_H}
respectively
of Lemma \ref{basic_MC_bdy_lemma}.
The role of the pair $(g,h)$
of metrics in Lemma \ref{basic_MC_bdy_lemma}
will be played by
$(m^2g_{\mathrm{euc}},m^2h)$,
this latter $h$ referring
of course to the auxiliary metric \eqref{auxmet}.
The role of $\phi$ in Lemma \ref{basic_MC_bdy_lemma}
will be played
by the inclusion map 
$\iota^{\vphantom{|}}_{\Sigma_{m,\xi}}$
of $\Sigma_{m,\xi}$ in $\B^3 \subset \R^3$.

We want the role of $\phi_u$
in Lemma \ref{basic_MC_bdy_lemma}
to be played by the deformed inclusion
$\iota_{m,\xi,u}$,
defined by \eqref{deformed_init_surf_inc}.
In the notation of Appendix \ref{app:graphs}
(specifically equation \eqref{deformed_phi})
\begin{equation*}
\iota_{m,\xi,u}
=
\iota^{\vphantom{|}}_{\Sigma_{m,\xi}}
  [u_{h,g_{\mathrm{euc}}},h,\nu_{m,\xi}^h],
\end{equation*}
where $\nu^h_{m,\xi}$ is the $h$ unit normal
to $\Sigma_{m,\xi}$,
as defined in \eqref{normal_in_aux_met}.
On the other hand we have the scaling identity
(as follows for example from
Remarks
\ref{scaling_of_MC_and_deformed_map}
and \ref{scaling_of_u_h})
\begin{equation*}
\iota^{\vphantom{|}}_{\Sigma_{m,\xi}}
  [u_{h,g_{\mathrm{euc}}},h,\nu_{m,\xi}^h]
=
\iota^{\vphantom{|}}_{\Sigma_{m,\xi}}
  [mu_{m^2h,m^2g_{\mathrm{euc}}},m^2h,\nu^{m^2h}_{m,\xi}],
\end{equation*}
$\nu^{m^2h}_{m,\xi}=m^{-1}\nu^h_{m,\xi}$
being the $m^2h$ unit normal to $\Sigma_{m,\xi}$
which is parallel to $\nu^h_{m,\xi}$.
Of course
$v_{m^2h,m^2g_{\mathrm{euc}}}$ is linear in $v$,
so
$
 mu_{m^2h,m^2g_{\mathrm{euc}}}
 =
 (mu)_{m^2h,m^2g_{\mathrm{euc}}}
$.
Accordingly,
the role of $u$ in Lemma~\ref{basic_MC_bdy_lemma}
will be played by $mu$,
with the latter $u$ the one in the statement
of the present lemma.
By the scaling law
for mean curvature
(as in Remark \ref{scaling_of_MC_and_deformed_map}),
the role of $H_u$ in Lemma \ref{basic_MC_bdy_lemma}
will be played by $m^{-1}H_{m,\xi,u}$,
as defined by
\eqref{deformed_init_surf_inc}--\eqref{deformed_MC_def},
when it exists.
We summarize the correspondence in Table \ref{table:correspondence}.

\begin{table}[htbp]
\centering\renewcommand{\arraystretch}{1.2} 
\begin{tabular}{>{\centering\arraybackslash}p{0.2\textwidth}|>{\centering\arraybackslash}p{0.2\textwidth}}
Lemma \ref{basic_MC_bdy_lemma} & Lemma \ref{quad_ests} \\
\hline
$g$ & $m^2g_{\mathrm{euc}}$ \\
\hline
$h$ & $m^2h$ \\
\hline
$\phi$ & $\iota^{\vphantom{|}}_{\Sigma_{m,\xi}}$ \\
\hline
$u$ & $mu$ \\
\hline
$\phi_u$ & $\iota_{m,\xi,u}$ \\
\hline
$H_u-H_0+J_\Sigma u$ & $m^{-1}Q_{m,\xi,u}$ \\
\hline
\end{tabular}
\caption{Application of
Lemma \ref{basic_MC_bdy_lemma}
to
Lemma \ref{quad_ests}.}
\label{table:correspondence}
\end{table}

Referring to 
item \ref{initsurfuniformgeometry}
of Proposition \ref{initsurfbasicprops}
(which in particular asserts the $C^2$ boundedness,
independently of $m$, of the second fundamental
form of $m\Sigma_{m,\xi}$)
and
the definition 
\eqref{aux_conf_fact}
of the conformal factor defining $h$,
it is clear that the assumptions
\eqref{assumptions_for_basic_MC_bdy_lemma}
of Lemma \ref{basic_MC_bdy_lemma}
are satisfied
with
$
 (g,h,\phi)
 =
 (
  m^2g_{\mathrm{euc}},
  m^2h,
  \iota^{\vphantom{|}}_{\Sigma_{m,\xi}}
 )
$,
possibly after further scaling $m^2g_{\mathrm{euc}}$
and $m^2h$ by a factor of the square of
$1$ plus the norm of second fundamental form
of $\Sigma_{m,\xi}$
with respect to $m^2g_{\mathrm{euc}}$,
which would introduce to the below estimates
a multiplicative constant
independent of $c$, $m$, and $\xi$
that we can safely suppress.
Let $\epsilon$ and $C_1$
be respectively the constants $\epsilon$ and $C$
provided by Lemma \ref{basic_MC_bdy_lemma}.

Next, in order to express the estimate
of item \ref{deform_quad_est_of_H}
of Lemma \ref{basic_MC_bdy_lemma}
in terms of the weighted norms
on the initial surfaces,
we want to compare the
$
  C^{\ell,\alpha}
  (
   \Sigma_{m,\xi},
   m^2\iota_{\Sigma_{m,\xi}}^*g_{\mathrm{euc}}
 )
$
norm,
written $\nm{\cdot}_{\ell,\alpha}$
for the remainder of this proof,
to the
$\nm{\cdot}_{\ell,\alpha,\delta}$
norm defined by \eqref{globalnorms}.
Because of the nonuniform weight function
appearing in this last definition
we do not really want a global comparison.
Rather, in order to complete the proof,
it suffices to estimate $Q_{m,\xi,u}$
on a neighborhood of each point of $\Sigma_{m,\xi}$.
Therefore,
because $H_{m,\xi,u}$
is a local operator,
we may assume
that $u$ has support contained
in an open disc $B_p$ in $\Sigma_{m,\xi}$
with $m^2g_{\mathrm{euc}}$ radius $1$
and center some $p \in \Sigma_{m,\xi}$.
Then $Q_{m,\xi,u}$,
whenever it exists,
will share this property.

Now suppose that $v$
is a function on $\Sigma_{m,\xi}$
with support contained in $B_p$.
Recall that the definition \eqref{globalnorms}
of $\nm{\cdot}_{\ell,\alpha,\delta}$
employs the metrics on the various
model surfaces used to construct
$\Sigma_{m,\xi}$,
and recall further the comparisons
that items \ref{MetCompTow}
and \ref{MetCompDiscAndCat}
of Proposition \ref{initsurfcomp}
make between these metrics
and the induced metric
$
 g^{\vphantom{|}}_{\Sigma_{m,\xi}}
 =\iota_{\Sigma_{m,\xi}}^*g_{\mathrm{euc}}
$.
Using these comparisons
and definition \eqref{globalnorms},
we obtain,
whenever $\ell \leq 2$,
$m > m_1$,
and $\abs{\xi} \leq c$,
\begin{equation*}
\frac{1}{10}
  e^{\delta s(p)}
  \nm{v}_{\ell,\alpha}
\leq
\nm{v}_{\ell,\alpha,\delta}
\leq
10
  (f(p))^{\ell+\alpha}
  e^{\delta s(p)}
  \nm{v}_{\ell,\alpha},
\end{equation*}
where
\begin{align*}
f(p)
  &=
  \begin{cases}
    1 
      & \mbox{if }
      p \in \towr_{m,\xi}^1
    \\
    m 
      & \mbox{if }
      p \in \Sigma_{m,\xi} \setminus \towr_{m,\xi}^1
  \end{cases}
&
s(p)
  &=
  \begin{cases}
    \varpi_{\towr_{m,\xi}}^{-1*}\dist_{\axis{\theta}}
      & \mbox{if }
      p \in \towr_{m,\xi}
    \\ 
    m^{1/2}
      & \mbox{if }
      p \in \Sigma_{m,\xi} \setminus \towr_{m,\xi}.
  \end{cases}
\end{align*}
(The function $f$ is needed
due to the difference in scale
between the region $\towr_{m,\xi}$
and the regions $\catr_{m,\xi}$
and $\discr_{m,\xi}$.)

For $v=mu$ we then have
\begin{equation*}
\nm{mu}_{2,\alpha}
\leq
10Cm^{-1}e^{-\beta s(p)}
\end{equation*}
(with $C$ as in the statement
of the present lemma).
In particular
$\lim_{m \to \infty} \nm{mu}_{2,\alpha} = 0$,
so there exists $m_2 \geq m_1$
such that whenever $m > m_2$
and $\abs{\xi} \leq c$,
we have
$\nm{mu}_{2,\alpha} < \epsilon$.
Item \ref{deform_well_defined}
of Lemma \ref{basic_MC_bdy_lemma}
therefore ensures
$H_{m,\xi,u}$ and so too $Q_{m,\xi,u}$
are well-defined under these same assumptions.

Moreover,
item \ref{deform_quad_est_of_H}
of Lemma \ref{basic_MC_bdy_lemma}
(bearing in mind the conformality
of $h$ to $g_{\mathrm{euc}}$
and Remark \ref{conformal_special_case})
implies the bound
\begin{equation*}
\nm{m^{-1}Q_{m,\xi,u}}_{0,\alpha} 
\leq
100\,C^2C_1m^{-2}e^{-2\beta s(p)},
\end{equation*}
whence
\begin{equation*}
\nm{Q_{m,\xi,u}}_{0,\alpha,\gamma}
\leq
\sup_{p \in \Sigma_{m,\xi}}
1000\,C^2C_1m^{-1}(f(p))^\alpha e^{(\gamma-2\beta)s(p)}.
\end{equation*}
Since we assume $\gamma<2\beta$,
there exists $m_0 \geq m_2$,
large enough in terms of $C$ and $C_1$
(itself independent of $m$),
such that the estimate
asserted in the statement of the lemma
holds whenever $m>m_0$ and $\abs{\xi} \leq c$.
\end{proof}

\subsection{Proof of the main theorem}\label{subs:ProofMainThm}

Most of the remaining work in proving Theorem \ref{thm:ConstructFirst} is done by the following lemma.

\begin{lemma}
[Existence of a fixed point]
\label{existence_of_fixed_point}
Let $\alpha \in \interval{0,1}$
and $\beta \in \interval{1/2,1}$.
There are constants $C,m_0>0$ such that
for every integer $m>m_0$
the map
$F_{m,\alpha/2}$
is well-defined by \eqref{map_to_prove_main_theorem}
and
has a fixed point
$
 \left(
   v_m,
   \xi_m 
 \right)
 \in
 C^{2,\alpha}_{\apr_m}(\Sigma_{m,0})
 \oplus
 \R
$
such that
\begin{equation}\label{eq:Est_Fixed_Point}
m^2\nm{\varsigma_{m,\xi}^{-1*} v_m}_{2,\alpha,\beta}
  +\abs{\xi_m}
\leq
C.
\end{equation}
\end{lemma}

\begin{proof}
Fix $\alpha \in \interval{0,1}$,
$\beta \in \interval{1/2,1}$,
and $\gamma \in \interval{\beta,1}$,
so that in particular
we have $0<\beta<\gamma<1<2\beta$.
For the remainder of the proof
we write $C_1$ for the constant $C(\alpha,\beta)$
from Lemma \ref{first-order_ests}
and we set $c\vcentcolon=2C_1$.
We will consider the map
$F_{m,\alpha/2}$
defined by \eqref{map_to_prove_main_theorem}
(with $\alpha/2$ in place of $\alpha$
as part of preparation to apply
the Schauder fixed point theorem below).
It follows from Proposition \ref{globalsol},
Lemma \ref{first-order_ests},
and Lemma \ref{quad_ests}
that
$F_{m,\alpha/2}$
is defined for sufficiently large $m$
and
is continuous
(with respect to the norm $\nm{\cdot}_{2,\alpha/2}+\abs{\cdot}$).
The same references along with
\eqref{weighted_ests_for_Sigma_0_to_Sigma_xi}
imply the existence
of
$m_1,C_2,C_3>0$
such that
for every integer $m>m_1$
and
every $\xi \in \IntervaL{-c,c}$
we have
\begin{align}
\label{C_one}
&\text{by Lemma \ref{first-order_ests}}
\qquad
\nm[\Big]{u_{m,\xi}^{(1)}}_{2,\alpha,\beta}
  +\abs[\Big]{\mu_{m,\xi}^{(1)}}
  \leq
  C_1m^{-2},
\\
\label{C_sigma}
&\text{by \eqref{weighted_ests_for_Sigma_0_to_Sigma_xi}}
\qquad
\begin{aligned}
\nm{\varsigma_{m,\xi}^{-1*}v}_{2,\alpha,\beta}
  &\leq
  C_2 \nm{v}_{2,\alpha,\beta}
  \quad \forall
  v \in C^{2,\alpha}_{\apr_m}(\Sigma_{m,0}),
\\
\nm{\varsigma_{m,\xi}^*u}_{2,\alpha,\beta}
  &\leq
  C_2 \nm{u}_{2,\alpha,\beta}
  \quad \forall
  u \in C^{2,\alpha}_{\apr_m}(\Sigma_{m,\xi}),
\end{aligned}
\\
\label{C_P}
&\text{by \eqref{global_resolvent_with_hom_bdy_data}
  and Proposition \ref{globalsol}
  }
\qquad
\begin{aligned}
&\nm{u}_{2,\alpha,\beta}
  +\abs{\mu}
\leq
C_3\nm{E}_{0,\alpha,\gamma}
\\
&\text{for} \quad
  (u,\mu) \vcentcolon= P_{m,\xi}E
\end{aligned}
\quad \forall
  E \in C^{0,\alpha}_{\apr_m}(\Sigma_{m,\xi}),
\\
\label{C_Q}
&\text{by Lemma \ref{quad_ests} }
\qquad
 \nm{u}_{2,\alpha,\beta}
  \leq cm^{-2}
  \Rightarrow
  \nm{Q_{m,\xi,u}}_{0,\alpha,\gamma}
  \leq
  m^{-3/4}
\quad \forall
u \in C^{2,\alpha}_{\apr_m}(\Sigma_{m,\xi})
\end{align}
(where we take $C$ in Lemma \ref{quad_ests}
to be the quantity $c$ in the present proof);
note that
$\varsigma_{m,\xi}$,
$\varsigma_{m,\xi}^{-1}$,
$P_{m,\xi}$,
$u_{m,\xi}^{(1)}$,
$\mu_{m,\xi}^{(1)}$,
and $Q_{m,\xi,u}$
are indeed all well-defined
under the above assumptions.

Now for each integer $m>m_1$ set
\begin{equation*}
D_m
\vcentcolon=
\left\{
  v \in C^{2,\alpha}_{\apr_m}(\Sigma_{m,0})
  \st
  \nm{v}_{2,\alpha,\beta} \leq m^{-5/2}
\right\}
\times
\IntervaL{-c,c}.
\end{equation*}
Then $D_m$ is a nonempty,
compact (by the Ascoli--Arzel\`{a} theorem),
convex
subset of $C^{2,\alpha/2}(\Sigma_{m,0}) \oplus \R$,
and,
by the preceding paragraph,
for $m$ large enough
$F_{m,\alpha/2}$ is defined on $D_m$
and
continuous under the
$\nm{\cdot}_{2,\alpha/2}+\abs{\cdot}$
norm.
If we can establish that
$F_{m,\alpha/2}(D_m) \subseteq D_m$
when $m$ is sufficiently large,
then the Schauder fixed point theorem
will be applicable and will ensure
that $F_{m,\alpha/2}(D_m)$
has a fixed point in $D_m$.
To that end suppose $(v,\xi) \in D_m$.
We then have
\begin{equation*}
\text{using \eqref{C_one} and \eqref{C_sigma}}
\qquad
\nm[\Big]{u_{m,\xi}^{(1)}+\varsigma_{m,\xi}^{-1*}v}_{2,\alpha,\beta} 
\leq
C_1m^{-2} + C_2 m^{-5/2}
\leq
cm^{-2},
\end{equation*}
assuming $m>m_2$, for some $m_2 \geq m_1$
sufficiently large in terms of $C_1=c/2$ and $C_2$.
We in turn obtain
\begin{equation}
\label{Q_bd}
\text{using \eqref{C_Q}}
\qquad
\nm[\Big]{Q_{m,\xi,u_{m,\xi}^{(1)}+\varsigma_{m,\xi}^{-1*}v}}_{0,\alpha,\gamma}
\leq
m^{-3/4}. 
\end{equation}
Defining $(u_{m,\xi,v},\mu_{m,\xi,v})$
as in \eqref{map_to_prove_main_theorem},
we next find
\begin{align*}
&\text{using \eqref{C_P} and \eqref{Q_bd}}
\qquad
\nm{u_{m,\xi,v}}_{2,\alpha,\beta}
  +\abs{\mu_{m,\xi,v}}
  \leq
  C_3m^{-11/4},
\\
&\text{using \eqref{C_sigma}}
\qquad
\nm{\varsigma_{m,\xi}^*u_{m,\xi,v}}_{2,\alpha,\beta}
  \leq
  C_2 C_3 m^{-11/4}
  \leq
  m^{-5/2},
\\
&\text{using \eqref{C_one}}
\qquad
\abs[\Big]{\mu_{m,\xi}^{(1)}+\mu_{m,\xi,v}}
  \leq
  C_1m^{-2} + C_3m^{-11/4}
  \leq
  cm^{-2},
\end{align*}
assuming $m>m_3$
for some $m_3 \geq m_2$
sufficiently large
in terms of $C_2$, $C_3$, and $C_1=c/2$.

Referring to the definition,
\eqref{map_to_prove_main_theorem},
of $F_{m,\alpha/2}$,
we have just verified that
$F_{m,\alpha/2}(D_m) \subseteq D_m$.
The proof is now completed by applying
the Schauder fixed point theorem,
as anticipated above,
and taking $m_0\vcentcolon=m_3$
and $C\vcentcolon=2c$.
\end{proof}

\begin{proof}[Proof of Theorem \ref{thm:ConstructFirst}]
Fix $\alpha \in \interval{0,1}$
and $\beta \in \interval{1/2,1}$.
Applying Lemma \ref{existence_of_fixed_point},
we have $m_0>0$ and $C>0$
such that for every integer $m>m_0$
there exist $\xi_m \in \R$
and $v_m \in C^{2,\alpha}_{\apr_m}(\Sigma_{m,0})$
such that the pair
$(v_m,\xi_m)$ is a fixed point of
$F_{m,\alpha/2}$,
given by \eqref{map_to_prove_main_theorem},
and satisfies estimate \eqref{eq:Est_Fixed_Point}.
We set
\begin{align*}
u_m^{(1)}&\vcentcolon=
u_{m,\xi_m}^{(1)},
&
u_m^{(2)}&\vcentcolon=
\varsigma_{m,\xi_m}^{-1*}v_m,
&
u_m&\vcentcolon=
u_m^{(1)}+u_m^{(2)},
\\
\mu_m^{(1)}&\vcentcolon=
\mu_{m,\xi_m}^{(1)},
&
\mu_m^{(2)}&\vcentcolon=
\mu_{m,\xi_m,v_m},
&
\mu_m&\vcentcolon=
\mu_m^{(1)}+\mu_m^{(2)}
\end{align*}
(for which we refer to \eqref{map_to_prove_main_theorem}
for the definition of $\mu_{m,\xi,v}$).
Then
\begin{align*}
H_{m,\xi_m,u_m}
&=
  H_{m,\xi_m,0}
  -J_{\Sigma_{m,\xi_m}}u_m
  +Q_{m,\xi_m,u_m}
\\[.5ex]
&=
  (\xi_m-m^2\mu_m^{(1)})\cokerval{\xi_m}
%\\&\hphantom{=(}
  -J_{\Sigma_{m,\xi_m}}u_m^{(2)}
  +Q_{m,\xi_m,u_m}
\\[.5ex]
&=
  (\xi_m-m^2\mu_m)\cokerval{\xi_m}
=0,
\end{align*}
where the first equality
is just definition \eqref{nonlinear_terms_def},
the second equality follows from
definition \eqref{sol_to_first_order},
and the third and fourth equalities
follow from definition \eqref{map_to_prove_main_theorem}
and the fact that $(v_m,\xi_m)$ is a fixed point.
It is also immediate from definitions
\eqref{sol_to_first_order}
and \eqref{map_to_prove_main_theorem}
that
for each $i=0,1$
the pair $(u_m^{(i)},\mu_m^{(i)})$
lies in the image of $P_{m,\xi}$,
so we also have
\begin{equation*}
B^{\mathrm{Robin}}_{\Sigma_{m,\xi_m}}u_m = 0.
\end{equation*}
Thus $u_m$ is a solution to
\eqref{system_to_prove_main_theorem}
with $\xi=\xi_m$.
It follows (using elliptic regularity in a by now standard fashion)
that in fact $u \in C^\infty_{\apr_m}(\Sigma_{m,\xi_m})$
and that the image
\begin{equation*}
\Sigma_{m-1}
\vcentcolon=
\iota_{m,\xi_m,u_m}(\Sigma_{m,\xi_m})
\end{equation*}
of \eqref{deformed_init_surf_inc}
is a 
free boundary minimal surface in $\B^3$
(though, a priori, not necessarily embedded),
as explained between
\eqref{deformed_MC_def}
and \eqref{system_to_prove_main_theorem}.
From Lemma \ref{first-order_ests}
and Lemma \ref{existence_of_fixed_point} (equation \eqref{eq:Est_Fixed_Point})
we obtain
\begin{equation}
\label{est_on_defining_function_of_final_surf}
\nm{u_m}_{2,\alpha,\beta}
  +m^{-2}\abs{\xi_m}
\leq
C_0m^{-2}
\end{equation}
for some $C_0 \geq C>0$ independent of $m$.

We claim that the remaining assertions
of Theorem \ref{thm:ConstructFirst}
now follow from
the properties
of the initial surfaces,
as summarized in
Proposition \ref{initsurfbasicprops},
and estimate
\eqref{est_on_defining_function_of_final_surf},
according to which
$\Sigma_{m-1}$ is a $C^{2,\alpha}$-small perturbation
of $\Sigma_{m,\xi_m}$.

In particular we can choose $m$ 
large enough
that \eqref{est_on_defining_function_of_final_surf}
and the properties
of the initial surface $\Sigma_{m,\xi_m}$
guarantee
embeddedness of $\Sigma_{m-1}$.
Indeed,
it is clear from the construction
of the initial surfaces
and from the embeddedness
of $\B^2$, $\K_0$, and $\tow$
that there exists $\delta>0$
such that at each point
$p \in m_k\Sigma_{m_k,\xi_{m_k}}$
we have embeddedness of
any smooth perturbation
of size less than $\delta$ (i.\,e. of any smooth normal graph whose defining function is $\delta$-small in $C^0$).
On the other hand, we know by \eqref{est_on_defining_function_of_final_surf} that the defining function $u_m$ of $\Sigma_{m-1}$ decays \emph{quadratically} with respect to $m$ and so, even magnifying by a factor $m$, the $\delta$-smallness condition will be satisfied for $m$ large enough, whence the embeddedness follows.

At that stage, having gained embeddedness, the fact that $\Sigma_{m-1}$
has genus $m-1$
and three boundary components
follows from
items \ref{initsurfgenus} and \ref{initsurfbdy}
of Proposition \ref{initsurfbasicprops}
(i.\,e. from the corresponding facts for the initial surface
$\Sigma_{m,\xi_m}$).
Lastly, that the symmetry group of $\Sigma_{m-1}$
contains $\apr_m$
is immediate from the corresponding property of the associated initial surface $\Sigma_{m,\xi_m}$, the $\apr_m$-equivariance of $u_m$
and the $\apr_m$-invariance of the auxiliary metric
and defining conformal factor.
Furthermore, that the symmetry group of $\Sigma_{m-1}$
is contained in $\apr_m$
then follows from
\eqref{est_on_defining_function_of_final_surf}
and
item \ref{initsurfsymgroup}
of Proposition \ref{initsurfbasicprops};
this can be established on rather abstract
and general grounds,
but instead of explaining the details of that approach
we offer a short ad hoc argument
that requires just the containment
$\Aut_{\B^3}(\Sigma_{m,\xi}) \geq \apr_m$
from item \ref{initsurfsymgroup}
of Proposition~\ref{initsurfbasicprops}.

Indeed, it is clear from the construction
(and in particular from the fact
that we have already established the containment
$\Aut_{\B^3}(\Sigma_{m-1}) \geq \apr_m$)
that
$\Sigma_{m-1}$ contains $m$ diameters of $\B^2$,
each contained in a line of reflectional symmetry
for $\Sigma_{m-1}$,
having equiangular intersections at the origin,
and one of which is $\axis{x} \cap \B^2$.
We call the collection of these diameters $L_m$,
and we will now show that,
for all sufficiently large $m$,
$\Sigma_{m-1}$ contains
no diameters outside $L_m$,
so that in particular
any element
of $\Aut_{\B^3}(\Sigma_{m-1})$
preserves $\bigcup L_m$.

To start, assuming that $m \geq 2$,
any symmetry of $\Sigma_{m-1}$
must take an element of $L_m$
to another diameter of $\B^2$
(given that the symmetry must preserve
the tangent plane to $\Sigma_{m-1}$ at the origin),
though a priori the image diameter
need not belong to $L_m$.
Since
$\apr_m$
contains all the reflections through
vertical planes bisecting the
intersection angles of $L_m$ and the cyclic subgroup of rotations (of angle a multiple of $2\pi/m$),
if $\Sigma_{m-1}$ were to contain
a diameter not in $L_m$,
then 
it would also have to contain a diameter
intersecting $\axis{x}$
at an angle in the interval
$\intervaL{0,\pi/2m}$.
Call the line containing this diameter $\ell_1$,
and call the intersection angle $\theta_0$.
Setting $\axis{x} \vcentcolon= \ell_0$
and inductively defining
$\ell_{j+1} \vcentcolon= \refl_{\ell_j}\ell_{j-1}$
for each integer $j \geq 1$,
it follows from
the reflection principle for minimal surfaces
(and induction on $j$)
that $\ell_j \cap \B^2 \subset \Sigma_{m-1}$
and $\refl_{\ell_j} \in \Aut_{\B^3}(\Sigma_{m-1})$
for all $j$,
and of course $\ell_j$ intersects $\axis{x}$
at angle $j\theta_0$.
Let $n$ be the least positive integer
such that $n\theta_0 \geq \pi/4m$.
Then $(n-1)\theta_0 < \pi/4m$
and, if $n>1$,
$\theta_0 < \pi/4m$,
so $n\theta_0 < \pi/2m$.
In all cases then
$\ell_n \subset \Sigma_{m-1}$
and $\ell_n$ meets $\axis{x}$
at some angle in the closed interval
$\IntervaL{\pi/(4m),\pi/(2m)}$.

On the other hand,
$m\Phi^{-1}(\Sigma_{m-1})$
is, on a neighborhood of $\axis{\theta}$,
a perturbation of $\tow^+$ which is of order
$m^{-1}$ in $C^0$ (here we are crucially appealing to \eqref{est_on_defining_function_of_final_surf}).
By item \ref{towcontainslines}
of Proposition \ref{karcher-scherk}
the spacing between the straight lines on
$\tow \cap \{\psi=0\}$ is $\pi$,
and we recall that
$\Phi$ takes straight lines
of constant $\theta$ in $\{\psi=0\}$
to straight lines in $\{z=0\}$ through the origin.
We conclude that in fact $\Sigma_{m-1}$
contains no diameters of $\B^2$
except those in $L_m$,
assuming of course that $m$ is sufficiently large
in terms depending on just the geometry of $\tow$.

Finally we assume
$\mathsf{S} \in \Aut_{\B^3}(\Sigma_{m-1})$
and we intend to show that in fact
$\mathsf{S} \in \apr_m$.
From the foregoing discussion we already know
that $\mathsf{S}$ preserves
$\bigcup L_m$.
Since $\apr_m$ acts transitively
on (the endpoints) $\Sp^1 \cap \bigcup L_m$
and is contained in
$\Aut_{\B^3}(\Sigma_{m-1})$,
we may without loss of generality assume
(by composing with elements of $\apr_m$ if necessary)
that $\mathsf{S}$ preserves each of the points
$(1,0,0) \in \axis{x}$ and $(0,0,1) \in \axis{z}$.
Then $\mathsf{S}$, which is an element of $\Ogroup(3)$ (cf. Section \ref{sec:Nota}), can be only the identity
(which belongs to $\apr_m$)
or instead the reflection through the plane $\{y=0\}$
(which does not).
Suppose that $\refl_{\{y=0\}}$
belongs to $\Aut_{\B^3}(\Sigma_{m-1})$.
Then $\refl_{\{z=0\}}$ does too,
since $\refl_{\axis{x}} \in \apr_m \leq \Aut_{\B^3}(\Sigma_{m-1})$ and $\refl_{\{y=0\}}\refl_{\{z=0\}}=\refl_{\axis{x}}$.
Because $\Sigma_{m-1}$ contains $\axis{x} \cap \B^2$,
a unit normal $\nu^{\vphantom{|}}_{\Sigma_{m-1}}$
to $\Sigma_{m-1}$ along $\axis{x} \cap \B^2$
cannot be contained in both $\{y=0\}$ and $\{z=0\}$.
Since, under the present assumptions,
reflection through each of these planes
is a symmetry of $\Sigma_{m-1}$,
$\nu^{\vphantom{|}}_{\Sigma_{m-1}}$
must everywhere along $\axis{x} \cap \B^2$
be orthogonal to one of these two planes.
In other words $\nu^{\vphantom{|}}_{\Sigma_{m-1}}$
is constant, regarded as a map to $\Sp^2$,
on $\axis{x} \cap \B^2$,
and this readily forces $\Sigma_{m-1}$ to be planar. 
Since however $\Sigma_{m-1}$ is, after all, not a totally geodesic disc (e.\,g. by virtue of its different topological type),
we conclude that in fact $\refl_{\{y=0\}}$
is not a symmetry of $\Sigma_{m-1}$,
and therefore
$\Aut_{\B^3}(\Sigma_{m-1})=\apr_m$,
ending the proof.
\end{proof}

\section{Desingularizing the union of two catenoidal annuli}\label{sec:OnlyCatsDesing}

As anticipated in the introduction, we will briefly describe here how to modify the construction above to obtain yet another novel family of free boundary minimal surfaces in the unit ball of $\R^3$. In particular, we will outline the proof of the following assertion:

\begin{theorem}\label{thm:ConstructSecond}
There exists a sequence $\{\Xi_n\}_{n\geq n_0}$ of  properly embedded, free boundary minimal surfaces in $\B^3$ such that:
\begin{enumerate}[label={\normalfont(\alph*)}]
    \item  $\Xi_n$ has genus zero, exactly $n+2$ boundary components and symmetry group coinciding with the prismatic group of order $4n$;
    \item as one lets $n\to\infty$ the surface $\Xi_n$ converges, in the sense of varifolds, to the union $\K_0 \cup -\K_0$; the convergence is smooth, with multiplicity one, away from the intersection $\K_0 \cap -\K_0$.
\end{enumerate}
\end{theorem}

The construction proceeds largely in parallel with that of the sequence $\{\Sigma_g\}$ which has been the focus of the core of this article. 
In the following discussion, for the sake of convenience we will recycle some notation from earlier sections,
in that certain objects in the construction of $\{\Xi_n\}$ are analogues of objects in the construction of $\{\Sigma_g\}$, 
and we will apply the same notation for such objects in view of the analogy.
For example, the six-ended Karcher--Scherk tower $\tow$ used in the construction of $\{\Sigma_g\}$
will play no role in the construction of $\{\Xi_n$\};
instead, in the current section $\tow$
will denote the four-ended Karcher--Scherk tower
we will use to desingularize the union $\K_0 \cup -\K_0$.

On the other hand, $\K_b$ continues to denote the same catenoidal annuli
constructed at the beginning of Section \ref{sec:Initial}
and $\omega_0$ the intersection
angle of $\K_0$ with $\{z=0\}$,
as per \eqref{eqn:defininiton_omega} and \eqref{eqn:values_a0_h0} in Remark \ref{rem:values_for_b=0};
in particular we recall that $0<\omega_0<\pi/4$.
We refer the reader to \cite{KarcherScherk}
and Section 2 of \cite{KapouleasEuclideanDesing}
as references for the four-ended Karcher--Scherk towers.
The following facts concerning our particular $\tow$
(and jointly analogous to Proposition \ref{karcher-scherk} and Remark \ref{towasymptotics})
are easily established.

\begin{figure}%
\includegraphics[width=\fbmsscale\textwidth,page=5]{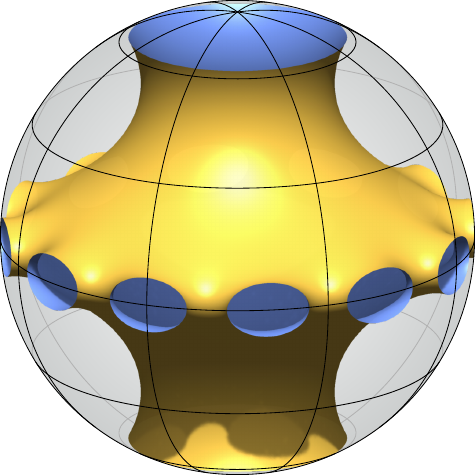}\hfill
\includegraphics[width=\fbmsscale\textwidth,page=6]{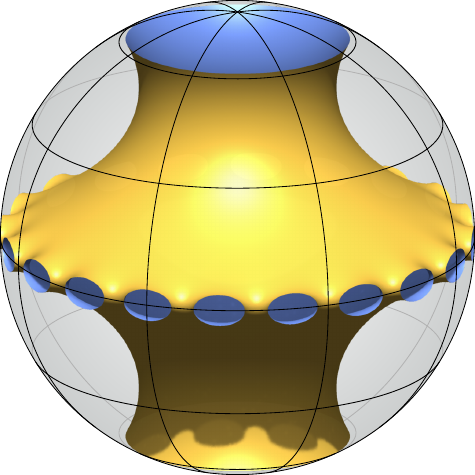}%
\caption{Visualization of $\Xi_n$ for $n=12$ and $n=20$.}%
\label{fig:genus0n12}%
\end{figure}

\begin{proposition}
[Desingularizing model]
\label{four-ended_tower}
Let $\R^3$ be endowed with a Cartesian coordinate system $(\sigma,\psi,\theta)$.

\begin{enumerate}[label={\normalfont(\roman*)}]
\item
There exist in $\R^3$
precisely two
complete embedded
minimal surfaces
having symmetry group
\begin{equation*}
\sk{
  \refl_{\{\sigma=0\}},
  \refl_{\{\psi=0\}},
  \refl_{\{\theta=0\}},
  \refl_{\{\theta=\pi\}}
},
\end{equation*}
genus $0$ quotient by
$
 \trans^{\axis{\theta}}_{2\pi}
 =
 \refl_{\{\theta=\pi\}}\refl_{\{\theta=0\}}
$,
and exactly four ends,
each asymptotic to a plane
that intersects $\{\psi=0\}$
at angle $\omega_0$
along a line parallel to $\axis{\theta}$.
These two surfaces are congruent
by means of the isometry $\trans^{\axis{\theta}}_\pi$;
we pick one such surface and call it $\tow$.

\item
Every symmetry of $\tow$ is even i.\,e. has positive sign as defined in \eqref{symsign}.

\item
Outside of a compact set, $\tow$ consists of four normal graphs over their respective
asymptotic planes and the corresponding defining functions converge to zero exponentially with rate one (together with their derivatives of all orders).

\item
There exists $b^{\mathrm{tow}}_{\omega_0}>0$
such that the unique end of $\tow$
in $\{\sigma \geq 0\} \cap \{\psi \geq 0\}$
is asymptotic to the plane
$
 \{
   \psi=b^{\mathrm{tow}}_{\omega_0}
   + \sigma \tan \omega_0
 \}
$.

\item 
For each integer $n \geq 1$
the surface
\begin{equation*}
\towquot_{(n)}
\vcentcolon=
\tow / \sk{\trans^{\axis{\theta}}_{2n\pi}}
\end{equation*}
has genus $n-1$ and $4$ ends,
and it meets
$\{\sigma=0\}/\sk{\trans^{\axis{\theta}}_{2\pi}}$
along $n$ smooth closed simple curves,
at a constant, right angle.
In particular the surface
\begin{equation*}
\tow^+ \vcentcolon= \tow \cap \{\sigma \geq 0\}
\end{equation*}
is a free boundary minimal surface
in $\{\sigma \geq 0\}$,
and the surface
\begin{equation*}
\towquot^+_{(n)}
\vcentcolon=
\tow^+ / \sk{\trans^{\axis{\theta}}_{2n\pi}}
\end{equation*}
has
genus $0$,
$2$ ends,
and $n$ boundary components.
\end{enumerate}
\end{proposition}

In analogy with the six-ended case
(that is: for the half tower with three ends)
we define the function
\begin{equation*}
\towcokergen
\in
C^\infty_{\Aut_{\R^3}(\tow^+)}(\tow^+)
\quad
\mbox{(cf. \eqref{towcokerandgen})}
\end{equation*} 
generating equivariant $\partial_\psi$
dislocations of the wings:
$\towcokergen$
is the scalar normal projection
(after a choice of global unit normal for $\tow^+$)
of the velocity
of a one-parameter family
of $\Aut_{\R^3}(\tow^+)$-equivariant
deformations of $\tow^+$
that fix the core identically
and translate the ends
in the $\psi$ direction.
In particular,
sufficiently far from
$\axis{\theta}$,
$\towcokergen$
agrees up to a sign
with the (scalar function associated to the) Jacobi field
induced by the ambient Killing field
$\partial_\psi$,
the signs being such that
$\towcokergen$
is $\Aut_{\R^3}(\tow^+)$-equivariant.

To avoid misunderstandings, we explicitly note that in the present section $\towcokergen$ has the same sign along the two wings (at corresponding points), while the opposite is actually true when considering either pair of ``outer'' wings of the six-ended Karcher--Scherk tower which we dealt with in Section \ref{sec:Initial}.

We then define
\begin{equation*}
\towcoker
\vcentcolon=
-J_{\tow^+}\towcokergen
\in
C^\infty_{\Aut_{\R^3}(\tow^+)}(\tow^+)
\end{equation*}
to be the correspondingly induced
linearized mean curvature,
$J_{\tow^+}$ being the Jacobi operator of $\tow^+$.

Since the Robin operator $B^{\mathrm{Robin}}_{\tow^+}$ for $\tow^+$ as a free boundary minimal surface in $\{\sigma \geq 0\}$ coincides with the Neumann conormal derivative operator,
we can (as we did for the six-ended tower in Section~\ref{sec:Linear}) reduce the $\Aut_{\R^3}(\tow^+)$-equivariant Robin boundary value problem for $J_{\tow^+}$ on the half tower $\tow^+$
to the problem of inverting $J_{\tow}$ with fully $\Aut_{\R^3}(\tow)$-equivariant data on the complete tower $\tow$.
Analogously to Lemma \ref{towker} we know
(in this specific case as a corollary of Proposition~2.8 in \cite{KapouleasEuclideanDesing})
that $\tow$ has no nontrivial
bounded and equivariant
Jacobi fields.
We then obtain the following
analogue of Proposition \ref{towsol}
(essentially a special case
of Proposition 6.1 
in \cite{KapouleasZouCloseToBdy},
which also applies the four-ended
towers in the free boundary setting).

\begin{proposition}
[Fredholm properties of the boundary value problem on $\tow^+$]
\label{two-ended_half_tow_sol}
For any $\alpha,\beta \in \interval{0,1}$
the map
\begin{align*}
C^{2,\alpha,\beta}_{\Aut_{\R^3}(\tow^+)}(\tow^+)
  &\to
  C^{0,\alpha,\beta}_{\Aut_{\R^3}(\tow^+)}(\tow^+)
    \oplus
    C^{1,\alpha}_{\Aut_{\R^3}(\tow^+)}(\partial \tow^+)
\\
u
  &\mapsto
  (J_{\tow^+}u, B^{\mathrm{Robin}}_{\tow^+}u)
\end{align*}
is Fredholm with Fredholm index $-1$,
and the pair
$(\towcoker,0)$
has nontrivial projection in the cokernel.
\end{proposition}

In analogy with the six-ended case
we define for each
$n \in \N^{\ast}$
and $\xi \in \R$
an $\Aut_{\R^3}(\tow^+)$-equivariant deformation
\begin{equation*}
\towbent^+_{n,\xi}
\quad
\mbox{(cf. \eqref{towbent})}
\end{equation*}
of $\tow^+$
that smoothly interpolates
between the identically fixed core
and the wings dislocated
(as described above)
by translation $\trans^{\axis{\psi}}_{\pm\xi/n}$
starting at distance
$O(1)$ from $\axis{\theta}$
and furthermore straightened
to exact half planes
at distance $O(n^{3/4})$
from $\axis{\theta}$.
We also define
(as in \eqref{bmxi})
\begin{align*}
\kappa_{n,\xi}
  &\vcentcolon=
  \frac{b^{\mathrm{tow}}_{\omega_0}}{n}
  +\frac{\xi}{n^2},
\\
b_{n,\xi}
  &\vcentcolon=
  \sin \kappa_{n,\xi},
\\
P^1_{n,\xi}
  &\vcentcolon=
  \{\psi = \kappa_{n,\xi} + \sigma \tan \omega_0\},
\end{align*}
so that in particular
the end of the
rescaled, deformed tower
$n^{-1}\towbent_{n,\xi}^+$
inside
$\{\sigma \geq 0\} \cap \{\psi \geq 0\}$
is asymptotic to $P^1_{n,\xi}$ (eventually coinciding with it).

Just as in the construction
of the initial surfaces $\Sigma_{m,\xi}$,
we can use
an arc-length parametrization
of a profile curve of $\K_{b_{n,\xi}}$
to identify a subset
of
$
 P^1_{n,\xi}
 /
 \sk{\trans^{\axis{\theta}}_{2\pi}}
$
with $\K_{b_{n,\xi}}$
outside an
$O(n^{-1})$ neighborhood of the equator $\Sp^1$.
In this way
we can transfer the defining function
of the wing of
$n^{-1}\towbent^+_{n,\xi}$
over $P^1_{n,\xi}$
to $\K_{b_n,\xi}$
and thereby obtain the corresponding
graph over $\K_{b_{n,\xi}}$
, away from $\Sp^1$.
By means of the reflection $\refl_{\{z=0\}}$
we likewise obtain a graph over
$-\K_{b_{n,\xi}}$,
away from the equator.
Referring to the symmetry group
of $\tow$ in Proposition \ref{four-ended_tower},
it is easy to see that
the
$\Aut_{\R^3}(\tow^+ \cap \{\psi>0\})$-equivariance
of the defining function,
the $\Ogroup(2)$-invariance
of $\K_{b_{n,\xi}}$,
and the application of the scale factor $n$
imply $\pri_n$-invariance
of the union of the above two graphs.
We also observe that
\eqref{Phintertwine}
implies that the map $\Phi$
from \eqref{PhiR3def}
intertwines $\Aut_{\R^3}(\tow^+)$
with $\pri_n$,
and we next apply $\Phi$
to an $O(1)$ neighborhood of $\axis{\theta}$
to wrap the core of
$n^{-1}\towbent_{n,\xi}^+$
around a neighborhood of $\Sp^1$.
Finally we construct a connected surface,
namely the initial surface
\label{Xi_initsurf_def}
\begin{equation*}
\gls{Xinxi}%\Xi_{n,\xi}
\quad
\text{(cf. \eqref{initsurfdef}),}
\end{equation*}
that smoothly and $\pri_n$-equivariantly
interpolates between this core
and each of the two graphs
over $\pm \K_{b_{n,\xi}}$.
As 
\glsuserii{m} %$n$
 here takes the role
that $m$ played
in the construction of $\Sigma_{m,\xi}$,
we assume $n$ is as large as needed
to ensure that $\Xi_{n,\xi}$
is a well-defined, smooth,
and embedded surface.

Parallel to definitions made for $\Sigma_{m,\xi}$
we define the corresponding regions
\begin{equation*}
\towr_{n,\xi}^1 \subset \towr_{n,\xi} \subset \Xi_{n,\xi},
\quad
\catr_{n,\xi} \subset \Xi_{n,\xi}
\quad
\mbox{(cf. \eqref{eqn:definition_regions} and \eqref{towr1})}
\end{equation*}
so that, essentially,
$\towr_{n,\xi}$
is
the intersection with $\Xi_{n,\xi}$
of an $O(n^{-1/2})$
neighborhood of the equatorial circle $\Sp^1$,
$\towr_{n,\xi}^1$
is a marginally smaller subdomain of $\towr_{n,\xi}$
in $\Xi_{n,\xi}$,
and
$\catr_{n,\xi}$
is the intersection with
$\Xi_{n,\xi} \cap \{z>0\}$
of the complement of an $O(n^{-3/4})$
neighborhood of $\Sp^1$.
(The disc region
$\discr_{n,\xi} \subset \Sigma_{m,\xi}$
of course finds no analogue in $\Xi_{n,\xi}$.)

We also define the maps (cf. \eqref{regionalprojections})
\begin{align*}
\varpi_{\towr_{n,\xi}}
  \colon
  M_{n,\xi}
  &\to
  \towquot^+_{(n)},
&
\varpi_{\catr_{n,\xi}}
  \colon
  K_{n,\xi}
  &\to
  \K_{b_{n,\xi}}
\end{align*}
exactly as we did in Section \ref{sec:Initial} (so that, in particular, the map
$\varpi_{\catr_{n,\xi}}$
is simply the nearest-point projection in $\R^3$); note that 
the maps $\varpi_{\towr_{n,\xi}}$,
$\varpi_{\catr_{n,\xi}}$
are diffeomorphisms onto their images. 
We use them to define,
for any $k \in \N^{\ast}=\left\{1,2,3,\ldots \right\}$
and $\alpha,\beta \in \interval{0,1}$,
the norms
\begin{equation*}
\nm{{}\cdot{}}_{k,\alpha,\beta}
\quad
\mbox{(cf. \eqref{globalnorms})}
\end{equation*}
on $\pri_n$-equivariant
(equivalently: $\pri_n$-invariant, since all elements of $\pri_n$ have sign $+1$)
functions on $\Xi_{n,\xi}$
just as in \eqref{globalnorms},
but with the term corresponding
to the disc omitted
(and with $m$ replaced by $n$).
We also define
on $\Xi_{n,\xi}$
the smooth, compactly supported,
$\pri_n$-equivariant
function
\begin{equation*}
H_{\Xi_{n,\xi}}^{\mathrm{dislocate}}
\vcentcolon=
\varpi_{\towr_{n,\xi}}^*\towcoker
\end{equation*}
extended to be constantly zero
on $\Xi_{n,\xi} \setminus \towr_{n,\xi}$.
(In fact, we slightly abuse notation
in the above definition
in that by $\towcoker$
we really mean the unique function
on $\towquot^+_{(n)}$ whose
pullback under the canonical
projection $\tow^+ \to \towquot^+_{(n)}$
is $\towcoker$.)

With the foregoing definitions
it is straightforward to verify
the following analogue
of Proposition~\ref{initsurfbasicprops}
and Proposition~\ref{initsurfcomp}
(jointly)
by mirroring the proofs
of the latter two results,
in particular
using Proposition~\ref{four-ended_tower}
in place of Proposition~\ref{karcher-scherk}
and Remark~\ref{towasymptotics}.

\begin{proposition}
[Initial surfaces]
\label{genus-zero_init_surfs}
For each $c>0$
there exists $n_0>0$
such that for every $\xi \in \IntervaL{-c,c}$
and every integer $n>n_0$
the initial surface 
\gls{Xinxi} %$\Xi_{n,\xi}$
has the following properties.
\begin{enumerate}[label={\normalfont(\roman*)}]
\item 
$\Xi_{n,\xi}$ is a connected
smooth surface with boundary
and is properly embedded in $\B^3$.

\item
$\Xi_{n,\xi}$ has genus $0$
and $n+2$ boundary components.

\item
$\Xi_{n,\xi}$ meets $\partial \B^3$ orthogonally.

\item
$\Xi_{n,\xi}$ has symmetry group
$\Aut_{\B^3}(\Xi_{n,\xi})=\pri_n$,
and every symmetry is even
(namely: it has positive sign in the sense
of definition \eqref{symsign}).

\item
For any $\alpha,\beta \in \interval{0,1}$
the mean curvature $H_{\Xi_{n,\xi}}$
of $\Xi_{n,\xi}$ satisfies the estimate
\begin{equation*}
\nm{
  H_{\Xi_{n,\xi}}
  -\xi H^{\mathrm{dislocate}}_{\Xi_{n,\xi}}
}_{0,\alpha,\beta}
\leq
\frac{C}{1-\beta}
\end{equation*}
for some $C>0$
independent of $n$, $c$, $\xi$ and $\alpha,\beta$.
\end{enumerate}
Furthermore,
for each integer $k \geq 0$,
as $n \to \infty$
the region $\catr_{n,\xi}$
converges in $C^k$ to $\K_0$
and for any $p \in \Sp^1$
the translated and rescaled region
$n(\towr_{n,\xi}-p)$
converges in $C^k$
on compact subsets
to $\tow^+$ (modulo ambient isometry),
with the convergence
in both cases
uniform
in $\xi \in \IntervaL{-c,c}$.
\end{proposition}

Exploiting the above convergence
and Proposition \ref{two-ended_half_tow_sol}
in place of Proposition \ref{towsol}
(and discarding the disc)
we next obtain the following analogue
of Proposition \ref{globalsol}.
For the continuity assertion
we must first choose
diffeomorphisms
$\varsigma_{n,\xi}\colon\Xi_{n,0} \to \Xi_{n,\xi}$
exactly as in
\eqref{Sigma_0_to_Sigma_xi}
and
\eqref{weighted_ests_for_Sigma_0_to_Sigma_xi}
(but with each $\Sigma$ replaced by $\Xi$,
each $m$ replaced by $n$,
and $\apr_m$ replaced by $\pri_n$).

\begin{proposition}
[Solutions on the initial surface modulo
approximate cokernel]
Assume $0<\alpha<1$, $0<\beta<\gamma<1$, and $c>0$.
There exists $n_0>0$ such that
for any integer $n>n_0$ and any $\xi \in \IntervaL{-c,c}$
there is a linear map
\begin{equation*}
\glsuserii{initialresol} %P_{\Xi_{n,\xi}}
\colon
C^{0,\alpha}_{\pri_n}(\Xi_{n,\xi})
\to
C^{2,\alpha}_{\pri_n}(\Xi_{n,\xi}) \oplus \R
\end{equation*}
such that
if $E \in C^{0,\alpha}_{\pri_n}(\Xi_{n,\xi})$ 
and $(u,\mu)=P_{\Xi_{n,\xi}}E$,
then
\begin{enumerate}[label={\normalfont(\roman*)}]
\item
\(
\nm{u}_{2,\alpha,\beta}+\abs{\mu}
  \leq 
  C\nm{E}_{0,\alpha,\gamma}
\)
for some constant $C>0$ independent of
$c$, $n$, $n_0$, $\xi$, and $E$;
\item
\(\displaystyle
\left\{
\begin{aligned}
n^{-2}J_{\Xi_{n,\xi}}u
  &=
  E
    +\mu H^{\mathrm{dislocate}}_{\Xi_{n,\xi}}
&&\text{ in }\Xi_{n,\xi},
\\
B_{\Xi_{n,\xi}}^{\mathrm{Robin}}u
  &=0
&&\text{ on }
\partial\Xi_{n,\xi};
\end{aligned}\right.
\)
\item
the map
\begin{align*}
  \R
  \oplus C^{0,\alpha}_{\pri_n}(\Xi_{n,0})
&\to 
C^{2,\alpha}_{\pri_n}(\Xi_{n,0})
  \oplus \R
\\
(
  \xi,
  E_0
)
&\mapsto
(
  \varsigma_{n,\xi}^*u_\xi,
  \mu_\xi
),
\end{align*}
where $(u_\xi,\mu_\xi)\vcentcolon=P_{\Xi_{n,\xi}}\varsigma_{n,\xi}^{-1*}E_0$, is continuous.
\end{enumerate}
\end{proposition}

All of the machinery of Section \ref{sec:Nonlin}
now carries over with only notational changes
to produce a solution to the nonlinear problem,
namely for any $n$ sufficiently large a parameter $\xi_n$
and a free boundary minimal graph
(with respect to the auxiliary metric $h$)
over the initial surface $\Xi_{n,\xi_n}$.
Just as in Section \ref{sec:Nonlin}
we then obtain bounds on the parameters $\xi_n$
and the defining functions for the graphs,
which, in conjunction with
Proposition \ref{genus-zero_init_surfs},
allow to complete the proof of
Theorem \ref{thm:ConstructSecond}.

\section{A Morse index bound and related conjectures}\label{sec:Index}

In this final section, we will look back at the two families of free boundary minimal surfaces we constructed and discuss some of their finer geometric properties. In particular, we shall be concerned with the study of their Morse index, whose definition has been recalled above in Section \ref{sec:Nota}. However, we will firstly focus on their \emph{equivariant} Morse index instead, which -- as already appeared in \cite{CarlottoFranzSchulz20} and \cite{FranzIndex} -- may sometimes be simpler to compute (or estimate), and yet provides enough significant information for some purposes. 

To that aim, let us start by recalling the relevant notion, in the special case of our interest (where, among other things, all properly embedded surfaces are automatically two-sided); the reader is referred to Section 3 of \cite{FranzIndex} for further details. 
Given a finite group $G$ of isometries of $\B^3$, and a $G$-equivariant free boundary minimal surface $\Sigma$ therein, we shall define its $G$-equivariant Morse index as the maximal dimension of a linear subspace of $C^{\infty}_G(\Sigma)$ where the standard Jacobi form $Q_{\Sigma}(\cdot,\cdot)$ is negative definite; 
consistently with the notation we employed throughout the paper, here $C^{\infty}_G(\Sigma)$ denotes the space of $G$-equivariant smooth functions on $\Sigma$, i.\,e. those smooth functions $u$ satisfying the identity $u \circ \mathsf{M}=(\sgn_\Sigma \mathsf{M})u$ for all $\mathsf{M} \in G$. 
By appealing to a suitable equivariant counterpart of the standard spectral theorem for compact self-adjoint operators one can equivalently characterize the $G$-equivariant Morse index by looking at the number of negative eigenvalues of the elliptic problem \eqref{eq:RobinEingenvProblem} on $L^2_G(\Sigma)$. 

Our first result in this section ensures, as we had mentioned in the introduction of the present paper, that the free boundary minimal surfaces constructed in Theorem \ref{thm:ConstructFirst} cannot possibly be obtained by means of a one-parameter min-max scheme, and thus (in some sense) exhibit some higher complexity than the family of surfaces constructed by Ketover in \cite{KetoverFB} (see also Appendix~\ref{app:convergence_large_g} below), and conjecturally of the Kapouleas--Li surfaces constructed in \cite{KapouleasLiDiscCCdesing}.

\begin{proposition}
\label{prop:DistiguishByEquivIndex} 
If $g\in\N$ is sufficiently large, then the surface $\Sigma_g^{\mathrm{CSW}}$ constructed in Theorem \ref{thm:ConstructFirst} has $\apr_{g+1}$-equivariant index greater than or equal to $2$ while the surface $\Sigma_g^{\mathrm{Ket}}$ from \cite{KetoverFB} has $\apr_{g+1}$-equivariant index equal to $1$.
\end{proposition}

The second clause, namely the fact that each surface $\Sigma_g^{\mathrm{Ket}}$ from \cite{KetoverFB} has equivariant index equal to~$1$ has been obtained in \cite{FranzIndex} (in fact it follows as a basic special case of the main theorem there), so we shall rather be concerned with the proof of the first clause instead. In turn, that follows from combining the next few lemmata. 

\begin{lemma}
\label{lem:negative-direction_K}
Let $\K_0$ be the catenoidal annulus constructed in Lemma \ref{lem:orthogonal-0} (equivalently: in Lemma \ref{lem:catenoid-K_b} for $b=0$). 
There exist $\varepsilon>0$ and a smooth, $\Aut_{\B^3}(\K_0)$-equivariant function $v_1$ on $\K_0$ with support in $\K_0\cap\{z\geq\varepsilon\}$ such that $Q_{\K_0}(v_1,v_1)<0$. 
\end{lemma}%

\begin{proof}
The height function $u=z$ is harmonic on $\K_0$, and satisfies the Robin boundary condition 
%$\eta_{\K_0}\cdot\nabla_{\K_0}u= u$ 
$B^{\mathrm{Robin}}_{\K_0}u=0$ 
as defined in \eqref{Robin_op_general_def} 
on the upper boundary circle of $\K_0$ as well as the Dirichlet boundary condition $u=0$ along the equatorial boundary component. 
In particular, the corresponding index form defined in \eqref{eq:JacobiQuadraticForm} reads  
\begin{align*}
Q_{\K_0}(u,u)&=\int_{\K_0}\bigl(-u\Delta_{\K_0} u-\abs{A_{\K_0}}^2u^2\bigr)+\int_{\partial\K_0}
\bigl(u B^{\mathrm{Robin}}_{\K_0}u\bigr)
=-\int_{\K_0}\abs{A_{\K_0}}^2u^2.
\end{align*}%
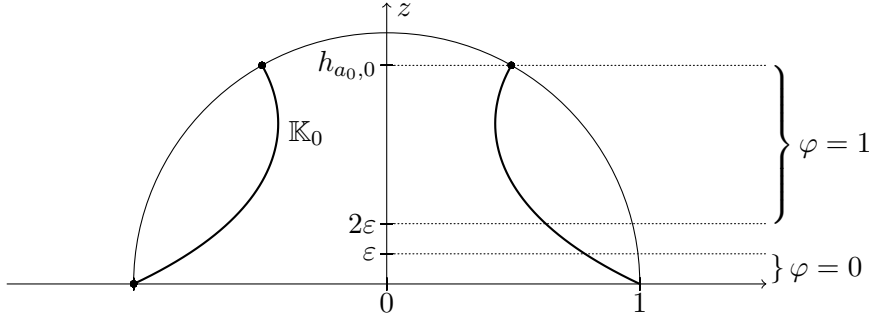
\begin{figure}%
\centering
\pgfmathsetmacro{\b}{0}
\pgfmathsetmacro{\a}{2.3328}
\pgfmathsetmacro{\h}{0.8703}
\pgfmathsetmacro{\eps}{0.12}
\begin{tikzpicture}[line cap=round,line join=round,baseline={(0,0)},scale=\unitscale,remember picture]
\draw[thick,domain=\b:\h,variable=\z,samples=50]plot({cosh(\a*\z-\a*\b-acosh(\a*sqrt(1-\b*\b)))/\a},{\z});
\draw[thick,domain=\b:\h,variable=\z,samples=50]plot({-cosh(\a*\z-\a*\b-acosh(\a*sqrt(1-\b*\b)))/\a},{\z});
\draw(-1/\a,0.6)node[right]{$\K_0$};
\draw plot[hdash](0,\h)node[left]{$h_{a_0,0}$};
\draw plot[bullet]({sqrt(1-\h*\h)},\h);
\draw plot[bullet]({-sqrt(1-\b*\b)},\b);
\draw plot[bullet]({-sqrt(1-\h*\h)},\h);
\draw[->](-1.5,0)--(1.5,0)coordinate(0-brace) ;
\draw[->](0,-0)--(0,1.12)node[below right,inner sep=0]{$~z$};
\draw plot[vdash](0,0)node[below]{$0$};
\draw plot[vdash](1,0)node[below]{$1$};
\draw plot[vdash](-1,0);
\draw(1,0)arc(0:180:1); 
\draw plot[hdash](0,\eps)node[left]{$\varepsilon$};
\draw plot[hdash](0,2*\eps)node[left]{$2\varepsilon$};
\draw[densely dotted]
(0,\eps)--++(1.5,0)coordinate(eps-brace)
(0,2*\eps)--++(1.5,0)coordinate(2eps-brace)
(0,\h)--++(1.5,0)coordinate(h-brace);
\draw(1.5,\eps/2)node[right,inner sep=0]{$\left.\tikz[remember picture]{\path(0-brace)--(eps-brace);}\right\}\varphi=0$};
\draw(1.5,\h/2+\eps)node[right,inner sep=0]{$\left.\tikz[remember picture]{\path(2eps-brace)--(h-brace);}\right\}\varphi=1$};
\end{tikzpicture}
\caption{Defining a cutoff function $\varphi$ on $\K_0$.}%
\label{fig:cutoff_on_K0}%
\end{figure}%
Given $\varepsilon>0$, let $\varphi_\varepsilon$ be a cutoff function on $\K_0$ which depends only on the height $z$, such that $\varphi_\varepsilon=0$ for $z\in[0,\varepsilon]$ and $\varphi_\varepsilon=1$ for $z\in[2\varepsilon,h_{a_0,0}]$ and such that $\varphi_\varepsilon$ is increasing in $z\in[\varepsilon,2\varepsilon]$ (cf. Figure~\ref{fig:cutoff_on_K0}). 
Then, $\sk{\nabla_{\K_0} u,\nabla_{\K_0}\varphi_\varepsilon}\geq0$ and  
\begin{align*}
\int_{\K_0}\bigl(-\varphi_\varepsilon u\Delta_{\K_0}(\varphi_\varepsilon u)\bigr)
=\int_{\K_0}\Bigl(-2\varphi_\varepsilon u\sk{\nabla_{\K_0} u,\nabla_{\K_0}\varphi_\varepsilon}-u^2\varphi_\varepsilon \Delta_{\K_0}\varphi_\varepsilon\Bigr) 
&\leq\int_{\K_0\cap\{z\in[\varepsilon,2\varepsilon]\}}u^2\abs{\Delta_{\K_0}\varphi_\varepsilon}. 
\end{align*}
Choosing $\varphi_\varepsilon$ such that $\abs{\Delta_{\K_0}\varphi_\varepsilon}\leq C\varepsilon^{-2}$ for some constant $C<\infty$ depending only on the geometry of $\K_0$, we obtain 
$u^2\abs{\Delta_{\K_0}\varphi_\varepsilon}\leq4C$ in $\K_0\cap\{z\in[\varepsilon,2\varepsilon]\}$. 
Moreover, a straightforward application of the coarea formula ensures that for all sufficiently small $\varepsilon>0$ the area of $\K_0\cap\{z\in[\varepsilon,2\varepsilon]\}$ is bounded from above by $2\pi \varepsilon/\sin(\omega_0)$, where $\omega_0$ is the contact angle defined in \eqref{eqn:defininiton_omega} (specified to $b=0$, cf. Remark \ref{rem:values_for_b=0}).
Therefore, 
\begin{align*}
Q_{\K_0}(\varphi_\varepsilon u,\varphi_\varepsilon u)
&\leq\int_{\K_0\cap\{z\in[\varepsilon,2\varepsilon]\}}u^2\abs{\Delta_{\K_0}\varphi_\varepsilon}
-\int_{\K_0}\abs{A_{\K_0}}^2\varphi_\varepsilon^2u^2
\leq \left(\frac{8\pi C}{\sin(\omega_0)}\right)\varepsilon
-\int_{\K_0}\abs{A_{\K_0}}^2\varphi_\varepsilon^2u^2.
\end{align*}
In particular, we obtain $Q_{\K_0}(\varphi_\varepsilon u,\varphi_\varepsilon u)<0$ provided that $\varepsilon>0$ is chosen sufficiently small. 
Setting $v_1=\varphi_\varepsilon u$ completes the proof. 
\end{proof}

To move on, we recall that $\towquot=\towquot_{(1)}$ denotes the quotient of the six-ended Karcher--Scherk tower with respect to vertical translations (generated by that of length $2\pi$) as defined by \eqref{towquotient}.

\begin{lemma}
\label{lem:negative-direction_M}
There exist $R>0$ and a smooth, $\Aut(\towquot)$-equivariant function $\tilde{v}_2$ on $\towquot$ with compact support in $\towquot\cap\axis{z}_{\leq 2R}$ such that $Q_{\towquot}(\tilde{v}_2,\tilde{v}_2)<0$. 
\end{lemma}

\begin{proof}
These assertions actually follow at once from our discussion in the first part of Section \ref{subs:LinTow}, as we explained in Remark \ref{rem:NegativeDirection}. 
\end{proof}

\begin{remark}
\label{rem:negative_direction_standard_Scherk}
For later reference, we also explicitly note that the very same conclusion as in Lemma~\ref{lem:negative-direction_M} holds when $\towquot$ denotes the period quotient of any standard Scherk tower ($k=2$). In that case, the constant function $1$ does satisfy the equivariance constraints and so we just need to multiply it by a cutoff function and (exactly as in Remark \ref{rem:NegativeDirection}) note that metric annuli on $\towquot$ between radii $R$ and $2R$ along the wings have \emph{linearly} growing area.
\end{remark}

We are now in the position to proceed with the proof of Proposition \ref{prop:DistiguishByEquivIndex}. 
Basically, we need to show one can effectively ``transplant'' the function $\tilde{v}_2$ in the previous lemma from the model block $\towquot$ to the initial surfaces $\Sigma_{m,\xi}$ and then to the actual minimal surfaces we constructed, so as to get a negative direction for the Jacobi forms of such surfaces. 
Roughly speaking, this conclusion relies on the fact that the map $\Phi$ defined in \eqref{PhiR3def} is close to the identity near the equatorial circle of $\Sp^2=\partial\B^3$. 

\begin{proof}[Proof of Proposition \ref{prop:DistiguishByEquivIndex}]
As already mentioned after the statement
of the proposition,
the $\apr_{g+1}$-equivariant index of $\Sigma_g^{\mathrm{Ket}}$
has been computed in \cite{FranzIndex},
so we turn to $\Sigma_g^{\mathrm{CSW}}$.
In the notation of \eqref{deformed_init_surf_inc}, the latter surface is by construction the graph
\begin{equation*}
\Sigma_g^{\mathrm{CSW}}
=
\iota_{g+1,\xi_{g+1},u_{g+1}}
  (\Sigma_{g+1,\xi_{g+1}})
\end{equation*}
with respect to the auxiliary metric
$h$ (defined in \eqref{auxmet})
of some function
$u_{g+1}\colon\Sigma_{g+1,\xi_{g+1}}\to\R$
over the initial surface
$\Sigma_{g+1,\xi_{g+1}}$
(defined in \eqref{initsurfdef})
for some $\xi_{g+1} \in \R$,
for which function and parameter
we have the estimate
\begin{equation}
\label{est_on_u_and_xi_for_index}
(g+1)^2\nm{u_{g+1}}_{2,0,1/2}
  + \abs{\xi_{g+1}}
\leq
C
\end{equation}
for some $C>0$ independent of $g$
(this is indeed the content of equation
\eqref{est_on_defining_function_of_final_surf}
in the proof of Theorem \ref{thm:ConstructFirst}).

Recall the quantity
$b_{m,\xi}$ from \eqref{bmxi}.
It follows from Lemma \ref{lem:catenoid-K_b},
Lemma \ref{lem:negative-direction_K}
and the estimate above for $\xi_{g+1}$
that there exist
$\epsilon>0$,
$\delta(\epsilon)>0$,
$C(\epsilon)>0$,
and, for all sufficiently large $g$,
functions
$v^{\K}_g \in C^\infty_{\Ogroup(2)}(\K_{b_{g+1,\xi_{g+1}}})$ such that 
\begin{align*}
Q_{\K_{b_{g+1,\xi_{g+1}}}}\bigl(v^\K_g,v^\K_g\bigr)
&<-2\delta(\epsilon), 
&
\nm{v^\K_g}_{C^2} &\leq C(\epsilon)
\end{align*}
and $v^\K_g$ has support contained in $\{z \geq \epsilon\}$. 
For each $n \in \N^{\ast}$
recall also the canonical projection
$
 \varpi_{(n)}
 \colon
 \R^3
 \to
 \R^3/\sk{\trans^{\axis{z}}_{2n\pi}}
$
and (as in Section \ref{sec:Initial})
let
$
 \pi_{(n)}
 \colon
 \R^3/\sk{\trans^{\axis{z}}_{2n\pi}}
 \to
 \R^3/\sk{\trans^{\axis{z}}_{2\pi}}
$
be the unique map
such that
$\varpi_{(1)}=\pi_{(n)} \circ \varpi_{(n)}$.
For each integer $g \geq 0$
we define the functions
\begin{equation*}
v_g^\tow
\vcentcolon=
\pi_{(g+1)}|_{\towquot_{(g+1)}}^*\tilde{v}_2,
\end{equation*}
with $\tilde{v}_2$ as in the statement of
Lemma \ref{lem:negative-direction_M}.
We then have
\begin{align*}
\varpi_{(g+1)}|_{\tow^+}^* v^\tow_g
&\in C^\infty_{\Aut_{\R^3}(\tow^+)}(\tow^+),
&
Q_{\towquot^+_{(g+1)}}(v^\tow_g,v^\tow_g)
&<-2\delta(R),
&
\nm{v^\tow_g}_{C^2} 
&\leq C(R)
\end{align*}
and $v^\tow_g$ has support contained in $\varpi_{(g+1)}(\axis{z}_{\leq 2R})$, 
for $R$ as in the statement of
Lemma \ref{lem:negative-direction_M},
some $\delta(R)$, $C(R)>0$
independent of $g$,
and
\begin{equation*}
Q_{\towquot^+_{(g+1)}}(u,u)
\vcentcolon=
\int_{\towquot^+_{(g+1)}}
  \Bigl(
    \abs[\big]{\nabla_{\towquot^+_{(g+1)}}u}^2
    -\abs[\big]{A_{\towquot^+_{(g+1)}}}^2u^2
  \Bigr).
\end{equation*}
(As we mentioned in Section \ref{sec:Nota}, since $\towquot^+_{(g+1)}$
is a free boundary minimal surface
in the quotiented half space
$
 \{x \geq 0\}
 /
 \sk{\trans^{\axis{z}}_{2(g+1)\pi}}
$,
whose boundary is totally geodesic,
this is the standard
quadratic form corresponding
to the second variation of area
of $\towquot^+_{(g+1)}$
through surfaces with boundary
on $\{x=0\}$.)

Recalling also the maps
$\varpi_{\catr_{m,\xi}}, \varpi_{\towr_{m,\xi}}$
defined in \eqref{regionalprojections},
by taking $g$ sufficiently large
in terms of $\epsilon$ and $R$
we can ensure that
the support of $v^\K_g$
is contained in the image of
$\varpi_{\catr_{g+1,\xi_{g+1}}}$
and the support of $v^\tow_g$
in the image of
$\varpi_{\towr_{g+1,\xi_{g+1}}}$.
Thus there are unique functions 
$v^\catr_g, v^\towr_g
  \in
  C^\infty_{\apr_{g+1}}(\Sigma_{g+1,\xi_{g+1}})$ such that
\begin{align*}
v^\catr_g|_{\catr_{g+1,\xi_{g+1}}}
&=\varpi_{\catr_{g+1,\xi_{g+1}}}^*v^\K_g, 
&
v^\catr_g|_{
    \Sigma_{g+1,\xi_{g+1}}
    \setminus
    \apr_{g+1}(\catr_{g+1,\xi_{g+1}})
    }&=0,
\\
v^\towr_g|_{\towr_{g+1,\xi_{g+1}}}
&=\varpi_{\towr_{g+1,\xi_{g+1}}}^*v^\tow_g,
&
v^\towr_g|_{
    \Sigma_{g+1,\xi_{g+1}}
    \setminus
    \apr_{g+1}(\towr_{g+1,\xi_{g+1}})
    }
&=0.
\end{align*}
Moreover, possibly taking $g$ even larger, $v^\catr_g$ is supported outside an $\epsilon/2$ neighborhood of $\Sp^1$, 
while $v^\towr_g$ is supported inside an $4R/g$ neighborhood of $\Sp^1$.
In particular $v^\catr_g$ and $v^\towr_g$ have disjoint supports for sufficiently large $g$;
since each of these supports is nonempty, the set $\{v^\catr_g,v^\towr_g\}$ is then linearly independent.

Even though the initial surface
$\Sigma_{g+1,\xi_{g+1}}$
is not minimal,
for the purposes of this proof
we define the quadratic form
$Q_{\Sigma_{g+1,\xi_{g+1}}}$
by \eqref{eq:JacobiQuadraticForm}.
By taking $g$ large enough
we can then ensure, by continuity, that
\begin{equation*}
Q_{\Sigma_{g+1,\xi_{g+1}}}\bigl(v^\catr_g,v^\catr_g\bigr)
  <
  -\delta(\epsilon)
\quad \mbox{and} \quad
Q_{\Sigma_{g+1,\xi_{g+1}}}\bigl(v^\towr_g,v^\towr_g\bigr)
  <
  -\delta(R),
\end{equation*}
where we emphasize
that the earlier introduced
strictly positive constants
$\delta(\epsilon)$
and $\delta(R)$
do not depend on $g$.
Specifically, for the first inequality
we use
items
\ref{MetCompDiscAndCat},
\ref{JacOpCompCat},
and \ref{BdyOpCompCat}
of Proposition \ref{initsurfcomp}
and we take $g$ large in terms
of the implicated universal constants,
including in particular
the area of $\K_0$
and the length of $\partial \K_0$.
Similarly
we apply
items \ref{MetCompTow},
\ref{JacOpCompTow},
and \ref{BdyOpCompTow}
of Proposition \ref{initsurfcomp},
and we also use the facts
that the Jacobi form
is invariant under
scaling of the ambient metric
and that
\begin{equation*}
Q_{\towquot^+_{(g+1)}}\bigl(v^\tow_g,v^\tow_g\bigr)
=
(g+1)Q_{\towquot}\bigl(\tilde{v}_2,\tilde{v}_2\bigr),
\end{equation*}
in order to accommodate the second inequality by taking $g$ sufficiently large in terms of universal constants,
including in particular $R$ and the length of $\partial\towquot$. 

\pagebreak[1]

Finally we claim that the functions $v^\catr_g$ and $v^\towr_g$ remain negative directions on $\Sigma_g$
(seen as a normal graph over $\Sigma_{g+1,\xi_{g+1}}$, in our usual sense) 
for $g$ sufficiently large, which, in view of the linear independence observed above, will complete the proof.
Indeed, it is clear from the construction of the initial surfaces
and from the uniform bound on $\xi_{g+1}$
in \eqref{est_on_u_and_xi_for_index}
that for each $k$ the $C^k$ norms of the second fundamental forms
of $\catr_{g+1,\xi_{g+1}}$ and $(g+1)\towr_{g+1,\xi_{g+1}}$
are uniformly bounded in $g$.
Of course each $C^k$ norm,
with respect to the ambient Euclidean metric,
of $h$ (the auxiliary metric we employed) and $h^{-1}$
are likewise uniformly bounded in $g$.
The bound in
\eqref{est_on_u_and_xi_for_index}
on the defining function $u_{g+1}$
then implies
that the induced metrics
and second fundamental forms
of $\catr_{g+1,\xi_{g+1}}$
and
$\iota_{g+1,\xi_{g+1},u_{g+1}}(\catr_{g+1,\xi_{g+1}})$
are $O((g+1)^{-2})$-close
in $C^1$ and $C^0$ respectively,
and likewise
the induced metrics
and second fundamental forms
of $(g+1)\towr_{g+1,\xi_{g+1}}$
and
$(g+1)\iota_{g+1,\xi_{g+1},u_{g+1}}(\towr_{g+1,\xi_{g+1}})$
are $O((g+1)^{-1})$-close
in $C^1$ and $C^0$ respectively,
all assuming large enough $g$.
The conclusion now follows
just as for the (regionwise) comparison above
between the index forms on the initial surfaces
and the model surfaces.
\end{proof}

\begin{remark}
\label{rem:equivariant-index-genus0}
By using the conclusion of Remark \ref{rem:negative_direction_standard_Scherk}
in place of Lemma \ref{lem:negative-direction_M}
the above proof is easily adapted
(with essentially just notational modifications)
to establish that
each free boundary minimal surface $\Xi_n$ of Theorem \ref{thm:ConstructSecond},
with $n$ sufficiently large,
also has equivariant index at least $2$.
\end{remark}

In general, the task of determining the Morse index of a minimal surface is a remarkably delicate one, and especially so in the free boundary case. 
To the best of our knowledge, within the class of free boundary minimal surfaces in $\B^3$ the value of the index has only been computed in the case of equatorial discs (which is trivial) and for critical catenoids (see the partly different arguments in \cite{Devyver2019}, \cite{Smith-Zhou2019} and \cite{Tran2020}). 
(On an incidental yet related note, we further remark that the Morse index of the critical Möbius band in $\B^4$ has proven in \cite{Medvedev2023} to equal five.) 
In particular, getting back to the three-dimensional Euclidean unit ball, we do not know the values of the index for \emph{any} infinite family, so that it is still unclear whether e.\,g. the growth rate with respect to the topological data (namely: the genus and the number of boundary components) should follow simpler, sometimes possibly even linear, or more subtle (and less universal) laws. 
That a linear, or rather affine, lower bound holds was established in \cite{AmbCarSha18-Index} and, independently, in \cite{Sar17}: 
the Morse index of any free boundary minimal surface in $\B^3$, say $\Sigma$, of genus $g$ and having $b$ boundary components satisfies the inequality
\[
\morseindex(\Sigma)\geq \frac{1}{3}(2g+b-1),
\]
an estimate that (albeit sub-optimal in the case of ``low topological complexity'') has not yet been refined in any way. This result was then complemented by the one of Lima (see \cite[Theorem 4]{Lim17}), that is an affine upper bound with a large (yet computable) numerical constant.
The network of conjectures that we are about to present in the second part of this section, partly based on numerical evidence using Brakke's \cite{Brakke1992} surface evolver, aims at shedding some new light on these delicate matters.

Setting up the simulation of a free boundary minimal surface using Brakke's surface evolver requires a (rough) initial triangulation of some surface with the desired topology, which can then be gradually refined. 
The free boundary condition is modelled by confining the boundary edges and vertices to the unit sphere using the appropriate ``level-set constraint'' \cite[Section~5.1]{Brakke1992}. 
The complexity of the simulation can be reduced dramatically by prescribing the expected symmetries using additional level-set constraints. 
The remaining (equivariant) instability of the surface in question can be overcome through a careful alternation between motion by mean curvature \cite[Section~6.4]{Brakke1992} and ``Hessian minimization'' \cite[Section~6.6]{Brakke1992}. 
The Morse index is then legible as the number of negative eigenvalues of the corresponding matrix of second derivatives of area.

\begin{conjecture}\label{conj:genusg-bc3-KL}
For every integer $g\geq2$ there exists a free boundary minimal surface $\Sigma_g^{\mathrm{KL}}$ in $\B^3$ with the folowing properties: 
\begin{enumerate}[label={\normalfont(\roman*)}]
\setlength{\itemsep}{-.5ex}
\item\label{conj:genusg-bc3-KL-i} $\Sigma_g^{\mathrm{KL}}$ has $3$ boundary components, genus $g$ and antiprismatic symmetry $\apr_{g+1}$.
\item\label{conj:genusg-bc3-KL-ii} $\area(\Sigma_g^{\mathrm{KL}})<\area(\B^2)+\area(
\gls{Kcrit} %\K_{\mathrm{crit}}
)$ and $\Sigma_g^{\mathrm{KL}}\to\B^2\cup\K_{\mathrm{crit}}$ in the sense of varifolds as $g\to\infty$. 
\item\label{conj:genusg-bc3-KL-iii} $\Sigma_g^{\mathrm{KL}}$ is congruent to the surface constructed by Kapouleas--Li \cite{KapouleasLiDiscCCdesing}  for all sufficiently large $g$ as well as to the surface $\Sigma_g^{\mathrm{Ket}}$ from \cite{KetoverFB}. 
\item\label{conj:genusg-bc3-KL-iv} The Morse index of $\Sigma_g^{\mathrm{KL}}$ is equal to $2g+6$ and its $\apr_{g+1}$-equivariant index is equal to $1$. 
\end{enumerate}
\end{conjecture}

\begin{figure}%
\includegraphics[width=\fbmsscale\textwidth,page=3]{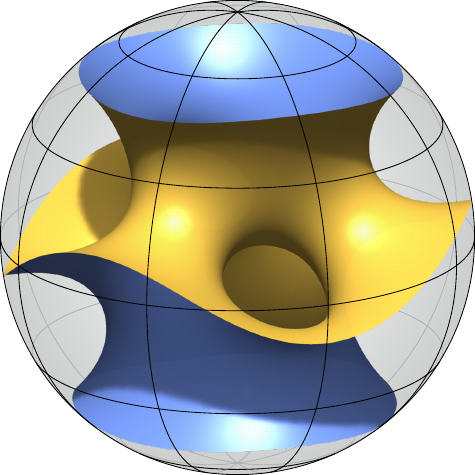}\hfill
\includegraphics[width=\fbmsscale\textwidth,page=4]{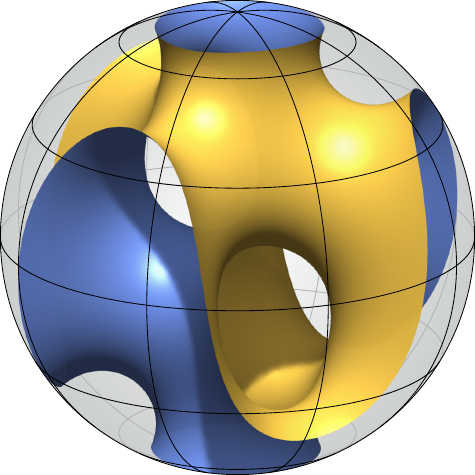}%
\caption{The conjectural free boundary minimal surfaces $\Sigma_2^{\mathrm{KL}}$ (left) and $\Sigma_2^{\mathrm{CSW}}$ (right).}%
\label{fig:genus2}%
\end{figure}

\emph{Heuristics and motivation.} 
As discussed in Appendix \ref{app:convergence_large_g}, the surfaces $\Sigma_g^{\mathrm{Ket}}$ and $\Sigma_g^{\mathrm{KL}}$ for sufficiently large $g$  satisfy \ref{conj:genusg-bc3-KL-i} and the convergence stated in \ref{conj:genusg-bc3-KL-ii} but it is open whether they are actually congruent. 
Our conjecture about their existence for low genus $g\geq2$ is based on numerical simulations which we visualize for $g=2$ in Figure~\ref{fig:genus2} (left image) and for $g=11$ in Figure~\ref{fig:genus11} (left image).  
A motivation for \ref{conj:genusg-bc3-KL-iv} is the fact that by \cite{Nayatani1993,Morabito2009} the complete Costa--Hoffman--Meeks surface $\Sigma_g^{\mathrm{CHM}}$ of genus $g$ in $\R^3$ has Morse index equal to $2g+3$.  
Recalling that the complete catenoid in $\R^3$ has index $1$, we obtain 
\begin{align}\label{eqn:formula:indexCHM}
\morseindex(\Sigma_g^{\mathrm{CHM}})=2g+3=2(g+1)+\morseindex(\text{catenoid}).
\end{align}
Here, we recover the factor $(g+1)$ which is also the order of the cyclic subgroup $\cyc_{g+1}<\apr_{g+1}$. 
In \cite{Devyver2019,Smith-Zhou2019,Tran2020} it was shown that the critical catenoid $\K_{\mathrm{crit}}$ in $\B^3$ has index $4$. 
Hence, replacing the contribution of the complete catenoid with that of the critical catenoid in equation \eqref{eqn:formula:indexCHM} we obtain  
\begin{align}\label{eqn:formula:indexKL}
\morseindex(\Sigma_g^{\mathrm{KL}})=2(g+1)+\morseindex(\K_{\mathrm{crit}})=2g+6
\end{align}
as conjectured in the first part of \ref{conj:genusg-bc3-KL-iv}. 
Moreover, the numerical data presented in Table \ref{table:index} are consistent with \eqref{eqn:formula:indexKL}. 
The surface $\Sigma_g^{\mathrm{Ket}}$ has been constructed via equivariant min-max methods and it follows from \cite{FranzIndex} that its equivariant index is equal to $1$. 
So if \ref{conj:genusg-bc3-KL-iii} is true then the second part of \ref{conj:genusg-bc3-KL-iv} follows. 

\begin{conjecture}\label{conj:genusg-bc3-CSW}
For every integer $g\geq2$ there exists a free boundary minimal surface $\Sigma_g^{\mathrm{CSW}}$ in $\B^3$ with the following properties: 
\begin{enumerate}[label={\normalfont(\roman*)}]
\setlength{\itemsep}{-.5ex}
\item\label{conj:genusg-bc3-CSW-i} $\Sigma_g^{\mathrm{CSW}}$ has $3$ boundary components, genus $g$ and antiprismatic symmetry $\apr_{g+1}$.
\item\label{conj:genusg-bc3-CSW-ii} $\area(\Sigma_g^{\mathrm{KL}})<\area(\Sigma_g^{\mathrm{CSW}})<\area(\B^2)+2\area(\K_{0})$ and $\Sigma_g^{\mathrm{CSW}}\to \K_{0}\cup\B^2\cup-\K_{0}$ in the sense of varifolds as $g\to\infty$. 
\item\label{conj:genusg-bc3-CSW-iii} $\Sigma_g^{\mathrm{CSW}}$ coincides with the surface constructed in Theorem \ref{thm:ConstructFirst} for all sufficiently large $g$. 
\item\label{conj:genusg-bc3-CSW-iv} The Morse index of $\Sigma_g^{\mathrm{CSW}}$ is greater than $3g+6$ and its $\apr_{g+1}$-equivariant index is equal to $2$. 
\end{enumerate}
\end{conjecture}

\emph{Heuristics and motivation.} 
We proved that the surfaces constructed in Theorem \ref{thm:ConstructFirst} satisfy \ref{conj:genusg-bc3-CSW-i} and the convergence stated in \ref{conj:genusg-bc3-CSW-ii}. 
It remains to determine the lowest value of $g\in\N$ for which properties \ref{conj:genusg-bc3-CSW-i} and \ref{conj:genusg-bc3-CSW-ii} are true. 
We conjecture that the answer to this question is again $g=2$, based on numerical simulation visualized for $g=2$ in Figure~\ref{fig:genus2} (right image) and for $g=11$ in Figure~\ref{fig:genus11} (right image). 
The numerical data presented in Table \ref{table:index} suggest that the index of $\Sigma_g^{\text{CSW}}$ is always odd with growth rate alternating between $2$ and $4$ but the explicit dependence on $g$ is not evident. 
Therefore, we only infer the lower bound $\morseindex(\Sigma_g^{\text{CSW}})>3g+6$ from the numerical data.
In Proposition \ref{prop:DistiguishByEquivIndex} above we proved that the surfaces constructed in Theorem \ref{thm:ConstructFirst} have equivariant index at least $2$ and we conjecture here that it is actually equal to $2$. 

\begin{conjecture}\label{conj:genus0-FPZ}
For every integer $n\geq2$ there exists a free boundary minimal surface $\Xi_n^{\mathrm{FPZ}}$ in $\B^3$ with the following properties: 
\begin{enumerate}[label={\normalfont(\roman*)}]
\setlength{\itemsep}{-.5ex}
\item\label{conj:genus0-FPZ-i} $\Xi_n^{\mathrm{FPZ}}$ has $n$ boundary components, genus zero and prismatic symmetry $\pri_{n}$ except for the case $n=2$ where the surface is congruent to the critical catenoid.
\item\label{conj:genus0-FPZ-ii} $\area(\Xi_n^{\mathrm{FPZ}})<2\area(\B^2)$ and $\Xi_n^{\mathrm{FPZ}}\to\B^2$ with multiplicity $2$ in the sense of varifolds as $n\to\infty$. 
\item\label{conj:genus0-FPZ-iii} $\Xi_n^{\mathrm{FPZ}}$ is congruent to the surface of genus zero constructed by Folha--Pacard--Zolotareva \cite{FolPacZol17} for all sufficiently large $n\in\N$. 
\item\label{conj:genus0-FPZ-iv} The Morse index of $\Xi_n^{\mathrm{FPZ}}$ is equal to $2n$ and its $\pri_{n}$-equivariant index is equal to $1$. 
\end{enumerate}
\end{conjecture}

\emph{Heuristics and motivation.} 
The boundary components of $\Xi_n^{\mathrm{FPZ}}$ are all aligned along the equator and historically this is the first infinite family of free boundary minimal surfaces ever described (see \cite{FraSch16}). 
Folha--Pacard--Zolotareva \cite[Theorem 1.1]{FolPacZol17} proved the existence of free boundary minimal surfaces satisfying \ref{conj:genus0-FPZ-i} and the convergence stated in \ref{conj:genus0-FPZ-ii} for all sufficiently large $n\in\N$. 
In \cite[Section~5]{KetoverFB} Ketover describes a variational construction of similar $\pri_{n}$-equivariant surfaces but the number of their boundary components is not controlled explicitly in the sense that additional boundary components could appear during the min-max procedure. 
If these surfaces are congruent to $\Xi_n^{\mathrm{FPZ}}$ then it follows from \cite{FranzIndex} that their equivariant index is equal to $1$. 
The numerical data presented in Table~\ref{table:index-genus0} suggest that the Morse index of $\Xi_n^{\mathrm{FPZ}}$ is equal to $2n$.  
It is rather difficult to simulate $\Xi_n^{\mathrm{FPZ}}$ for $n\geq8$ because then the surface is already extremely close to the doubling of the equatorial disc which means that the half necks along the boundary become too tiny. 
Therefore, we obtain fewer data points compared to the other families of free boundary minimal surfaces. 
The case $n=2$ is special because a result of McGrath \cite{McGrath2018} implies that a $\pri_2$-equivariant free boundary minimal surface in $\B^3$ with genus zero and two boundary components is congruent to the critical catenoid which is known to have index $4$ by \cite{Devyver2019,Smith-Zhou2019,Tran2020}. 

\begin{conjecture}\label{conj:genus0-CSW}
For every integer $n\geq3$ there exists a free boundary minimal surface $\Xi_n^{\mathrm{CSW}}$ in $\B^3$ with the following properties: 
\begin{enumerate}[label={\normalfont(\roman*)}]
\setlength{\itemsep}{-.5ex}
\item\label{conj:genus0-CSW-i} $\Xi_n^{\mathrm{CSW}}$ has $n+2$ boundary components, genus zero and prismatic symmetry $\pri_{n}$ except for the case $n=4$ where the the surface has octahedral symmetry (see Figure \ref{fig:genus0n3}).
\item\label{conj:genus0-CSW-ii} $\area(\Xi_n^{\mathrm{CSW}})<2\area(\K_{0})$ and $\Xi_n^{\mathrm{CSW}}\to \K_{0}\cup-\K_{0}$ in the sense of varifolds as $n\to\infty$. 
\item\label{conj:genus0-CSW-iii} $\Xi_n^{\mathrm{CSW}}$ coincides with the surface constructed in Theorem \ref{thm:ConstructSecond} for all sufficiently large $n\in\N$. 
\item\label{conj:genus0-CSW-iv} For $n\geq6$ the Morse index of $\Xi_n^{\mathrm{CSW}}$ is equal to $2n+6$ and its $\pri_{n}$-equivariant index is equal to $2$. 
\end{enumerate}
\end{conjecture}

\begin{figure}%
\includegraphics[width=0.32\textwidth,page=7]{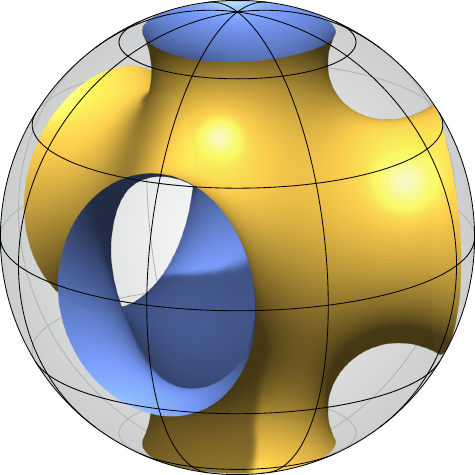}\hfill
\includegraphics[width=0.32\textwidth,page=8]{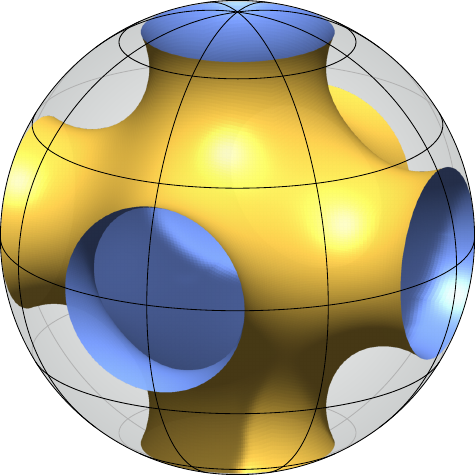}\hfill
\includegraphics[width=0.32\textwidth,page=9]{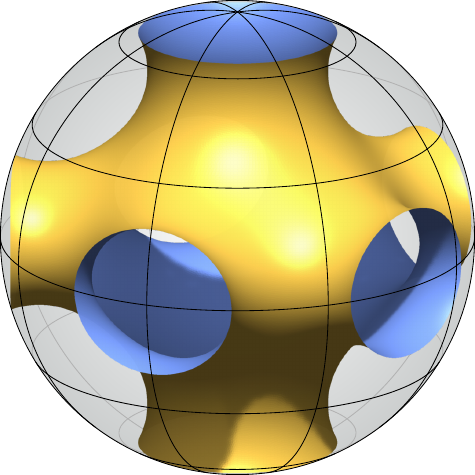}%
\caption{Visualization of the free boundary minimal surfaces 
$\Xi_3^{\mathrm{CSW}}$, 
$\Xi_4^{\mathrm{CSW}}$ and 
$\Xi_5^{\mathrm{CSW}}$ (left to right) which are conjectured to be congruent to  
$\Gamma_5^{\mathrm{max}}$,
$\Gamma_6^{\mathrm{max}}$ and 
$\Gamma_7^{\mathrm{max}}$ respectively.}%
\label{fig:genus0n3}%
\end{figure}	 

\emph{Heuristics and motivation.} 
For all sufficiently large $n\in\N$ we proved that the surfaces constructed in Theorem \ref{thm:ConstructSecond} satisfy \ref{conj:genus0-CSW-i} and the convergence stated in \ref{conj:genus0-CSW-ii}. 
For small $n\geq3$ the existence of free boundary minimal surfaces satisfying \ref{conj:genus0-CSW-i} and \ref{conj:genus0-CSW-iv} remains open and our conjecture is again based on simulations, which are visualized for $n\in\{3,4,5\}$ in Figure \ref{fig:genus0n3} and for $n\in\{12,20\}$ in Figure~\ref{fig:genus0n12}. 
It is conceivable that $\Xi_4^{\mathrm{CSW}}$ is congruent to a free boundary minimal surface constructed by Ketover \cite[Theorem 6.1]{KetoverFB} using equivariant min-max methods. 
We exclude the case $n=2$ here because we expect that $\Xi_2^{\mathrm{CSW}}$ would be congruent to the surface  $\Xi_4^{\mathrm{FPZ}}$ described in Conjecture~\ref{conj:genus0-FPZ} by means of a rotation by angle $\pi/2$ around a horizontal axis. 
In Remark \ref{rem:equivariant-index-genus0} we explain why the surfaces constructed in Theorem \ref{thm:ConstructSecond} have equivariant index at least $2$ and we conjecture here that it is actually equal to $2$. 
Regarding their full Morse index, the numerical data presented in Table \ref{table:index-genus0} suggest
\begin{align}\label{eqn:index-genus0-CSW}
\morseindex(\Xi_n^{\mathrm{CSW}})&=
\begin{cases}
	2n+6, & \text{ if $n\geq6$, } \\
	3\bigl((n+2)-1\bigr), & \text{ if $n\in\{3,4,5\}$.  }
\end{cases}
\end{align}
The dichotomy in \eqref{eqn:index-genus0-CSW} indicates that $\Xi_3^{\mathrm{CSW}}$, $\Xi_4^{\mathrm{CSW}}$ and $\Xi_5^{\mathrm{CSW}}$ can also be seen as members of another family of free boundary minimal surfaces described in Conjecture \ref{conj:genus0-max} \ref{conj:genus0-max-iii}--\ref{conj:genus0-max-iv} below.

\begin{table} 
\flushright
\pgfmathsetmacro{\xmax}{18}
\pgfmathsetmacro{\ymax}{65}
\pgfmathsetmacro{\yscale}{0.163}
\begin{tikzpicture}[line cap=round,line join=round,baseline={(0,0)},scale={(\textwidth-2em)/(\xmax+1)/1cm},mark size=1.5pt]
\begin{scope}[yscale=\yscale]
\draw[help lines,ultra thin](0,0)grid[xstep=1,ystep=1](\xmax+.5,\ymax+1);
\draw(0,0)grid[xstep=\xmax+1,ystep=10](\xmax+.5,\ymax+1);
\draw[->](0,0)--(\xmax+.5,0);
\draw[->](0,0)--(0,\ymax+1);
\foreach\index in {10,20,...,\ymax}{\draw plot[hdash](0,\index)node[left]{$\index$};}
\foreach\genus in {2,...,\xmax}{\draw plot[vdash](\genus,0);}
\draw[color={cmyk,1:magenta,0.5;yellow,1},semithick] plot[mark=*] coordinates {
( 2,13)
( 3,19)
( 4,23)
( 5,27)
( 6,31)
( 7,33)
( 8,35)
( 9,39)
(10,41)
(11,45)
(12,47)
(13,49)
(14,53)
(15,55)
(16,59)
(17,61) 
(18,65) 
}; 
\draw[domain=2:18,variable=\g,samples=17,color={cmyk,1:magenta,0.5;cyan,1},semithick] plot[mark=*](\g,2*\g+6);
\path(2,3*2+6)coordinate(3g+6start)(\xmax,3*\xmax+6)coordinate(3g+6end)
(2,2*2+6)coordinate(2g+6start)(\xmax,2*\xmax+6)coordinate(2g+6end);
\end{scope} 
\draw[dotted,semithick](3g+6start)--(3g+6end)node[near end,sloped,below]{$3g+6$};
\draw[color={cmyk,1:magenta,0.5;cyan,1}](2g+6start)--(2g+6end)node[pos=0.8,sloped,below]{$2g+6$};
\begin{scope}[yscale=1.66,shift={(0,0.15)},every node/.style={inner sep=-3pt}]
\draw (0.3,-1*\baselineskip)node[anchor=base]{genus $g$}; 
\draw(-1,-1.35*\baselineskip)--++(\xmax+1.5,0);
\draw[color={cmyk,1:magenta,0.5;yellow,1}] (0.3,-2*\baselineskip)node[anchor=base]{$\operatorname{index}\bigl(\Sigma_g^{\text{CSW}}\bigr)$}; 
\draw(-1,-2.35*\baselineskip)--++(\xmax+1.5,0);
\draw[color={cmyk,1:magenta,0.5;cyan,1}] (0.3,-3*\baselineskip)node[anchor=base]{$\operatorname{index}\bigl(\Sigma_g^{\text{KL}}\bigr)$}; 
\foreach[count=\k]\indexCSW in {13,19,23,27,31,33,35,39,41,45,47,49,53,55,59,61,65}{
\pgfmathsetmacro{\genus}{int(\k+1)}
\pgfmathsetmacro{\indexKL}{int(2*\genus+6)}
\draw (\genus,-1*\baselineskip)node[anchor=base]{$\genus$}; 
\draw[color={cmyk,1:magenta,0.5;yellow,1}] (\genus,-2*\baselineskip)node[anchor=base]{$\indexCSW$}; 
\draw[color={cmyk,1:magenta,0.5;cyan,1}] (\genus,-3*\baselineskip)node[anchor=base]{$\indexKL$}; 
}
\end{scope} 
\end{tikzpicture}
\caption{Numerically computed Morse indices of $\apr_{g+1}$-invariant free boundary minimal surfaces with genus $g$ and $b=3$ boundary components.}
\label{table:index}%
\bigskip
\bigskip
\pgfmathsetmacro{\ymax}{42}
\begin{tikzpicture}[line cap=round,line join=round,baseline={(0,0)},scale={(\textwidth-2em)/(\xmax+1)/1cm},mark size=1.5pt]
\begin{scope}[yscale=\yscale]
\draw[help lines,ultra thin](0,0)grid[xstep=1,ystep=1](\xmax+.5,\ymax+1);
\draw(0,0)grid[xstep=\xmax+1,ystep=10](\xmax+.5,\ymax+1);
\draw[->](0,0)--(\xmax+.5,0);
\draw[->](0,0)--(0,\ymax+1);
\foreach\index in {10,20,...,40}{\draw plot[hdash](0,\index)node[left]{$\index$};}
\foreach\Cn in {2,...,\xmax}{\draw plot[vdash](\Cn,0);}
\draw[color={cmyk,1:magenta,0.5;yellow,1},semithick] plot[mark=*] coordinates {
(3  ,12)
(4  ,15)
(5  ,18)
(6  ,18)
(7  ,20)
(8  ,22)
(9  ,24)
(10 ,26)
(11 ,28)
(12 ,30)
(13 ,32)
(14 ,34)
(15 ,36)
(16 ,38)
(17 ,40)
(18 ,42)
}; 
\draw[domain=2:7,variable=\n,samples=6,color={cmyk,1:magenta,0.5;cyan,1},semithick] plot[mark=*](\n,2*\n);
\draw[color={cmyk,1:magenta,0.5;cyan,1},densely dotted,semithick]
(7,2*7)coordinate(2nstart)--(\xmax,2*\xmax)coordinate(2nend); 
\draw[dotted,semithick](5.2,3*5.2+3)coordinate(3n+3start)--(13,3*13+3)coordinate(3n+3end);
\path(2,2*2+6)coordinate(2n+6start)--(\xmax,2*\xmax+6)coordinate(2n+6end);
\end{scope} 
\path[color={cmyk,1:magenta,0.5;cyan,1}](2nstart)--(2nend)node[pos=0.59,sloped,below]{$2n$}; 
\path[color={cmyk,1:magenta,0.5;yellow,1}](2n+6start)--(2n+6end)node[pos=0.75,sloped,above]{$2n+6$};
\path(3n+3start)--(3n+3end)node[pos=0.41,sloped,above]{$3(n+1)\qquad=3(b-1)$};
\begin{scope}[yscale=1.66,shift={(0,0.15)},every node/.style={inner sep=-3pt}]
\draw (0.3,-1*\baselineskip)node[anchor=base]{$n$}; 
\draw(-1,-1.35*\baselineskip)--++(\xmax+1.5,0);
\draw[color={cmyk,1:magenta,0.5;yellow,1}](0.3,-2*\baselineskip)node[anchor=base]{$\operatorname{index}\bigl(\Xi_n^{\mathrm{CSW}}\bigr)$}; 
\draw(-1,-2.35*\baselineskip)--++(\xmax+1.5,0);
\draw[color={cmyk,1:magenta,0.5;cyan,1}] (0.3,-3*\baselineskip)node[anchor=base]{$\operatorname{index}\bigl(\Xi_n^{\text{FPZ}}\bigr)$}; 
\foreach[count=\k]\indexCSW in {~,12,15,18,18,20,22,24,26,28,30,32,34,36,38,40,42}{
\pgfmathsetmacro{\Cn}{int(\k+1)}
\draw(\Cn,-1*\baselineskip)node[anchor=base]{$\Cn$}; 
\draw[color={cmyk,1:magenta,0.5;yellow,1}](\Cn,-2*\baselineskip)node[anchor=base]{$\indexCSW$};  
}
\foreach\Cn in {2,...,7}{
\pgfmathsetmacro{\indexFPZ}{int(2*\Cn)}
\draw[color={cmyk,1:magenta,0.5;cyan,1}] (\Cn,-3*\baselineskip)node[anchor=base]{$\indexFPZ$}; 
}
\end{scope} 
\end{tikzpicture}
\caption{Numerically computed Morse indices of $\pri_n$-invariant free boundary minimal surfaces with genus zero and $b=n+2$ (top) respectively $b=n$ (bottom) boundary components.}
\label{table:index-genus0}%
\end{table}

The following conjecture is best explained in the context of the Steklov eigenvalue problem, because the existence of free boundary minimal surfaces  in $\B^3$ satisfying items \ref{conj:genus0-max-i} and \ref{conj:genus0-max-ii} below has already been discussed in the literature (cf. \cite{FraSch16,Karpukhin-Stern, Girouard2021}). 
However, in the context of the present discussion, the essence of Conjecture \ref{conj:genus0-max} is property \ref{conj:genus0-max-iv}. 

\begin{conjecture}\label{conj:genus0-max}
For every integer $b\geq3$ there exists a free boundary minimal surface $\Gamma_b^{\mathrm{max}}$ in $\B^3$ with the following properties: 
\begin{enumerate}[label={\normalfont(\roman*)}]
\setlength{\itemsep}{-.5ex}
\item\label{conj:genus0-max-i} $\Gamma_b^{\mathrm{max}}$ has $b$ boundary components, genus zero and maximizes the area among all embedded free boundary minimal surfaces in $\B^3$ with the same topology.
\item\label{conj:genus0-max-ii} $\area(\Gamma_b^{\mathrm{max}})<\area(\partial\B^3)$ and $\Gamma_b^{\mathrm{max}}\to\partial\B^3$ in the sense of varifolds as $b\to\infty$. 
\item\label{conj:genus0-max-iii}For $b\in\{5,6,7\}$ the surface $\Gamma_{b}^{\mathrm{max}}$ is congruent to the surface $\Gamma_{b-2}^{\mathrm{CSW}}$ described in Conjecture~\ref{conj:genus0-CSW}. 
\item\label{conj:genus0-max-iv} For all $b\geq3$ the Morse index of $\Gamma_b^{\mathrm{max}}$ is equal to $3(b-1)$. 
\end{enumerate}
\end{conjecture}

\emph{Heuristics and motivation.} 
On any given surface $\Sigma$ with nonempty boundary $\partial\Sigma$ and outer unit conormal $\eta$ along $\partial\Sigma$, the spectrum of the \emph{Steklov eigenvalue problem} 
\begin{align}\label{eqn:Steklov}
\left\{
\begin{aligned}
\Delta_{\Sigma} u&=0, &&\text{ in $\Sigma$, }\\
\eta\cdot\nabla_{\Sigma}u&=\sigma u &&\text{ on $\partial\Sigma$ } 
\end{aligned}\right. 
\end{align} 
is discrete, and given by a sequence of eigenvalues $0=\sigma_0<\sigma_1\leq\sigma_2\leq\ldots\to\infty$. 
Fraser and Schoen \cite{FraSch11} observed that an immersed surface $\Sigma\subset\B^n$ is a free boundary minimal surface if and only if the ambient coordinate functions restricted to $\Sigma$ solve \eqref{eqn:Steklov} with eigenvalue $\sigma=1$. 
A conjecture by Fraser and Li \cite{FraLi14} states that in this case $\sigma=1$ is actually the \emph{first} Steklov eigenvalue, which can be characterized variationally as
\begin{align*}
\sigma_1(\Sigma,g)
&=\inf\Biggl\{
\frac{\int_{\Sigma}\abs{\nabla_{\Sigma} u}_{g}^2}{\int_{\partial\Sigma}u^2}
\st u\in C^{\infty}(\Sigma),~ \int_{\partial\Sigma}u=0
\Biggr\}
\end{align*}
depending on $\Sigma$ and the Riemannian metric $g$ on $\Sigma$. 
Fraser and Schoen \cite{FraSch16} proved that if $g_{\mathrm{max}}$ is a smooth metric on $\Sigma$ maximizing the \emph{scale-invariant} first Steklov eigenvalue 
\begin{align}\label{eqn:scale-invariant-Steklov-ev}
\overline{\sigma}_1(\Sigma,g)
&\vcentcolon=\sigma_1(\Sigma,g)\operatorname{length}(\partial\Sigma,g)
\end{align}
then there exist independent first eigenfunctions $u_1,\ldots,u_n$ which give a free boundary branched minimal immersion $U=(u_1,\ldots,u_n)$ into the unit ball $\B^n$ of $\R^n$ for some $n\geq3$ such that $U$ is an isometry on $\partial\Sigma$ up to a rescaling of the metric.  
In the special case that $\Sigma$ is a surface of genus zero, a result obtained independently by Fraser and Schoen \cite[Theorem 2.3]{FraSch16}, Jammes \cite[Theorem 1.5]{Jamm14} and by Karpukhin, Kokarev and Polterovich \cite[Corollary 1.3]{Karpukhin2014} implies that the multiplicity of the first Steklov eigenvalue is at most $n=3$. 
Moreover, again in the genus zero case, \cite[Proposition 8.1]{FraSch16} states that a free boundary minimal immersion $(u_1,u_2,u_3)\colon\Sigma\to\B^3$ by first Steklov eigenfunctions must actually be an embedding.

Given $b\geq2$, one can also consider an embedded free boundary minimal surface $\Gamma_b^{\max}$ in $\B^3$ with genus zero, $b$ boundary components and largest possible area. 
By an elementary computation (see \cite[Theorem 5.4]{FraSch11}, cf. \cite{LiSurvey}) the boundary length of any free boundary minimal surface in $\B^3$ coincides with twice its area. 
If the aforementioned conjecture by Fraser--Li is true, i.\,e. if $\sigma_1=1$ on any free boundary minimal surface in $\B^3$, then the induced metric on $\Gamma_b^{\mathrm{max}}$ is a maximizer of \eqref{eqn:scale-invariant-Steklov-ev}, provided that a smooth maximizer exists.   
A result by Karpukhin and Stern \cite[Corollary 1.4]{Karpukhin-Stern} then implies item \ref{conj:genus0-max-ii} (see also the work by Girouard and Lagacé \cite[Corollary~1.4]{Girouard2021}). 

Item \ref{conj:genus0-max-iii} is again based on numerical simulations. 
In \cite{Oudet2021} Kao, Osting and Oudet developed numerical methods to maximize the scale-invariant first Steklov eigenvalue on surfaces of genus zero. 
Their results confirm that the corresponding free boundary minimal surfaces have prismatic (i.\,e. ``bipyramidal'') symmetry $\pri_{b-2}$ for $b\in\{5,7\}$ boundary components respectively octahedral symmetry for $b=6$ boundary components (cf. \cite[Table 2]{Oudet2021}). 
It is important to note that $\Gamma_b^{\mathrm{max}}$ does not necessarily exhibit any symmetries. 
In fact, we have numerical evidence which indicates that $\Gamma_{61}^{\mathrm{max}}$ has trivial symmetry group. 
Therefore, we do not consider the equivariant index in the context of Conjecture \ref{conj:genus0-max}.   

Our conjectured formula \ref{conj:genus0-max-iv} for the Morse index of $\Gamma_b^{\mathrm{max}}$ is consistent with the second case in equation \eqref{eqn:index-genus0-CSW} (see also Table \ref{table:index-genus0}). 
A ``translation'' tangential to $\partial\B^3$ of any boundary component of a free boundary minimal surface with largest possible area should decrease its area to second order and for each boundary component there are two such independent directions of translation. 
A third independent deformation which decreases area to second order can be conceived as a ``pinching'' of the neck which forms around any boundary component of a surface of genus zero.   
Therefore, the growth rate of the Morse index of $\Gamma_b^{\mathrm{max}}$ depending on $b$ should indeed be at least three. 
One could interpret any ambient rotation of $\Gamma_b^{\mathrm{max}}$ 
in $\B^3$ as being generated by a suitable linear combination of 
tangential translations of its boundary components as described above.
Since there is a three-dimensional subspace of such ambient rotations 
which clearly does not contribute to the Morse index, the expected 
formula is indeed $3b-3$ rather than $3b$.

%%%%%%%%%%%%%%%%%%%%%%%%%%%%%%%%%%%%%%%%%%%%%%%%%%%%%%%%
\appendix
%%%%%%%%%%%%%%%%%%%%%%%%%%%%%%%%%%%%%%%%%%%%%%%%%%%%%%%%

\section{Parametrization of the Karcher--Scherk towers}
\label{app:Karcher--Scherk}In this appendix we present an explicit parametrization of $\tow_\vartheta$ (cf. Proposition \ref{karcher-scherk}), specifically its Enneper--Weierstrass representation, and we use this parametrization to compute the asymptotic planes of the tower. 
To begin we shall briefly recall the structure of the Enneper--Weierstrass representation of minimal surfaces in $\R^3$;
we refer the reader to \cite[Section 1.4]{KarcherTokyo} or \cite[Section 1.6]{CM11} for a complete introduction to the classical theory, 
though we will rather follow the normalization convention of \cite{KarcherScherk}.

Let $(x,y,z)\colon \Omega \to \R^3$ be a two-sided minimal immersion with image $\Sigma$;
we take $\Omega$ to be a Riemann surface, possibly with boundary, 
by pulling back the conformal structure induced on $\Sigma$ by the ambient Euclidean metric and a choice of unit normal.
Let $v\colon\Omega\to\C$ be the stereographic projection of the Gauss map and let $dh$ be the complex-valued one-form on $\Omega$ which at any point is the differential of the locally defined
(and unique up to an additive imaginary constant)
holomorphic (analytic) function having real part $2z$.
Then $v$ is a meromorphic function and $dh$ is a holomorphic differential, and the \emph{Enneper--Weierstrass representation} of $\Sigma$ reads:
\begin{align}\label{EWrep}
x(w)&=\Re\int^{w}_{0}\biggl(\frac{1}{v}-v\biggr) \, dh, &
y(w)&=\Re\int^{w}_{0}i\biggl(\frac{1}{v}+v\biggr) \, dh, &
z(w)&=\Re\int^{w}_{0}2 \, dh,
\end{align}
where $0\in\Omega$ simply denotes a reference basepoint.
Conversely, one can start with data a Riemann surface $\Omega$,
a meromorphic function $v$, and a holomorphic one-form $dh$,
and one can attempt to define a minimal immersion by \eqref{EWrep};
indeed, if the integrals in \eqref{EWrep} are path-independent,
then \eqref{EWrep} defines a possibly branched conformal minimal immersion $(x,y,z)\colon \Omega \to \R^3$. 
A particularly important special case, which suffices for the purposes of this appendix, is when $\Omega\subset \R^2\equiv \C$ with the standard complex (thus: conformal) structure; 
if such a domain $\Omega$ contains the origin then it is rather customary to take it as the basepoint for the integrals above (which justifies the corresponding notation).

In Section 2.5.1 of \cite{KarcherScherk} Karcher presents a one-parameter family of Enneper--Weierstrass data, which we will analyze and verify represents the family $\tow_\vartheta$ of Proposition \ref{karcher-scherk}. 
(The same family and data are also briefly studied in Section 5.3.2 of \cite{ChenTraizetKSI}.) 
Karcher's data are indexed by $\phi \in \interval{0,\pi/2}$, a distinguished angle in the domain 
(as visualized in Figure \ref{fig:Karcher-Scherk-Weierstrass}):
\begin{align}\label{eqn:Karcher-data}
\begin{aligned}
\Omega_\phi
&=\{w\in\C\st\abs{w}\leq1\}\setminus\{\pm i,\pm e^{i\phi},\pm e^{-i\phi}\}, 
\\
v(w)&=\frac{w^2+r}{1+r w^2}, 
\\
dh&=\frac{1+r^2+(w^2+w^{-2})r}{\bigl(w^2+w^{-2}-2\cos(2\phi)\bigr)(w^2+1)}dw,
\end{aligned}
\end{align}
where $r\in\interval{-1,1}$ depends on $\phi$ and is uniquely specified by means of the equation
\begin{align}\label{eqn:20211219-r}
\frac{4r}{(1-r)^2}&=\frac{2\sin(\phi)-1}{\cos^2(\phi)}.
\intertext{Solving equation \eqref{eqn:20211219-r} for $(1-r)\in\interval{0,2}$ yields}\notag
1-r&=\frac{2 \cos \phi}{\cos \phi + \sqrt{1-(1-\sin \phi)^2}}.
\end{align}
(In general equation \eqref{eqn:20211219-r} has two real solutions for $1-r$, with exactly the one above yielding $r \in \interval{-1,1}$, confirming that $r$ is well-defined by \eqref{eqn:20211219-r}.) 
To simplify expressions involving $1-\sin\phi$ we introduce the ($\phi$-dependent) angle $\vartheta\in \interval{0,\pi/2}$ by letting.
\begin{equation}
\label{eqn:theta_phi}
\cos(\vartheta) \vcentcolon=  1 - \sin(\phi).
\end{equation}
We will soon see that $\vartheta$ can be -- so to say -- interpreted as an angle in the target,
namely the wing angle for $\tow_\vartheta$
(see Figure \ref{fig:Karcher-Scherk-Weierstrass}).
With this definition we obtain 
\begin{equation}
\label{eqn:r_expr}
r=\frac{\sin\vartheta-\cos\phi}{\sin\vartheta+\cos\phi}.
\end{equation}
From \eqref{EWrep}, \eqref{eqn:Karcher-data} and \eqref{eqn:20211219-r}
it follows that
\begin{equation}
\label{eqn:complex_integrals}
\begin{aligned}
  x(w) &= -(1-r^2) \Re \int_0^w
    \frac{\zeta^2-1}
      {(\zeta^2-e^{2\phi i})(\zeta^2-e^{-2\phi i})}
      \, d\zeta,
    \\
  y(w) &= -(1+r^2) \Im \int_0^w
    \frac{(\zeta^2-e^{2\xi i})(\zeta^2-e^{-2\xi i})}
      {(\zeta^2-e^{2\phi i})(\zeta^2-e^{-2\phi i})
        (\zeta^2+1)} \, d\zeta,
  \\
  z(w) &= 2 \Re \int_0^w
    \frac{r\zeta^4 + (r^2+1)\zeta^2 + r}
      {(\zeta^2-e^{2\phi i})(\zeta^2-e^{-2\phi i})
        (\zeta^2+1)} \, d\zeta,
\end{aligned}
\end{equation}
where $\xi \in \Interval{0, 2\pi}$ is uniquely specified by
\begin{equation}
e^{2\xi i}
  =
  \frac{-2r}{r^2+1} + i\frac{r^2-1}{r^2+1}.
\end{equation}
Note that each integrand is in fact holomorphic on the domain $\Omega_\phi$,
which is simply connected.
Integrating and employing \eqref{eqn:r_expr}
and \eqref{eqn:theta_phi}, we obtain
\begin{equation}\label{eqn:original_parametrisation}
\begin{aligned}
x(w) &= \frac{\sin \vartheta}{(\sin \vartheta + \cos \phi)^2}
  \log \frac{\abs{w+e^{i\phi}}\abs{w+e^{-i\phi}}}
        {\abs{w-e^{i\phi}}\abs{w-e^{-i\phi}}}, \\
y(w) &= \frac{1}{(\sin \vartheta + \cos \phi)^2}
        \log \frac{\abs{w-i}}{\abs{w+i}}
  + \frac{\cos \vartheta}{(\sin \vartheta + \cos \phi)^2}
    \log \frac{\abs{w-e^{i\phi}}\abs{w+e^{-i\phi}}}
              {\abs{w+e^{i\phi}}\abs{w-e^{-i\phi}}}, \\
z(w) &= \frac{1}{(\sin \vartheta + \cos \phi)^2}
  \biggl(
\sum_{\epsilon_1, \epsilon_2 = \pm 1} 
    \epsilon_1\epsilon_2 \Arg
    \frac{\epsilon_1 e^{\epsilon_2 i\phi} - w}
         {\epsilon_1e^{\epsilon_2i\phi}}
  -\sum_{\epsilon=\pm 1}
    \epsilon\Arg\frac{\epsilon i - w}{\epsilon i}
  \biggr),
\end{aligned}
\end{equation}
where each instance of $\Arg \frac{\zeta-w}{\zeta}$ measures the (counterclockwise) signed angle in $\intervaL{-\pi,\pi}$ from the vector connecting $0$ to $\zeta$ to the vector connecting $w$ to $\zeta$. 
We also introduce the rescaled parametrization 
\begin{align}\label{eqn:rescaled_parametrization}
(X,Y,Z)\vcentcolon=(\sin \vartheta + \cos \phi)^2(x,y,z).
\end{align} 

\begin{figure}[t]%
\pgfmathsetmacro{\angle}{30}
\pgfmathsetmacro{\phipar}{asin(1-cos(\angle))}
\pgfmathsetmacro{\globalscale}{2.5}
\begin{tikzpicture}[line cap=round,line join=round,scale=\globalscale,baseline={(0,0)}]
\draw(0,0)node[scale=\globalscale]{\includegraphics[page=4]{figures-KarcherScherk}};
\draw(0,0)node[scale=-\globalscale]{\includegraphics[page=4]{figures-KarcherScherk}};
\draw(0,0)node[scale=\globalscale,xscale=-1]{\includegraphics[page=4]{figures-KarcherScherk}};
\draw(0,0)node[scale=\globalscale,yscale=-1]{\includegraphics[page=4]{figures-KarcherScherk}};
\draw[latex-](\phipar:1.005)-|++(0.1,0.33)node[above]{$e^{i\phi}$};
\draw[latex-](-\phipar:1.005)-|++(0.1,-0.33)node[below]{$e^{-i\phi}$};
\draw[latex-](\phipar:-1.005)-|++(-0.1,-0.33)node[below]{$-e^{i\phi}$\!};
\draw[latex-](-\phipar:-1.005)-|++(-0.1,0.33)node[above]{$-e^{-i\phi}$\!};
\draw[latex-](0,1.005)|-++(-0.33,0.1)node[left]{$i$};
\draw[latex-](0,-1.005)|-++(-0.33,-0.1)node[left]{$-i$};
\begin{scope}[ultra thick]
\pgfmathsetmacro{\gap}{2/3}
\draw[color={cmyk,1:magenta,0.5;yellow,1}](0,0)--(0,{1-rad(\gap)});
\draw[color={cmyk,1:magenta,0.5;cyan,1}](0,0)--(1,0);
\draw[color={cmyk,1:yellow,1;cyan,0.5}](\gap+\phipar:1)arc(\gap+\phipar:90-\gap:1);
\draw[color={cmyk,1:magenta,1;cyan,0.5}](1,0)arc(0:-\gap+\phipar:1); 
\end{scope}
\end{tikzpicture}%
\hfill%
\pgfmathsetmacro{\thetaO}{45}%
\pgfmathsetmacro{\phiO}{75}%
\pgfmathsetmacro{\radius}{9}%
\pgfmathsetmacro{\globalscale}{0.55}%
\pgfmathsetmacro{\offset}{(sin(\angle)-tan(\angle)/2)*ln(2/cos(\angle)-1)-2*sin(\angle)*ln(1/cos(\angle)-1)}%
\tdplotsetmaincoords{\thetaO}{\phiO}%
\begin{tikzpicture}[tdplot_main_coords,line cap=round,line join=round,scale=\globalscale,baseline={(0,-2)}]
\draw(0,0,-2*pi)node[scale=\globalscale]{\includegraphics[page=3]{figures-KarcherScherk}};
\draw(0,0,0)node[scale=\globalscale]{\includegraphics[page=2]{figures-KarcherScherk}};
\draw(0,0,-2*pi+pi/2)--++(0,0,pi/2);
\draw(0,0,0)node[scale=\globalscale]{\includegraphics[page=1]{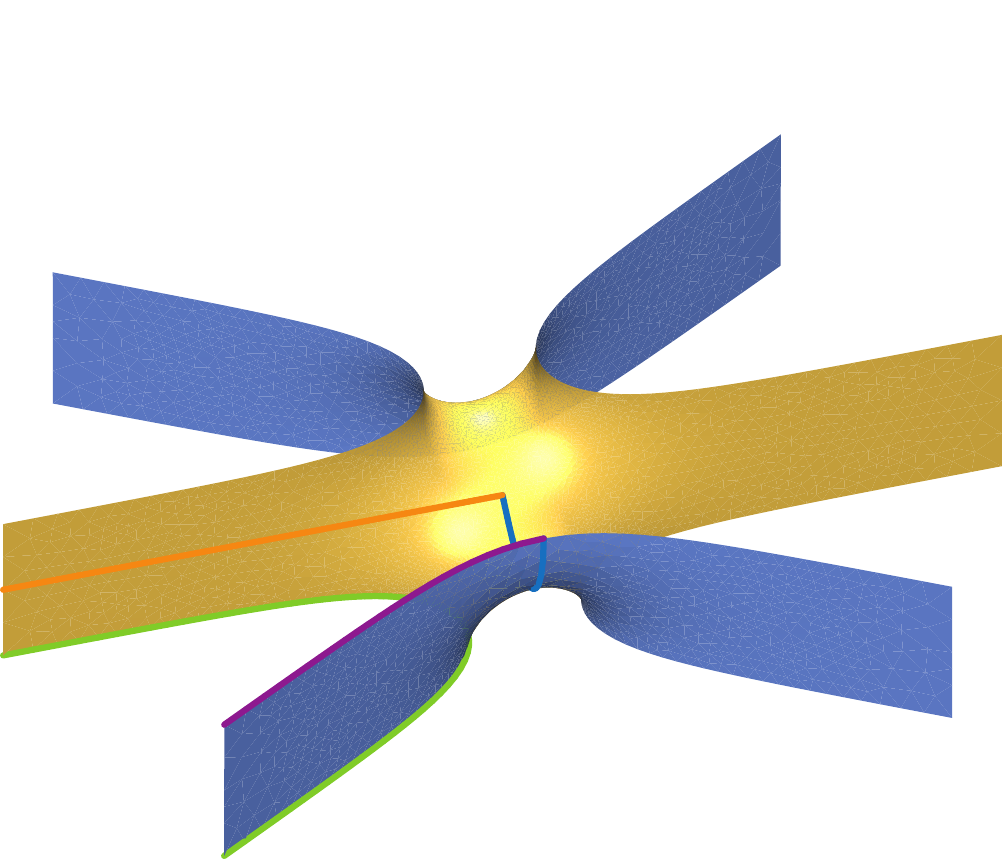}};
\coordinate(bottom)at (current bounding box.south);
\pgfresetboundingbox
\path(bottom);
\draw[-latex](0,0,-pi/2)--(0,0,5)node[above]{$z$};
\draw[dotted](0,0,-pi/2)--(0,0,-3*pi);
\draw[-latex](-2/3*\radius,0,0)--(\radius,0,0)node[right]{$x$};
\draw[-latex](0,-\radius,0)--(0,\radius,0)node[above,pos=0.97]{$y$};
\pgfmathsetmacro{\radius}{\radius*0.8175}
\draw
(\offset,0,0)--({\radius*sin(\angle)+\offset},{\radius*cos(\angle)},0)
(\offset,0,0)--({\radius*sin(\angle)+\offset},{-\radius*cos(\angle)},0)
(-\offset,0,0)--({-\radius*sin(\angle)-\offset},{\radius*cos(\angle)},0)
(-\offset,0,0)--({-\radius*sin(\angle)-\offset},{-\radius*cos(\angle)},0)
;
\tdplottransformmainscreen{0}{1}{0}
\pgfmathsetmacro{\VecHx}{\tdplotresx}
\pgfmathsetmacro{\VecHy}{\tdplotresy}
\tdplottransformmainscreen{0}{0}{1}
\pgfmathsetmacro{\VecVx}{\tdplotresx}
\pgfmathsetmacro{\VecVy}{\tdplotresy}
\path(0,-\radius,0)--++(0,0,-4*pi)
coordinate[pos=0](0pi)
coordinate[pos=0.25](1pi)
coordinate[pos=0.5](2pi)
coordinate[pos=0.75](3pi)
coordinate[pos=1](4pi)
;
\draw[stealth-stealth](0pi)--(1pi)node[midway,right,inner sep=1.5pt,cm={\VecHx ,\VecHy ,\VecVx ,\VecVy ,(0,0)}]{$\pi$};
\draw[stealth-stealth](1pi)--(3pi)node[midway,right,inner sep=1.5pt,cm={\VecHx ,\VecHy ,\VecVx ,\VecVy ,(0,0)}]{$2\pi$};
\fill[black!66] (0,0,-pi/2)circle(2pt)node[above left,cm={\VecHx ,\VecHy ,\VecVx ,\VecVy ,(0,0)}]{$\scriptstyle (0,0,0)$};
\tdplottransformmainscreen{0}{1}{0}
\pgfmathsetmacro{\VecHx}{\tdplotresx}
\pgfmathsetmacro{\VecHy}{\tdplotresy}
\tdplottransformmainscreen{-1}{0}{0}
\pgfmathsetmacro{\VecVx}{\tdplotresx}
\pgfmathsetmacro{\VecVy}{\tdplotresy}
\begin{scope}[color={cmyk,1:magenta,1;yellow,0.5}]
\path[very thick](0,0,0)--(\offset,0,0)node[midway,right,cm={\VecHx ,\VecHy ,\VecVx ,\VecVy ,(0,0)}]{\Large$b^{\text{\text{tow}}}_\vartheta$};
\fill(0,-0.05pt,0)--++(0,0.1pt,0)--++(\offset,0,0)--++(0,-0.1pt,0)--cycle;
\end{scope}
\tdplotdrawarc[thick]{(\offset,0,0)}{2/3*\radius}{90}{90-\angle}{cm={\VecHx ,\VecHy ,\VecVx ,\VecVy ,(0,0)},anchor=-10}{\Large$\vartheta$}
\draw[dashed,thin](\offset,0,0)--++(0,\radius,0);
\end{tikzpicture}
\caption{Domain $\Omega_\phi$ (left) and corresponding image (right, colored region) of the rescaled Weierstrass parametrization \eqref{eqn:rescaled_parametrization} of a Karcher--Scherk tower with wing angle $\vartheta=\pi/6$ and corresponding domain angle $\phi\approx0.043\,\pi$. 
One and a half vertical periods are displayed. 
To clarify the asymptotics, the $x$-$y$-axes are displayed at height $z=\pi/2$.}%
\label{fig:Karcher-Scherk-Weierstrass}%
\end{figure}

\begin{lemma}[Symmetries]\label{lem:Karcher-Scherk_symmetries} 
The map $(X,Y,Z)\colon\Omega_\phi\to\R^3$ defined in \eqref{eqn:original_parametrisation}--\eqref{eqn:rescaled_parametrization} satisfies   
\begin{align}
\label{eqn:symmetry_under_complex_conj}
(X,Y,Z)(\overline{w})&=(X,-Y,Z)(w), 
\\ 
\label{eqn:symmetry_antipodal}
(X,Y,Z)(-w)&=-(X,Y,Z)(w).
\end{align}
\end{lemma}

\begin{proof}
Let $f\colon\Omega_\phi\to\C$ be any of the three integrands in \eqref{eqn:complex_integrals}. 
Then it is straightforward to verify that $f$ commutes with complex conjugation, that is  $f(\overline{\zeta})=\overline{f(\zeta)}$. 
This proves \eqref{eqn:symmetry_under_complex_conj} since in \eqref{eqn:complex_integrals}, we take the imaginary part of the integral for $Y$ and the real part for $X$ and $Z$. 

Since the integrands in \eqref{eqn:complex_integrals} are even functions in the sense that $f(\zeta)=f(-\zeta)$, the corresponding integrals are odd functions and \eqref{eqn:symmetry_antipodal} follows. 
Indeed, if $\gamma\subset\Omega_\phi$ is a suitable path from $0$ to $w$, then $-\gamma\subset\Omega_\phi$ is a suitable path from $0$ to $-w$ thanks to the symmetry of the domain $\Omega_\phi$. 
\end{proof}

As a result of the rescaling \eqref{eqn:rescaled_parametrization} the image of the parametrization in question has ``vertical range'' equal to $\pi$ independently of the parameter $\phi$. 
More precisely: 

\begin{lemma}\label{lem:vertical_height}
Let $Z\colon\Omega_\phi\to\R$ be as in \eqref{eqn:original_parametrisation}--\eqref{eqn:rescaled_parametrization}. 
Then, $\abs{Z}\leq\pi/2$ and $\abs{Z(w)}=\pi/2\;\Leftrightarrow\;\abs{w}=1$. 
In fact $Z(e^{i\alpha})=\pi/2$ if $\alpha\in\Interval{0,\phi}$ and $Z(e^{i\alpha})=-\pi/2$ if $\alpha\in\interval{\phi,\pi/2}$. 
\end{lemma}

\begin{proof}
Let $w\in\Omega_\phi$ with $\abs{w}=1$. 
By Lemma \ref{lem:Karcher-Scherk_symmetries} we may assume $w=e^{i\alpha}$ for some $\alpha\in\Interval{0,\pi/2}\setminus\{\phi\}$. 
Then, using the third formula in \eqref{eqn:original_parametrisation},
\begin{align*}
Z(e^{i\alpha})
&=\Arg\bigl(1-e^{i(\alpha-\phi)}\bigr)
-\Arg\bigl(1+e^{i(\alpha-\phi)}\bigr) 
\\&\hphantom{{}={}}
-\Arg\bigl(1-e^{i(\alpha+\phi)}\bigr) 
+\Arg\bigl(1+e^{i(\alpha+\phi)}\bigr) 
\\&\hphantom{{}={}}
-\Arg\bigl(1-e^{i(\alpha-\frac{\pi}{2})}\bigr) 
+\Arg\bigl(1+e^{i(\alpha-\frac{\pi}{2})}\bigr).
\end{align*}
By elementary geometric considerations in the complex plane we have for any angle $\sigma$
\begin{align*}
\Arg(1+e^{i\sigma})-\Arg(1-e^{-i\sigma})&=
\begin{cases}
\hphantom{-}\frac{\pi}{2} & \text{ if $0<\sigma<\pi$,} \\
-\frac{\pi}{2} & \text{ if $-\pi<\sigma<0$.} \\
\end{cases}
\end{align*}
Therefore, $Z(e^{i\alpha})=\pi/2$ if $0\leq\alpha<\phi$ and $Z(e^{i\alpha})=-\pi/2$ if $\phi<\alpha<\pi/2$. 
This proves $\abs{Z}=\pi/2$ on the boundary $\partial\Omega_\phi$ (except at the six points
where $Z$ is undefined).
It follows immediately from the definition of $Z$
that we then also have
$
\limsup_{w \to \omega} \abs{Z(\omega)}
=
\pi/2
$
for each $\omega \in \{\pm i, \pm e^{i\phi}, \pm e^{-i\phi}\}$.
Now, for each $\epsilon>0$,
define the closed subset $\Omega_\phi^\epsilon$
to be $\Omega_\phi$ less the six open discs
with radius $\epsilon$
and centers $\pm i, \pm e^{i\phi}, \pm e^{-i\phi}$.
Then $Z$ is continuous on each
$\Omega_\phi^\epsilon$
and harmonic on the interior.
Since we may take $\epsilon>0$ as small as desired,
by the maximum principle together
with the preceding information
on the behavior of
$Z|_{\partial \Omega_\phi^\epsilon}$
we conclude that $\abs{Z} \leq \pi/2$
on $\Omega_\phi$.
Since $Z$ is not constant,
we moreover have
$\abs{Z} < \pi/2$
in the interior of $\Omega_\phi$. 
\end{proof} 

In the following statement and throughout this appendix, we shall generically refer to an \emph{end} as the image of any small connected open set containing a puncture (whenever such an image is unbounded, i.\,e. the punctures are not removable) and we shall then say that a surface parametrized by $(X,Y,Z)\colon\Omega\subset\C\to\R^3$ has an end asymptotic to an affine plane $\Pi\subset\R^3$ if there exists a point $w_0$ in the closure of $\Omega$ such that
\[
\Omega\ni w\to w_0\quad\Rightarrow\quad\abs{(X,Y,Z)}(w)\to\infty \text{ and } \dist_{\Pi}\circ(X,Y,Z)(w)\to 0
\] 
recalling notation \eqref{eqn:definition_distancefunction} concerning the distance from a set in Euclidean $\R^3$. 
For the Karcher--Scherk towers we will later strengthen this notion of asymptotics to the property of being graphical over the asymptotic plane in question with exponential decay (cf. Lemmata \ref{lem:exp_asy} and \ref{lem:CylMSE}). 

For the remainder of this appendix it will be convenient to distinguish the real-valued functions 
$X,Y,Z$ defined in \eqref{eqn:rescaled_parametrization} from the standard coordinates in the target which we  still denote by $x,y,z$. 

\begin{lemma}[Ends]\label{lem:btow_theta}
The image of $(X,Y,Z)\colon\Omega_\phi\to\R^3$ as given in \eqref{eqn:original_parametrisation}--\eqref{eqn:rescaled_parametrization} has six ends. 
Two of them are asymptotic to the plane $\{x=0\}$ and the other four are asymptotic to the planes 
$\{x=\pm b^{\mathrm{tow}}_\vartheta + y\tan \vartheta\}$ respectively $\{x=\pm b^{\mathrm{tow}}_\vartheta - y\tan \vartheta\}$, where 
\begin{align}\label{eqn:btow_theta}
\gls{btow}=\bigl(\sin(\vartheta)-\tfrac{1}{2}\tan(\vartheta)\bigr)\log\bigl(2\sec(\vartheta)-1\bigr)
-2\sin(\vartheta)\log\bigl(\sec(\vartheta)-1\bigr).
\end{align}
\end{lemma}
 
\begin{proof}
By the symmetries shown in Lemma \ref{lem:Karcher-Scherk_symmetries} it suffices to consider the limits 
\begin{align*}
\lim_{w \to i}\bigl(X,Y\bigr)(w) &=(0,-\infty), &
\lim_{w \to e^{i\phi}}\bigl(X,Y\bigr)(w) &=(\infty,-\infty).
\end{align*}
We directly obtain that the ends corresponding to the limits $w\to\pm i$ are asymptotic to the plane $\{x=0\}$. 
To determine the asymptotic planes for the remaining wings, and in particular their offset $b^{\text{\text{tow}}}_\vartheta$ from the origin (see Figure \ref{fig:Karcher-Scherk-Weierstrass}), we compute  
\begin{align*}
b^{\text{\text{tow}}}_\vartheta\vcentcolon=
\lim_{w \to e^{i\phi}}\Bigl(X(w) + Y(w)\tan(\vartheta)\Bigr)
&=\sin(\vartheta)\log\frac{\abs{e^{i\phi}+e^{-i\phi}}^2}{\abs{e^{i\phi}-e^{-i\phi}}^2}
+\tan(\vartheta)\log\frac{\abs{e^{i\phi}-i}}{\abs{e^{i\phi}+i}}
\\
&=2\sin(\vartheta)\log\abs{\cot\phi}
+\frac{1}{2}\tan(\vartheta)\log\biggl(\frac{1-\sin(\phi)}{1+\sin(\phi)}\biggr),
\end{align*}
where we used $\abs{e^{i\phi}\pm i}^2=(\cos(\phi))^2+(\sin(\phi)\pm1)^2=2\pm2\sin(\phi)$. 
Recalling the relation \eqref{eqn:theta_phi} between $\vartheta$ and $\phi$, in particular $\cot(\phi)=\sqrt{2\sec(\vartheta)-1}/(\sec(\vartheta)-1)$, we obtain
\begin{align*}
b^{\text{\text{tow}}}_\vartheta
&=\bigl(\sin(\vartheta)-\tfrac{1}{2}\tan(\vartheta)\bigr)\log\bigl(2\sec(\vartheta)-1\bigr)
-2\sin(\vartheta)\log\bigl(\sec(\vartheta)-1\bigr).
\qedhere
\end{align*}
\end{proof}

\begin{lemma}[Gauss map]\label{lem:gaussmap}
The image of the Gauss map $N\colon\Omega_\phi\to\Sp^2$ corresponding to the parametrization \eqref{eqn:original_parametrisation}--\eqref{eqn:rescaled_parametrization} is contained in the southern hemisphere and $N(w)$ is horizontal (i.\,e. on the equator) if and only if $\abs{w}=1$. 
\end{lemma}

\begin{proof}
Recalling that $v\colon\Omega_\phi\to\C$ given in \eqref{eqn:Karcher-data} is the stereographic projection of $N$, we have  
\begin{align*}
N=\frac{1}{1+\abs{v}^2}\Bigl(2\Re(v),\,2\Im(v),\,\abs{v}^2-1\Bigr).	
\end{align*} 
To prove the claim it suffices to show $\abs{v}^2=\abs{w^2+r}^2/\abs{1+rw^2}^2\leq1$. 
Indeed, given any $r\in\interval{-1,1}$ and any $w\in\Omega_\phi$, we have
\(\abs{w^2+r}^2-\abs{1+rw^2}^2=(1-r^2)\bigl(\abs{w}^4-1\bigr)\leq0\)
with equality if and only if $\abs{w}=1$.  
\end{proof}

\begin{lemma}[Vertical graphicality]\label{lem:graphicality}
The image of the map $(X,Y,Z)\colon\Omega_\phi\to\R^3$ defined in  \eqref{eqn:original_parametrisation}--\eqref{eqn:rescaled_parametrization} is a graph over a domain in the horizontal plane $\{z=0\}$.  
\end{lemma}

\begin{remark}
In particular, Lemma \ref{lem:graphicality} implies that the map $(X,Y,Z)\colon\Omega_\phi\to\R^3$ is an embedding. 
Karcher \cite{KarcherScherk,KarcherTokyo} concludes the embeddedness of his singly periodic minimal surfaces indirectly by analyzing the corresponding conjugate surface. 
The authors of \cite{Bucaj2013} found an alternative, more direct approach to prove the embeddedness of the Karcher--Scherk surfaces with higher dihedral symmetry and remark that a similar approach might also work for the less symmetric surfaces $\tow_\vartheta$ but to the best of our knowledge this has not been established, yet.   
Our argument is yet again different, using only straightforward properties of the Enneper--Weierstrass parametrization \eqref{eqn:complex_integrals}--\eqref{eqn:original_parametrisation}. 
\end{remark}

\begin{proof}
\emph{Claim 1.} $X(w)=0\,\Leftrightarrow\,\Re(w)=0$.

By \eqref{eqn:original_parametrisation}, the condition $X(w)=0$ is equivalent to 
$\abs{w+e^{i\phi}}\abs{w+e^{-i\phi}}
=\abs{w-e^{i\phi}}\abs{w-e^{-i\phi}}$. 
Abbreviating $a\vcentcolon=\Re(w)$, $b\vcentcolon=\Im(w)$, $c\vcentcolon=\cos(\phi)$ and $s\vcentcolon=\sin(\phi)$, we have 
\begin{align}\label{eqn:20220221}
\begin{aligned}
\abs[\big]{w\pm e^{i\phi}}^2
\abs[\big]{w\pm e^{-i\phi}}^2	
&=\Bigl((a\pm c)^2+(b\pm s)^2\Bigr)
\Bigl((a\pm c)^2+(b\mp s)^2\Bigr)
\\
&=(a\pm c)^4+2(b^2+s^2)(a\pm c)^2 
+(b^2-s^2)^2.
\end{aligned}
\end{align}
Hence, $\abs{w+e^{i\phi}}\abs{w+e^{-i\phi}}-\abs{w-e^{i\phi}}\abs{w-e^{-i\phi}}=f\bigl((a+c)^2\bigr)-f\bigl((a-c)^2\bigr)$ where $f\colon\Interval{0,\infty}\to\R$ is given by $f(t)=t^2+2(b^2+s^2)t$. 
Since $f$ is injective, $f\bigl((a+c)^2\bigr)=f\bigl((a-c)^2\bigr)$ is equivalent to $(a+c)^2=(a-c)^2$ which in turn is equivalent to $0=a=\Re(w)$ since by assumption $c=\cos(\phi)\neq0$. 

\emph{Claim 2.}
Let $\Omega_\phi^+=\{w\in\Omega_\phi\st \Re(w)\geq0,~\Im(w)\geq0\}$ be the intersection of the domain $\Omega_\phi$ with the first quadrant. 
Then the level sets of the restriction $X\colon\Omega_\phi^+\to\R$ are connected and $X|_{\Omega_\phi^+}\geq0$.

\begin{figure}%
\centering
\pgfmathsetmacro{\angle}{30}
\pgfmathsetmacro{\phipar}{asin(1-cos(\angle))}
\pgfmathsetmacro{\globalscale}{2.5*2}
\begin{tikzpicture}[line join=round,scale=\globalscale,baseline={(0,0)}]
\begin{scope}[line width=4pt]
\pgfmathsetmacro{\gap}{0}
\draw[color={cmyk,1:magenta,0.5;yellow,1}](0,0)--(0,{1-rad(\gap)})
node[midway,left]{$\gamma_0$};
\draw[color={cmyk,1:magenta,0.5;cyan,1}](0,0)--(1,0)
node[near end,below]{$\gamma_1$}arc(0:-\gap+\phipar:1);
\draw[color={cmyk,1:yellow,1;cyan,0.5}](\gap+\phipar:1)arc(\gap+\phipar:90-\gap:1)
node[midway,above right,inner sep=1pt,circle]{$\gamma_2$};
\end{scope}
\draw[latex-](\phipar:1)++(\phipar:0.005)-|++(0.1,0.33)node[above]{$e^{i\phi}$};
\begin{scope}
\clip(1,0)arc(0:90:1)|-cycle;
\fill[black!10](0,0)rectangle(1,1);
\draw(  0.0763,  0.0000)..controls(  0.0757,  0.1033)and(  0.0804,  0.2518)..(  0.0833,  0.3130)..controls(  0.0919,  0.4756)and(  0.0926,  0.4961)..(  0.1110,  0.6886)..controls(  0.1147,  0.7297)and(  0.1387,  0.9296)..(  0.1490,  0.9887);
\draw(  0.1518,  0.0000)..controls(  0.1516,  0.1133)and(  0.1580,  0.2247)..(  0.1655,  0.3081)..controls(  0.1800,  0.4604)and(  0.1939,  0.5416)..(  0.2145,  0.6499)..controls(  0.2342,  0.7488)and(  0.2614,  0.8593)..(  0.2914,  0.9565);
\draw(  0.2256,  0.0000)..controls(  0.2242,  0.0912)and(  0.2353,  0.2289)..(  0.2457,  0.3008)..controls(  0.2635,  0.4304)and(  0.2800,  0.4961)..(  0.3080,  0.6010)..controls(  0.3447,  0.7270)and(  0.3787,  0.8127)..(  0.4217,  0.9066);
\draw(  0.2972,  0.0000)..controls(  0.2986,  0.1248)and(  0.3059,  0.1829)..(  0.3194,  0.2724)..controls(  0.3344,  0.3667)and(  0.3644,  0.4791)..(  0.4017,  0.5764)..controls(  0.4508,  0.6989)and(  0.4795,  0.7508)..(  0.5364,  0.8439);
\draw(  0.3659,  0.0000)..controls(  0.3659,  0.0977)and(  0.3784,  0.2029)..(  0.3957,  0.2783)..controls(  0.4100,  0.3501)and(  0.4463,  0.4623)..(  0.4953,  0.5599)..controls(  0.5116,  0.5968)and(  0.5829,  0.7158)..(  0.6337,  0.7734);
\draw(  0.4311,  0.0000)..controls(  0.4291,  0.0910)and(  0.4498,  0.2121)..(  0.4646,  0.2644)..controls(  0.4803,  0.3256)and(  0.5082,  0.4071)..(  0.5556,  0.4926)..controls(  0.5900,  0.5535)and(  0.6416,  0.6293)..(  0.7142,  0.6999);
\draw(  0.4926,  0.0000)..controls(  0.4951,  0.0850)and(  0.4964,  0.1068)..(  0.5156,  0.2008)..controls(  0.5341,  0.2849)and(  0.5717,  0.3691)..(  0.6005,  0.4175)..controls(  0.6547,  0.5117)and(  0.7162,  0.5766)..(  0.7785,  0.6273);
\draw(  0.5503,  0.0000)..controls(  0.5482,  0.0574)and(  0.5638,  0.1577)..(  0.5781,  0.2034)..controls(  0.6038,  0.2945)and(  0.6365,  0.3506)..(  0.6740,  0.4053)..controls(  0.7139,  0.4645)and(  0.7745,  0.5215)..(  0.8295,  0.5582);
\draw(  0.6041,  0.0000)..controls(  0.6025,  0.0760)and(  0.6222,  0.1600)..(  0.6313,  0.1880)..controls(  0.6475,  0.2419)and(  0.6744,  0.2998)..(  0.7072,  0.3464)..controls(  0.7709,  0.4285)and(  0.7886,  0.4423)..(  0.8690,  0.4945);
\draw(  0.6536,  0.0000)..controls(  0.6546,  0.0592)and(  0.6611,  0.1120)..(  0.6780,  0.1674)..controls(  0.6994,  0.2363)and(  0.7242,  0.2779)..(  0.7556,  0.3198)..controls(  0.7920,  0.3681)and(  0.8523,  0.4150)..(  0.8993,  0.4369);
\draw(  0.6996,  0.0000)..controls(  0.6971,  0.0402)and(  0.7122,  0.1283)..(  0.7239,  0.1596)..controls(  0.7377,  0.2049)and(  0.7666,  0.2574)..(  0.7885,  0.2825)..controls(  0.8277,  0.3296)and(  0.8475,  0.3482)..(  0.9224,  0.3858);
\draw(  0.7422,  0.0000)..controls(  0.7424,  0.0353)and(  0.7462,  0.0807)..(  0.7583,  0.1285)..controls(  0.7705,  0.1756)and(  0.7968,  0.2270)..(  0.8215,  0.2561)..controls(  0.8456,  0.2880)and(  0.8982,  0.3284)..(  0.9401,  0.3409);
\draw(  0.7819,  0.0000)..controls(  0.7814,  0.0350)and(  0.7876,  0.0789)..(  0.7887,  0.0861)..controls(  0.7925,  0.1284)and(  0.8135,  0.1810)..(  0.8445,  0.2252)..controls(  0.8660,  0.2583)and(  0.9251,  0.2975)..(  0.9531,  0.3026);
\draw(  0.8192,  0.0000)..controls(  0.8197,  0.0418)and(  0.8213,  0.0609)..(  0.8260,  0.0968)..controls(  0.8283,  0.1170)and(  0.8423,  0.1689)..(  0.8713,  0.2078)..controls(  0.8999,  0.2430)and(  0.9274,  0.2583)..(  0.9628,  0.2701);
\draw(  0.8553,  0.0000)..controls(  0.8552,  0.0482)and(  0.8549,  0.0491)..(  0.8581,  0.0971)..controls(  0.8661,  0.1454)and(  0.8777,  0.1706)..(  0.8957,  0.1945)..controls(  0.9221,  0.2237)and(  0.9370,  0.2315)..(  0.9700,  0.2430);
\draw(  0.8917,  0.0000)..controls(  0.8894,  0.0365)and(  0.8862,  0.0912)..(  0.8875,  0.1048)..controls(  0.8891,  0.1344)and(  0.9009,  0.1618)..(  0.9153,  0.1818)..controls(  0.9407,  0.2083)and(  0.9428,  0.2073)..(  0.9754,  0.2202);
\draw(  0.9333,  0.0000)..controls(  0.9249,  0.0419)and(  0.9274,  0.0305)..(  0.9144,  0.0796)..controls(  0.9106,  0.1190)and(  0.9095,  0.1266)..(  0.9250,  0.1609)..controls(  0.9410,  0.1865)and(  0.9527,  0.1922)..(  0.9794,  0.2015);
\draw(  0.9995,  0.0308)..controls(  0.9766,  0.0292)and(  0.9432,  0.0731)..(  0.9463,  0.0715)..controls(  0.9291,  0.1003)and(  0.9285,  0.1291)..(  0.9352,  0.1447)..controls(  0.9418,  0.1637)and(  0.9592,  0.1815)..(  0.9824,  0.1864);
\draw(  0.9973,  0.0732)..controls(  0.9759,  0.0721)and(  0.9555,  0.0907)..(  0.9519,  0.1029)..controls(  0.9453,  0.1138)and(  0.9447,  0.1350)..(  0.9496,  0.1443)..controls(  0.9557,  0.1625)and(  0.9713,  0.1708)..(  0.9847,  0.1743);
\draw(  0.9956,  0.0933)..controls(  0.9812,  0.0925)and(  0.9693,  0.1010)..(  0.9658,  0.1069)..controls(  0.9551,  0.1214)and(  0.9569,  0.1270)..(  0.9602,  0.1439)..controls(  0.9682,  0.1588)and(  0.9744,  0.1610)..(  0.9863,  0.1648);
\draw(  0.9944,  0.1054)..controls(  0.9882,  0.1034)and(  0.9765,  0.1102)..(  0.9742,  0.1127)..controls(  0.9641,  0.1216)and(  0.9665,  0.1313)..(  0.9669,  0.1381)..controls(  0.9719,  0.1532)and(  0.9796,  0.1546)..(  0.9875,  0.1573);
\draw(  0.9935,  0.1135)..controls(  0.9860,  0.1130)and(  0.9785,  0.1167)..(  0.9757,  0.1224)..controls(  0.9717,  0.1283)and(  0.9718,  0.1353)..(  0.9742,  0.1408)..controls(  0.9772,  0.1473)and(  0.9828,  0.1502)..(  0.9884,  0.1516);
\draw(  0.9929,  0.1193)..controls(  0.9877,  0.1186)and(  0.9817,  0.1217)..(  0.9803,  0.1247)..controls(  0.9763,  0.1297)and(  0.9767,  0.1354)..(  0.9789,  0.1399)..controls(  0.9796,  0.1420)and(  0.9859,  0.1477)..(  0.9891,  0.1472);
\draw(  0.9924,  0.1233)..controls(  0.9885,  0.1226)and(  0.9843,  0.1254)..(  0.9835,  0.1267)..controls(  0.9794,  0.1303)and(  0.9812,  0.1369)..(  0.9814,  0.1371)..controls(  0.9837,  0.1421)and(  0.9856,  0.1426)..(  0.9896,  0.1438);
\draw(  0.9920,  0.1262)..controls(  0.9887,  0.1262)and(  0.9880,  0.1265)..(  0.9857,  0.1283)..controls(  0.9833,  0.1310)and(  0.9832,  0.1337)..(  0.9837,  0.1358)..controls(  0.9839,  0.1378)and(  0.9872,  0.1412)..(  0.9900,  0.1413);
\draw(  0.9917,  0.1283)..controls(  0.9899,  0.1279)and(  0.9880,  0.1290)..(  0.9867,  0.1303)..controls(  0.9865,  0.1303)and(  0.9847,  0.1339)..(  0.9855,  0.1349)..controls(  0.9861,  0.1377)and(  0.9881,  0.1390)..(  0.9902,  0.1394);
\draw(  0.9915,  0.1298)..controls(  0.9901,  0.1295)and(  0.9883,  0.1305)..(  0.9876,  0.1315)..controls(  0.9868,  0.1325)and(  0.9866,  0.1345)..(  0.9872,  0.1356)..controls(  0.9885,  0.1376)and(  0.9884,  0.1373)..(  0.9904,  0.1380);
\end{scope}
\end{tikzpicture}
\caption{Level sets of $X$ restricted to the domain $\Omega_\phi^+$ (cf. proof of Lemma \ref{lem:graphicality}, Claim 3).}%
\label{fig:levelsets}%
\end{figure}
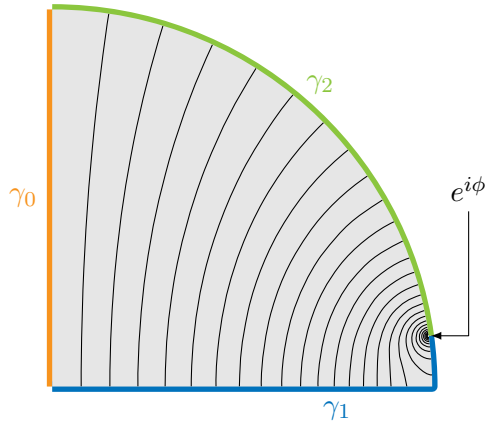

We divide $\partial\Omega_\phi^+\setminus\{i,e^{i\phi}\}=\gamma_0\cup\gamma_1\cup\gamma_2$ into the union of the three piecewise smooth, disjoint subsets 
\begin{align*}
\gamma_0&\vcentcolon=\{w\in\partial\Omega_\phi^+\st \Re(w)=0,~w\neq i\},   \\
\gamma_1&\vcentcolon=\{w\in\partial\Omega_\phi^+\st \Re(w)>0,~\Im(w)=0\}\cup\{e^{i\alpha}\st \alpha\in\interval{0,\phi}\}, \\
\gamma_2&\vcentcolon=\{e^{i\alpha}\st \alpha\in\interval{\phi,\tfrac{\pi}{2}}\}
\end{align*}
as visualized in Figure \ref{fig:levelsets}. 
By Claim 1, $X$ vanishes identically on $\gamma_0$. 
Moreover, we claim that with respect to the canonical (counterclockwise) orientation, $X$ is strictly increasing (from $0$ to $+\infty$) along $\gamma_1$ and strictly decreasing (from $+\infty$ to $0$) along $\gamma_2$. 
Indeed, since $\abs{e^{i\alpha}+e^{\pm i\phi}}=2\cos\bigl((\alpha\mp\phi)/2\bigr)$
and $\abs{e^{i\alpha}-e^{\pm i\phi}}=2\abs[\big]{\sin\bigl((\alpha\mp\phi)/2\bigr)}$ we have 
\begin{align}\label{eqn:X(exp(ia))}
X(e^{i\alpha})
&=(\sin \vartheta)\log\abs[\Big]{ 
\cot\Bigl(\frac{\alpha-\phi}{2}\Bigr)
\cot\Bigl(\frac{\alpha+\phi}{2}\Bigr)
}
\end{align}
using the first formula in \eqref{eqn:original_parametrisation}. 
Then, the monotonicity along  $\gamma_2$ and along the circular arc of $\gamma_1$ is evident from \eqref{eqn:X(exp(ia))}.  
Along the straight piece of $\gamma_1$ the argument $w$ is real-valued in the interval $\intervaL{0,1}$ and for such arguments, 
the first integrand in \eqref{eqn:complex_integrals} is real valued as well with the right sign:  
recalling $\abs{r}<1$ from \eqref{eqn:20211219-r}, we have indeed for any $\zeta\in[0,1]$ 
\begin{align}\label{eqn:X_increasing}
\frac{-(1-r^2)(\zeta^2-1)}{(\zeta^2-e^{2\phi i})(\zeta^2-e^{-2\phi i})}
=\frac{-(1-r^2)(\zeta^2-1)}{\zeta^4-2\zeta^2\cos(2\phi)+1}\geq0.
\end{align}
Since $X\colon\Omega_\phi^+\to\R$ is harmonic, its level sets do not contain any closed curves. 
Hence, appealing to the monotonicity of $X$ along the boundary pieces $\gamma_1$ and $\gamma_2$, each level set must connect a point on $\gamma_1$ to a point on $\gamma_2$. 
In particular, the level sets in question are connected. 
Moreover, we obtain $X\geq0$ on $\partial\Omega_\phi^+$ and hence in all of $\Omega_\phi^+$ by the maximum principle for harmonic functions.  

\emph{Claim 3.} The image of the restricted map $(X,Y,Z)\colon\Omega_\phi^+\to\R^3$ is a graph over a domain in the horizontal plane. 

Towards a contradiction, suppose that there exist $w_0,w_1\in\Omega_\phi^+$ such that $(X,Y)(w_0)=(X,Y)(w_1)$ but $Z(w_0)\neq Z(w_1)$. 
By Sard's theorem we may additionally assume that the image of $(X,Y,Z)$ intersects the vertical plane $P=\{(x,y,z)\in\R^3\st x=X(w_0)\}$ transversally. 
Then, by Claim 2 there exists a smooth curve $\gamma\colon[0,1]\to\Omega_\phi^+$ such that $\gamma(0)=w_0$, $\gamma(1)=w_1$ and such that $X(\gamma(t))$ is constant in $t\in[0,1]$, i.\,e. $(X,Y,Z)\circ\gamma\colon[0,1]\to P$ is a well-defined, smooth curve. 
The mean value theorem yields $t_*\in\interval{0,1}$ such that the derivative of $Y\circ\gamma$ vanishes at $t_*$. 
This implies that the normal vector $N(\gamma(t_*))$ is horizontal. 
However, since $X\circ\gamma$ is constant, $\gamma(t_*)$ must be in the interior of $\Omega_\phi$ by the same monotonicity argument as in the proof of Claim 2 and we obtain a contradiction with Lemma \ref{lem:gaussmap}. 

\emph{Conclusion}. 
For any $w\in\Omega_\phi^+$ we have $\Im(w)\geq0$ by definition and therefore 
\begin{align*}
\abs{w-i}&\leq\abs{w+i},    &
\abs{w-e^{i\phi}}&\leq\abs{w+e^{i\phi}},   &
\abs{w+e^{-i\phi}}&\leq\abs{w-e^{-i\phi}}
\end{align*}
for any $\phi\in\interval{0,\pi/2}$. 
Then, $Y\leq0$ in $\Omega_\phi^+$ follows directly from the expression for $y(w)$ in \eqref{eqn:original_parametrisation}.
Moreover, in Claim 2 we showed $X\geq0$ in $\Omega_\phi^+$.
Hence, the images of $(X,Y,Z)$ and $-(X,Y,Z)$, respectively $(X,-Y,Z)$, restricted to $\Omega_\phi^+$ intersect only along the common boundary of $\Omega_\phi^+$ and $-\Omega_\phi^+$, respectively its complex conjugate. 
The statement then follows from Claim 3 and Lemma~\ref{lem:Karcher-Scherk_symmetries}.
\end{proof}

\begin{lemma}[Exponential asymptotics]\label{lem:exp_asy}
The unit normal vector $N$ of $(X,Y,Z)\colon\Omega_\phi\to\R^3$ defined in  \eqref{eqn:original_parametrisation}--\eqref{eqn:rescaled_parametrization} converges exponentially with rate $1$ along the ends of the image, i.\,e. there exist constants $C,\delta>0$ such that the following implications hold.
\begin{align*}
\begin{aligned}
\abs{w\mp i}&<\delta &&\Rightarrow &
\abs[\big]{N(w)-\lim_{\quad\mathclap{w\to\pm i}\quad}N(w)}&\leq C e^{\pm Y(w)}
\\
\abs{w\mp e^{i\phi}}&<\delta &&\Rightarrow &
\abs[\big]{N(w)-\lim_{\quad\mathclap{w\to\pm e^{i\phi}}\quad}N(w)}&\leq C e^{\mp\sin(\vartheta)X(w)\pm\cos(\vartheta)Y(w)}
\\
\abs{w\mp e^{-i\phi}}&<\delta &&\Rightarrow &
\abs[\big]{N(w)-\lim_{\quad\mathclap{w\to\pm e^{-i\phi}}\quad}N(w)}&\leq C e^{\mp\sin(\vartheta)X(w)\mp\cos(\vartheta)Y(w)}
\end{aligned}
\end{align*}
\end{lemma}

\begin{proof}
We recall that the Enneper--Weierstrass datum $v\colon\Omega_\phi\to\C$ given by $v(w)=(w^2+r)/(1+rw^2)$ in \eqref{eqn:Karcher-data} is the stereographic projection of the Gauss map $N$. 
By Lemma \ref{lem:gaussmap}, the image of $N$ is contained in the southern hemishpere, where the distance between any pair of points is comparable to the distance of their image under stereographic projection. 
Therefore, it suffices to prove exponential convergence of $v(w)$ along the ends. 
Moreover, as in the proof of Lemma \ref{lem:btow_theta}, it suffices to consider the limits $w\to i$ and $w\to e^{i\phi}$ due to the symmetries listed in Lemma \ref{lem:Karcher-Scherk_symmetries}. 

As $w\to i$ we have $v(w)\to -1$ and $Y(w)\to-\infty$. 
Using formulae \eqref{eqn:Karcher-data} and \eqref{eqn:original_parametrisation}, we compute 
\begin{align*}
v(w)+1&=\frac{(w^2+1)(1+r)}{1+r w^2},   
&
e^{-Y(w)}&=\frac{\abs{w+i}}{\abs{w-i}}
\left(\frac{\abs{w+e^{i\phi}}\abs{w-e^{-i\phi}}}{\abs{w-e^{i\phi}}\abs{w+e^{-i\phi}}}\right)^{\cos\vartheta}  
\end{align*}
where $r\in\interval{-1,1}$ was defined in \eqref{eqn:20211219-r}, and obtain  
\begin{align*}
\abs{v(w)+1}e^{-Y(w)}
&=\frac{\abs{1+r}}{\abs{1+rw^2}} \abs{w+i}^2
\left(\frac{\abs{w+e^{i\phi}}\abs{w-e^{-i\phi}}}{\abs{w-e^{i\phi}}\abs{w+e^{-i\phi}}}\right)^{\cos\vartheta}			
\end{align*}
which converges to $4(1+r)/(1-r)$ as $w\to i$. 
In particular, there exist constants $C,\delta>0$ such that 
$\abs{v(w)+1}e^{-Y(w)}\leq C$ whenever $\abs{w-i}<\delta$ and the desired exponential rate of convergence follows. 

As $w\to e^{i\phi}$ we have $v(w)\to(e^{2i\phi}+r)/(1+e^{2i\phi}r)=\vcentcolon\zeta$ and compute 
\begin{align*}
v(w)-\zeta&=\frac{w^2+r-(1+rw^2)\zeta}{1+rw^2}
=\frac{(1-r\zeta)w^2-(\zeta-r)}{1+rw^2}
=\frac{(1-r\zeta)(w^2-e^{2i\phi})}{1+rw^2}
\end{align*}
where we used that $(\zeta-r)=e^{2i\phi}(1-r^2)/(1+e^{2i\phi}r)=e^{2i\phi}(1-r\zeta)$. 
Moreover, 
\begin{align*}
&\sin(\vartheta)X(w)-\cos(\vartheta)Y(w)
\\
&=\sin^2(\vartheta)\log \frac{\abs{w+e^{i\phi}}\abs{w+e^{-i\phi}}}{\abs{w-e^{i\phi}}\abs{w-e^{-i\phi}}}
-\cos(\vartheta)\log \frac{\abs{w-i}}{\abs{w+i}}
-\cos^2(\vartheta)
\log\frac{\abs{w-e^{i\phi}}\abs{w+e^{-i\phi}}}{\abs{w+e^{i\phi}}\abs{w-e^{-i\phi}}}
\\
&=\left(\log\frac{\abs{w+e^{i\phi}}}{\abs{w-e^{i\phi}}}
+\bigl(\sin^2(\vartheta)-\cos^2(\vartheta)\bigr)\log
\frac{\abs{w+e^{-i\phi}}}{\abs{w-e^{-i\phi}}}
-\cos(\vartheta)\log \frac{\abs{w-i}}{\abs{w+i}}\right)
\to\infty  
\end{align*}
as $w\to e^{i\phi}$. 
Using $\sin^2(\vartheta)-\cos^2(\vartheta)=-\cos(2\vartheta)$, we obtain 
\begin{align*}
\abs{v(w)-\zeta}e^{\sin(\vartheta)X(w)-\cos(\vartheta)Y(w)}
&=
\frac{\abs{1-r\zeta}}{\abs{1+rw^2}}
\abs{w+e^{i\phi}}^2 
\left(\frac{\abs{w-e^{-i\phi}}}{\abs{w+e^{-i\phi}}}\right)^{\cos(2\vartheta)}
\left(\frac{\abs{w+i}}{\abs{w-i}}\right)^{\cos(\vartheta)}
\end{align*}
which again stays bounded as $w\to e^{i\phi}$ such that the claimed exponential rate of convergence follows. 
\end{proof}

It turns out that, in order to reconcile the notation in this appendix (which is consistent with \cite{KarcherScherk}) with that employed in the core of the paper (which is much more convenient for our specific purposes), we will have to switch $X$ and $Y$ and so rather consider the map 
\begin{equation}
\label{standardtowerparamPRELIM}
\begin{aligned}
F\colon \Omega_{\phi}
&\to\R^3 \\
w&\mapsto (Y,X,Z)(w).
\end{aligned}
\end{equation}
(To avoid ambiguities, we iterate that the standard coordinates in $\R^3$ are still denoted by $x,y,z$; here we simply agree that the function $Y$ is now placed in the slot of the $x$-coordinate.)

That said, the following propositon ensures that one can actually extend the domain of the parametrization defined in \eqref{eqn:original_parametrisation}--\eqref{eqn:rescaled_parametrization} to the Riemann sphere with six punctures, provided we take the quotient in the target with respect to vertical translations of length an integer multiple of $2\pi$.

\begin{proposition}\label{prop:towerEWsummary}
There exists a well-defined map
\begin{equation}
\label{standardtowerparam}
\begin{aligned}
G\colon(\C \cup \{\infty\})\setminus\{\pm i, \pm e^{i\phi}, \pm e^{-i\phi}\}
&\to\R^3/\sk{\trans^{\axis{z}}_{2\pi}} \\
w&\mapsto (Y,X,Z)(w)
\end{aligned}
\end{equation}
that extends the map $F$ in \eqref{standardtowerparamPRELIM} once we quotient by $\trans^{\axis{z}}_{2\pi}$ in the target and is an embedding, whose image we shall denote by $\widetilde{\Gamma}_{\phi}$.  Moreover,  
\begin{enumerate}[label={\normalfont(\roman*)}]
\item\label{prop:towerEWsummary-i} 
$\widetilde{\Gamma}_{\phi}$ is a complete $\refl_{\{x=0\}}$-invariant minimal surface in the quotient $\R^3/\sk{\trans^{\axis{z}}_{2\pi}}$; 
\item\label{prop:towerEWsummary-ii} 
$\widetilde{\Gamma}_{\phi}$ has genus $0$ and exactly six ends with two ends asymptotic to the cylinder $\{y=0\}$ 
and the remaining four ends asymptotic to $\{y=\pm b^{\mathrm{tow}}_\vartheta + x \tan \vartheta\}$ and $\{y=\pm b^{\mathrm{tow}}_\vartheta - x \tan \vartheta\}$.
\item\label{prop:towerEWsummary-iii} 
$\widetilde{\Gamma}_{\phi}$ meets $\{x=0\} \cap \{z=\pi/2\}$ at a single point, which lies in $\{y>0\}$ 
(cf. Figure \ref{fig:halftower}).
\end{enumerate}
\end{proposition}

The image of $G$ corresponds to a full vertical period of the Karcher--Scherk tower like the gray region of the surface visualized on the right in Figure \ref{fig:Karcher-Scherk-Weierstrass} modulo an identification between the top and the bottom boundary curves. Note that, since we take the quotient with respect to vertical translations, the ends of the image of $G$ are asymptotic to cylinders rather than vertical planes. For later reference (cf. Appendix \ref{app:CylAnalysis} below) it is indeed convenient to also consider the corresponding six half cylinders, namely the quotients of the sets $\{x\geq 0\}$, $\{x\leq 0\}$,
$\{y=b^{\mathrm{tow}}_\vartheta + x \tan \vartheta, x\geq 0\}$, $\{y=-b^{\mathrm{tow}}_\vartheta + x \tan \vartheta, x\leq 0\}$, $\{y= b^{\mathrm{tow}}_\vartheta - x \tan \vartheta, x\leq 0\}$ and $\{y= -b^{\mathrm{tow}}_\vartheta - x \tan \vartheta, x\geq 0\}$. We will generically employ the letter $\Lambda$ to denote any of these.

Lastly, a clarification: both in the previous statement and in the following proof we will still employ $x,y,z$ to denote the \emph{quotiented} coordinates in $\R^3/\sk{\trans^{\axis{z}}_{2\pi}}$.

\begin{proof}
Lemma \ref{lem:vertical_height} and Lemma \ref{lem:gaussmap} imply that the image of $F\colon\Omega_\phi\to\R^3$ is contained in $\{\abs{z}\leq\pi/2\}$ and meets the horizontal planes $\{z=\pi/2\}$ and $\{z=-\pi/2\}$ orthogonally along the boundary pieces of $\Omega_\phi$. 
In particular, these are planes of symmetry: 
we may reflect the image of $F\colon\Omega_\phi\to\R^3$ across the plane $\{z=\pi/2\}$ and translate the resulting surface periodically by $2\pi$ in the vertical direction to obtain a complete, embedded, singly periodic minimal surface henceforth denoted $\Gamma_{\phi}$.  

The parametrization \eqref{eqn:original_parametrisation}--\eqref{eqn:rescaled_parametrization} is manifestly well-defined on the domain of $G$ with branch cuts for each $\Arg(\cdot)$ appearing in the expression for $Z(w)$. 
The branch cuts can then be removed at the cost of taking the quotient, i.\,e. this can be done modulo $2\pi$ in the third component. 
Hence, taking the quotient with respect to vertical translations $\trans^{\axis{z}}_{2\pi}$ in the target, $G$ is well-defined. 

Furthermore,
with the aid of the elementary identities
\begin{equation*}
\log
  \frac{\abs{-\overline{w}^{-1}+\zeta}}
    {\abs{-\overline{w}^{-1}-\zeta}}
=
-\log \frac{\abs{w+\zeta}}{\abs{w-\zeta}},
\quad
\Arg \frac{\zeta+\overline{w}^{-1}}{\zeta-\overline{w}^{-1}}
  -\pi
\equiv
\Arg \frac{\zeta-w}{\zeta+w} \bmod 2\pi
\quad
(\abs{\zeta}=1),
\end{equation*}
it is readily verified directly from
\eqref{eqn:original_parametrisation}
that the first two coordinates of $G(-1/\overline{w})$
coincide with the corresponding ones of $-G(w)$,
while the third one differs from that of $G(w)$
by addition of $\pi$ (modulo $2\pi$, for we are working in the quotient).

We conclude from Lemma \ref{lem:graphicality} (and Lemma \ref{lem:vertical_height}) that $G$ is actually a (conformal) diffeomorphism onto its image.
In particular, $\widetilde{\Gamma}_{\phi}$ has genus $0$ because it is parametrized over the punctured Riemann sphere. 
By the unique continuation property for minimal surfaces, $\widetilde{\Gamma}_{\phi}$ coincides with $ \Gamma_{\phi}/\sk{\trans^{\axis{z}}_{2\pi}}$. 
Hence, up to switching $x$ and $y$, it has the same asymptotic behavior as described in Lemma \ref{lem:btow_theta}. 

It follows from the embeddedness of $\widetilde{\Gamma}_{\phi}$ together with the symmetry \eqref{eqn:symmetry_under_complex_conj} and Lemma \ref{lem:vertical_height} that the inverse image under $G$ of $\{x=0\} \cap \{z=\pi/2\}$ is just the singleton set $\{1\}$. 
We recall that \eqref{eqn:X_increasing} implies $X(1)>X(0)=0$ and hence $G(1)\in\{y>0\}$.  
\end{proof}

With slight yet convenient abuse of notation, in the core of the present paper (in particular in Section \ref{sec:Initial}) 
we will \emph{a posteriori} modify the target of the map $G$ above to $\widetilde{\Gamma}_{\phi}=\towquot_{\vartheta}$ (where the link between the angles $\phi$ and $\vartheta$ is given by equation \eqref{eqn:theta_phi}) so that it upgrades to a conformal diffeomorphism.

\section{Analysis on asymptotically cylindrical surfaces}\label{app:CylAnalysis}

We shall start here with the discussion of the most basic elliptic boundary value problem on a half cylinder. 
From \eqref{towquotient} we recall that $\towquot_{\vartheta}$ denotes the quotient of the Karcher--Scherk tower with respect to vertical translations (of length an integer multiple of $2\pi$). 
As we are about to see, the ends of $\towquot_{\vartheta}$ approach the corresponding asymptotic cylinders at \emph{exponential} rate, and thus this model case is arbitrarily well approximated when we deal with the actual analysis on the (bent) tower.

\begin{lemma}[Poisson problem on the half cylinder with Dirichlet data]
\label{lem:cylsol}
Let $\alpha,\beta \in \interval{0,1}$ and let $\cyl$ be the upper unit cylinder, equipped with standard coordinates $(\theta,\rho)\in\Sp^1\times\Interval{0,\infty}$.   
\begin{enumerate}[label={\normalfont(\roman*)}]
\item\label{lem:cylsol-i} 
For any data $E \in C^{0,\alpha}(\cyl,e^{-\beta \rho })$
and $f \in C^{2,\alpha}(\partial \cyl)$
there is a unique bounded function $u$ on $\cyl$
such that $\Delta_\cyl u = E$ and $u|_{\partial \cyl}=f$;
moreover there exists a unique real number $\mu$
such that
\begin{equation}
\abs{\mu}+\nm{u-\mu: C^{2,\alpha}(\cyl,e^{-\beta \rho })}
\leq
C\left(\nm{E: C^{0,\alpha}(\cyl,e^{-\beta \rho })}
       +\nm{f: C^{2,\alpha}(\partial \cyl)}
\right)
\end{equation}
for some constant $C>0$ which is independent of the data $E$ and $f$. 
\item\label{lem:cylsol-ii}
In particular there exists a unique bounded linear map
$P_\cyl\colon C^{0,\alpha}(\cyl,e^{-\beta \rho }) \to C^{2,\alpha}(\cyl,e^{-\beta \rho })$
such that for all $E \in C^{0,\alpha}(\cyl,e^{-\beta \rho })$ we have $\Delta_\cyl (P_\cyl E)=E$ and $(P_\cyl E)|_{\partial \cyl}$ is a constant depending on $E$.
\end{enumerate}
\end{lemma}

Before we proceed with the proofs, let us convene on the lighter notation to be employed (cf. Section~\ref{sec:Nota}) and recall some basic facts on weighted H\"older spaces on manifolds with cylindrical ends. 
Recalling definition \eqref{eqn:definition_weighted_norm} we set $\nm{{}\cdot{}}_{k,\alpha,\beta}\vcentcolon=\nm{{}\cdot{}: C^{k,\alpha}(\cyl,e^{-\beta \rho })}$
and $C^{k,\alpha,\beta}(\cyl)\vcentcolon=C^{k,\alpha}(\cyl,e^{-\beta \rho })$. 
For any $t_2\geq t_1\geq 0$ we shall further define the sets 
\begin{align*}
\cyl(t_1,t_2)\vcentcolon=\cyl \cap \{t_1 \leq \rho  \leq t_2\}.
\end{align*}
The weighted Hölder spaces and norms are well-defined for any real $\beta$. 
Given any $\beta_1<\beta_2$ and $\alpha_1,\alpha_2\in\interval{0,1}$ as well as nonnegative integers $k_1 \leq k_2$ such that $k_1+\alpha_1<k_2+\alpha_2$, we have compactness of the embedding 
$C^{k_2,\alpha_2,\beta_2}(\cyl)\hookrightarrow C^{k_1,\alpha_1,\beta_1}(\cyl)$.
Indeed, suppose $\{v_n\}$ is a bounded sequence in $C^{k_2,\alpha_2,\beta_2}(\cyl)$. 
Then by compactness of the embedding $C^{k_2,\alpha_2}(\cyl(0,T)) \hookrightarrow C^{k_1,\alpha_1}(\cyl(0,T))$
for each $T>0$, a standard diagonal argument gives that
$\{v_n\}$ has a subsequence $\{w_n\}$
such that $\{w_n|_{\cyl(0,N)}\}$
converges in $C^{k_1,\alpha_1}(\cyl(0,N))$
for every integer $N \geq 1$.
On the other hand, clearly
\begin{equation}
\nm{w_n|_{\cyl\setminus \cyl(0,N)}}_{k_1,\alpha_1,\beta_1}
\leq
e^{N(\beta_1-\beta_2)}
\sup \{\nm{v_j}_{k_2,\alpha_2,\beta_2}\}
\end{equation}
for all $n$.
Together these facts imply
that $\{w_n\}$ is Cauchy in $C^{k_1,\alpha_1,\beta_1}(\cyl)$, and thus converges in that Banach space.

For any $\alpha \in \interval{0,1}$, $\beta \in \R$ and $u \in C^{2,\alpha}_{\mathrm{loc}}(\cyl)$ the weighted Schauder estimate 
\begin{equation}
\label{weightedschauderoncyl}
\nm{u}_{2,\alpha,\beta}
\leq C\bigl(\nm{u}_{0,0,\beta}
  +\nm{\Delta_\cyl u}_{0,\alpha,\beta}
  +\nm{u|_{\partial \cyl}}_{2,\alpha}
\bigr)
\end{equation}
holds for some constant $C>0$ depending on $\alpha,\beta$ but not on $u$.
Indeed, by standard Schauder estimates there exists a constant $K=K(\alpha,\beta)>0$ such that for any $u \in C^{2,\alpha}_{\mathrm{loc}}(\cyl)$ we have 
\begin{align*}
\nm{u|_{\cyl(t+1,t+3)}}_{2,\alpha}
&\leq K\left(\nm{u|_{\cyl(t,t+4)}}_0
    +\nm{\Delta_\cyl u|_{\cyl(t,t+4)}}_{0,\alpha}
  \right)
  \mbox{ for all } t \geq 0,
  \\
\nm{u|_{\cyl(0,3)}}_{2,\alpha}
&\leq K\left(\nm{u|_{\cyl(0,4)}}_0
    +\nm{\Delta_\cyl u|_{\cyl(0,4)}}_{0,\alpha}
    +\nm{u|_{\partial \cyl}}_{2,\alpha}
  \right),
\end{align*}
whence we obtain \eqref{weightedschauderoncyl} with $C=Ke^{3\beta}>0$.

\begin{proof}[Proof of Lemma \ref{lem:cylsol}]
We first verify that \ref{lem:cylsol-i} implies \ref{lem:cylsol-ii}.
Given $E$ we set $f\vcentcolon=0$, and take $u$ and $\mu$ as guaranteed by \ref{lem:cylsol-i}. 
Then $P_\cyl E\vcentcolon=u-\mu$ has the desired properties. 
If $P_1$ and $P_2$ are two such operators, then for any datum $E$ the functions $P_1E-(P_1E)|_{\partial \cyl}$ and $P_2E-(P_2E)|_{\partial \cyl}$ both solve the Poisson problem $\Delta_{\Lambda} u=E$ with trivial boundary data,
so the uniqueness claim in \ref{lem:cylsol-i} implies that
\[
P_1E-P_2 E= (P_1E)|_{\partial \cyl}-(P_2E)|_{\partial \cyl}
\]
but the exponential decay of the left-hand side forces $(P_1E)|_{\partial \cyl}=(P_2E)|_{\partial \cyl}$ hence, in turn, $P_1E=P_2 E$. Thus
in the end we conclude that
$P_1=P_2$.

We then turn our attention to \ref{lem:cylsol-i} for the remainder of the proof and start with the uniqueness claim. 
If $u_1$ and $u_2$ are two bounded functions on $\cyl$ satisfying $u_1{}|_{\partial \cyl}=u_2{}|_{\partial \cyl}$ and $\Delta_\cyl u_1 = \Delta_\cyl u_2$,
then their difference $u\vcentcolon=u_1-u_2$ is a bounded harmonic function vanishing on $\partial \cyl$. 
Moreover, $u$ defines a bounded harmonic function on the upper half plane which is periodic in the horizontal direction and vanishes on the boundary. 
By the reflection principle for harmonic functions, we can then extend it by odd reflection to a bounded entire harmonic function on $\R^2$. 
Liouville's theorem then implies that this function must be constant. 
Since $u$ vanishes on a line, it follows that it vanishes everywhere, establishing the asserted uniqueness.

For existence with the claimed estimates we first reduce as follows to the case where the datum $E$ is compactly supported. 
Recalling the notation \eqref{eqn:definition_cutoff} for cutoff functions we set $E_n\vcentcolon=(\cutoff{n+1}{n} \circ \rho )E$ for each $n\in\N$, so that each $E_n$ is compactly supported, 
$E_n\to E$ in $C^0$ and $\nm{E_n}_{0,\alpha,\beta} \leq C\nm{E}_{0,\alpha,\beta}$
for a constant $C>0$ independent of $E$ and $n$. 
Assuming that the claims of the present lemma hold for compactly supported $E$, we then obtain a sequence of functions
$u_n$ on $\cyl$ and a sequence of reals $\mu_n$ solving $\Delta_\cyl u_n = E_n$ with $u_n|_{\partial \cyl}=f$ and satisfying the estimate
\[
\abs{\mu_n}+\nm{u_n-\mu_n}_{2,\alpha,\beta}
\leq  C\bigl(\nm{E}_{0,\alpha,\beta} + \nm{f}_{2,\alpha}\bigr).
\]
By compactness of the embedding $C^{2,\alpha,\beta}(\cyl) \times \R \hookrightarrow C^2(\cyl) \times \R$, 
the sequence $\{(u_n-\mu_n, \mu_n)\}$ has a converging subsequence (which we do not rename), namely 
$\mu_n \to \mu$ in $\R$ and $u_n-\mu_n \to v$ in $C^2$ for some $\mu \in \R$ and $v \in C^{2,\alpha,\beta}(\cyl)$.
Then, the pair $u\vcentcolon=v+\mu$ and $\mu$ satisfies all the claims in~\ref{lem:cylsol-i}.
Thus we may indeed assume that $E$ is compactly supported.
A conformal change transforms the given Poisson problem to a Poisson problem on the unit disc (with the same boundary data but new interior data which nevertheless are $C^{0,\alpha}$ on the disc), so a bounded solution $u \in C^{2,\alpha}_{\mathrm{loc}}(\cyl)$ exists.  

For the following estimates, we will make use of the inequality
\begin{equation}
\label{highfrequencyestimate}
\nm{u}_{C^0}
\leq
C\lambda_k^{-\frac{2}{2+\dim M}}\nm{du}_{C^0},
\end{equation}
which holds for any $C^1$ function $u$ on a closed Riemannian manifold $(M,g)$
provided $u$ is $L^2(g)$-orthogonal
to the direct sum of the eigenspaces corresponding to the first
$k-1$ eigenvalues of the Laplacian on $(M,g)$, henceforth denoted $V\vcentcolon=V_1\oplus V_2 \oplus \ldots\oplus V_{k-1}$
assuming $k \geq 2$;
here $\lambda_k$ is the $k$\textsuperscript{th} eigenvalue, and $C>0$ is a constant depending on just $M$ (its volume and curvature).
This inequality follows immediately from the variational characterization of the eigenvalues in concert with the bound 
$\nm{u}_{C^0}^{2+\dim M}\leq C\nm{u}_{L^2}^2\nm{du}_{C^0}^{\dim M}$,
which can itself be established by bounding below $\abs{u}$, in terms of $\nm{du}_{C^0}$  (assumed nonzero since $k \geq 2$),
on a neighborhood of a point where it is maximized, to get in turn a lower bound on $\nm{u}_{L^2}$.

Applying \eqref{highfrequencyestimate} with $M=\Sp^1$ (isometric to each cross section of $\cyl$) in \eqref{weightedschauderoncyl}
and taking $k$ sufficiently large in terms of the universal constants appearing in the two estimates, we obtain
\begin{equation}
\label{highfrequencyschauder}
\nm{u}_{2,\alpha,\beta}
\leq C\bigl(\nm{\Delta_\cyl u}_{0,\alpha,\beta}
  +\nm{u|_{\partial \cyl}}_{2,\alpha}\bigr)
\end{equation}
provided that $u|_{\{\rho =t\}}$ is $L^2(\Sp^1)$ orthogonal to the subspace $V$ for all $t \geq 0$.

Let $\{e_n\}_{n \in \Z}$ be an Hilbertian basis of $L^2(\Sp^1,\R)$, consisting of eigenfunctions of $\Delta_{\Sp^1}$ so that, in particular $\Delta_{\Sp^1}e_n=-n^2e_n$. 
For each integer $n$ and any bounded continuous function $v$ on $\cyl$ we define the function $v_n\colon\Interval{0,\infty}\to\R$ by
\begin{equation}
\label{fouriercoeffdef}
v_n(t)\vcentcolon=\int_{\{\rho =t\}} (e_n \circ \theta)v|_{\{\rho =t\}}.
\end{equation}
By the Parseval identity and the H\"older inequality on $\Sp^1$, we get
\begin{equation}\label{eq:Parseval}
\sum_{n\in\Z}\abs{v_n(t)}^2=\nm{v(t)}^2_{L^2}\leq 2\pi\nm{v(t)}^2_{L^\infty}
\end{equation}
where we have denoted by $v(t)$ the restriction of the function $v$ to the set $\{\rho =t\}$; in particular, if $\abs{v(t)}$ is bounded by a constant $C$ (or, respectively, by $Ce^{-\beta t}$) then, apart from a multiplicative factor $\sqrt{2\pi}$ the same bound holds true for $v_n(t)$, for any $n\in\Z$. 
We further note that one can derive a H\"older bound on each function $v_n$ given a corresponding bound on $v$. 
In particular, from \eqref{eq:Parseval} it follows at once that for any $n\in\Z$ there holds
\begin{equation}\label{eq:HolderParsF}
\nm{f_n(e_n\circ\theta)}_{2,\alpha}\leq C \nm{f}_{2,\alpha},
\end{equation}
with the Fourier coefficients $f_n$ of $f$ defined in the obvious way, 
and again directly from \eqref{eq:Parseval} together with the very definition \eqref{fouriercoeffdef} we further get
\begin{equation}\label{eq:HolderParsE}
\nm{(E_n\circ \rho )(e_n \circ \theta)}_{0,\alpha,\beta}\leq C\nm{E}_{0,\alpha,\beta}
\end{equation}
when considering the decomposition of the datum $E$ instead.
Furthermore, performing this spectral decomposition for the function $u$ we have that $u_n \in C^{2,\alpha}_{\mathrm{loc}}(\Interval{0,\infty})$ for any $n\in\Z$ and
\begin{equation}
\label{ODEsys}
\begin{cases}
\ddot{u}_n-n^2u_n=E_n,  & (\text{equivalently: }
\Delta_\cyl\bigl((u_n \circ \rho )(e_n \circ \theta)\bigr)
  =(E_n \circ \rho )(e_n \circ \theta),)  \\
u_n(0)=f_n, \\
u_n \mbox{ is bounded.}
\end{cases}
\end{equation}
For any function $v$ on the cylinder $\cyl$, later to be specified to be $u$ or $E$, and the usual boundary datum $f$ we introduce the decompositions $v=v_{\mathrm{low}}+v_{\mathrm{high}}$ and $f=f_{\mathrm{low}}+f_{\mathrm{high}}$, where 
\begin{align*}
v_{\mathrm{low}}(\theta,\rho)&\vcentcolon=\sum_{\abs{n}<N} (v_n \circ \rho )(e_n \circ \theta), 
&
f_{\mathrm{low}}(\theta)&\vcentcolon=\sum_{\abs{n}<N} f_n(e_n \circ \theta),
\end{align*}
where $v_{\mathrm{high}}$ and $f_{\mathrm{high}}$ are, in turn, defined by these four equalities and $N$ has been chosen large enough
(in terms of universal constants only, independently of $u$) 
so that estimate \eqref{highfrequencyschauder} holds with $u_{\mathrm{high}}$ in place of $u$.
Then, by \eqref{ODEsys}, we get at once that $\Delta_\cyl u_{\mathrm{low}}=E_{\mathrm{low}}$ and $u_{\mathrm{low}}|_{\partial \cyl}=f_{\mathrm{low}}$, hence by linearity
\[\left\{
\begin{aligned}
\Delta_\cyl u_{\mathrm{high}}
  &= E_{\mathrm{high}}, \\
u_{\mathrm{high}}|_{\partial \cyl}
  &=f_{\mathrm{high}}.
\end{aligned}\right.
\]
Therefore, since the triangle inequality gives
\begin{align}\notag
\nm{E_{\mathrm{high}}}_{0,\alpha,\beta}
&\leq\nm{E}_{0,\alpha,\beta}+\nm{E_{\mathrm{low}}}_{0,\alpha,\beta}, & 
\nm{f_{\mathrm{high}}}_{2,\alpha}
&\leq\nm{f}_{2,\alpha}+\nm{f_{\mathrm{low}}}_{2,\alpha}, 
\intertext{it suffices to note (from \eqref{eq:HolderParsE} and \eqref{eq:HolderParsF}) that}
\label{eq:LowModeData}
\nm{E_{\mathrm{low}}}_{0,\alpha,\beta}
&\leq C(N)\nm{E}_{0,\alpha,\beta}, &
\nm{f_{\mathrm{low}}}_{2,\alpha}
&\leq C(N)\nm{f}_{2,\alpha},
\end{align}
to finally conclude, appealing to \eqref{highfrequencyschauder}, that
\begin{equation}\label{eq:EstimateHighModes} 
\nm{u_{\mathrm{high}}}_{2,\alpha,\beta}
\leq C(\nm{E}_{0,\alpha,\beta}+\nm{f}_{2,\alpha}).
\end{equation}
It remains to estimate the finite sum $u_{\mathrm{low}}$, for which it obviously suffices to estimate each $u_n$ with $\abs{n}<N$, 
and by virtue of \eqref{weightedschauderoncyl} and \eqref{eq:LowModeData} it in fact suffices to suitably bound $\abs{e^{\beta t}u_n(t)}$.
The boundary value problem \eqref{ODEsys} 
(with boundedness a sort of boundary condition at infinity)
can easily be solved explicitly in terms of $E_n$ and $f_n$,
from which expressions, and the bounds \eqref{eq:HolderParsE}, \eqref{eq:HolderParsF},
one can obtain the required estimate. 
For $n \neq 0$ one can alternatively apply the maximum principle, comparing $u_n$ with the function
\begin{equation}
v(t)\vcentcolon=\frac{\sqrt{2\pi}}{(1-\beta^2)}\bigl(\nm{f}_0+\nm{E}_{0,0,\beta}\bigr)e^{-\beta t}.
\end{equation}
Then $u_n(0) \leq v(0)$ and $\ddot{v}-n^2v \leq E_n$ pointwise provided $\abs{n} \geq 1$. 
Since $E_n$ is compactly supported (because so is $E$) we have $\ddot{u}_n(t)-n^2u_n(t)=0$ for all sufficiently large $t$ and since $u_n$ is bounded, we obtain $u_n(t)=C_n e^{-n t}$ for some real $C_n$ and all $t$ sufficiently large.  
Then it follows, again assuming $\abs{n} \geq 1$,  that $u_n(L)<v(L)$ for $L$ sufficiently large.
By the maximum principle applied on $[0,L]$ (for a fixed, large value of $L$) we then have
$u_n(t) \leq v(t)$ pointwise, 
and analogously we have $u_n(t) \geq -v(t)$,
but in conjunction with \eqref{weightedschauderoncyl} these inequalities imply the desired estimate for $u_n$.

Finally to dispense with $u_0$ (which, we recall, is bounded in $C^0$) we identify $\mu$ as the constant value it takes outside the support of $E_0$ and we appeal to the explicit expression of the solution, that is 
\begin{align}\notag
u_0(t)&=f_0+\int_0^t \int_0^s E_0(\tau) \, d\tau \, ds  -t\int_0^\infty E_0(s) \, ds,
\intertext{which can be rewritten (by just changing the order of integration in the double integral) as}
\label{eq:0modeSol}
u_0(t)&=f_0 - \int_0^t sE_0(s) \, ds - t\int_t^\infty E_0(s) \, ds.
\intertext{Letting $t\to+\infty$ in \eqref{eq:0modeSol} we get indeed}\notag
\mu&=f_0 - \int_0^\infty sE_0(s) \, ds
\shortintertext{hence}\notag
u_0(t)-\mu&=\int_t^\infty (s-t)E_0(s) \, ds,
\end{align}
yielding the remaining bound on
$\abs{\mu}+\sup \abs{e^{\beta t}(u_0(t)-\mu)}$, and thereby
ending the proof.
\end{proof}

For the following Lemma, which states that the asymptotics of $\towquot_{\vartheta}$ are exponential with rate $1$, we recall the notation \eqref{eqn:definition_weighted_norm} for weighted Hölder norms. 

\begin{lemma}\label{lem:CylMSE}
Let $W_0$ be any end of $\towquot_{\vartheta}$ and let $\cyl$ be its asymptotic half cylinder as described in Proposition \ref{prop:towerEWsummary}~\ref{prop:towerEWsummary-ii}, which we equip with standard cylindrical coordinates $(\theta,\rho)\in\Sp^1\times\Interval{0,\infty}$ and a unit normal vector field $\nu$ in $\R^3/\sk{\trans^{\axis{z}}_{2\pi}}$. 
If $R$ is chosen sufficiently large, then $W_0\supset\graph(w\nu)$, where the defining function $w\colon\Sp^1\times\Interval{R,\infty}\to\R$ satisfies 
\begin{align*}
\nm[\big]{w:C^{k}(\Sp^1\times\Interval{R,\infty},\,e^{-\rho})}&\leq C(k)
\end{align*}
for any $k\in\N$ and some finite constant $C(k)$. We then let $W\vcentcolon=\graph(w\nu)$ and refer to it as a wing of $\towquot_{\vartheta}$.
\end{lemma}

\begin{proof}
It follows from the analysis presented in Appendix \ref{app:Karcher--Scherk}, and most importantly from Lemma~\ref{lem:exp_asy}, which states that the unit normal vector along $W$ converges for $\rho\to\infty$ with exponential rate $1$, that $W$ is globally the normal graph of a smooth function $w$, provided that $R$ is sufficiently large, and
 the gradient of the defining function decays exponentially along $\cyl$ so that $\abs{\partial w/\partial\rho}\leq C e^{-\rho}$ for all $\rho\geq R$. 
This implies the desired bounds (for $k=1$) on $w$ since we know from Proposition \ref{prop:towerEWsummary}~\ref{prop:towerEWsummary-ii} that $w(\theta,\rho)$ must decay to zero for $\rho\to\infty$. 
The higher-order bounds then follow from minimality, i.\,e. exploiting the minimal surface equation in a standard fashion.
\end{proof}

We exploit the previous results to discuss solvability of the Poisson problem along an end $W$ of the actual (quotient) tower $\towquot$. 
By Lemma~\ref{lem:CylMSE}, $W$ is a normal graph over its asymptotic half cylinder $\cyl$ in $\R^3/\sk{\trans^{\axis{z}}_{2\pi}}$ 
with coordinates $(\theta,\rho)\in\Sp^1\times\Interval{R,\infty}$ and defining function $w\in C^{k}(\cyl,e^{-\rho})$ provided that $R>0$ is chosen sufficiently large. 
In the following, we implicitly pull back all functions and operators on $W$ to $\Sp^1\times\Interval{R,\infty}$. 
For example, given the canonical parametrization $\varphi\colon\Sp^1\times\Interval{R,\infty}\to W$ 
defined by $\varphi(p)=p+w(p)\nu(p)$ where $\nu$ denotes the unit normal along $\Lambda$, and given a function $v\colon\Sp^1\times\Interval{R,\infty}\to\R$
we simply write $\Delta_W v$ rather than $\varphi^*\Delta_W(v\circ \varphi^{-1})$ when applying the Laplace--Beltrami operator on $W$.  
Since the defining function $w$ is exponentially decreasing
there exists a constant $C$ such that
for all $\beta,\gamma \in \R$
with $\beta<\gamma+2$
and for all
$
 v\in C^{2,\alpha}(\Sp^1\times\Interval{R,\infty},
 e^{-\gamma\rho})
$
\begin{align}
\label{eqn:20220405-1}
\nm{\Delta_W v-\Delta_\cyl v}_{0,\alpha,\beta}
\leq C e^{(\beta-\gamma-2)R}\nm{v}_{2,\alpha,\gamma}.
\end{align}
This estimate follows from the standard formula for the Laplace--Beltrami operator in local coordinates. 
Indeed, with respect to the coordinates $(x^1,x^2)=(\theta,\rho)\in\Sp^1\times\Interval{R,\infty}$, the Riemannian metric $g$ on $W$ and its inverse are  given by 
\begin{align*}
g_{ij}&=\delta_{ij}+\partial_i w\,\partial_j w, &
g^{ij}&=\delta_{ij}-\frac{\partial_i w\,\partial_j w}{1+\abs{\nabla_{\Lambda} w}^2}
\end{align*}
which means that $g_{ij}-\delta_{ij}$
and all its coordinate derivatives
decay like $e^{-2\rho}$;
the estimate then follows in view of the identity
\begin{align*}
\Delta_W v-\Delta_\Lambda v	
&=g^{ij}\biggl(\frac{\partial^2 v}{\partial x^i\partial x^j}
-\Gamma^{k}_{ij}\frac{\partial v}{\partial x^k}\biggr)
-\delta_{ij}\frac{\partial^2 v}{\partial x^i\partial x^j}. 
\end{align*}

\begin{corollary}\label{cor:wingsol}
 Let $R>0$ and $W$ as in Lemma \ref{lem:CylMSE}, and let $\alpha,\beta\in\interval{0,1}$.
If $R$ is sufficiently large, then:
\begin{enumerate}[label={\normalfont(\roman*)}]
\item\label{cor:wingsol-i} 
For any data $E \in C^{0,\alpha}(W,e^{-\beta\rho})$ and $f \in C^{2,\alpha}(\partial W)$ there is a unique bounded $u\colon W\to\R$ such that $\Delta_W u = E$ and $u|_{\partial W}=f$;
moreover there exists a unique $\mu\in\R$ such that
\begin{align*}
\abs{\mu}+\nm{u-\mu}_{2,\alpha,\beta}
&\leq C\Bigl(\nm{E}_{0,\alpha,\beta}+\nm{f}_{2,\alpha}\Bigr)
\end{align*}
for some constant $C>0$ which is independent of the data $E$ and $f$. 
\item\label{cor:wingsol-ii}
There exists a unique bounded linear map
$P_W\colon C^{0,\alpha}(W,e^{-\beta\rho})\to C^{2,\alpha}(W,e^{-\beta\rho})$
such that for all $E \in C^{0,\alpha}(W,e^{-\beta\rho})$ we have $\Delta_W(P_W E)=E$ and $(P_W E)|_{\partial W}$ is a constant depending on $E$.
\end{enumerate}
\end{corollary}

\begin{proof}
Exactly as in the proof of Lemma \ref{lem:cylsol}, one can show at once that \ref{cor:wingsol-i} implies \ref{cor:wingsol-ii}, so it suffices to prove \ref{cor:wingsol-i}.  For the sake of convenience, let us set $\cyl_R\vcentcolon=\cyl(R,\infty)$.
Item \ref{lem:cylsol-i} of Lemma \ref{lem:cylsol}
asserts, in particular, invertibility of the map
\begin{align*}
T_{\cyl_R}
  \colon
  \R
    \oplus C^{2,\alpha,\beta}(\cyl_R)
  &\to
  C^{0,\alpha,\beta}(\cyl_R)
    \oplus  C^{1,\alpha}(\partial \cyl_R)
\\
(\mu, v)
  &\mapsto
  \bigl(
    \Delta_{\cyl} (v+\mu),~
    (v+\mu)|_{\partial \cyl_R}
  \bigr);
\end{align*}
if we now define the map $T_W$
in exactly the same way as $T_{\cyl_R}$
but with each instance of $\cyl_R$
replaced by $W$,
then \eqref{eqn:20220405-1}
with $\gamma=\beta$
ensures that
the operator norm of
$
 (1,\varphi^*)
   T_W
   (\varphi^{-1*},\varphi|_{\partial \cyl_R}^{-1*})
 -T_{\cyl_R}
$
is bounded by $C e^{-2R}$.
We therefore conclude that $T_W$
is itself invertible
for all sufficiently large $R$.
For this we are also using the fact that,
by virtue of the $C^k$ estimates
for $g_{ij}-\delta_{ij}$ preceding
the statement of the corollary,
taking $R$ sufficiently large
in terms of a universal constant
guarantees that
$
 \nm{v}_{k,\alpha,\beta}
 \leq 2\nm{\varphi^*v}_{k,\alpha,\beta}
 \leq 4\nm{v}_{k,\alpha,\beta}
$
for each $k=0,1,2$
and every $v: W \to \R$.

Thus we have confirmed in item \ref{cor:wingsol-i}
existence, uniqueness within the domain of $T_W$,
and the claimed estimate.
To complete the proof
it remains only to establish
that any bounded solution $u$
of $\Delta_Wu=E$ with $u|_{\partial W}=f$
actually belongs to the domain of $T_W$
whenever $(E,f)$ belongs to its target,
but this follows from the analogous assertion
of item \ref{lem:cylsol-i} of Lemma \ref{lem:cylsol}
for $\Delta_\cyl$ by again appealing to
\eqref{eqn:20220405-1},
now with $\gamma=0$.
(Note that uniqueness among bounded solutions
is now a consequence of uniqueness
within the domain of $T_W$.)
\end{proof}

\section{Graphical deformations of immersed hypersurfaces}
\label{app:graphs}
Let $(M,g)$ be a geodesically complete smooth Riemannian manifold,
let $\phi\colon\Sigma \to M$ be a smooth two-sided immersion
of a smooth manifold $\Sigma$ without boundary,
with $\dim M = 1 + \dim \Sigma$,
and let $\nu_g$,
a section of $\phi^*TM$,
be a unit normal
to $\phi$ with respect to $g$.
We shall consider the scalar-valued
second fundamental form
$A[\phi,g,\nu_g]$ of $\phi$ with respect to $g$
and $\nu_g$ and, correspondingly, the
scalar-valued mean curvature 
$H[\phi,g,\nu_g]$, which
is then the $\phi^*g$ trace of
$A[\phi,g,\nu]$.
Note that signs in these definitions are chosen so
to enforce agreement with the conventions discussed in Section \ref{sec:Nota} (cf. equation \eqref{eq:PointwiseJacobi}).

For any function $u\colon \Sigma \to \R$, we consider the corresponding normal graph over $\Sigma$, i.\,e.
we define the deformed map
$\phi[u,g,\nu_g]\colon \Sigma \to M$ by 
\begin{equation}
\label{deformed_phi}
\phi[u,g,\nu_g](p)
\vcentcolon=
\exp^{(M,g)}_{\phi(p)} u(p)\nu_g(p),
\end{equation}
where $\exp^{(M,g)}\colon TM \to M$
is the exponential map of $(M,g)$.
For $u$ suitably small,
$\phi[u,g,\nu_g]$
is also an immersion,
and in this appendix we are interested
in the variation
with respect to $\phi$ and $u$
of the induced metric
and mean curvature
(possibly with respect to an ambient metric
different from the metric $g$ used to define
$\phi[u,g,\nu_g]$),
as well as of the set $\phi[u,g,\nu_g]^{-1}(\Gamma)$
and corresponding intersection angle,
under special conditions,
for a given embedded hypersurface
$\Gamma \subset M$.

\begin{remark}
\label{scaling_of_MC_and_deformed_map}
Note that the following scaling laws hold true for any $\lambda>0$:
\[
\nu_{\lambda^2g}=\lambda^{-1}\nu_g, \
H[\phi,\lambda^2g,\nu_{\lambda^2g}] =\lambda^{-1}H[\phi,g,\nu], \
\phi[\lambda u, \lambda^2g, \nu_{\lambda^2g}]=\phi[u,g,\nu_g].
\]
\end{remark}

Of course one could also consider variations
with respect to $g$,
and indeed variations with respect to $\phi$
can be reduced to the former
via a suitable diffeomorphism,
but we have no need for such a level of generality.
In fact, while it would be possible to treat
all cases of interest to us in a unified way,
the discussion would become excessively complicated
for our purposes,
so we rather split the results we need
between
Lemma \ref{lem:mc_and_met_of_graphs_euc}
and
Lemma \ref{basic_MC_bdy_lemma}.

\begin{remark}
Concerning the application of the following results in the main core of this article, whenever $\Sigma$ is a smooth (compact) submanifold with boundary we preliminarily consider an extension 
$\Sigma \subset \Sigma'$
for some smooth manifold $\Sigma'$ 
without boundary
(and of the same dimension as $\Sigma$),
and correspondingly $\phi=\phi'|_\Sigma$
for some smooth two-sided immersion $\phi'\colon \Sigma' \to M$.
In that case, i.\,e. if $\Sigma$ has a (smooth) boundary, by a smooth function on $\Sigma$
we mean a function which has a smooth extension
to an open neighborhood of $\Sigma$ in $\Sigma'$.
\end{remark}

\paragraph{Euclidean case with variable base immersion.}
First we consider the special case when
$(M,g)=(\R^n,g_{\mathrm{euc}})$
is the Euclidean space of dimension $n \geq 2$.
In this context, with $\Sigma$ and
$\phi \colon \Sigma\to M$
as in the first paragraph of the present appendix, we set
\begin{align*}
&\begin{aligned}
g_\phi&\vcentcolon=\phi^*g_{\mathrm{euc}},
\\
A_\phi&\vcentcolon=A[\phi,g_{\mathrm{euc}},\nu],
\end{aligned}
&\begin{aligned}
\phi[u]&\vcentcolon=\phi[u,g_{\mathrm{euc}},\nu]
=\phi+u\nu,
\\
g[\phi,u]&\vcentcolon=\phi[u]^*g_{\mathrm{euc}},
\end{aligned}
\end{align*}
interpreting $\nu$ as a map $\Sigma \to \R^n$
in the right-most expression for $\phi[u]$. Furthermore, we shall introduce the following convenient notation: given $p \in \Sigma$, for any nonnegative integer $\ell$
and any tensor field $S$ (possibly a function)
on $\Sigma$
set
\begin{equation}
\label{eq:PointwiseNorm}
\abs{S}_\ell
\vcentcolon=
\sum_{j=0}^\ell \abs{(D_{\phi^*g}^jS)(p)}_{\phi^*g},
\end{equation}
where $D_{\phi^*g}$ indicates covariant differentiation
via the Levi-Civita connection on $\Sigma$
induced by $\phi^*g$ and $\abs{{}\cdot{}}_{\phi^*g}$
indicates the appropriate norm at $p$ induced by $\phi^*g$.

We require only quite coarse estimates,
but, since it is easy in this setting to write down more detailed information,
we briefly do so now before stating 
Lemma \ref{lem:mc_and_met_of_graphs_euc}.
An elementary computation yields
\begin{equation}
\label{gphiu_euc}
g[\phi,u]_{ab}
=
g_{ab}
  -2uA_{ab}
  +u^2A_{ac}A_{bd}g^{cd}
  +u_{,a}u_{,b},
\end{equation}
where we have used abstract-index notation
(and in particular $u_{,a}$
represents the one-form $du$)
and have written
$g_{ab}$ for $g_\phi$, $g^{ab}$ for $(g_\phi)^{-1}$,
and $A_{ab}$ for $A_\phi$.
As it is well-known,
$\phi[u]$ is an immersion
for $u$ pointwise sufficiently small
in terms of $A_\phi$,
assuming also $u \in C^1_{\mathrm{loc}}(\Sigma)$.
If, moreover,
$u \in C^2_{\mathrm{loc}}(\Sigma)$,
then in this case
we set
\begin{align*}
A[\phi,u]
&\vcentcolon=
  A\bigl[\phi[u], g_{\mathrm{euc}},\nu[u]\bigr],
&
H[\phi,u]
&\vcentcolon=
  H\bigl[\phi[u], g_{\mathrm{euc}},\nu[u]\bigr],
\end{align*}
where $\nu[u]$ is the unit normal for $\phi[u]$
such that $(p,t) \mapsto \nu[tu](p)$
is continuous on $\Sigma \times \IntervaL{0,1}$
and $\nu[0]=\nu$.
Another elementary computation yields
\begin{equation}
\label{Aphiu_euc}
\begin{aligned}
\bigl(1+\abs{du}_g^2\bigr)^{1/2}A[\phi,u]_{ab}
  &=
  A_{ab} + u_{;ab} - uA_{ac}A_{bd}g^{cd}
    + uu_{,c}\widetilde{g}^{cd}A_{ab;d}
\\
  &\hphantom{{}={}}
    + u_{,c}u_{,a}A_{bd}\widetilde{g}^{cd}
    - uu_{,c}u_{,a}A_{be}A_{df}g^{ef}\widetilde{g}^{cd}
\\
  &\hphantom{{}={}}
    + u_{,c}u_{,b}A_{ad}\widetilde{g}^{cd}
    - uu_{,c}u_{,b}A_{ae}A_{df}g^{ef}\widetilde{g}^{cd}
\\
  &\hphantom{{}={}}
    -u^2u_{,c}A_{ab;e}A_{df}g^{ef}\widetilde{g}^{cd},
\end{aligned}
\end{equation}
with notation as above,
$\widetilde{g}^{ab}|_p$ denoting
the metric on $T_p^*\Sigma$
dual (or inverse) to
\begin{equation} 
\label{gtilde_euc}
\widetilde{g}_{ab}
=
g[\phi,u]_{ab}
  -u_{,a}u_{,b},
\end{equation}
at the point $p$, and each semicolon indicating covariant differentiation
with respect to the Levi-Civita connection
induced on $\Sigma$ by $g_\phi$.

\begin{lemma}
[Mean curvature and induced metrics of graphs in the Euclidean space]
\label{lem:mc_and_met_of_graphs_euc}
Let
$k$ be a nonnegative integer
and
$
 \phi_1,\phi_2
 \colon
 \Sigma
 \to
 \R^n
$
smooth, two-sided immersions
of a smooth manifold $\Sigma$ of dimension $n-1$
(with $n \geq 2$)
with
$\nu_1,\nu_2$
corresponding choices of unit normals.
Given $p\in\Sigma$, any nonnegative integer $\ell$ and any tensor field $S$ (possibly a function) we employ the notation $|S|_{\ell}$ as per \eqref{eq:PointwiseNorm}, with $\phi_1^*g$ in lieu of $\phi^*g$.
There exist
\[
\epsilon=\epsilon(k,n,\abs{A_{\phi_1}}_{k+1}),
~
C=C(k,n,\abs{A_{\phi_1}}_{k+1})
>
0
\]
(respectively
nonincreasing and nondecreasing in $\abs{A_{\phi_1}}_{k+1}$)
such that 
$\phi_1[u]$ and $\phi_2[u]$
are well-defined immersions
on a neighborhood of $p$
and for
any smooth function $u\colon \Sigma \to \R$,
provided
\begin{equation*}
\abs{u}_{k+2}
  + \abs{g_{\phi_2}-g_{\phi_1}}_{k+1}
  + \abs{A_{\phi_2}-A_{\phi_1}}_{k+1}
<
\epsilon,
\end{equation*}
the following estimates hold:
\begin{enumerate}[label={\normalfont(\roman*)}]
\item
\label{met_comp_euc_same_def_funcs}
  $
   \abs{g[\phi_2,u]-g[\phi_1,u]}_{k+1}
   \leq
   C(
      \abs{g_{\phi_2}-g_{\phi_1}}_{k+1}
      + \abs{A_{\phi_2}-A_{\phi_1}}_{k+1}
    )
  $
\item
\label{met_comp_euc_same_background}
  $
   \abs{g[\phi_1,u]-g[\phi_1,0]}_{k+1}
   \leq
   C\abs{u}_{k+2}
  $
\item
  $
   \abs{(g[\phi_2,u]-g[\phi_1,u])-(g[\phi_2,0]-g[\phi_1,0])}_{k+1}
   \leq
   C
   \abs{u}_{k+2}
     (\abs{g_{\phi_2}-g_{\phi_1}}_{k+1}+\abs{A_{\phi_2}-A_{\phi_1}}_{k+1})
  $
\item
  $
   \abs{H[\phi_2,u]-H[\phi_1,u]}_k
   \leq
   C(
      \abs{g_{\phi_2}-g_{\phi_1}}_{k+1}
      + \abs{A_{\phi_2}-A_{\phi_1}}_{k+1}
    )
  $
\item
\label{mc_comp_euc_same_background}
  $
   \abs{H[\phi_1,u]-H[\phi_1,0]}_k
   \leq
   C\abs{u}_{k+2}
  $
\item
\label{mc_quad_comp_euc}
  $\abs[\big]{(H[\phi_2,u]-H[\phi_1,u])-(H[\phi_2,0]-H[\phi_1,0])}_k
   \leq C\abs{u}_{k+2}\bigl(\abs{g_{\phi_2}-g_{\phi_1}}_{k+1}+\abs{A_{\phi_2}-A_{\phi_1}}_{k+1}\bigr)$
\end{enumerate}
\end{lemma}

\begin{proof}
As indicated above
(and as it is also clear from \eqref{gphiu_euc}),
$\phi[u]$ is an immersion
for sufficiently small $u$.
The first estimates then follow from the smooth dependence
at any $p \in \Sigma$,
as exhibited by \eqref{gphiu_euc}
and
\eqref{Aphiu_euc}
(along with \eqref{gtilde_euc}),
of $g[\phi,u](p)$
on $g_\phi(p)$, $A_\phi(p)$, $u(p)$, and $du(p)$.

More in detail,
writing
$\mathrm{Sym}^{(+)}_{n-1}$
for the set of (positive-definite) symmetric
$(n-1) \times (n-1)$ real matrices,
there exists a smooth function
(given by \eqref{gphiu_euc})
\begin{equation*}
\mathbf{g}
\colon
\mathrm{Sym}^+_{n-1}
  \times
  \mathrm{Sym}_{n-1}
  \times \R 
  \times \R^{n-1}
\to
\mathrm{Sym}_{n-1}
\end{equation*}
such that
for all $p \in \Sigma$
\begin{equation*}
g[\phi,u]|_p=\mathbf{g}(g_\phi|_p,A_\phi|_p,u(p),du|_p),
\end{equation*}
where by each evaluation at $p$
we really mean the corresponding tensor
(or scalar or vector)
expressed with respect to a
$g_{\phi_1}$-orthonormal basis at $p$.

Of course we then also have
\begin{equation}
\label{FTC}
\begin{aligned}
g[\phi_2,u]|_p-g[\phi_1,u]|_p
  &=
  \int_0^1
    \partial_s
    \mathbf{g}(g_s|_p, A_s|_p, u(p), du|_p)
    \, ds,
\\
g[\phi_1,u]|_p-g[\phi_1,0]|_p
  &=
  \int_0^1 \partial_t
    \mathbf{g}(
      g_{\phi_1}|_p, A_{\phi_1}|_p, tu(p), tdu|_p
    ) \, dt,
\\
((g[\phi_2,u]-g[\phi_1,u])-(g[\phi_2,0]-g[\phi_1,0]))|_p
  &=
  \int_0^1 
    \int_0^1
      \partial_s \partial_t
      \mathbf{g}(g_s|_p, A_s|_p, tu(p), tdu|_p)
    \, dt
  \, ds,
\end{aligned}
\end{equation}
where we have set
\begin{align*}
g_s &\vcentcolon= sg_{\phi_1}+(1-s)g_{\phi_2}, & 
A_s &\vcentcolon= sA_{\phi_1}+(1-s)A_{\phi_2}.
\end{align*}
Write $\abs{{}\cdot{}}$ as usual for the Euclidean norm on each $\R^d$ but also for the Euclidean (or Frobenius) norm on $\mathrm{Sym}_{n-1}$,
let $I$ denote the $(n-1) \times (n-1)$ identity matrix, and, for each $\delta>0$, set
\begin{equation*}
\begin{aligned}
K_\delta
&\vcentcolon=
\{G \in \mathrm{Sym}_{n-1} \st \abs{G-I} \leq \delta\}
  \times
  \{B \in \mathrm{Sym}_{n-1} \st \abs{B} \leq \nm{A_{\phi_1}}+\delta\}
\\
&\hphantom{\vcentcolon=}
  \times
  \IntervaL{-\delta,\delta}
  \times
  \{v \in \R^{n-1} \st \abs{v} \leq \delta\}.
\end{aligned}
\end{equation*}
Then $K_\delta$ is convex and compact
and, for $\delta$ sufficiently small,
contained in the domain of $\mathbf{g}$,
so that in particular we have a $C^2$
bound on $\mathbf{g}|_{K_\delta}$.
Using \eqref{FTC},
the first three items of the lemma now follow
with $k=-1$
(which is not actually included
in the lemma's statement but makes sense
for these items).

The higher-$k$ cases can be likewise secured
by differentiating
\eqref{gphiu_euc}
to obtain smooth functions
(with matrix values and arguments)
$\mathbf{g}^{(\ell)}$ such that
for all $p \in \Sigma$
\begin{equation*}
D_{\phi_1^*g}^\ell g[\phi,u]|_p
=
\mathbf{g}^{(\ell)}
  (
    g_\phi|_p,
    \ldots,
    D_{\phi_1^*g}^\ell g_\phi|_p,~
    A_\phi|_p,
    \ldots,
    D_{\phi_1^*g}^\ell A_\phi|_p,~
    u(p),
    \ldots,
    D_{\phi_1^*g}^{\ell+1}u|_p
  )
\end{equation*}
and then appropriately restricting their domains
and applying the obvious analogues
of \eqref{FTC}.
Clearly we can also ensure invertibility
of $g[\phi,u]$ and control the norm
of $g[\phi,u]^{-1}$
(and its derivatives)
by additional restrictions.
In view of \eqref{Aphiu_euc}
we can then prove
the final three items
in essentially the same fashion.
\end{proof}

\paragraph{Generators of deformations under alternative metrics.}
Suppose
again
that $(M,g)$ is a complete Riemannian manifold.
Suppose further that $h$ is another Riemannian metric on $M$
such that $(M,h)$ is also complete
and let $\nu_h$ be that unit normal to $\phi$
with respect to $h$ which has pointwise positive
inner product (with respect to either $g$ or $h$)
with $\nu_g$,
so that
\begin{equation*}
\nu_h=\frac{N}{\abs{N}_h},
\quad \mbox{where} \quad
N \vcentcolon= ((\nu_g)^{\flat_g})^{\sharp_h},
\end{equation*}
where $(\nu_g)^{\flat_g}$
is the one-form
(section of $\phi^*T^*M$)
dual to $\nu_g$ under the metric $g$,
$((\nu_g)^{\flat_g})^{\sharp_h}$
is the vector field
(section of $\phi^*TM$)
dual to this last one-form under the metric $h$,
and $\abs{N}_h$ is the $h$ norm of $N$.
Then there is a unique function
$\nu_h^\perp\colon\Sigma \to \interval{0,\infty}$
and there is a unique vector field
$\nu_h^\top$ on $\Sigma$
satisfying
\begin{equation}
\label{nu_h_ortho_decomp}
\nu_h = \nu_h^\perp \nu_g + d\phi \, \nu_h^\top.
\end{equation}
(Note that $\nu_h^\top$ is a section of $T\Sigma$,
so that $d\phi \, \nu_h^\top$
is a section of $\phi^*TM$.)

Given any function $u \colon \Sigma \to \R$,
we define on $\Sigma$ the pointwise rescaled function
$u_{h,g}$ and the vector field $u_{h,g}^\top$ by
\begin{equation}
\label{u_h}
u_{h,g}
  =
  \frac{1}{\nu_h^\perp}u
\quad \mbox{and} \quad
u_{h,g}^\top
  \vcentcolon= 
  u_{h,g}\nu_h^\top.
\end{equation}
The geometric motivation for these definitions
is that,
in view of \eqref{nu_h_ortho_decomp},
$u_{h,g}$ is the unique function such that
$u$ is the (scalar) $g$-normal component
of the initial velocity field
of the one-parameter family
$\phi[tu_{h,g},h,\nu_h]$,
and $d\phi \, u_{h,g}^\top$
is then the corresponding tangential component:
\begin{equation*}
\frac{d}{dt}\bigg\vert_{t=0}\phi[tu_{h,g},h,\nu_h]
=u\nu_g + d\phi \, u_{h,g}^\top.
\end{equation*}

\begin{remark}
\label{scaling_of_u_h}
Note that the following scaling laws hold true for any $\lambda>0$:
\[
 u_{\lambda^2 h, \lambda^2 g}
 =
 u_{h,g}, \
 u_{\lambda^2 h, \lambda^2 g}^\top
 =
 \lambda^{-1}u_{h,g}^\top.
\]
\end{remark}

\begin{remark}
\label{conformal_special_case}
Note that
\begin{equation*}
h=\rho^2g
~\Rightarrow~
(
  u_{h,g}=\rho u
  \mbox{ and }
  u_{h,g}^\top=0
),
\end{equation*}
in the special case that $h$ is conformal to $g$.
\end{remark}

\paragraph{Mean curvature and intersection with a given hypersurface
under graphical deformations.}

The lemma below 
contains conditions
under which the deformed map $\phi[u_{h,g},h,\nu_h]$
is an immersion
and, in that event,
presents fundamental information
on its mean curvature,
the intersection of its image
with a hypersurface of $M$
under certain further assumptions,
and its unit normal along this intersection.
In our application we will take
$h$ to be conformal to $g$,
but in the lemma below
we allow the possibility that $\nu_h^\top \neq 0$,
since this more general situation
can be treated with no additional effort
and could be useful in a setting other than $\B^3$.
At each point of $\Sigma$
the quantities of interest
depend on $u$, $\phi$, $g$, and $h$
only in a neighborhood of $p$
(or $\phi(p)$),
and so all assertions of the lemma
are actually local in nature.
However,
in an effort to keep the statement simple,
we present our assumptions and estimates
in global terms.

\begin{lemma}[Variation of mean curvature
and intersection with a given hypersurface]
\label{basic_MC_bdy_lemma}
Let $\alpha \in \interval{0,1}$.
With notation as above,
assume that
\begin{equation}
\label{assumptions_for_basic_MC_bdy_lemma}
\begin{aligned}
\nm{A_\Sigma}_2
  +\nm{\Riem[g]}_2
  +\nm{\overline{D}^g h}_3
&\leq
1,
\\
\nm{h}_0
  +\nm{h^{-1}}_0
&\leq
100 \sqrt{\dim M},
\end{aligned}
\end{equation}
where
$A_\Sigma \vcentcolon= A[\phi,g,\nu_g]$,
$\Riem[g]$ is the Riemann curvature tensor of $(M,g)$,
$\overline{D}^g$ is the Levi-Civita connection induced by $g$,
and each $\nm{{}\cdot{}}_k$
is the $C^k$ norm induced by $g$ on the appropriate vector bundle.
There exists
$\epsilon(\dim M)>0$,
depending on just the dimension of $M$,
such that the following statements hold for any function $u$ on $\Sigma$
satisfying $\nm{u}_{2,\alpha}<\epsilon(\dim M)$.
Here and below
$
 \nm{{}\cdot{}}_{k,\alpha}
 =
 \nm{{}\cdot{}}_{C^{k,\alpha}(\Sigma,\phi^*g)}
$.
\begin{enumerate}[label={\normalfont(\roman*)}]
\item \label{deform_well_defined}
      The map
      $\phi_u \vcentcolon= \phi[u_{h,g}, h, \nu_h]$
      (as defined in \eqref{deformed_phi},
      with $u_{h,g}$ itself defined in \eqref{u_h})
      is an immersion
      with well-defined mean curvature
      \begin{equation*}
      H_u
      \vcentcolon=
      H\bigl[\phi[u_{h,g},h,\nu_h],g,\nu_u\bigr]
      \end{equation*}
      relative to $g$ and the unit normal $\nu_u$
      chosen such that
      the map $(p,t) \mapsto \nu_{tu}(p)$
      is continuous on $\Sigma \times \IntervaL{0,1}$
      and 
      $\nu^{\vphantom{|}}_0=\nu_g$.
\item \label{deform_var_of_H} 
      We have
      \begin{equation*} 
      \left.\frac{d}{dt}\right|_{t=0}
        H_{tu}
      =
      -J_\Sigma u + u_{h,g}^\top H_0.
      \end{equation*}
\item \label{deform_quad_est_of_H}
      There holds the quadratic estimate
      \begin{equation*}
      \nm{H_u-H_0 + J_\Sigma u}_{0,\alpha}
      \leq
      C(\dim M)\nm{u}_{2,\alpha}\bigl( 
        \nm{u}_{2,\alpha}
        +\nm{H_0}_{1,\alpha}
      \bigr)
      \end{equation*}
      for some constant $C(\dim M)$
      depending on $\dim M$
      but independent of $u$.
      The term $\nm{H_0}_{1,\alpha}$
      on the right-hand side can be omitted
      in case $u_{h,g}^\top$ vanishes identically.
\item \label{deform_stays_on_Gamma}
      Suppose $\Gamma \subset M$
      is a geodesically complete, embedded,
      two-sided hypersurface.
      Assume that $\Gamma$
      is totally geodesic under $h$,
      that
      $g|_p=h|_p$
      (as metrics on $T_pM$)
      for each $p \in \Gamma$,
      and that
      $\nu_g|_p \in T_{\phi(p)}\Gamma$
      for each $p \in \phi^{-1}(\Gamma)$.
      Then
      $
       \phi_u(\phi^{-1}(\Gamma))
       \subset
       \Gamma
      $.
\item \label{deform_Robin_cond_implies_orthogonality} 
      Continue to make the assumptions of
      item \ref{deform_stays_on_Gamma}.
      Let $\nu^{\vphantom{|}}_\Gamma$
      be a choice of unit normal to $\Gamma$
      relative to $g$ (and $h$),
      let $A_\Gamma$ be the
      scalar-valued second fundamental form
      of $\Gamma$ with respect to $g$
      and $\nu^{\vphantom{|}}_\Gamma$
      (with sign convention as above, see Section \ref{sec:Nota}),
      and let $\eta$ be the $\phi^*g$
      unit conormal to $\Sigma$
      on $\phi^{-1}(\Gamma)$
      such that
      $
       d\phi \, \eta
       =
       \nu^{\vphantom{|}}_\Gamma \circ \phi
      $.
      If $u$ satisfies the Robin condition
      \begin{equation*}
      \eta u - (A_\Gamma \circ \phi)(\nu_g,\nu_g)u
      =
      0
      \quad \mbox{ on }  \phi^{-1}(\Gamma),  
      \end{equation*}
      then $\nu_u|_p \in T_{\phi_u(p)}\Gamma$
      for each $p \in \phi^{-1}(\Gamma)$.
      \end{enumerate}
\end{lemma}

\begin{proof}
We start by introducing the map
\begin{align*}
\Phi_u \colon \Sigma \times \R &\to M \\
(p,t) &\mapsto \phi_{tu}(p).
\end{align*}
Writing $\overline{D}^h$ for the $h$
Levi-Civita connection on $TM$,
there is a unique connection $D^h$ on $\Phi_u^*TM$
satisfying the chain rule
$
 D^h_V (X \circ \Phi_u)
 =
 (\overline{D}^h_{d\Phi_u V}X) \circ \Phi_u
$
for any smooth sections
$V$ of $T(\Sigma \times \R)$
and $X$ of $TM$;
$D^h$ is torsion-free in the sense that
$
 D^h_V d\Phi_u W - D^h_W d\Phi_u V
 =
 d\Phi_u [V,W]
$
and compatible with $h$ in the sense that
$
 V(h \circ \Phi_u)(X,Y)
 =
 (h \circ \Phi_u)(D^h_V X, Y)
   +(h \circ \Phi_u)(X, D^h_V Y)
$
for any smooth sections $V,W$ of $T(\Sigma \times \R)$
and $X,Y$ of $\Phi_u^*TM$.

Defining $T_u$, a section of $\Phi_u^*TM$,
by
$T_u \vcentcolon= d\Phi_u \partial_t$,
we have $D^h_{\partial_t}T_u=0$
and,
for any section $V$ of $T\Sigma$,
$D^h_{\partial_t}d\Phi_uV = D^h_V T_u$.
We also compute
\begin{equation}
\label{system_for_pushforward}
\left\{\begin{aligned}
d\Phi_uV|_{t=0}
  &=
  d\phi \, V,
\\
D^h_{\partial_t}d\Phi_uV|_{t=0}
  &=
  (Vu_{h,g})\nu_h + S_hV,
\\
D^h_{\partial_t}D^h_{\partial_t}\Phi_*V
  &=
  (R^h \circ \Phi)(T_u,\Phi_*V)T_u,
\end{aligned}\right.
\end{equation}
where
$S_h$ is the shape operator of $\Sigma$
relative to $h$
(obtained by raising via $\phi^*h$
an index of $A[\phi,h,\nu_h]$)
and
$R^h$ is the Riemann curvature tensor of $(M,h)$
(with suitable sign and ordering conventions).
Item \ref{deform_well_defined}
now follows in view of the bounds
assumed on the background geometry
($g$, $h$, and $A_\Sigma$).

Item \ref{deform_var_of_H}
is an immediate consequence
of the usual formula for the variation of mean curvature
in conjunction with
the definition of $\phi_u$,
in item \ref{deform_well_defined}
via \eqref{deformed_phi},
and the definition of $u_{h,g}$, in \eqref{u_h}.
To prove item \ref{deform_quad_est_of_H},
we use the identity
\begin{equation*}
(H_u-H_0-\partial_t|_{t=0}H_{tu})|_p
=
\int_0^1 \int_0^t \partial_s^2 H_{su} \, ds \, dt
\end{equation*}
and the bound
$
 \nm{\partial_s^2H_{su}}_{0,\alpha}
 \leq
 C(\dim M)\nm{u}_{2,\alpha}^2
$,
which can be obtained with the aid of
\eqref{system_for_pushforward}
for all $u$ sufficiently small
in terms of the assumed bounds
on the background geometry.

Item \ref{deform_stays_on_Gamma}
is obvious,
given that $\nu_h=\nu_g$ on $\phi^{-1}(\Gamma)$
under the assumptions.
For item \ref{deform_Robin_cond_implies_orthogonality}
we first claim that the Robin condition posited on $u$
is equivalent to the Neumann condition
$\eta u_{h,g} = 0$ on $u_{h,g}$
(and we note that $\eta$
is also a $\phi^*h$ unit conormal
to $\Sigma$ on $\phi^{-1}(\Gamma)$).
This equivalence follows immediately
from the definition \eqref{u_h},
the fact
that
$
 g|_\Gamma
 =
 h|_\Gamma
$
(and so $\nu_h=\nu_g$ on $\phi^{-1}(\Gamma)$),
and the computation
\begin{equation*}
\left.\eta \nu_h^\perp\right|_{\phi^{-1}(\Gamma)}
=
(A_\Gamma \circ \phi)(\nu_g,\nu_g),
\end{equation*}
which in turn follows from the assumption that $\Gamma$
is totally geodesic with respect to $h$.

Next we will show that the Neumann condition
$\eta u_{h,g} = 0$ implies the conclusion
of item \ref{deform_Robin_cond_implies_orthogonality}.
To this end
let $\overline{\eta}$ be any smooth extension
of $\eta$ to all of $\Sigma$
and let $\overline{\nu}^{\vphantom{|}}_\Gamma$
be any smooth extension of $\nu^{\vphantom{|}}_\Gamma$
to all of $M$
and define sections $X,Y,T$ of $\Phi_u^*TM$ by
\begin{align*}
X
&\vcentcolon=
  \overline{\nu}^{\vphantom{|}}_\Gamma \circ \Phi_u,
\\
Y
&\vcentcolon=
  d\Phi_u\overline{\eta},
  \quad \mbox{and}
\\
T(p,t)
&\vcentcolon=
  d\Phi_1|_{(p,tu(p))} \partial_t.
\end{align*}
In particular
$T({}\cdot{},0)=\nu_h$,
$(h \circ \Phi_u)(T,T)=1$,
$T_u = d\Phi_u \partial_t = u_{h,g}T$,
and
$D^h_{\partial_t}T=0$.
(Here and below
$u_{h,g}$ has been extended fiberwise constantly
to $\Sigma \times \R$.)
By the assumption that $\Gamma$
is totally geodesic under $h$
we also have
$(h \circ \Phi_u)(X,T)=0$
and
$D^h_{\partial_t}X=0$
on $\phi^{-1}(\Gamma) \times \R$.
Given any section $Z$ of $\Phi_u^*TM$,
we further define the section
\begin{equation}
\label{Z_perp_def}
Z^\top
  \vcentcolon=
  Z-(h \circ \Phi_u)(X,Z)X.
\end{equation}
Note that
$
 D^h_{\partial_t}(Z^\top)
 =
 (D^h_{\partial_t}Z)^\top
$
on $\phi^{-1}(\Gamma) \times \R$.
We will complete the proof
by showing that $Y^\top$ vanishes
on $\phi^{-1}(\Gamma) \times \R$.

The assumptions of item \ref{deform_stays_on_Gamma}
that
$g|_\Gamma=h|_\Gamma$
and 
$\nu_g|_p \in T_{\phi(p)}\Gamma$
for each
$p \in \phi^{-1}(\Gamma)$
clearly imply
\begin{equation}
\label{Y_top_initial_value}
Y^\top({}\cdot{},0)=0
\mbox{ on }
\phi^{-1}(\Gamma).
\end{equation}
We also compute
\begin{equation*}
D^h_{\partial_t}Y
=
D^h_{\overline{\eta}}(u_{h,g}T)
=
(\overline{\eta} u_{h,g})T
  +u_{h,g} D^h_{\overline{\eta}}T,
\end{equation*}
so that in particular
\begin{equation}
\label{D_dt_Y_initial_normal}
(h \circ \phi)(\nu_h,~ D^h_{\partial_t}Y|_{t=0})
=
\eta u_{h,g}
\mbox{ on } 
\phi^{-1}(\Gamma)
\end{equation}
and, for any smooth vector field $W$ on $\Sigma$,
\begin{equation}
\label{D_dt_Y_initial_tangential}
(h \circ \phi)(d\phi_u W,~  D^h_{\partial_t}Y|_{t=0})
=
-u_{h,g} A_\Sigma^h(\eta, W)
\mbox{ on } \phi^{-1}(\Gamma),
\end{equation}
where $A_\Sigma^h \vcentcolon= A[\phi,h,\nu_h]$
is the second fundamental form
of $\phi$ relative to the ambient metric
$h$ and the choice of unit normal $\nu_h$.
However, $d\phi \, \eta = X \circ \phi$
on $\phi^{-1}(\Gamma)$,
so the assumption that $\Gamma$
is totally geodesic with respect to $h$
implies that
$A_\Sigma^h(\eta,W)=0$
whenever $W$ is $\phi^*h$ orthogonal to $\eta$.
(Here we use the fact that for each
$p \in \phi^{-1}(\Gamma)$
and each $w \in T_p\Sigma$
with $\phi^*h(w,\eta|_p)=0$
our assumptions
ensure the existence of a curve
which lies on $\Sigma$,
passes through $p$ with velocity $w$,
and has image under $\phi$ contained in $\Gamma$.)

Recall that
$D^h_{\partial_t}Y^\top=(D^h_{\partial_t}Y)^\top$
on $\phi^{-1}(\Gamma) \times \R$
and that $X(\cdot,0)=d\phi \, \eta$
on $\phi^{-1}(\Gamma)$,
so that
$D^h_{\partial_t}Y^\top(\cdot,0)$
is 
by \eqref{Z_perp_def}
obviously orthogonal to
$d\phi \, \eta$ everywhere on $\phi^{-1}(\Gamma)$.
Thus, provided $\eta u_{h,g} = 0$,
it follows from
\eqref{D_dt_Y_initial_normal}
and \eqref{D_dt_Y_initial_tangential}
that
\begin{equation}
\label{Y_top_initial_derivative}
D^h_{\partial_t}Y^\top({}\cdot{},0)=0
\mbox{ on } \phi^{-1}(\Gamma).
\end{equation}

Next we compute
\begin{equation*}
D^h_{\partial_t}D^h_{\partial_t}Y
=
D^h_{\partial_t}D^h_{\overline{\eta}}u_{h,g}T
=
u_{h,g}^2(R^h \circ \Phi_u)(T,Y)T.
\end{equation*}
Since $T$ is everywhere
on $\phi^{-1}(\Gamma) \times \R$
tangential to $\Gamma$
and since $\Gamma$ is $h$ totally geodesic,
using the Codazzi equation
we obtain
\begin{equation}
\label{Y_top_ODE}
D^h_{\partial_t}D^h_{\partial_t}Y^\top
=
u_{h,g}^2(R^h \circ \Phi_u)(T,Y^\top)T
\mbox{ on } \phi^{-1}(\Gamma)
\end{equation}
The proof is now concluded by the observation that the ODE system
\eqref{Y_top_ODE}
subject to the initial conditions
\eqref{Y_top_initial_value}
and \eqref{Y_top_initial_derivative}
has only the trivial solution
$Y^\top=0$
on $\phi^{-1}(\Gamma) \times \R$.
\end{proof}

\section{Asymptotic behavior of the Ketover free boundary minimal surfaces}\label{app:convergence_large_g}

Considering a pair of free boundary minimal surfaces with the same topology and symmetry group, one might ask whether any of the two surfaces can be obtained using a suitable equivariant min-max approach. 
In fact, as mentioned in the introduction of the present article, Ketover \cite{KetoverFB} used equivariant min-max theory to construct a family $\{\Sigma_g^{\mathrm{Ket}}\}_{g\in\N}$ of free boundary minimal surfaces satisfying the properties \ref{prop:convergence_large_g-i}--\ref{prop:convergence_large_g-iii} given in Proposition \ref{prop:convergence_large_g} below. 
The construction involves sweepouts of $\B^3$ whose slices are equivariantly isotopic to the free boundary minimal surfaces $\Sigma_g^{\mathrm{KL}}$ constructed by Kapouleas--Li \cite{KapouleasLiDiscCCdesing} as well as to our surfaces $\Sigma_g^{\mathrm{CSW}}$ found in Theorem~\ref{thm:ConstructFirst} (for any given $g$ for which the corresponding surfaces exist).
Min-max theory does \emph{not} answer the question whether or not the resulting free boundary minimal surface $\Sigma_g^{\mathrm{Ket}}$ coincides with either $\Sigma_g^{\mathrm{KL}}$ or $\Sigma_g^{\mathrm{CSW}}$.
Nevertheless, one can show that $\Sigma_g^{\mathrm{Ket}}$ and $\Sigma_g^{\mathrm{KL}}$ 
have the same asymptotic behavior as the genus $g$ tends to infinity, namely convergence in the sense of varifolds to the union of the horizontal disc and the critical catenoid. 
This statement is part of \cite[Theorem~1.1]{KetoverFB}; in the proof however (cf. section~4.3 of \cite{KetoverFB}) the possibility that $\Sigma_g^{\mathrm{Ket}}$ has the same asymptotic behavior as our surface $\Sigma_g^{\mathrm{CSW}}$, namely convergence to $\K_0\cup\B^2\cup -\K_0$ for $g\to\infty$, is actually not excluded.  
More precisely, the argument for Claim~4 in the proof of Proposition \ref{prop:convergence_large_g} below is missing in \cite{KetoverFB}, and the scope of this appendix is to fill this gap exploiting
the study of the catenoidal annuli in Section \ref{subsec:building_blocks}. 
Following the ideas in \cite{KetoverFB} we formulate a full proof of the convergence result Proposition~\ref{prop:convergence_large_g} with a similar approach as in Section~3 of \cite{Buzano2021}. 
Especially \cite[Lemma 2.9]{Buzano2021}, which we restate here for the convenience of the reader, is crucial for several of the arguments we are about to present. 
As defined in \eqref{cycdef}, $\cyc_{g+1}$ denotes the cyclic group of order $g+1$, identified with the subgroup of $\SOgroup(3)$ corresponding to rotations around the $z$-axis of angles multiple of $2\pi/(g+1)$. 

\begin{lemma}[{\cite[Lemma 2.9]{Buzano2021}}]\label{lem:genus}
Given $1\leq g\in\N$, let $\Sigma\subset\R^3$ be any closed, connected, embedded $\cyc_{g+1}$-equivariant surface of genus $\gamma\in\{1,\ldots,g\}$. 
If $\Sigma$ is disjoint from the axis of rotation then $\gamma=1$. 
If $\Sigma$ intersects the axis of rotation then the number of intersections is $4$ and $\gamma=g$. 
\end{lemma}

\begin{corollary}\label{cor:genus}
Given $1\leq g\in\N$, let $\Omega\subset\R^3$ be any convex, bounded, $\cyc_{g+1}$-equivariant domain with piecewise smooth boundary and let $\Sigma\subset\Omega$ be any connected, properly embedded, $\cyc_{g+1}$-equivariant surface of genus $\gamma\in\{1,\ldots,g\}$, possibly with boundary. 
Then $\Sigma$ has genus $\gamma\in\{1,g\}$ and if $\Sigma$ intersects the axis of rotation then $\gamma=g$. 
\end{corollary}

To avoid ambiguities, we wish to stress that -- as in the core of this article -- when writing \emph{properly embedded} we mean that the surface in question is compact, embedded, that its boundary lies in the boundary of the ambient domain and that there are no additional interior contact points. 
In particular, it follows from the convexity of $\Omega$ that any such surface must be orientable.

For our purposes, we shall say that an open domain $\Omega\subset\R^3$ (as in the statement of the corollary) has \emph{piecewise smooth} boundary if there exist finitely many smooth closed curves lying on $\partial\Omega$ whose complement consists of (finitely many) open domains each admitting a smooth parametrization up to (and including) the boundary. For instance, that is the case if we consider the domain obtained by intersecting the open unit ball in $\R^3$ with a open half space. One could easily work in greater generality, but at the cost of unnecessarily tedious details.

\begin{proof} 
By definition, the genus of $\Sigma$ is equal to the genus of the closed surface which is obtained by closing up each boundary component of $\Sigma$ by gluing in a topological disc. 
In view of Lemma~\ref{lem:genus} it suffices to show that this procedure can be done while preserving embeddedness and $\cyc_{g+1}$-equivariance. 
Without loss of generality, we may assume that $\Omega$ is a ball $B$ around the origin. 
Indeed, possibly after a translation, we have $0\in\Omega$. Then we may choose a ball $B\supset\Omega$ around $0$ and since $\Omega$ is convex, $\Sigma$ can be extended radially to a properly embedded, $\cyc_{g+1}$-equivariant surface in $B$ having the same genus as $\Sigma$. 
Such an extension will not in general be smooth; however, as we are about to describe we will take care of the smoothening at the end of the proof.

Let $\beta\subset\partial\Sigma\subset\partial B$ be any boundary component of $\Sigma$ and let $S_\beta$ be the one of the two domains in $\partial B$ bounded by $\beta$ which has smaller area. (If both have equal area, the choice is arbitrary and in that case the choice causes no problems for the following argument.)  
Given $R_\beta\geq1$, let $D_\beta\vcentcolon=[1,R_\beta]\beta\cup R_\beta S_\beta$, where the multiplication is to be understood as scaling. 
Then, $D_\beta$ is a topological disc. 
Choosing $R_\beta=1+\area(S_\beta)^{1/2}$ we can achieve that the corresponding discs $D_\beta$, as one varies $\beta$, are pairwise disjoint. 
Indeed, let $\beta_1$ and $\beta_2$ be two arbitrary boundary components of $\Sigma$ which we label such that $\area(S_{\beta_1})\leq\area(S_{\beta_2})$. 
Assuming $\beta_1\neq\beta_2$ we have $\beta_1\cap\beta_2=\emptyset$. 
If $S_{\beta_1}\cap S_{\beta_2}=\emptyset$ then $D_{\beta_1}\cap D_{\beta_2}=\emptyset$ follows trivially. 
If $S_{\beta_1}\cap S_{\beta_2}\neq\emptyset$ then $S_{\beta_1}\subset S_{\beta_2}$ and $R_{\beta_1}<R_{\beta_2}$ which implies $D_{\beta_1}\cap D_{\beta_2}=\emptyset$. 
Smoothing the union $\bigcup_\beta D_\beta\cup\Sigma$ equivariantly, we obtain a closed, embedded, $\cyc_{g+1}$-equivariant surface to which Lemma~\ref{lem:genus} applies. 
\end{proof}

In the proof of Proposition \ref{prop:convergence_large_g} below we will frequently use the subadditivity of the genus: 
given smooth embedded surfaces $\Sigma_1$ and $\Sigma_2$, having at most the smooth boundary in common, we have \(\genus(\Sigma_1)+\genus(\Sigma_1)\leq\genus(\Sigma_1\cup\Sigma_2)\). 
In particular, if the surface $\Sigma$ is a subset of the surface $\Gamma$, then $\genus(\Gamma\setminus\Sigma)\leq\genus(\Gamma)-\genus(\Sigma)\leq \genus(\Gamma)$. 

\begin{proposition}\label{prop:convergence_large_g}
Let $\{\Gamma_g\}_{g\geq 1}$ be a sequence of properly embedded, free boundary minimal surfaces in the Euclidean unit ball $\B^3$ with the following properties for each $1\leq g\in\N$: 
\begin{enumerate}[label={\normalfont(\roman*)}]
\item\label{prop:convergence_large_g-i} $\Gamma_g$ is $\dih_{g+1}$-equivariant, and contains the horizontal axes 
$\xi_1,\ldots,\xi_{g+1}$, where we define $\xi_\ell=\{(r\cos(\ell\pi/(g+1)),r\sin(\ell\pi/(g+1)),0)\st r\in[-1,1]\}$ for each $\ell\in\{1,\ldots,g+1\}$;
\item\label{prop:convergence_large_g-ii} $\genus(\Gamma_g)\leq g$ 
and the number of boundary components of $\Gamma_g$ is bounded uniformly in $g$;
\item\label{prop:convergence_large_g-iii} $\area(\B^2)<\area(\Gamma_g)<\area(\B^2\cup\K_{\mathrm{crit}})$, 
where 
$\gls{Kcrit} %\K_{\mathrm{crit}}
\subset\B^3$ 
denotes the critical catenoid and $\B^2\subset\B^3$ the horizontal unit disc. 
\end{enumerate}
Then $\Gamma_g$ converges to $\B^2\cup\K_{\mathrm{crit}}$ as $g\to\infty$ in the sense of varifolds and locally smoothly (with multiplicity one) away from the intersection circle $\B^2\cap\K_{\mathrm{crit}}$. 
Moreover, $\Gamma_g$ has genus $g$ and exactly three boundary components for all sufficiently large $g$.  
\end{proposition}

\begin{proof}
In what follows, we shall consider any subsequence of $\{\Gamma_g\}_{g\geq 1}$ without relabelling; we further recall the notation $\N^{\ast}=\left\{1,2,3,\ldots\right\}$ for the set of positive integers.
Since we assume uniform bounds on the area and thus on the boundary length of $\Gamma_g$ in \ref{prop:convergence_large_g-iii}, there 
exists a further subsequence which converges in the sense of varifolds to a stationary integral varifold $\Gamma_{\infty}$.
Assumption \ref{prop:convergence_large_g-i} implies that the support of $\Gamma_{\infty}$ is rotationally symmetric around the vertical axis 
\begin{align*}
\xi_0\vcentcolon=\{(0,0,t)\st t\in\R\}
\end{align*}
and contains the horizontal disc $\B^2$ (cf. Claim 1 in section 3 of \cite{Buzano2021}). 
With slight abuse of notation, let $A_g$ (in lieu of $A_{\Gamma_g}$) denote the second fundamental form of $\Gamma_g$ and $B_r(x)$ the open ball of radius $r>0$ around some $x\in\B^3$. 
As in \cite[(4.26)]{KetoverFB} we consider the set 
\begin{align}\label{eqn:blow-up-set}
\Lambda\vcentcolon=\biggl\{x\in\B^3\st\inf_{r>0}\Bigl(\liminf_{g\to\infty}\int_{\Gamma_g\cap B_r(x)}\abs{A_g}^2 \Bigr)\geq\varepsilon_0\biggr\},
\end{align}
where $\varepsilon_0>0$ is given by the $\varepsilon$-regularity theorem of Choi--Schoen \cite[Proposition 2]{ChoSch85} (for boundary points one appeals to \cite[Theorem 5.1]{FraLi14} instead). 
In particular, given any $x_0\in\B^3\setminus\Lambda$ there exists $r>0$ 
and a constant $C$, which depends only on the background geometry and hence can be chosen uniformly in $g$, such that 
\begin{align}\label{eqn:Choi-Schoen}
\max_{0\leq\sigma\leq r}\sigma^2\sup_{B_{r-\sigma}(x_0)}\abs{A_g}^2\leq C. 
\end{align}
The uniform bound \eqref{eqn:Choi-Schoen} implies (this is now standard, but for a reference see e.\,g. \cite{Langer1985}) that the convergence $\Gamma_g\to\Gamma_{\infty}$ is smooth in $B_{r/2}(x_0)$. 
As a consequence, the limit $\Gamma_{\infty}$ is a smooth, embedded minimal surface away from $\Lambda$ and, away from $\Lambda$, $\Gamma_{\infty}$ meets the sphere $\partial\B^3$ orthogonally. 
In what follows we analyze the structure of the singular set $\Lambda$.

\emph{Claim 1.} $\Gamma_g$ restricted to any neighborhood of any $x_0\in\Lambda\setminus\xi_0$ has unbounded genus as $g\to\infty$. 

\begin{proof}[Proof of Claim 1]
If $x_0\in(\Lambda\setminus\xi_0)\setminus\partial\B^3$ is an interior point then there exists $r_0>0$ such that $B_{r_0}(x_0)\subset\B^3$. 
Moreover, there exists $0<r<r_0$ such that $\partial B_r(x_0)$ intersects $\Gamma_g$ transversally for all $g$ by Sard's Theorem. 
In particular, the genus of $\Gamma_g\cap B_{r}(x_0)$ is well-defined. 
Given any $\gamma,g_0\in\N$, $g_0\geq 1$, we assume towards a contradiction that $\genus(\Gamma_g\cap B_{r}(x_0))\leq\gamma$ for all $g\geq g_0$. 
Then,  Ilmanen's \cite[Lecture 3]{Ilmanen1998} localized Gauss--Bonnet estimate implies  
\begin{align}\label{eqn:20210304-1}
\sup_{g\geq g_0}\int_{\Gamma_g\cap B_{r/2}(x_0)}\abs{A_g}^2\leq C.
\end{align}
Let $c$ be the horizontal circle around $\xi_0$ passing through $x_0$. 
Then $c\subset\Lambda$ since $\Lambda$ inherits the rotational symmetry of $\Gamma_{\infty}$.   
Given any $n\in\N^\ast$ there exist  points $x_1,\ldots,x_n\in c\cap B_{r/2}(x_0)$ with pairwise distance at least $2\delta>0$ depending only on $r$, $n$ and the radius of $c$. 
By definition \eqref{eqn:blow-up-set},
\begin{align}\label{eqn:20210304-2}
\sup_{g\geq g_0}\int_{\Gamma_g\cap B_{r/2}(x_0)}\abs{A_g}^2 
\geq\sup_{g\geq g_0}\sum^{n}_{k=1}
\int_{\Gamma_g\cap B_{\delta}(x_k)}\abs{A_g}^2 \geq\frac{n\varepsilon_0}{2}. 
\end{align}
Choosing $n>2C/\varepsilon_0$ we obtain a contradiction with \eqref{eqn:20210304-1}. 
Thus, there exists a subsequence of indices $g\to\infty$ along which $\genus(\Gamma_g\cap B_{r}(x_0))\to\infty$.

If $x_0\in(\Lambda\setminus\xi_0)\cap\partial\B^3$ is a boundary point then Ilmanen's localized Gauss--Bonnet estimate does not apply directly. 
However, we may extend $\Gamma_g$ across $\partial\B^3$ by spherical inversion:  
for some $r_0\in\interval{0,1}$ we set $\tilde\Gamma_g\vcentcolon=\Gamma_g\cup I(\Gamma_g\setminus B_{1-r_0})$, where the map $I\colon\R^3\setminus\{0\}\to\R^3\setminus\{0\}$ given by $x\mapsto\abs{x}^{-2}x$ is conformal satisfying $I^*g_{\mathrm{euc}}=\phi^{4}g_{\mathrm{euc}}$ for $\phi(x)=\abs{x}^{-1}$. 
In particular, $\tilde\Gamma_g$ inherits the uniform area bound.  
Moreover, the mean curvature $\tilde{H}$ of $I(\Gamma_g\setminus B_{1-r_0})$ satisfies 
$\abs{\tilde{H}\circ I}=\abs{4\phi^{-3}\partial_\nu\phi}$ (cf. \cite[(B.2)]{CarLi19}), where $\nu$ is a choice of unit normal on $\Gamma_g$. 
Therefore, the mean curvature of $\tilde\Gamma_g$ is continuous across the free boundary on $\partial\B^3$ and bounded uniformly with respect to $g$. 
With a similar formula, one checks that the second fundamental form is also continuous across the interface. 
Hence, Ilmanen's argument can be applied at $x_0$ for $\tilde\Gamma_g$. 
Arguing as above in \eqref{eqn:20210304-1}--\eqref{eqn:20210304-2}, the localized Gauss--Bonnet argument then yields a subsequence $g\to\infty$ along which $\genus(\tilde\Gamma_g\cap B_{r}(x_0))\to\infty$ for some fixed $0<r<r_0$.  
A priori, $\genus(\tilde\Gamma_g\cap B_{r}(x_0)\cap\B^3)$ could be bounded, but then, by construction $\genus(\tilde\Gamma_g\cap B_{r}(x_0)\setminus\B^3)$ stays bounded as well and $\tilde\Gamma_g\cap B_{r}(x_0)\cap\partial\B^3$ would contain an unbounded number of closed curves contradicting the assumption that $\partial\Gamma_g$ has a uniformly bounded number of connected components. 
(Here we use that the genus is defined as the maximum number of disjoint simple closed curves which can be removed from a surface without disconnecting it.)
Therefore, $\genus(\Gamma_g\cap B_{r}(x_0))=\genus(\tilde\Gamma_g\cap B_{r}(x_0)\cap\B^3)$ is unbounded in $g$ as claimed. 
\end{proof}

\begin{remark}
The argument used to prove Claim 1 does not apply if $x_0\in\Lambda\cap\xi_0$ because then the radius of the circle $c$ would vanish. 
This subtlety seems to have been neglected in \cite{KetoverFB}.  
We further, explicitly note that $\Gamma_g\cap B_r(x_0)$ may a priori be disconnected, so for instance consisting of a certain, eventually large number of connected components having say genus equal to one. This aspects also needs to be dealt with in our discussion. 
\end{remark}

\emph{Claim 2.}
$\Lambda\subset\B^2\cup\xi_0$.  
Moreover, $\Lambda\cap\B^2\setminus\xi_0$ consists of at most one circle, $\Lambda\cap \xi_0$ is finite and disjoint from $\partial\B^3$, 
and $\Gamma_{\infty}$ is a smooth minimal surface away from $\Lambda\cap\B^2$ (not only away from $\Lambda$).

\begin{proof}[Proof of Claim 2]
Towards a contradiction, suppose that there exists $x_0\in\Lambda\setminus(\B^2\cup\xi_0)$. 
Appealing to the dihedral symmetry, we may assume that $x_0$ is in the upper half ball.  
By Sard's Theorem there exists $0<r<\dist(x_0,\B^2)$ such that the restriction of $\Gamma_g$ to the domain $\Omega_r\vcentcolon=\{(x,y,z)\in\B^3\st z\geq r\}$ is properly embedded in $\Omega_r$. 
Since $x_0\in\Omega_r$ there exists some large $g\in\N^{\ast}$ such that at least one connected component $Q_g$ of $\Gamma_g\cap\Omega_r$ satisfies $\genus(Q_g)\geq1$ by Claim 1. 
Then, $Q_g$ must be $\cyc_{g+1}$-equivariant. 
If not, the orbit of $Q_g$ under the dihedral group $\dih_{g+1}$ would have $2(g+1)$ connected components of genus at least one and since all of them are contained in $\Gamma_g$ by \ref{prop:convergence_large_g-i}, this would contradict assumption \ref{prop:convergence_large_g-ii}. 
Hence, Corollary~\ref{cor:genus} implies $\genus(Q_g)\in\{1,g\}$.
If $\genus(Q_g)=g$, the dihedral symmetry of $\Gamma_g$ implies $\genus(\Gamma_g)\geq2\genus(Q_g)=2g$ in contradiction with assumption~\ref{prop:convergence_large_g-ii}. 
If $\genus(Q_g)=1$ then $\genus(\Gamma_g\setminus Q_g)\leq g-1$ and thus, as a consequence of Corollary~\ref{cor:genus}, the connected component $O_g$ of $\Gamma_g\setminus Q_g$ containing the origin must have genus zero. 
We claim that $\rot_{\xi_1}^{\pi}Q_g\subset O_g$ which leads to a contradiction because $\genus(O_g)=0$ and $\genus(\rot_{\xi_1}^{\pi}Q_g)=1$. 
Since $\Gamma_g$ is connected, there exists a path $\rho\subset\Gamma_g\setminus\rot_{\xi_1}^{\pi} Q_g$ connecting the origin to $\rot_{\xi_1}^{\pi}Q_g$. 
If $\rho$ is disjoint from $Q_g$ then $\rot_{\xi_1}^{\pi}Q_g\subset O_g$ follows since we defined $O_g$ to be the connected component of $\Gamma_g\setminus Q_g$ containing the origin. 
If $\rho$ intersects $Q_g$ before it reaches $\rot_{\xi_1}^{\pi}Q_g$, then we simply replace $\rho$ by $\rot_{\xi_1}^{\pi}\rho$. 
This completes the proof of $\Lambda\subset\B^2\cup\xi_0$.  

Being rotationally symmetric, $\Lambda\cap\B^2\setminus\xi_0$ is a union of circles. 
Suppose, $\Lambda\cap\B^2\setminus\xi_0$ contains two circles $c_1,c_2$ with radii $0<r_1<r_2\leq 1$. 
Let $r\in\interval{r_1,r_2}$ such that the restriction of $\Gamma_g$ to the ball $B_r$ of radius $r$ around the origin is properly embedded in  $B_r$. 
Let $x_0\in c_1$ and $0<\varepsilon<\min\{r_1,r-r_1\}/4$. 
If $g\in\N^{\ast}$ is sufficiently large, Claim 1 implies that $\Gamma_g\cap B_\varepsilon(x_0)$ has a connected component $U_{g,\varepsilon}$ with $\genus(U_{g,\varepsilon})\geq1$. 
The $\cyc_{g+1}$-orbit $\cyc_{g+1}U_{g,\varepsilon}$ of $U_{g,\varepsilon}$ must be connected because otherwise, 
$\genus(\cyc_{g+1}U_{g,\varepsilon})\geq g+1$ in contradiction to $\genus(\Gamma_g)=g$. 
In particular, we may consider the connected component of $\Gamma_g\cap B_r$ containing $\cyc_{g+1}U_{g,\varepsilon}$ which has at least genus $2$ by construction and therefore must have full genus $g$, by Corollary~\ref{cor:genus}. 
Consequently, $\genus(\Gamma_g\setminus B_r)=0$ in contradiction with Claim~1 applied to neighborhoods of points in $c_2$. 

Now let $x_0\in\Lambda\cap\xi_0\setminus\B^2$ be arbitrary. 
Then there exist $0<r<\dist(x_0,\B^2)$ and $g_0\in\N^{\ast}$ such that any connected component of $\Gamma_g\cap B_r(x_0)$ has genus $0$ or $1$ for all $g\geq g_0$:
otherwise, Corollary \ref{cor:genus} would imply $\genus(\Gamma_g\cap B_r(x_0))\geq g$, and so $\genus(\Gamma_g)\geq2g$ by dihedral symmetry.
Once the genus is uniformly bounded, White's result \cite[Theorem~1.1]{White2018} implies that the limit $\Gamma_{\infty}\cap B_r(x_0)$ is smooth. 
Furthermore, given this conclusion, we note that $\Lambda\cap\xi_0$ has to be finite because any point $p\in\Lambda\cap\xi_0$ would force (thanks to the maximum principle and to the fact that $\Gamma_{\infty}$ is rotationally symmetric around the vertical axis) to contain a whole flat horizontal disc passing through $p$, and the number of those discs is bounded by virtue of the uniform area bound we are assuming. 

Lastly, we need to note that -- straight from the same smoothness conclusion and again the maximum principle -- the set $\Lambda$ cannot possible contain the north or south pole.
\end{proof}

\emph{Claim 3.} 
$\Lambda\cap\B^2=\vcentcolon c$ is a circle of positive radius and $\Gamma_{\infty}$ is not just a multiple of $\B^2$. Moreover,  $\Gamma_{\infty}\setminus \B^2$ has at least one pair of isometric connected components.
Each connected component of  $\Gamma_{\infty}\setminus \B^2$ is a rotationally symmetric, minimal annulus meeting $\partial\B^3$ orthogonally along one of its boundary components and with $c$ as the other boundary component.  

\begin{proof}[Proof of Claim 3]
Towards a contradiction, suppose $\Lambda\cap \B^2\subset\{0\}$ (possibly including the case $\Lambda\cap \B^2=\emptyset$). 
Then Claim 2 implies that $\Lambda$ is a discrete set contained in the vertical axis $\xi_0$ and that $\Gamma_{\infty}\setminus\{0\}$ is a smooth, embedded minimal surface. 
By \cite{Gulliver1976} (see also \cite[Proposition 1]{ChoSch85}) the singularity at the origin is removable and $\Gamma_{\infty}$ is in fact a smooth, embedded free boundary minimal surface in $\B^3$. 
In particular, $\Gamma_{\infty}$ is connected and since it contains $\B^2$ it must be a multiple of $\B^2$. 
The dihedral symmetry implies that the multiplicity $m$ must be odd. 
In fact $m=1$ by the upper area bound in \ref{prop:convergence_large_g-iii} recalling that $\area(\B^2\cup\K_{\mathrm{crit}})<3\area(\B^2)$. At that stage, by now standard arguments (see e.\,g. Claim 4 in the proof of Theorem 2.3 in \cite{Sharp}) ensure that the convergence of $\Gamma_g$ to $\Gamma_\infty$ must be smooth and graphical (with multiplicity one) at all points, which in particular implies that the surface $\Gamma_g$ must be a topological disc at least for $g$ sufficiently large; 
then Nitsche's theorem \cite{Nitsche1985} and the dihedral symmetry assumption imply $\Gamma_g=\B^2$ for sufficiently large $g$ (cf. \cite[Proposition 2.1]{KetoverFB}). 
This however would contradict the lower area bound in \ref{prop:convergence_large_g-iii} and establishes the first part of Claim 3. 

Since $\Gamma_{\infty}\setminus \B^2$ is nonempty, the dihedral symmetry implies that it has at least two connected components, one in the upper and one in the lower half ball. 
Let $Q$ be the closure of any connected component of $\Gamma_{\infty}\setminus \B^2$. 
Then, $Q$ is a minimal surface which is properly embedded in a half ball and therefore (e.\,g. by the maximum principle) must intersect $\partial\B^3\setminus \B^2$. 
Moreover, $Q$ meets $\partial\B^3\setminus \B^2$ orthogonally because the convergence $\Gamma_g\to\Gamma_\infty$ is smooth away from $\Lambda$ and $\Lambda$ is disjoint from $\partial\B^3\setminus \B^2$.  
Consequently, $Q$ also intersects $\B^2$ because a free boundary minimal surface in $\B^3$ can not be contained in a half ball due to the Frankel property \cite[Lemma 2.4]{FraLi14}.
The intersection $Q\cap\B^2$ must coincide with $c$ again thanks to the smooth convergence away from $\Lambda$.
\end{proof}

\emph{Claim 4.} 
The radius of the circle $c=\Lambda\cap\B^2$ is strictly smaller than $1$. 

\begin{proof}[Proof of Claim 4]
Suppose that $c$ coincides with the equator $\B^2\cap\partial\B^3$. 
By Claim 3, $\Gamma_{\infty}\setminus \B^2$ has a pair $\Theta,\Theta'$ of isometric connected components. 
By Corollary \ref{cor:area_K} we have $\area(\Theta)=\area(\Theta')>\pi$. 
This implies $\area(\Gamma_{\infty})\geq\area(\B^2\cup\Theta\cup\Theta')>3\pi$ which contradicts the fact that $\area(\Gamma_{g})<3\pi$ for all $g$ by assumption~\ref{prop:convergence_large_g-iii}. 
\end{proof}

\emph{Conclusion.}
Assumption \ref{prop:convergence_large_g-ii} and Claim 1 imply $1\leq\genus(\Gamma_g)\leq g$ provided that $g$ is sufficiently large. 
Since $0\in\Gamma_g$ by assumption \ref{prop:convergence_large_g-i}, Corollary \ref{cor:genus} then yields $\genus(\Gamma_g)=g$ as claimed. 
Moreover, Lemma \ref{lem:genus} implies that once all boundary components of $\Gamma_g$ are closed up $\cyc_{g+1}$-equivariantly by topological discs, the resulting surface intersects the vertical axis $\xi_0$ exactly four times. 
This means that at most three boundary components of $\Gamma_g$ can wind around the vertical axis $\xi_0$ and by \ref{prop:convergence_large_g-i} at most two of them can be disjoint from $\B^2$. 
Appealing to Claim 3, the dihedral symmetry
and the fact that the convergence $\Gamma_g\to\Gamma_\infty$ is smooth away from $\Lambda$ and $\Lambda$ is disjoint from $\partial\B^3$ (which also builds on Claim 4), we obtain that $\Gamma_\infty\setminus\B^2$ consists of exactly two minimal annuli. As a result, $\Gamma_g$ has exactly three boundary components if $g$ is sufficiently large. 

It remains to determine the exact shape of $\Gamma_\infty$. 
By Claim 4, the singular set $c$ has a positive distance from the sphere $\partial\B^3$ and hence from the boundary of $\Gamma_\infty$. 
Therefore, the blow-up of $\Gamma_\infty$ around $x_0\in c$ is a stationary varifold $W$ in $\R^3$ without boundary. 
Since $\Gamma_\infty$ is rotationally symmetric, the support of $W$ is of the form $X\times\R$ for some $X\subset\R^2$.
Since $\Gamma_\infty\setminus\B^2$ has exactly two components, the profile $X$ consists of exactly $4$ rays emerging from the origin $x_0$; two of them correspond to the horizontal disc $\B^2\subset\Gamma_\infty$ and hence form a straight line $\xi$. 
Stationarity implies that the configuration must be balanced, i.\,e. the union of the remaining two rays must again form a straight line $\zeta$. 
Since $X$ is symmetric with respect to $\xi$, the intersection of $\xi$ with $\zeta$ must be orthogonal. 
We conclude that $\Gamma_\infty$ is the union of the horizontal disc $\B^2$ with a smooth, rotationally symmetric free boundary minimal surface $S$ which intersects $\B^2$ orthogonally. 
According to \cite{FraSch11}, such a surface $S$ coincides with the critical catenoid $\K_{\text{crit}}$. 
Recalling that the choice of our initial subsequence was arbitrary, the convergence $\Gamma_g\to\B^2\cup\K_{\text{crit}}$ as $g\to\infty$ follows. 
\end{proof}

\begin{remark}[Behavior for \emph{low} genus]
If $g\in\N^{\ast}$ is sufficiently large, then Proposition \ref{prop:convergence_large_g} provides full control on the topology of the free boundary minimal surface $\Sigma_g^{\mathrm{Ket}}$ which have been constructed in \cite{KetoverFB} via equivariant min-max methods. 
For small values of $g\in\N$ this approach still yields the existence of certain $\dih_{g+1}$-equivariant free boundary minimal surfaces $\Sigma_g^{\mathrm{Ket}}$ but without any control on their boundary connectivity. 
While numerical simulations (cf. Conjecture \ref{conj:genusg-bc3-KL} and Figure \ref{fig:genus2}, left image) confirm that $\Sigma_g^{\mathrm{Ket}}$ does have exactly three boundary components for all integers $g\geq2$, we lack any evidence for this statement in the case $g=1$. 
It is conceivable that the min-max construction of $\Sigma_1^{\mathrm{Ket}}$ loses the upper and the lower boundary component of the corresponding model surfaces, so that $\partial\Sigma_1^{\mathrm{Ket}}$ is in fact connected. 
In this sense, the conclusion of Proposition \ref{prop:convergence_large_g} may -- loosely speaking -- be regarded as sharp. 
\end{remark}

%===== GLOSSARY =====================================

\clearpage
\printglossaries 
\clearpage 

%===== BIBLIOGRAPHY =====================================

\setlength{\parskip}{1ex plus 1pt minus 1pt}
\bibliography{fbms-bibtex}

%% Addresses and affiliations
\printaddress

\end{document}

%% file: glossaryentries.tex
\newglossaryentry{ball}{
name={\ensuremath{\B^3}, \ensuremath{\Sp^2}, \ensuremath{\B^2}, \ensuremath{\Sp^1}},
description={closed unit ball and sphere, equatorial disc and circle},
user1={\ensuremath{\B^3}},
user2={\ensuremath{\Sp^2}},
user3={\ensuremath{\B^2}},
user4={\ensuremath{\Sp^1}},
sort={A a},
}

\newglossaryentry{axis}{
name={\ensuremath{\axis{x},\axis{y},\axis{z}}},
description={positively directed coordinate axis},
sort={A b},
}

\newglossaryentry{tube}{
name={\ensuremath{S_{>s}}, \ensuremath{S_{=s}}, ...},
description={(complements and boundaries of) tubular neighborhoods},
user1={\ensuremath{S_{\sim s}}}, 
sort={A c},
}

\newglossaryentry{refl}{
name={\ensuremath{\refl_V}, \ensuremath{\rot_{\axis{\ell}}^t}, \ensuremath{\trans^{\axis{\ell}}_t}},
description={reflections, rotations, translations},
user1={\ensuremath{\refl_V}},
user2={\ensuremath{\rot_{\axis{\ell}}^t}},
user3={\ensuremath{\trans^{\axis{\ell}}_t}},
sort={A d},
}

\newglossaryentry{Aut}{
name={\ensuremath{\Aut_M(S)}},
description={group of self congruences of $S$ in $M$}, 
sort={A e},
}

\newglossaryentry{sym}{
name={\ensuremath{\apr_m}, \ensuremath{\pri_m}, \ensuremath{\pyr_m}},
description={(anti)prismatic and pyramidal symmetry groups},
user1={\ensuremath{\apr_m}},
user2={\ensuremath{\pri_m}},
user3={\ensuremath{\pyr_m}},
sort={A f},
}

\newglossaryentry{nu}{
name={\ensuremath{\nu_\Sigma}, \ensuremath{A_\Sigma}, \ensuremath{H_\Sigma}},
description={unit normal, second fundamental form, mean curvature},
user1={\ensuremath{\nu_\Sigma}},
user2={\ensuremath{A_\Sigma}},
user3={\ensuremath{H_\Sigma}},
sort={B a},
}

\newglossaryentry{Jacobi}{
name={\ensuremath{J_\Sigma}},
description={Jacobi operator}, 
sort={B b},
}

\newglossaryentry{conormal}{
name={\ensuremath{\eta_\Sigma}, \ensuremath{B^{\mathrm{Robin}}_\Sigma}},
description={outward unit conormal, Robin boundary operator},
user1={\ensuremath{\eta_\Sigma}},
user2={\ensuremath{B^{\mathrm{Robin}}_\Sigma}}, 
sort={B c},
}

\newglossaryentry{indexform}{
name={\ensuremath{Q_\Sigma}, \ensuremath{\morseindex(\Sigma)}},
description={Jacobi (or index) form, Morse index of $\Sigma$}, 
user1={\ensuremath{Q_\Sigma}},
user2={\ensuremath{\morseindex(\Sigma)}}, 
sort={B d},
}

\newglossaryentry{norm}{
name={\ensuremath{\nm{{}\cdot{}}_{k,\alpha,\beta}}},
description={weighted H\"{o}lder norm (on the cylinder, on the tower, on $\Sigma_{m,\xi}$)},
user1={\ensuremath{\nm{u}_{k,\alpha,\beta}}}, 
sort={B f},
}

\newglossaryentry{symsign}{
name={\ensuremath{\sgn_\Sigma(\mathsf{M})}},
description={sign of self congruence $\mathsf{M}$ of two-sided hypersurface $\Sigma$}, 
sort={B g},
}

\newglossaryentry{CG}{
name={\ensuremath{C^{k,\alpha}_G(\Sigma)}},
description={H\"{o}lder space of $G$-equivariant functions},
sort={B h},
}

\newglossaryentry{cutoff}{
name={\ensuremath{\cutoff{a}{b}}},
description={cutoff function},
sort={B i},
}

\newglossaryentry{Kb}{
name={\ensuremath{\K_b}},
description={one-parameter family of catenoidal annuli in $\B^3$},
sort={C a},
}

\newglossaryentry{Kcrit}{
name={\ensuremath{\K_{\mathrm{crit}}}},
description={critical catenoid},
sort={C ab},
}

\newglossaryentry{omegab}{
name={\ensuremath{\omega_b}},
description={intersection angle between $\K_b$ and plane $\{z=b\}$},
sort={C b},
}

\newglossaryentry{tow}{
name={\ensuremath{\tow_\vartheta}, \ensuremath{\tow}, \ensuremath{\tow^+}},
description={Karcher--Scherk towers and half tower},
user1={\ensuremath{\tow_\vartheta}},
user2={\ensuremath{\tow}},
user3={\ensuremath{\tow^+}},
sort={D a},
}

\newglossaryentry{btow}{
name={\ensuremath{b^{\mathrm{tow}}_\vartheta}},
description={$y$-intercept of an asymptotic plane of $\tow$},
sort={D b},
}

\newglossaryentry{varpi}{
name={\ensuremath{\varpi_{(n)}}, \ensuremath{\varpi}},
description={canonical projections from $\R^3$ onto quotients by translations}, 
user1={\ensuremath{\varpi_{(n)}}},
user2={\ensuremath{\varpi}},
sort={D c},
}

\newglossaryentry{towquot}{
name={\ensuremath{\towquot_{(n)}}, \ensuremath{\towquot_{(n)}^+}, \ensuremath{\towquot}, \ensuremath{\towquot^+}},
description={quotients of towers and half towers by translations},
user1={\ensuremath{\towquot_{(n)}}},
user2={\ensuremath{\towquot_{(n)}^+}},
user3={\ensuremath{\towquot}},
user4={\ensuremath{\towquot^+}},
sort={D d},
}

\newglossaryentry{piAut}{
name={\ensuremath{\pi_{\Aut(\towquot)}}, \ensuremath{\pi_{\Aut(\towquot^+)}}},
description={projections onto subspace of equivariant functions on $\towquot$, $\towquot^+$},
user1={\ensuremath{\pi_{\Aut(\towquot)}}},
user2={\ensuremath{\pi_{\Aut(\towquot^+)}}}, 
sort={D e},
}

\newglossaryentry{towbent}{
name={\ensuremath{\towbent_{m,\xi}^+}},
description={towers with dislocated and straightened wings},
sort={D f},
}

\newglossaryentry{towcoker}{
name={\ensuremath{\towcokergen}, \ensuremath{\towcoker}},
description={generator of dislocations and induced mean curvature
   on towers},
user1={\ensuremath{\towcokergen}},
user2={\ensuremath{\towcoker}},  
sort={D g},
}

\newglossaryentry{m}{
name={\ensuremath{m}, \ensuremath{n}},
description={large prescribed integer data for the two constructions},
user1={\ensuremath{m}},
user2={\ensuremath{n}}, 
sort={E a},
}

\newglossaryentry{xi}{
name={\ensuremath{\xi}},
description={real parameter for the constructions, to be solved for in terms of $m$ or $n$}, 
sort={E b},
}

\newglossaryentry{Phi}{
name={\ensuremath{\Phi}},
description={spherical-coordinate covering map of a neighborhood of $\Sp^1$}, 
sort={E c},
}

\newglossaryentry{bmxi}{
name={\ensuremath{b_{m,\xi}}},
description={height parameter for $\K_b$ used in construction with data $(m,\xi)$}, 
sort={E d},
}

\newglossaryentry{Sigmamxi}{
name={\ensuremath{\Sigma_{m,\xi}}},
description={initial surface with genus $m-1$ and $3$ boundary components}, 
sort={E e},
}

\newglossaryentry{Xinxi}{
name={\ensuremath{\Xi_{n,\xi}}},
description={initial surface with genus $0$ and $n+2$ boundary components}, 
sort={E f},
}

\newglossaryentry{region}{
name={\ensuremath{\discr_{m,\xi}}, \ensuremath{\catr_{m,\xi}}, \ensuremath{\towr_{m,\xi}}},
description={disc, catenoid, and tower regions on $\Sigma_{m,\xi}$},
user1={\ensuremath{\discr_{m,\xi}}},
user2={\ensuremath{\catr_{m,\xi}}}, 
user3={\ensuremath{\towr_{m,\xi}}}, 
sort={E g},
}

\newglossaryentry{regionalprojections}{
name={\ensuremath{\varpi_{\discr_{m,\xi}},\varpi_{\catr_{m,\xi}},\varpi_{\towr_{m,\xi}}}},
description={diffeomorphism of regions onto subsets of models $\B^2,\K_{b_{m,\xi}},\tow^+$},
user1={\ensuremath{\varpi_{\discr_{m,\xi}}}},
user2={\ensuremath{\varpi_{\catr_{m,\xi}}}}, 
user3={\ensuremath{\varpi_{\towr_{m,\xi}}}}, 
sort={E h},
}

\newglossaryentry{coker}{
name={\ensuremath{\coker}},
description={mean curvature induced by dislocation on $\Sigma_{m,\xi}$},
sort={E i},
}

\newglossaryentry{varsigma}{
name={\ensuremath{\varsigma_{m,\xi}}},
description={identification diffeomorphism $\Sigma_{m,0} \to \Sigma_{m,\xi}$},
sort={E ea},
}

\newglossaryentry{modelresol}{
name={\ensuremath{P^m_{\B^2}}, \ensuremath{P^m_{\K_b}}, \ensuremath{P_{\tow^+}}},
description={resolvent operators on models},
user1={\ensuremath{P^m_{\B^2}}},
user2={\ensuremath{P^m_{\K_b}}}, 
user3={\ensuremath{P_{\tow^+}}}, 
sort={F ab},
}

\newglossaryentry{initialresol}{
name={\ensuremath{P_{\Sigma_{m,\xi}}}, \ensuremath{P_{\Xi_{n,\xi}}}},
description={resolvent operators on initial surfaces},
user1={\ensuremath{P_{\Sigma_{m,\xi}}}},
user2={\ensuremath{P_{\Xi_{n,\xi}}}},  
sort={F ac},
}

\newglossaryentry{geuc}{
name={\ensuremath{g_{\mathrm{euc}}}},
description={Euclidean metric on $\R^3$},
sort={F ax},
} 

\newglossaryentry{auxmet}{
name={\ensuremath{h=\Omega^4 g_{\mathrm{euc}}}},
description={the auxiliary metric, a conformal metric on $\R^3$},
user1={\ensuremath{h}},
sort={F b},
}

\newglossaryentry{iota}{
name={\ensuremath{\iota_{m,\xi,u}}},
description={graphical deformation of the inclusion $\Sigma_{m,\xi} \hookrightarrow \R^3$},
sort={F c},
}

\newglossaryentry{deformedMC}{
name={\ensuremath{H_{m,\xi,u}}, \ensuremath{Q_{m,\xi,u}}},
description={mean curvature of \gls{iota} and its nonlinear part},
user1={\ensuremath{H_{m,\xi,u}}},
user2={\ensuremath{Q_{m,\xi,u}}}, 
sort={F d},
}